\documentclass[reqno]{amsart}
\usepackage[margin = 1in]{geometry}
\usepackage{amsmath, amssymb, amsthm, fancyhdr, verbatim, graphicx, relsize, mathtools, xfrac}
\usepackage{enumerate}
\usepackage{enumitem} 
\usepackage[all]{xy}
\usepackage[dvipsnames]{xcolor}
\usepackage{mathrsfs}
\usepackage{tikz}
\usetikzlibrary{arrows.meta, positioning, calc, patterns, decorations.pathmorphing, decorations.pathreplacing}
\usepackage{tikz-cd}
\usetikzlibrary{arrows.meta,backgrounds,calc}
\definecolor{CFink}{HTML}{334958}
\definecolor{CFone}{HTML}{238589}
\definecolor{CFtwo}{HTML}{B17E46}
\definecolor{CFthree}{HTML}{75629D}
\definecolor{CFprimeone}{HTML}{7655AD}
\definecolor{CFprimetwo}{HTML}{3E80BA}
\definecolor{CFprimethree}{HTML}{409D73}
\definecolor{CFprimefour}{HTML}{C2AF39}
\definecolor{CFprimefive}{HTML}{D88638}
\definecolor{CFprimesix}{HTML}{C95655}
\definecolor{CFwire}{HTML}{CAD2D8}
\tikzset{
  cf dot/.style={circle,inner sep=0pt,minimum size=4.2pt,
    outer sep=1pt,draw=none,fill=#1},
  cf group/.style={font=\normalsize,text=CFink,fill=none,inner sep=2pt},
  cf axis/.style={draw=CFink,line width=.65pt,
    -{Stealth[length=4.5pt,width=3.2pt]}},
  cf plane/.style={draw=#1!32,line width=.45pt,fill=#1,fill opacity=.022},
  cf edge/.style={line width=.85pt},
  cf route/.style={line width=1.35pt},
  cf change/.style={draw=CFink!80,line width=1.05pt,
    dash pattern=on 3pt off 2.5pt},
  cf guide/.style={line width=.4pt,densely dotted},
}
\usetikzlibrary{calc}
\usepackage[T1]{fontenc} 
\usepackage{framed}
\usepackage[hypertexnames=false]{hyperref}
\usepackage[OT2, T1]{fontenc}  
\usepackage[titletoc]{appendix}
\usepackage{bbm}
\usepackage{adjustbox}
\usepackage{amssymb} 
\usepackage{float}

\numberwithin{equation}{subsection}

\DeclareSymbolFont{cyrletters}{OT2}{wncyr}{m}{n}
\DeclareMathSymbol{\Sha}{\mathalpha}{cyrletters}{"58}

\usepackage{color}

\newcommand{\F}{\mathbf{F}}
\newcommand{\RR}{\mathbf{R}}

\newcommand{\G}{\mathbf{G}}
\newcommand{\tr}[0]{\operatorname{tr}}
\newcommand{\wt}[1]{\widetilde{#1}}

\newcommand{\Q}{\mathbf{Q}}
\newcommand{\Z}{\mathbf{Z}}

\newcommand{\mf}[1]{\mathfrak{#1}}

\newcommand{\sgn}{\operatorname{sgn}}
\newcommand{\Gal}{\operatorname{Gal}}

\newcommand{\ol}[1]{\overline{#1}}
\newcommand{\wh}[1]{\widehat{#1}}

\newcommand{\co}{\colon}

\newcommand{\ld}{{}^L}

\newcommand{\surj}{\twoheadrightarrow}

\newcommand\cA{\mathcal{A}}
\newcommand\cB{\mathcal{B}}

\newcommand\cE{\mathcal{E}}
\newcommand\cF{\mathcal{F}}
\newcommand\cG{\mathcal{G}}
\newcommand\cH{\mathcal{H}}

\newcommand\cL{\mathcal{L}}
\newcommand\cM{\mathcal{M}}

\newcommand\cO{\mathcal{O}}

\newcommand\cS{\mathcal{S}}
\newcommand\cT{\mathcal{T}}

\newcommand\frG{\mathfrak{G}}

\newcommand\frg{\mathfrak{g}}
\newcommand\frh{\mathfrak{h}}

\newcommand\frp{\mathfrak{p}}

\newcommand\frt{\mathfrak{t}}

\DeclareMathOperator{\GL}{GL}
\DeclareMathOperator{\SL}{SL}

\DeclareMathOperator{\ab}{ab}

\DeclareMathOperator{\Tr}{Tr}

\DeclareMathOperator{\PGL}{PGL}
\DeclareMathOperator{\Sp}{Sp}
\DeclareMathOperator{\SO}{SO}
\DeclareMathOperator{\Hom}{Hom}

\DeclareMathOperator{\ord}{ord}
\DeclareMathOperator{\Aut}{Aut}

\DeclareMathOperator{\Nm}{Nm}

\DeclareMathOperator{\Lie}{Lie}

\DeclareMathOperator{\ad}{ad}

\DeclareMathOperator{\unr}{unr}
\DeclareMathOperator{\Res}{Res}
\DeclareMathOperator{\Frac}{Frac}

\DeclareMathOperator{\ind}{ind}

\DeclareMathOperator{\FS}{FS}
\DeclareMathOperator{\cInd}{c-Ind}

\DeclareMathOperator{\red}{red}

\DeclareMathOperator{\Kal}{Kal}

\DeclareMathOperator{\der}{der}

\DeclareMathOperator{\rec}{rec}
\DeclareMathOperator{\rk}{rk}

\DeclareMathOperator{\Fr}{Fr}

\RequirePackage{xspace}

\newcommand{\rH}{\ensuremath{\mathrm{H}}\xspace}

\newcommand{\rT}{\ensuremath{\mathrm{T}}\xspace}

\newcommand{\bF}{\mathbf{F}}

\renewcommand{\ss}{\mathrm{ss}}

\newtheorem{thm}{Theorem}[subsection]
\newtheorem{lemma}[thm]{Lemma}
\newtheorem{prop}[thm]{Proposition}

\theoremstyle{remark}
\newtheorem{remark}[thm]{Remark} 
 
\newtheorem{defn}[thm]{Definition}
\newtheorem{ansatz}[thm]{Ansatz}

\newtheorem{example}[thm]{Example}

\makeatletter
\def\th@remark{%
  \thm@headfont{\bfseries}%
  \normalfont 
  \thm@preskip \thm@preskip 
  \thm@postskip\thm@preskip
}
\def\imod#1{\allowbreak\mkern5mu({\operator@font mod}\,\,#1)}
\makeatother

\numberwithin{equation}{subsection}
\numberwithin{figure}{subsection}

\title[Local Langlands functoriality for Yu's supercuspidals I]{Local Langlands functoriality for Yu's supercuspidals I: Kaletha's parametrization}

\author{Sean Cotner and Tony Feng}

\begin{document}

\begin{abstract}
This is the first of two papers dedicated to the explicit computation of the Fargues--Scholze correspondence. In this first part, we construct (inspired by, and extending, Kaletha's explicit Local Langlands Correspondence for non-singular supercuspidal representations) semisimple inertial L-parameters for all cuspidal representations arising from Yu's constructions, with arbitrary coefficients. We then establish a characterization of this correspondence in terms of certain functoriality properties. This characterization will be used in the sequel paper in order to compute the Fargues--Scholze correspondence after restriction to inertia.
\end{abstract}

\maketitle

\tableofcontents


\section{Introduction}

This paper is the first in a two-part series which computes the Fargues--Scholze correspondence \cite{FS} on inertia for supercuspidal representations arising from Yu's construction. 

\subsection{Informal preview of results}\label{ssec:II-informal-preview}
Let $p$ be a prime and let $G$ be a reductive group over a nonarchimedean local field $F$ of residue characteristic $p$. 

\subsubsection{The Local Langlands Correspondence} The Local Langlands Correspondence (LLC) predicts a relationship between the smooth irreducible representations of $G(F)$ and the Galois representations (more generally, $L$-parameters) into the L-group $\ld G$. Several rather distinct approaches to the LLC have now emerged, at least in characteristic $0$. In this paper, we focus on the following ones. 
\begin{enumerate}
    \item The ``classical'' approach stemming from the seminal work of Laumon--Rapoport--Stuhler \cite{LRS93}, Harris--Taylor \cite{HT01}, and Henniart \cite{Hen00} on the Local Langlands Correspondence for $\GL_n$. This has been extended to certain other families of groups, as we shall review in \S \ref{ssec:related-work}. 
    \item The ``explicit'' approach developed by Kaletha \cite{Kal19}, \cite{Kal21b}. This has earlier roots in the depth 0 correspondence of DeBacker--Reeder \cite{DR09}, and draws inspiration from the character formulas of Adler--Spice and DeBacker--Spice \cite{AS09,DS18}. 
    \item The ``geometric'' approach of Fargues--Scholze \cite{FS}, which is based on adapting the Geometric Langlands program and Lafforgue's work \cite{Laff18} to the Fargues--Fontaine curve. Other geometric approaches due to Genestier--Lafforgue \cite{GL18} and Zhu \cite{Zhu25} will also be reviewed in \S \ref{ssec:related-work}.
\end{enumerate}

Each of these approaches has its own advantages and disadvantages. The classical approach furnishes a complete Local Langlands Correspondence, but only for specific families of groups, and it is difficult to make it explicit. Kaletha's correspondence is explicit and applies uniformly to quite general groups, but only for specific families of representations, and it carries technical assumptions (e.g., that $p$ not be too small relative to $G$). The Fargues--Scholze correspondence is completely general and uniform in $G$, and it is part of a broader categorical framework, but it is difficult to compute explicitly, and its fundamental expected properties (such as surjectivity and finiteness of fibers) remain open in general. Therefore, it is of interest and utility to reconcile these approaches. In this paper and its sequel \cite{CF26b}, we will compare (2) and 
(3).

\subsubsection{Our results} 
One of the main results \cite[Theorem~10.1.1]{CF26b} of this two-part series \emph{computes the Fargues--Scholze correspondence on inertia for cuspidal representations arising from Yu's construction}. Moreover, for non-singular cuspidal representations, we calculate the restriction of the Fargues--Scholze parameter to an explicit finite index subgroup. The answer is compatible with Kaletha's explicit Local Langlands parametrization (when the latter is applicable). 

In particular, when $G$ splits after a tamely ramified extension and $p$ does not divide the order of the absolute Weyl group of $G$, work of Fintzen \cite{Fin21}, \cite{Fin22}, building on work of Kim \cite{Kim07}, shows that Yu's construction \cite{Yu01} exhausts all cuspidal representations, so our results determine the Fargues--Scholze parameters on inertia for all irreducible smooth representations of $G(F)$. In this case, we are also able to deduce important qualitative consequences for the Fargues--Scholze correspondence, including: 
\begin{itemize}
    \item \cite[Theorem~10.1.1(2)]{CF26b} A characterization of the representations with supercuspidal L-parameter.
    \item \cite[Theorem~10.2.1]{CF26b} Depth preservation, and finiteness of the fibers. 
    \item \cite[Theorem~10.3.1]{CF26b} Surjectivity onto inertial $L$-parameters.
\end{itemize}

There has been considerable recent progress towards some of these results in \cite{DL26}, \cite{Ete23}, \cite{Fu26}, \cite{GHS24}, and \cite{BPHT25} under various hypotheses on $F$ and $G$, which we will recall in more detail in the introduction to \cite{CF26b}. One motivation for the ``soft'' results bulleted above is that they are among the inputs needed in the strategy of Hansen--Mann \cite{HM26} to prove the categorical local Langlands conjecture of \cite{FS}. A depth $0$ version of the categorical conjecture has already been established by other means by Zhu \cite{Zhu25} for his version of categorical local Langlands, and he has compared the resulting Local Langlands Correspondence to DeBacker--Reeder in \cite[\S 5.3.4]{Zhu25} when $G$ is unramified of adjoint type.

\begin{remark}
There has been much recent work on the compatibility of (3) with the semisimplification of (1). In \cite[Theorem I.9.6(ix)]{FS}, Fargues--Scholze proved compatibility for $G = \GL_n$. Most recently, Daniels--van Hoften--Kim--Zhang \cite[Theorem II]{DvHKZ26} (building on \cite{FS}, \cite{HKW22}, \cite{BMHN24}, \cite{Ham25}, \cite{Han26}, and \cite{Pen26}) establish this compatibility over arbitrary $p$-adic local fields for most inner forms of classical groups. This does not imply our results, even for the groups to which it applies, because the compatibility of (1) and (2) was only previously known in special cases \cite{OT21}, \cite{Tok23}, \cite{Oi26}. We note that our results imply inertial compatibility of (1) and (2) in the cases in which (1) and (3) are known to be equivalent.
\end{remark}

\subsubsection{Our methods}\label{ssec:II-intro-methods} 
Our methods are completely different from prior approaches. In particular, they are purely local and representation-theoretic, and uniform in the group. The starting point is the theory of \emph{modular functoriality} as developed in \cite{F24}, building on ideas of Treumann--Venkatesh \cite{TV}, which gives concrete constructions of functorial (in the sense of Langlands) descents of $\ol\F_\ell$-representations in certain very special situations.

The basic idea is to use modular functoriality to probe the Fargues--Scholze parameters modulo arbitrarily large prime numbers. In particular, although the applications described above are to representation theory with characteristic $0$ coefficients, \emph{the proofs make essential use of miracles occurring in positive characteristic}. The key advantage of positive characteristic is the (provocatively phrased) slogan that ``there are many more Jacquet functors in positive characteristic.'' 

More precisely, the theory of modular functoriality supplies ``exceptional'' functors from smooth characteristic $\ell$ representations of $G(F)$ to smooth characteristic $\ell$ representations to a subgroup $H(F)$, which ``behave like Jacquet functors''. In principle, this provides a mechanism to inductively bootstrap understanding of the Local Langlands correspondence. However, this mechanism is subtle to use for at least two reasons: 
\begin{enumerate}
\item we only obtain congruential information modulo a prime $\ell$, and 
\item it is not obvious which pairs $(H,G)$ can arise in modular functoriality.
\end{enumerate} 
For these reasons it is not immediately clear how much power can be leveraged from modular functoriality. 
In this paper and its sequel, we will demonstrate how to carve out of the web of modular functoriality a systematic strategy to compute the Fargues--Scholze correspondence on inertia (and more), for all supercuspidals arising from Yu's construction. Roughly speaking, the proof involves using modular functoriality and independence of $\ell$ to pass to an extension of $F$ of very large prime degree, after which $G(F)$ has a lot of torsion of large prime order, and then to use the same methods again to pass to an unramified twisted Levi $F$-subgroup of $G$ (see Figure \ref{fig:cf26a-inductive-path}).

\begin{figure}[!htbp]
  \centering
  \resizebox{.6\linewidth}{!}{%
\begin{tikzcd}[
  arrow style=tikz,
  arrows={-{Stealth[length=5pt,width=3.6pt]}},
  cells={nodes={inner sep=0pt,minimum size=0pt}},
  labels={font=\scriptsize,fill=white,inner sep=2pt},
  execute before arrows={
    \filldraw[cf plane=CFprimesix] (4.5,2.45) rectangle (13.5,8.70);
    \filldraw[cf plane=CFprimefive] (3.6,1.85) rectangle (12.6,8.10);
    \filldraw[cf plane=CFprimefour] (2.7,1.25) rectangle (11.7,7.50);
    \filldraw[cf plane=CFprimethree] (1.8,.65) rectangle (10.8,6.90);
    \filldraw[cf plane=CFprimetwo] (.9,.05) rectangle (9.9,6.30);
    \filldraw[cf plane=CFprimeone] (0,-.55) rectangle (9,5.70);
    \draw[CFwire,line width=.45pt] (0,-.55)--(4.5,2.45);
    \draw[CFwire,line width=.45pt] (0,5.70)--(4.5,8.70);
    \draw[CFwire,line width=.45pt] (9,5.70)--(13.5,8.70);
    \draw[cf guide,CFprimeone,opacity=.32] (8.10,.55)--(9,.55)--(9,-.55);
    \draw[cf guide,CFprimetwo,opacity=.32] (9,2.55)--(9.9,2.55)--(9.9,.05);
    \draw[cf guide,CFprimethree,opacity=.32] (7.35,3.15)--(10.8,3.15)--(10.8,.65);
    \draw[cf guide,CFprimefour,opacity=.32] (8.25,5.35)--(11.7,5.35)--(11.7,1.25);
    \draw[cf guide,CFprimefive,opacity=.32] (6.55,5.95)--(12.6,5.95)--(12.6,1.85);
    \draw[cf guide,CFprimesix,opacity=.32] (7.45,8.15)--(13.5,8.15)--(13.5,2.45);
%
    \node[cf dot=CFprimeone,minimum size=5pt] (cfA) at (8.10,.55) {};
    \node[cf dot=CFprimeone,minimum size=5pt] (cfB) at (8.10,1.95) {};
    \node[cf dot=CFprimetwo,minimum size=5pt] (cfJ) at (9,2.55) {};
    \node[cf dot=CFprimetwo,minimum size=5pt] (cfD) at (6.45,2.55) {};
    \node[cf dot=CFprimethree,minimum size=5pt] (cfC) at (7.35,3.15) {};
    \node[cf dot=CFprimethree,minimum size=5pt] (cfK) at (7.35,4.75) {};
    \node[cf dot=CFprimefour,minimum size=5pt] (cfE) at (8.25,5.35) {};
    \node[cf dot=CFprimefour,minimum size=5pt] (cfF) at (5.65,5.35) {};
    \node[cf dot=CFprimefive,minimum size=5pt] (cfG) at (6.55,5.95) {};
    \node[cf dot=CFprimefive,minimum size=5pt] (cfL) at (6.55,7.55) {};
    \node[cf dot=CFprimesix,minimum size=5pt] (cfH) at (7.45,8.15) {};
    \node[cf dot=CFprimesix,minimum size=5.5pt] (cfI) at (4.95,8.15) {};
    \node[cf group,anchor=north] at (8.10,.48) {$G_F$};
    \node[cf group,anchor=south east] at (8.00,2.03) {$G_{F'}$};
    \node[cf group,anchor=north east] at (6.38,2.48) {$H_{F'}$};
    \node[cf group,anchor=south west] at (8.12,5.44) {$H_{F''}$};
    \node[cf group,anchor=north east] at (5.58,5.28) {$M_{F''}$};
    \node[cf group,anchor=south west] at (7.52,8.22) {$M_{F'''}$};
    \node[cf group,anchor=north west] at (5.02,8.00) {$T_{F'''}$};
%
    \draw[cf axis] (-.38,-.05)--(-.38,6.10);
    \node[text=CFink,anchor=south] at (-.38,6.22) {field};
    \foreach \h/\f in {.55/F,1.95/F',3.55/F'',5.15/F'''}{
      \draw[CFink,line width=.5pt] (-.44,\h)--(-.32,\h);
      \node[text=CFink,font=\small,anchor=east] at (-.52,\h) {$\f$};
    }
    \draw[cf axis] (0,-.94)--(9.25,-.94);
    \foreach \s/\t in {.45/0,2.95/s_2,5.55/s_1,8.10/s}{
      \draw[CFink,line width=.5pt] (\s,-1.00)--(\s,-.88);
      \node[text=CFink,font=\small,anchor=north] at (\s,-1.06) {$\t$};
    }
    \node[text=CFink,anchor=north] at (4.45,-1.55) {semisimple rank};
    \draw[cf axis] (8.65,-.78333)--(14.10,2.85);
    \foreach \x/\y/\c/\i in {9/-.55/CFprimeone/1,9.9/.05/CFprimetwo/2,10.8/.65/CFprimethree/3,11.7/1.25/CFprimefour/4,12.6/1.85/CFprimefive/5,13.5/2.45/CFprimesix/6}{
      \node[cf dot=\c,minimum size=5pt] at (\x,\y) {};
      \node[text=\c,font=\small,anchor=north west,inner sep=3pt]
        at (\x,\y) {$(\ell_{\i})$};
    }
    \node[text=CFink,anchor=west] at (14.18,2.88)
      {$\operatorname{Spec}\mathbb Z$};
  }
]
  \arrow[from=cfA,to=cfB,draw=CFprimeone,cf route]
  \arrow[from=cfB,to=cfJ,cf change]
  \arrow[from=cfJ,to=cfD,draw=CFprimetwo,cf route]
  \arrow[from=cfD,to=cfC,cf change]
  \arrow[from=cfC,to=cfK,draw=CFprimethree,cf route]
  \arrow[from=cfK,to=cfE,cf change]
  \arrow[from=cfE,to=cfF,draw=CFprimefour,cf route]
  \arrow[from=cfF,to=cfG,cf change]
  \arrow[from=cfG,to=cfL,draw=CFprimefive,cf route]
  \arrow[from=cfL,to=cfH,cf change]
  \arrow[from=cfH,to=cfI,draw=CFprimesix,cf route]
\end{tikzcd}}
  \caption{Cartoon of the strategy to compute the Local Langlands correspondence. We identify specific paths within the web of modular functoriality that facilitate an induction on semisimple rank, down to the case of tori. We organize modular functoriality along three axes: the base field (vertical axis), the semisimple rank (horizontal axis), and the coefficient characteristic. Vertical edges represent congruences along cyclic base change. Dashed edges represent invocations of ``independence of $\ell$''. Horizontal edges represent congruences to a twisted Levi. The extensions $F'/F$, $F''/F'$, and $F'''/F''$ are unramified of degrees $\ell_1$, $\ell_3$, and $\ell_5$, respectively.}
  \label{fig:cf26a-inductive-path}
\end{figure}
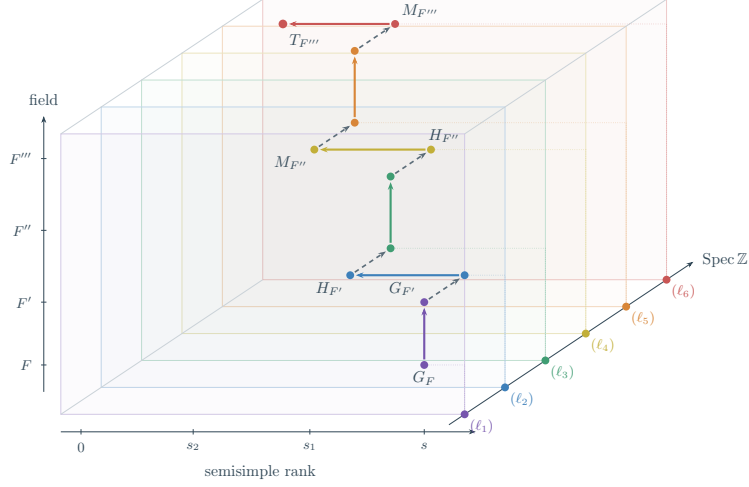

In fact, the above strategy only applies on the nose to relatively ``unramified'' settings, and one more step (again involving modular functoriality and independence of $\ell$) is needed to reduce to this setting.

\subsection{The role of this paper} 

Let us highlight some of the main outcomes of this paper. First, as noted above, our strategy requires passage through positive characteristic coefficients, even for the ultimate characteristic 0 applications. Moreover, to obtain finiteness of L-packets, we need to (partially) compute the Fargues--Scholze L-parameters of arbitrary cuspidal representations. Therefore, we define (following the construction of Kaletha for non-singular supercuspidal representations with characteristic $0$ coefficients) explicit semisimplified \emph{inertial} L-parameters for cuspidal representations with arbitrary coefficients.\footnote{A technical point is that our definitions are only made on the level of Yu data defining a cuspidal representation. There is more than one Yu datum associated to a given representation, and we do not know that our definitions are independent of this choice because we have not proven the analogue of \cite[Theorem 6.6]{HM08}, \cite[Corollary 3.5.5]{Kal19} with positive characteristic coefficients. For the applications, this distinction will be irrelevant, so we will elide this in the remainder of the introduction.} For ``non-singular'' cuspidal representations, we are able to define full $L$-parameters, which coincide with those constructed by Kaletha.

One of the difficulties in comparing different approaches to the LLC is that, beyond the case of $\GL_n$, there is no known list of properties characterizing it. One of our main results gives a partial characterization of the explicit (inertial) L-parameters above, using known properties of $\rho^{\FS}$ and two forms of functoriality described below. This will be described in more detail shortly in \S \ref{ssec:II-intro-parameter}. Both of these forms of functoriality will be verified in \cite{CF26b} for the Fargues--Scholze correspondence.

Given a cuspidal representation $\pi$, we define ``descents'' of $\pi$ to cuspidal representations of unramified twisted Levi subgroups $H$ of $G$, under some hypotheses on $H$. These descents should be thought of as adjoint to the ``functorial transfer'' corresponding to a natural $L$-homomorphism $\ld j_{H,G} \co \ld H \rightarrow \ld G$. Similarly, for a cyclic $\ell$-extension $E/F$ we define (under some hypotheses) a cuspidal representation $\pi_E$ of $G(E)$, which should be regarded as the ``base change lifting'' in the sense of Langlands. A major subtlety in both cases is that the sign characters of \cite{FKS23} show up in nontrivial ways in these constructions. This will be sketched further in \S \ref{ssec:II-intro-functoriality}.

\subsection{Kaletha's parametrization}\label{ssec:II-intro-parameter} Throughout the rest of this introduction, we assume that $G$ splits after a tamely ramified extension of $F$ and that $p$ does not divide the order of the absolute Weyl group of $G$. For an algebraically closed field $k$ of characteristic $\neq p$, let $\Pi_k(G)_{\mathrm{cusp}}$ denote the set of irreducible cuspidal smooth $k$-representations of $G(F)$ up to isomorphism, and let $\Phi_k^{\ss}(G)$ denote the set of semisimple L-parameters $W_F \to \ld G(k)$ up to $\wh G(k)$-conjugation.

In \cite{Kal21b}, when $k$ is of characteristic $0$, Kaletha singled out a class $\Pi_k(G)_{\mathrm{cusp,ns}}$ of irreducible supercuspidal smooth $k$-representations which he called \emph{non-singular}, and he defined a map of sets
\[
\rho^{\Kal}\co \Pi_k(G)_{\mathrm{cusp,ns}} \to \Phi_k^{\ss}(G)
\]
with the property that $\rho^{\Kal}(\pi)$ is always a \emph{supercuspidal} L-parameter, i.e., there is no proper parabolic subgroup of $\wh G$ normalized by $\rho^{\Kal}(\pi)$. Roughly, the recipe is as follows: from the representation $\pi$, one can extract a ($G(F)$-conjugacy class of) pair(s) $(T, \theta)$ consisting of an elliptic maximal $F$-torus $T \subset G$ and a character $\theta\co T(F) \to k^\times$. From this, one defines
\begin{equation}\label{eqn:I-intro-kal-defn}
\rho^{\Kal}(\pi) = \left(W_F \xrightarrow{\ld\theta} \ld T(k) \xrightarrow{\ld j_{T,G}} \ld G(k)\right),
\end{equation}
where $\ld\theta\co W_F \to \ld T(k)$ is the L-parameter arising from the LLC for tori and $\ld j_{T,G} \co \ld T \to \ld G$ is a certain L-embedding. One of the innovations of \cite{Kal21b}, developed from \cite{FKS23}, is the definition of $\ld j_{T,G}$, which is highly non-canonical.\footnote{In particular, despite our notation, $\ld j_{T,G}$ depends on $\pi$ and not just $T$!} Supporting evidence for this definition of $\rho^{\Kal}(\pi)$ comes from the Harish--Chandra character formulas and from stability and the endoscopic character identities proved for certain endoscopic elements in \cite[\S 4.3, Theorem 4.4.4]{FKS23} at least under additional technical hypotheses on $F$. 


The first goal of this paper is to extend Kaletha's definition in two directions, as we explain in \S\ref{sec:kaletha-param}: first, one can allow $k$ to be of positive characteristic with few changes. Second, for arbitrary cuspidal $\pi$, one can still extract a pair $(T, \theta)$ and attempt to use \eqref{eqn:I-intro-kal-defn} to define a semisimple L-parameter. However, if $\pi$ is not non-singular, then this cannot be the ``correct'' semisimple L-parameter in general because it has finite image mod center, while the correct L-parameter should not (in characteristic $0$); moreover, the pair $(T, \theta)$ is not unique, so this definition depends on a choice. We will show that both objections disappear for the \emph{inertial} (semisimple) L-parameter 
\[
\rho_I^{\Kal}(\pi) = \left(I_F \xrightarrow{\ld\theta} \ld T(k) \xrightarrow{\ld j_{T,G}} \ld G(k) \right).
\]
Similar restricted L-parameters were proposed in \cite{CDT25} and \cite{DF26}; the former proposes a candidate for the restriction of the L-parameter to the ``greatest nontrivial depth'', while the latter proposes a candidate for the restriction to all of wild inertia. We propose $\rho_I^{\Kal}(\pi)$ as a construction of the ``true'' semisimple inertial L-parameter associated to $\pi$. The following theorem, which will be stated far more precisely in Theorem~\ref{thm:llc-partial-characterization}, gives some justification for believing this.

\begin{thm}[Informal; see Theorem~\ref{thm:llc-partial-characterization}]\label{thm:I-intro-llc-partial-characterization}
    Let $k$ be an algebraically closed field of characteristic $\neq p$. Suppose that for every finite tamely ramified extension $E/F$ and every twisted Levi $E$-subgroup $H \subset G_E$, we are given a map of sets
    \[
    \rho = \rho_{E,H}\co \Pi_k(H)_{\mathrm{cusp}} \to \Phi_k^{\ss}(H)
    \]
    satisfying the following conditions:
    \begin{enumerate}
        \item Agreement with class field theory in the case of tori.
        \item Compatibility with formation of central characters.
        \item Compatibility with homomorphisms which are isomorphisms up to center.
        \item A weak form of the existence of base change liftings along certain cyclic extensions of prime degree $\neq p$ for a certain class of cuspidal representations.
        \item Existence of descent to unramified twisted Levi subgroups.
    \end{enumerate}
    Then
    \begin{enumerate}[label=(\Alph*)]
        \item $\rho(\pi)|_{I_F} \sim \rho_I^{\Kal}(\pi)$ for all $\pi \in \Pi_k(G)_{\mathrm{cusp}}$.
        \item If $\pi$ is non-singular, then there is an explicit integer $N$ (depending on $\pi$) such that
        \begin{equation}\label{eqn:I-intro-restriction-to-finite-index}
        \rho(\pi)|_{W_{F_N}} \sim \rho^{\Kal}(\pi)|_{W_{F_N}}.
        \end{equation}
        Moreover, $\rho(\pi)$ is irreducible if and only if $\rho^{\Kal}(\pi)$ is irreducible.
    \end{enumerate}
\end{thm}

\begin{remark}The integer $N$ in \eqref{eqn:I-intro-restriction-to-finite-index} depends on the torus $T$ associated to $\pi$. If $T$ is maximally unramified (i.e., $T$ contains a maximal unramified $F$-torus), then one can take $N = 1$. In Example~\ref{ex:llc-classical}, we study the case that $T \cap G_{\der}$ is totally ramified\footnote{Using assumption (5) of Theorem~\ref{thm:llc-partial-characterization}, one can reduce to this case by replacing $G$ by an unramified twisted Levi. However, the bounds become more complicated because these twisted Levis are products (up to center) of several smaller rank groups, and the bound one obtains is roughly the least common multiple of the bounds associated to such factors.}, and we show that if $G$ is an unramified form of $\GL_r$ then one can take $N = r$, while if $G$ is an unramified form of $\SO_{2r+1}$, $\Sp_{2r}$, or $\SO_{2r}$ then one can take $N \leq 2^{\lceil\log_2(2r)\rceil}$ to be a power of $2$. In Example~\ref{ex:llc-exceptional}, we give a brief indication of how to establish concrete bounds on $N$ when $G$ is an unramified form of an exceptional group, if $\rho$ is compatible with Weil restriction and products; in all cases, one can show $N \leq 36$. We will show in \cite[Theorem~10.4.3]{CF26b} that if $G_{F^{\unr}} \cong \SL_r$, then one can take $N = 1$.
\end{remark}

\begin{remark}\label{rem:modular-functoriality-black-box-II}
In \cite{CF26b}, we will apply Theorem~\ref{thm:I-intro-llc-partial-characterization} to $\rho=\rho^{\FS}$. Properties (1), (2), and (3) follow from \cite[Theorem I.9.6(i), (iii), (v)]{FS}. Properties (4) and (5) are established in \cite{CF26b} using modular functoriality \cite[Theorem~9.1.1 and Example~9.1.2]{F24}, which applies at every prime $\ell\neq p$. Theorem~\ref{thm:llc-partial-characterization} gives a sharper statement with more flexible hypotheses and explicit bounds.
\end{remark}

The very rough idea of the proof of Theorem~\ref{thm:I-intro-llc-partial-characterization} is to use (5) to reduce to the case that $T\cap G_{\der}$ is totally ramified, then use (4) (which applies after this reduction) to pass to a field extension $E/F$ such that $T_E$ is elliptic and maximally unramified. The a priori surprising feature of Theorem~\ref{thm:I-intro-llc-partial-characterization} is that assumption (4) will only concern \emph{wild inertia}, so this naive argument only determines $\rho(\pi)|_{P_F}$ a priori. However, once one passes to the case that $T\cap G_{\der}$ is totally ramified, the situation becomes extremely rigid, to the point that the L-parameter $\rho(\pi)$ is almost completely determined by its restriction to $P_F$.

\subsection{Functoriality}\label{ssec:II-intro-functoriality}

To make hypotheses (4) and (5) in Theorem~\ref{thm:I-intro-llc-partial-characterization} more precise, we must specify the sense in which they hold for $\rho^{\Kal}$ and $\rho_I^{\Kal}$. We start with (5).

\subsubsection{Unramified twisted Levis}
Recall that a closed $F$-subgroup $H \subset G$ is an \emph{unramified twisted Levi} if $H$ becomes a Levi factor of a parabolic subgroup of $G$ after passing to an unramified extension of $F$. For such a subgroup $H$, there is a canonical L-embedding $\ld j_{H,G}\co \ld H \to \ld G$, whose definition we recall in Definition~\ref{defn:canonical-l-embeddings}.

\begin{thm}[Theorem~\ref{thm:kaletha-functoriality}]\label{thm:I-intro-kaletha-unramified-twisted-levi}
    Let $\pi$ be an irreducible cuspidal $k$-representation of $G(F)$, let $(T, \theta)$ be a torus-character pair associated to $\pi$ as above, and let $H \subset G$ be an unramified twisted Levi $F$-subgroup containing $T$. Then there is an explicit finite set of pairs $\{(L, \pi_L)\}$, where $L$ is a Levi $F$-subgroup of $H$ and $\pi_L$ is a cuspidal irreducible smooth $k$-representation of $L(F)$ such that
    \[
    \rho_I^{\Kal}(\pi) \sim \ld j_{L,G} \circ \rho_I^{\Kal}(\pi_L).
    \]
    If $\pi$ is non-singular, then $L = H$ for every such pair, $\pi_H$ is non-singular, and
    \[
    \rho^{\Kal}(\pi) \sim \ld j_{H,G} \circ \rho^{\Kal}(\pi_H).
    \]
\end{thm}

The existence of \emph{some} pair $(L, \pi_L)$ as in Theorem~\ref{thm:I-intro-kaletha-unramified-twisted-levi} has limited content, at least when $G$ is quasi-split, since \cite{Kal21b} already shows that every irreducible L-parameter is in the image of $\rho^{\Kal}$; the content of the theorem is that $\pi_L$ is constructed explicitly. The fact that this is subtle can be seen from Remark~\ref{remark:nontrivial-sign-char}, which shows that if $\pi$ is non-singular and $(T, \theta)$ is the torus-character pair associated to $\pi$, then although the pair $(T_H, \theta_H)$ associated to $\pi_H$ satisfies $T_H = T$, the character $\theta_H$ need not be $N_G(T)(F)$-conjugate to $\theta$. What is true is that $\theta_H$ is conjugate to $\theta\epsilon$ for some sign character $\epsilon\co T(F) \to k^\times$, but $\epsilon$ is generally nontrivial and depends on $\pi$; we will discuss this in \S\ref{ss:I-intro-proofs}. Beyond dealing with sign issues, the main content in the construction of Theorem~\ref{thm:I-intro-kaletha-unramified-twisted-levi} is to carefully track the semi-rational Lusztig series associated to the special fibers of Bruhat--Tits group schemes, as introduced in \cite[\S 2.8]{Cot26b}.

Assumption (5) in Theorem~\ref{thm:I-intro-llc-partial-characterization} is a slight variant of the statement that Theorem~\ref{thm:I-intro-kaletha-unramified-twisted-levi} holds with $\rho$ (resp.\ $\rho|_{I_F}$) in place of $\rho^{\Kal}$ (resp.\ $\rho_I^{\Kal}$) for a \emph{single} pair $(L, \pi_L)$.

\subsubsection{Cyclic field extensions}

Let $E/F$ be a cyclic extension of degree $\ell \neq p$. If $\pi$ is an irreducible cuspidal $k$-representation of $G(F)$, then we do not have a general recipe for producing a representation $\pi_E$ of $G(E)$ such that $\rho^{\Kal}(\pi_E) \sim \rho^{\Kal}(\pi)|_{W_E}$. Even for depth $0$ representations, this appears to be complicated: see, for example, the treatment of ramified $\mathrm{U}(2,1)$ in \cite[\S 7]{AL10}. However, there are two situations in which matters simplify considerably.

\begin{thm}[Proposition~\ref{prop:kal-base-change}]\label{thm:I-intro-large-bc}
    Let $\pi$ be an irreducible cuspidal $k$-representation of $G(F)$, let $\ell \neq p$ be a prime, and let $E/F$ be the unramified extension of degree $\ell$. If $\ell$ is sufficiently large, then there is an explicit irreducible cuspidal $k$-representation $\pi_\ell$ of $G(E)$ such that
    \[
    \rho_I^{\Kal}(\pi_\ell) \sim \rho_I^{\Kal}(\pi).
    \]
    If $\pi$ is non-singular, then $\pi_\ell$ is non-singular and
    \[
    \rho^{\Kal}(\pi_\ell) \sim \rho^{\Kal}(\pi)|_{W_E}.
    \]
\end{thm}

Once again, the significance of Theorem~\ref{thm:I-intro-large-bc} is that the representation $\pi_\ell$ is explicit. The construction of $\pi_\ell$ in Theorem~\ref{thm:I-intro-large-bc} involves the Glauberman correspondence \cite{Gla68}, which, roughly speaking, provides base change lifts for finite reductive groups under field extensions of large prime degree.\footnote{It is likely true that the ``correct'' base change lifting is given by Shintani descent. As explained in \cite[Remark~4.3.1]{Cot26b}, Shintani descent and the Glauberman correspondence differ by a twist, which when reduced modulo $\ell$ corresponds to the $\ell$-Frobenius endomorphism. Thus our construction also involves a certain twist.} Plugging this into Yu's construction leads to the existence of $\pi_\ell$. The sign issues described above are relatively minor in this case by Lemma~\ref{lemma:fks-base-change} and Lemma~\ref{lemma:tasho-base-change}.

Next, consider the case that $\pi$ is \emph{toral}, i.e., that the pair $(T, \theta)$ associated to $\pi$ is toral in the sense of \cite[Definition 3.7]{CO21}.\footnote{The word \emph{toral} has two competing definitions in the literature; our definition is the less restrictive one.} Any toral irreducible cuspidal $k$-representation is non-singular.

\begin{thm}[Proposition~\ref{prop:kal-small-degree-bc}]\label{thm:I-intro-small-bc}
    Let $\pi$ be a toral irreducible cuspidal $k$-representation of $G(F)$, let $\ell \neq p$ be a prime, let $(T,\theta)$ be a torus-character pair associated to $\pi$, and let $E/F$ be a cyclic extension of prime degree $\ell\neq p$ such that $T_E$ is elliptic. There is an explicit irreducible cuspidal $k$-representation $\pi_E$ of $G(E)$ such that if either $\ell$ is odd, or $\ell = 2$ and $k = \ol\F_2$, we have
    \[
    \rho^{\Kal}(\pi_E) \sim \rho^{\Kal}(\pi)|_{W_E}.
    \]
\end{thm}

The same sign issues described above for unramified twisted Levis arise in the setting of Theorem~\ref{thm:I-intro-small-bc}, particularly when $E/F$ is ramified. If $\ell = 2$, then one can still define a candidate base change lifting $\pi_E$ for arbitrary $k$, but the sign issues are even more severe; this is the reason for the restriction $k = \ol\F_2$.

Assumption (4) of Theorem~\ref{thm:I-intro-llc-partial-characterization} is the statement that, in the setting of Theorem~\ref{thm:I-intro-small-bc}, one has
\[
\rho(\pi_E)|_{P_F} \sim \rho(\pi)|_{P_F}.
\]
We emphasize that this is a fairly weak hypothesis: it only concerns toral representations, and it only requires functoriality after restriction to wild inertia. Although Theorem~\ref{thm:I-intro-small-bc} shows that $\pi_E$ is an actual base change lifting of $\pi$ from $F$ to $E$, adding this as a further assumption to Theorem~\ref{thm:I-intro-llc-partial-characterization} does not significantly improve \eqref{eqn:I-intro-restriction-to-finite-index}. In any case, for the Fargues--Scholze correspondence we can only prove that $\pi_E$ is a base change lifting of $\pi$ in characteristic $\ell$.

\subsection{Sign characters}\label{ss:I-intro-proofs} One important theme of the explicit Local Langlands Correspondence is the necessity of certain subtle sign characters, as alluded to above and treated definitively in \cite{FKS23}. For simplicity, we will assume in the following discussion that $\pi$ is non-singular, but the issues are essentially the same in general. 

In the proof, we see sign characters arise in two independent ways. First, if $T \subset H \subset G$ are as in Theorem~\ref{thm:I-intro-kaletha-unramified-twisted-levi}, then there is a naive way to construct a candidate $\pi_H$ using Yu's construction \cite{Yu01}, which we temporarily call $\pi_H'$, such that $(T, \theta)$ is a torus-character pair associated to $\pi_H'$. From $\pi_H'$, one obtains an L-embedding $\ld j_{T,H}\co \ld T \to \ld H$ as above. However, in general,
\begin{equation}\label{eqn:I-intro-l-embedding-not-transitive}
\ld j_{T,G} \circ \ld\theta \not\sim \ld j_{H,G} \circ \ld j_{T,H} \circ \ld\theta,
\end{equation}
so $\pi_H'$ cannot be the correct functorial descent.\footnote{Note that this does not contradict \cite[Proposition 6.9]{Kal21a}, which looks similar to the negation of this statement but whose notation implicitly refers to different L-embeddings than those appearing in \eqref{eqn:I-intro-l-embedding-not-transitive}.} This discrepancy gives rise to one sign character. 

The second source of sign characters is (as explained in \cite{FKS23}) that Yu's construction should itself be twisted by a sign character which depends on the data involved in the definition. This suggests that the naive construction of $\pi_H'$ should also be twisted by a sign character; finding the ``correct'' twist, however, requires a careful analysis of \cite{FKS23}. The precise sign characters that intervene appear slightly ad hoc; they will be partially explained by the Tate cohomology of the Weil representation in \cite{CF26b}, where they will ultimately arise from the sign characters in \cite{Ger77}.

\subsection{Context and related work}\label{ssec:related-work}
In this subsection we discuss further context on the different approaches to the Local Langlands correspondence, and related work on their compatibility.

\subsubsection{The classical approach}
After the Local Langlands Correspondence was established for $\GL_n$ \cite{LRS93, HT01,Hen00}, it was extended to more general groups in work of many people.

Over $p$-adic fields (in characteristic 0), twisted endoscopy was used to extend the correspondence to quasi-split symplectic and special orthogonal groups by Arthur \cite{Art13}, to quasi-split unitary groups by Mok \cite{Mok15}, and to their inner forms of unitary groups by Kaletha--M\'{i}nguez--Shin--White \cite{KMSW14}. M\oe glin--Renard \cite[\S 3.6]{MoeRen18} establish the tempered correspondence for pure inner forms of quasi-split special orthogonal and unitary groups. Applying the Langlands classification to these tempered correspondences, including those for Levi subgroups, gives the correspondence for all irreducible representations.\footnote{For even special orthogonal groups, these endoscopic classifications retain the usual ambiguity under the outer automorphism induced by the full orthogonal group.} 

The theta correspondence gives another construction: for even orthogonal groups, this was used by Chen--Zou \cite{CZ21}, and one can deduce a correspondence for special orthogonal groups up to outer automorphism. Gan--Takeda and Gan--Tantono \cite{GT11,GT14} also deduce the correspondence for $\mathrm{GSp}_4$ and its non-split inner form; Gan--Takeda and Choiy treat $\mathrm{Sp}_4$ and its non-split inner form \cite{GT10,Cho17}. For $\mathrm{G}_2$, Gan--Savin \cite{GS23} use an exceptional theta correspondence, while Aubert--Xu \cite[Theorem 10.2.1]{AX22} give an explicit construction when $p\ne 2,3$.

Other classical constructions include Hiraga--Saito's work on inner forms of $\SL_n$ \cite{HiSa12}, completed over arbitrary nonarchimedean local fields by Aubert--Baum--Plymen--Solleveld \cite{ABPS16}. They use restriction from inner forms of $\GL_n$, whose correspondence follows from the Jacquet--Langlands correspondence. For split odd special orthogonal groups, Jiang--Soudry \cite{JS03} established the correspondence for generic supercuspidal representations using local converse theorems and descent.

Over local fields of positive characteristic, the method of close fields has been used to construct a Local Langlands correspondence in some situations: Ganapathy \cite{Gan15} treats $\mathrm{GSp}_4$ when $p>2$, Ganapathy--Varma \cite{GV17} treat split classical groups under bounds on $p$, and \cite{ABPS16} treats inner forms of $\SL_n$ without such bounds. Gan--Lomel\'{i} \cite[\S 7]{GanLom18} give a different construction of parameter maps for quasi-split classical groups using globalization and Lafforgue's work, under some hypotheses. 

In a different direction, for coefficients of characteristic $\ell\ne p$, Vign\'{e}ras \cite{Vig01} constructs the semisimple correspondence for $\GL_n(F)$ over an arbitrary nonarchimedean local field $F$. Her proof uses congruences between automorphic representations following Khare \cite{Kha01}. 

\subsubsection{The geometric approach}

The ``geometric'' approach to the Local Langlands Correspondence itself divides into several different approaches.

The earliest is that of of Genestier--Lafforgue \cite{GL18}, which constructs a semisimplified Local Langlands correspondence over function fields. It uses the global Langlands correspondence of V. Lafforgue \cite{Laff18} plus local-global compatibility. Later, Li-Huerta \cite{LH23} has shown that the Genestier--Lafforgue and Fargues--Scholze correspondences agree in this case.

Zhu \cite{Zhu25} has developed his own approach to the categorical Local Langlands Conjecture. While this does not yet have a ``spectral action'' (hence no construction of L-parameters attached to higher depth representations), Zhu is able to construct a fully faithful embedding of depth zero blocks into the spectral category. The recent work \cite{GHILZ26} compares the relevant categories of sheaves in the approaches of Fargues--Scholze and Zhu. For $G = \GL_n$, this had been done earlier by Ben-Zvi--Chen--Helm--Nadler \cite{BCHN24}. 

Building on Zhu's work, Helm--Solleveld--Xu \cite[Theorem~1.2]{HSX26} construct fully faithful functors from Bernstein blocks with generic supercuspidal support to ind-coherent sheaves on parameter stacks for a class of quasi-split groups over $p$-adic fields that includes symplectic, special orthogonal, and unitary groups, and $G_2$. This in turn has applications to the Fargues--Scholze correspondence through the work of Hansen--Mann \cite{HM26}.

\subsection{Outline of the paper}

We now give some details on the various sections appearing in this work.

In \S\ref{ssec:notation}, we detail the various conventions on notations and definitions which we will use.

In \S\ref{sec:sign-characters}, we recall the sign characters defined in \cite{FKS23} which will be used throughout. In \S\ref{ss:chi-data-L-embedding}, we recall the definition of sets of $\chi$-data and their associated L-embeddings from \cite{LS87} (mainly following \cite{Kal21a}), and we define the particular L-embeddings which will be of use to us. In \S\ref{sec:sign-l-embedding-compatibility}, we establish certain compatibilities between the sign characters of \S\ref{sec:sign-characters} and the L-embeddings of \S\ref{ss:chi-data-L-embedding} under base change and descent to unramified twisted Levi subgroups, and in particular we show how to modify \eqref{eqn:I-intro-l-embedding-not-transitive} to a true statement. Since several of the proofs involve rather long and complicated calculations, they are mostly relegated to Appendix~\ref{app:rootwise-sign-chi-calculations}.

In \S\ref{ss:yu-construction}, we recall the twisted Yu construction from \cite{FKS23}, which involves modifying the classical Yu construction \cite{Yu01} with a certain sign character. With an eye towards applications to functoriality, we carefully analyze the behavior of this construction under base change and passage to unramified twisted Levis.

In \S\ref{sec:kaletha-param}, we describe our extension of Kaletha's parametrization in detail, and we prove some basic compatibility results for it. In \S\ref{sec:bc-functoriality-yu-data}, we describe the explicit base change liftings required for Theorems~\ref{thm:I-intro-large-bc} and \ref{thm:I-intro-small-bc}, and we prove those theorems. The proofs mainly involve applying the results of \S\ref{sec:sign-l-embedding-compatibility}. In \S\ref{ss:kaletha-functoriality}, we describe the explicit descent to unramified twisted Levi subgroups used in Theorem~\ref{thm:I-intro-kaletha-unramified-twisted-levi}, and we prove that theorem. In \S\ref{ss:sign-example}, we give an explicit example showing that the sign issues can be quite subtle.

Finally, in \S\ref{sec:partial-characterization}, we prove Theorem~\ref{thm:I-intro-llc-partial-characterization}, establishing a partial characterization of the LLC in terms of Theorems~\ref{thm:I-intro-kaletha-unramified-twisted-levi} and \ref{thm:I-intro-small-bc}.

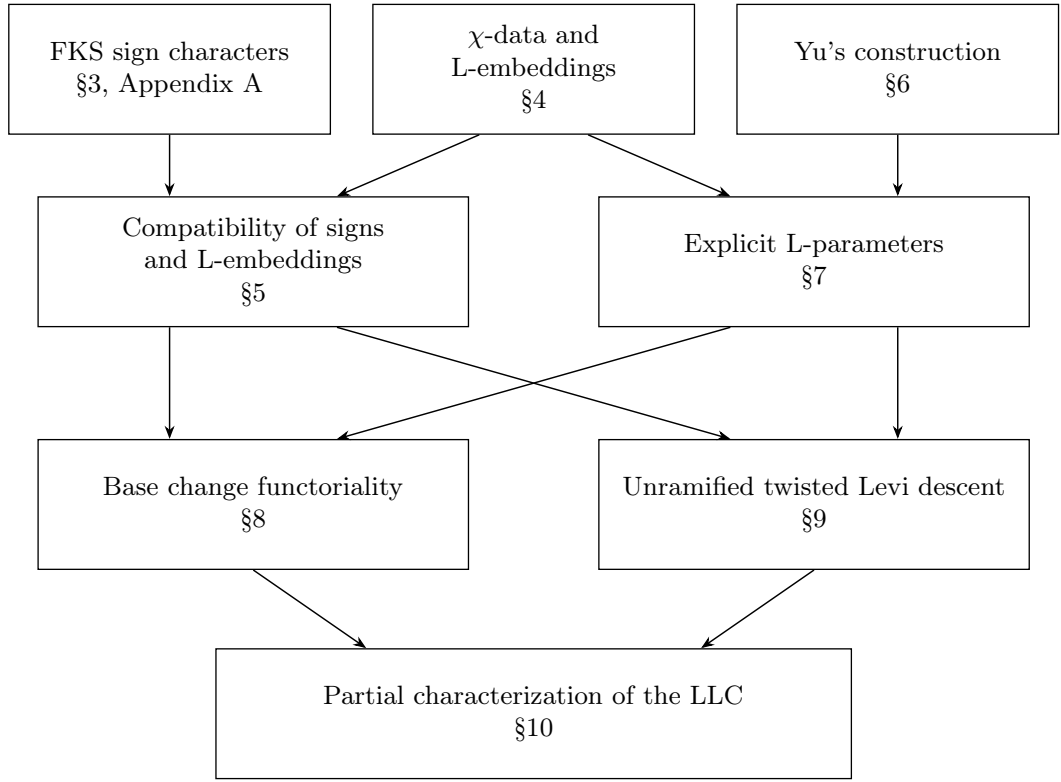
\begin{figure}[!htbp]
  \centering
  \hypersetup{hidelinks}
  \tikzset{
    lf box/.style={draw, line width=.45pt, rectangle, align=center,
      inner xsep=7pt, inner ysep=7pt, font=\small,
      text width=3.35cm, minimum height=1.55cm, fill=white},
    lf wide/.style={lf box, text width=4.65cm, minimum height=1.55cm},
    lf result/.style={lf box, text width=7.1cm},
    lf input/.style={lf box, dashed},
    lf arrow/.style={-{Stealth[length=4.5pt,width=3.4pt]}, line width=.5pt},
    lf external/.style={lf arrow, dashed},
    lf note/.style={font=\scriptsize, align=center, fill=white, inner sep=2pt}
  }
  \resizebox{.85\textwidth}{!}{%
    \begin{tikzpicture}[x=1cm,y=1cm]
      \node[lf box] (signs) at (-4.35,0)
        {FKS sign characters\\
         \S\ref{sec:sign-characters}, Appendix~\ref{app:rootwise-sign-chi-calculations}};
      \node[lf box] (chi) at (0,0)
        {$\chi$-data and\\L-embeddings\\
         \S\ref{ss:chi-data-L-embedding}};
      \node[lf box] (yu) at (4.35,0)
        {Yu's construction\\
         \S\ref{ss:yu-construction}};
    
      \node[lf wide] (compatibility) at (-3.35,-2.3)
        {Compatibility of signs\\and L-embeddings\\
         \S\ref{sec:sign-l-embedding-compatibility}};
      \node[lf wide] (parameters) at (3.35,-2.3)
        {Explicit L-parameters\\
         \S\ref{sec:kaletha-param}};
      \node[lf wide] (basechange) at (-3.35,-5.2)
        {Base change functoriality\\
         \S\ref{sec:bc-functoriality-yu-data}};
      \node[lf wide] (levi) at (3.35,-5.2)
        {Unramified twisted Levi descent\\
         \S\ref{ss:kaletha-functoriality}};
      \node[lf result] (characterization) at (0,-7.7)
        {Partial characterization of the LLC\\
         \S\ref{sec:partial-characterization}};
    
      \draw[lf arrow] (signs.south) -- ([xshift=-1cm]compatibility.north);
      \draw[lf arrow] ([xshift=-.65cm]chi.south) -- ([xshift=1cm]compatibility.north);
      \draw[lf arrow] ([xshift=.65cm]chi.south) -- ([xshift=-1cm]parameters.north);
      \draw[lf arrow] (yu.south) -- ([xshift=1cm]parameters.north);
      \draw[lf arrow] ([xshift=-1cm]compatibility.south) -- ([xshift=-1cm]basechange.north);
      \draw[lf arrow] ([xshift=1cm]compatibility.south) -- ([xshift=-1cm]levi.north);
      \draw[lf arrow] ([xshift=-1cm]parameters.south) -- ([xshift=1cm]basechange.north);
      \draw[lf arrow] ([xshift=1cm]parameters.south) -- ([xshift=1cm]levi.north);
      \draw[lf arrow] (basechange.south) -- ([xshift=-2cm]characterization.north);
      \draw[lf arrow] (levi.south) -- ([xshift=2cm]characterization.north);
    \end{tikzpicture}
  }
  \caption{Leitfaden of the paper.}
  \label{fig:cf26a-leitfaden-normal}
\end{figure}

\subsection{AI methodology} 

After an initial draft of this paper was written, AI tools (specifically GPT 5.6 and Claude Fable 5) were used for proofreading, revision, and the creation of all tables and figures. Finally, GPT 6 was used for literature search in order to give appropriate attributions in the introduction. AI tools were not used in any other way. 

\subsection{Acknowledgements}

We thank Jeff Adler, Ad\`ele Bourgeois, Charlotte Chan, Stephen DeBacker, Jessica Fintzen, Michael Harris, Alex Hazeltine, Alex Ivanov, Josh Lansky, Monica Nevins, Sian Nie, Peter Scholze, David Schwein, Jack Sempliner, Loren Spice, Jay Taylor, and Cheng-Chiang Tsai for helpful conversations. 

S.C.~acknowledges support from the National Science Foundation under Award No.\ 2402231 and the European Research Council (ERC) under the
European Union’s Horizon 2020 research and innovation programme (grant agreement no.\ 950326).

T.F.~was supported by the NSF (grants DMS-2302520 and DMS-2441922), the Simons Foundation, and the Alfred P. Sloan Foundation.

\section{Notation and conventions}\label{ssec:notation}

\subsection{Local fields}
Let $F$ be a local field with ring of integers $\cO_F$ and residue field $\F_q$ of characteristic $p > 0$. We fix a separable closure $\ol F/F$. A separable algebraic extension of $F$ is always understood as a subfield of $\ol F$. For any finite separable field extension $E/F$ and integer $n \geq 1$, we denote by $E_n$ the unramified subextension of $\ol F/E$ of degree $n$. We write $F^{\unr}\subset \ol F$ for the maximal unramified extension of $F$. Unless another field is specified, $\varpi$ denotes a uniformizer of $F$.

We write $W_F$ for the Weil group of $F$, $I_F \triangleleft W_F$ for the inertia subgroup, and $P_F \triangleleft W_F$ for the wild inertia subgroup. 

We write $\Fr$ for geometric Frobenius in $W_F/I_F$, and also for a chosen lift to $W_F$.

\subsection{Reductive groups}
We denote by $G$ a connected reductive group over $F$. (Our convention is that reductive groups over fields are not necessarily assumed to be connected.) All representations and characters of $G(F)$ are understood to be smooth unless otherwise stated. If $C \subset G$ is a closed subscheme, then we denote by $Z_G(C)$ (resp.\ $N_G(C)$) the scheme-theoretic centralizer (resp.\ normalizer) of $C$ in $G$.

If $H$ is a group scheme of finite type over a field $k$, then we use $H^\circ$ to denote the identity component of $H$. The notation $H^0$, which is sometimes used instead of $H^\circ$, will be reserved for Yu's construction. We write $\pi_0(H)=H/H^\circ$ for the group scheme of connected components and $H_{\red}$ for the underlying reduced subscheme. Similarly, if $\cO$ is a discrete valuation ring and $\cG$ is a smooth separated $\cO$-group scheme, then we will use $\cG^\circ$ to denote the (open) relative identity component of $\cG$ over $\cO$, a smooth finite type $\cO$-group scheme with connected fibers.

For a connected reductive group $H$, we write $Z(H)$ for its center, $H_{\der}$ for its derived subgroup, $H_{\ad}=H/Z(H)$ for its adjoint quotient, and $H_{\ab}=H/H_{\der}$ for its abelianization. We write $\rk H$ for its absolute rank. For semisimple $H$, $\pi_1(H)$ denotes its algebraic fundamental group. We use $\frg=\Lie(G)$ and $\Lie^*(G)$ for the linear dual of $\Lie(G)$.

If $T$ is an $F$-torus, then $T(F)_{\mathrm{b}}$ denotes the maximal bounded subgroup of $T(F)$.

We make compatible choices of square roots of $q$ in $\ol\Q_\ell$, $\ol\Z_\ell$, and $\ol\F_\ell$. We use $\ind$ for ordinary induction in finite or finite-index settings and $\cInd$ for compact induction.

\subsection{Dual groups}
The Langlands dual group $\wh{G}$ is regarded over $\Z[1/p]$, though we will often base change to a field of characteristic $\neq p$ without changing the notation. 

We write $\ld G=\wh G\rtimes W_F$ for the L-group of $G$, with the usual pinning-preserving action of $W_F$ on $\wh G$. If $\rho_1,\rho_2$ are (inertial) L-parameters valued in $\ld G(k)$, then $\rho_1\sim\rho_2$ means that $\rho_1$ and $\rho_2$ are $\wh G(k)$-conjugate. We reserve \emph{equality} of L-parameters for actual equality of chosen representatives by homomorphisms.

\subsection{Bruhat--Tits buildings}
We let $\cB(G)$ denote the extended Bruhat--Tits building of $G$ over $F$. For a point $x \in \cB(G)$, we let $[x] = [x]_G$ denote the corresponding point of $\cB(G_{\der})$. Let $\cG_x$ (resp.\ $\cG_{[x]}$) denote the smooth separated $\cO_F$-group scheme with generic fiber $G$ and $\cG_x(\cO_F) = G(F)_x$ (resp.\ $\cG_{[x]}(\cO_F) = G(F)_{[x]}$), the stabilizer of $x$ (resp.\ $[x]$) in $G(F)$, as in \cite[Remark 8.3.4]{KP}. Note that $\cG_x$ is of finite type, but $\cG_{[x]}$ might not be. Let $\ol G_x$ (resp.\ $\ol G_{[x]}$) denote the quotient of the special fiber of $\cG_x$ (resp.\ $\cG_{[x]}$) by the unipotent radical of its identity component.

For $r\geq 0$, the notations $G(F)_{x,r}$ and $G(F)_{x,r+}$ refer to the Moy--Prasad filtration at $x$; analogous notation is used for the Lie algebra. We implicitly use the normalized valuation $v$ of $F$ in this definition; however, if $E/F$ is a finite separable extension then $G(E)_{x,r}$ and $G(E)_{x,r+}$ will be defined using the unique extension of $v$ to $E$. Thus $G(F)_{x,r} = G(E)_{x,r} \cap G(F)$.

\subsection{Coefficients}\label{sssec:coefficient-conventions}
We use $k$ for coefficient rings, which need not be fields. Our coefficient prime $\ell$ is always distinct from $p$. We write $\ol\Z_\ell$ for the valuation ring of $\ol\Q_\ell$, with residue field $\ol\F_\ell$.

For a $k$-algebra $k'$ and a $k$-module or representation $V$, we write $V_{k'}=V\otimes_k k'$. Subscripts on group schemes and homomorphisms likewise indicate base change.

\section{The FKS sign character}\label{sec:sign-characters}\label{ss:fks}

There is a notorious sign error in \cite{Yu01}, noticed by Loren Spice and arising from a misprint in \cite[Theorem 2.4(b)]{Ger77}; see the introduction to \cite{Fin21b}. This has led to substantial ``sign adjustments'' in the literature, the culmination of which is \cite{FKS23}. The latter paper introduces a complicated sign character $\epsilon = \epsilon_G$ and shows (among other things) that if Yu's construction is adjusted by $\epsilon$ then the errant results in \cite{Yu01} become valid.

The character $\epsilon$ is the product of five individually complicated sign characters, called $\epsilon_{\sharp,x}$, $\epsilon_{\flat,0}$, $\epsilon_f$, $\epsilon_{\flat,1}$, and $\epsilon_{\flat,2}$ in \cite{FKS23}. The character $\epsilon_{\sharp,x}$ was introduced in \cite[\S 4.3]{DS18}; the character $\epsilon_{\flat,0}$ first appeared in \cite[Proposition 5.27]{Kal21a}; and the character $\epsilon_f$ is the quotient of two sign characters which appear in \cite[Definition 4.7.3]{Kal19}. 

Write $\F=\F_q$ for the residue field of $F$. Whenever a coefficient ring $k \in \{\ol\Q_\ell,\ol\Z_\ell,\ol\F_\ell\}$ is fixed in the following constructions, we also fix once and for all an additive character $\psi = \psi_k\co F \to k^\times$ which is trivial on the maximal ideal $\frp_F \subset \cO_F$ and whose restriction to $\cO_F$ is the inflation of a nontrivial character $\psi^0\co \F_q \to k^\times$ of the form $\F_q \xrightarrow{\Tr_{\F_q/\F_p}} \F_p \xrightarrow{\ol\psi} k^\times$.
The characters $\psi_{\ol\Q_\ell}$ and $\psi_{\ol\F_\ell}$ are chosen to be compatible with $\psi_{\ol\Z_\ell}$ by extension of scalars and reduction, respectively. 

Assume that $p \neq 2$ and that $G$ is a connected reductive $F$-group of adjoint type, and let $M$ be a tamely ramified twisted Levi $F$-subgroup of $G$. Fix an embedding $\cB(M) \subset \cB(G)$ of enlarged Bruhat--Tits buildings, and let $x \in \cB(M)$. Let $\cM_x$ be the smooth affine $\cO_F$-group scheme with $\cM_x(\cO_F)=M(F)_x$, and let $\ol M_x$ be the quotient of its special fiber by the unipotent radical of its identity component. If $T\subset G$ is a tamely ramified maximal torus, let $\ol T$ denote the analogous quotient of the special fiber of the finite-type N\'eron model of $T$. Then \cite[Theorem 3.4]{FKS23} exhibits a sign character
\[
\epsilon_x^{G/M}\co \ol M_x(\F)\longrightarrow\{\pm1\},
\]
whose definition we now recall.

\subsection{Preliminary notions}\label{ssec:FKS-preliminaries}
For an $F$-torus $T \subset M$, let $\Phi(G_{\ol F}, T_{\ol F})$ (resp.\ $\Phi((G/M)_{\ol F}, T_{\ol F})$) denote the set of nonzero weights (which we will call roots, even if $T$ is not maximal) for the action of $T_{\ol F}$ on $\mathfrak{g}_{\ol F}$ (resp.\ $(\mathfrak{g}/\mathfrak{m})_{\ol F}$). Let $\Gamma = \Gal(\ol F/F)$ and $\Sigma = \Gamma \times \{\pm 1\}$, so $\Gamma$ and $\Sigma$ act on $\Phi(G_{\ol F}, T_{\ol F})$ and $\Phi((G/M)_{\ol F}, T_{\ol F})$.

\begin{defn}\label{defn:symmetric-roots}
    A root $\alpha \in \Phi(G_{\ol F}, T_{\ol F})$ is said to be \textit{symmetric} if $-\alpha$ lies in the $\Gamma$-orbit of $\alpha$; otherwise $\alpha$ is said to be \textit{asymmetric}. For each absolute root $\alpha \in \Phi(G_{\ol F}, T_{\ol F})$, let $F_\alpha$ be the finite extension of $F$ such that $\Gal(\ol F/F_\alpha)$ is the stabilizer of $\alpha$ under the $\Gamma$-action, and let $F_{\pm\alpha}$ be the subextension of $F_\alpha$ such that $\Gal(\ol F/F_{\pm\alpha})$ is the (setwise) stabilizer of the set $\{\pm\alpha\}$ under the $\Gamma$-action.
    
    A symmetric root $\alpha$ is \textit{unramified} if the quadratic extension $F_\alpha/F_{\pm\alpha}$ is unramified; otherwise $\alpha$ is \textit{ramified}. 

Below we indicate the corresponding subsets of $\Phi((G/M)_{\ol F}, T_{\ol F})$ with self-explanatory subscripts, e.g., $\Phi((G/M)_{\ol F},T_{\ol F})_{\mathrm{sym,unram}}$ denotes the subset of symmetric unramified roots in $\Phi((G/M)_{\ol F},T_{\ol F})$.
\end{defn}

Let $\F_\alpha$ (resp.\ $\F_{\pm \alpha}$) denote the residue field of $F_\alpha$ (resp.\ $F_{\pm \alpha}$). If $\alpha$ is unramified, so that $\F_\alpha \neq \F_{\pm\alpha}$, then let $\F_\alpha^1$ denote the kernel of the norm map $\Nm_{\F_\alpha/\F_{\pm\alpha}}\co \F_\alpha^\times \to \F_{\pm\alpha}^\times$. Let $\sgn_{\F_\alpha^\times}$ (resp.\ $\sgn_{\F_\alpha^1}$) denote the unique nontrivial character $\F_\alpha^\times \to \{\pm 1\}$ (resp.\ $\F_\alpha^1 \to \{\pm 1\}$). Let $\ord_x(\alpha) = \ord_x(\alpha)_F$ denote the set of real numbers $t$ such that $\mathfrak{g}_\alpha(F_\alpha)_{x, t+} \subsetneq \mathfrak{g}_\alpha(F_\alpha)_{x, t}$.

If $M \subset G$ is a twisted Levi $F$-subgroup, then we will write $M_{\mathrm{sc}}$ to denote the preimage of $M \cap G_{\der}$ in the universal cover $G_{\mathrm{sc}} \to G_{\der}$.

\subsection{Definition of the sign character}\label{sssec:fks-sign-character-definition}
Fix a \emph{good} element\footnote{Here we follow the definitions in \cite{FKS23}, which differ from those of \cite{Yu01}; see \cite[Remark 4.1.3]{FKS23} for a discussion of the difference.} $X \in \Lie^*(M_{\mathrm{sc},\ab})(F) \subset \Lie^*(M_{\mathrm{sc}})(F)$, i.e., an element such that there exists $r \in \mathbf{R}$ such that
\[
\ord(\langle X, H_\alpha\rangle) = -r \quad \text{ for all $\alpha \in \Phi((G/M)_{\ol F}, S_{\ol F})$},
\]
where $S$ ranges over all tamely ramified maximal $F$-tori in $M$ and $H_\alpha = \mathrm{d}\alpha^\vee(1)$.\footnote{Generic elements (the definition of which we will recall later) are good, so such elements will arise in the setting of Yu's construction.} For a tamely ramified maximal $F$-torus $T \subset M$ such that $x\in\cB(T)$, the character $\epsilon_x^{G/M}|_{\ol T(\F)}$ is characterized as a product
\begin{equation}\label{eqn:def-of-epsilon}
\epsilon_x^{G/M}|_{\ol T(\F)} = \epsilon_{\sharp,x}^{G/M} \cdot \epsilon_{\flat,0}^{G/M} \cdot \epsilon_f^{G/M} \cdot \epsilon_{\flat,1}^{G/M} \cdot \epsilon_{\flat,2}^{G/M}
\end{equation}
where each term on the right hand side is a character of $\ol T(\F)$. Remarkably, \cite[Theorem 3.4 and Corollary 3.6]{FKS23} shows that there is a unique character $\epsilon_x^{G/M}\co \ol M_x(\F) \to \{\pm 1\}$ whose restriction to $\ol T(\F)$ is given by \eqref{eqn:def-of-epsilon} for every tamely ramified maximal $F$-torus $T \subset M$ with $x \in \cB(T)$.

We now define the sign characters appearing in \eqref{eqn:def-of-epsilon}. Fix a tamely ramified maximal $F$-torus $T \subset M$ such that $x\in\cB(T)$, and fix $\gamma \in \ol T(\F)$. Define first
\begin{equation}\label{eqn:FKS-1}
\epsilon_{\sharp,x}^{G/M}(\gamma) = \prod_{\substack{\alpha \in \Phi((G/M)_{\ol F}, T_{\ol F})_{\mathrm{asym}}/\Sigma \\ r/2 \in \ord_x(\alpha)}} \sgn_{\F_\alpha^\times}(\alpha(\gamma)) \cdot \prod_{\substack{\alpha \in \Phi((G/M)_{\ol F},T_{\ol F})_{\mathrm{sym,unram}}/\Gamma \\ r/2 \in \ord_x(\alpha)}} \sgn_{\F_\alpha^1}(\alpha(\gamma)).
\end{equation}

Let $Z_M$ denote the maximal central torus of $M$, and for $\alpha \in \Phi((G/M)_{\ol F}, T_{\ol F})$ let $\alpha_M = \alpha|_{Z_M}$. Let $e(\alpha/\alpha_M) = e(\alpha/\alpha_M)_F$ denote the ramification degree of $F_\alpha/F_{\alpha_M}$, where $F_{\alpha_M}$ is the finite extension of $F$ corresponding to $\alpha_M$. We extend the terminology of Definition~\ref{defn:symmetric-roots} to the nonzero weights of $Z_M$: the field $F_{\pm\alpha_M}$ is the fixed field of the setwise stabilizer of $\{\pm\alpha_M\}$ in $\Gamma$, and $\alpha_M$ is symmetric ramified precisely when $-\alpha_M \in \Gamma\cdot\alpha_M$ and $F_{\alpha_M}/F_{\pm\alpha_M}$ is ramified. Define
\begin{equation}\label{eqn:FKS-2}
\epsilon_{\flat,0}^{G/M}(\gamma) = \prod_{\substack{\alpha \in \Phi((G/M)_{\ol F},T_{\ol F})_{\mathrm{asym}}/\Sigma \\ \alpha_M \in \Phi((G/M)_{\ol F},(Z_M)_{\ol F})_{\mathrm{sym,ram}} \\ 2 \nmid e(\alpha/\alpha_M)}} \sgn_{\F_\alpha^\times}(\alpha(\gamma)) \cdot \prod_{\substack{\alpha \in \Phi((G/M)_{\ol F},T_{\ol F})_{\mathrm{sym,unram}}/\Gamma \\ \alpha_M \in \Phi((G/M)_{\ol F},(Z_M)_{\ol F})_{\mathrm{sym,ram}} \\ 2 \nmid e(\alpha/\alpha_M)}} \sgn_{\F_\alpha^1}(\alpha(\gamma)).
\end{equation}

For a symmetric ramified root $\alpha \in \Phi((G/M)_{\ol F}, T_{\ol F})$, choose a nonzero element $X_\alpha \in \mathfrak{g}_\alpha(F_\alpha)$, and choose $\tau \in \Gamma$ such that $\tau \cdot \alpha = -\alpha$. It is observed in \cite[\S 3.1]{Kal15} that
\begin{equation}\label{eq:c-alpha}
c_\alpha=\frac{[X_\alpha,\tau X_\alpha]}{H_\alpha}
\end{equation}
lies in $F_{\pm\alpha}^\times$. Let
\[
\kappa_\alpha\co F_{\pm \alpha}^\times/\Nm_{F_\alpha/F_{\pm\alpha}}(F_\alpha^\times) \cong \{\pm 1\}
\]
denote the unique isomorphism. Replacing $X_\alpha$ by $aX_\alpha$, with $a\in F_\alpha^\times$, multiplies \eqref{eq:c-alpha} by $a\tau(a)=\Nm_{F_\alpha/F_{\pm\alpha}}(a)$. Thus $\kappa_\alpha(c_\alpha)$ is independent of the choice of $X_\alpha$; we denote it by $f_{(G,T)}(\alpha)$. Moreover, $f_{(G,T)}(\alpha)$ depends only on the $\Gal(\ol F/F)$-orbit of $\alpha$. If $\frp_\alpha$ denotes the maximal ideal of $\cO_{F_\alpha}$, then we define
\begin{equation}\label{eqn:FKS-3}
\epsilon_f^{G/M}(\gamma) = \prod_{\substack{\alpha \in \Phi((G/M)_{\ol F},T_{\ol F})_{\mathrm{sym,ram}}/\Gamma \\ \alpha(\gamma) \in -1 + \frp_\alpha}} f_{(G,T)}(\alpha).
\end{equation}

For any $\alpha \in \Phi(G_{\ol F}, T_{\ol F})$, let $\ell_{p'}(\alpha^\vee) = \ell_{G,p'}(\alpha^\vee) \in \{1,2,3\}$ denote the prime-to-$p$ part of the normalized square length of $\alpha^\vee$, as in \cite[Definition 5.4.1]{FKS23}. Let $e_\alpha$ denote the ramification degree of $F_\alpha/F$. We define then
\begin{equation}\label{eqn:FKS-4}
\epsilon_{\flat,1}^{G/M}(\gamma) = \prod_{\substack{\alpha \in \Phi((G/M)_{\ol F},T_{\ol F})_{\mathrm{sym,ram}}/\Gamma \\ \alpha(\gamma) \in -1 + \frp_\alpha}} (-1)^{[\F_\alpha:\F]+1}\sgn_{\F_\alpha^\times}(e_\alpha \ell_{p'}(\alpha^\vee)).
\end{equation}
By \cite[Lemma 5.6.5]{FKS23}, the integer $e(\alpha/\alpha_M)$ is odd for every symmetric ramified root $\alpha$, so the exponent in \eqref{eqn:FKS-5} is an integer. Finally, define
\begin{equation}\label{eqn:FKS-5}
\epsilon_{\flat,2}^{G/M}(\gamma) = \prod_{\substack{\alpha \in \Phi((G/M)_{\ol F},T_{\ol F})_{\mathrm{sym,ram}}/\Gamma \\ \alpha(\gamma) \in -1 + \frp_\alpha}} \sgn_{\F_\alpha^\times}(-1)^{(e(\alpha/\alpha_M) - 1)/2}.
\end{equation}

We have now defined $\epsilon_x^{G/M}$ as a character of $\ol M_x(\F)$; inflation along $M(F)_x\to\ol M_x(\F)$ gives a character of $M(F)_x$. For a general connected reductive $G$, set $G_{\mathrm{ad}}=G/Z(G)$ and let $M_{\mathrm{ad}}$ be the image of $M$ in $G_{\mathrm{ad}}$. If $[x]_G$ denotes the image of $x$ in $\cB(G_{\der}) = \cB(G_{\ad})$, then we define (cf.\ \cite[Lemma 4.1.2 and Definition 4.1.10]{FKS23}):
\[
\epsilon_x^{G/M}\co M(F)_{[x]_G}\longrightarrow M_{\mathrm{ad}}(F)_{[x]_G}
\xrightarrow{\ \epsilon_x^{G_{\mathrm{ad}}/M_{\mathrm{ad}}}\ }\{\pm1\}.
\]
We emphasize that, despite the notation, the character $\epsilon_x^{G/M}$ depends on the good element $X$, although it only does so through $r$.

\subsubsection{Unramified twisted Levis}

We record here a couple of basic properties concerning the behavior of $\epsilon_x^{G/M}$ under passage to unramified twisted Levi subgroups of $G$; this will be studied further in \S\ref{ss:unram-levi-signs}.

\begin{lemma}\label{lemma:unramified-levi-not-ramified-symmetric}
Let $H \subset G$ be an unramified twisted Levi, and let $T \subset H$ be an $F$-subtorus containing $Z_H$. Every symmetric ramified root in $\Phi(G_{\ol F}, T_{\ol F})$ lies in $\Phi(H_{\ol F}, T_{\ol F})$.
\end{lemma}

\begin{proof}
Choose a finite unramified extension $E/F$ for which $H_E$ is a Levi factor of an $E$-parabolic subgroup $P\subset G_E$, and let $U$ be the unipotent radical of $P$. Observe that
\begin{equation}\label{eq:root-difference}
\Phi(G_{\ol F},T_{\ol F})-\Phi(H_{\ol F},T_{\ol F})
=\Phi(U_{\ol F},T_{\ol F})\sqcup
-\Phi(U_{\ol F},T_{\ol F}),
\end{equation}
where $\Gal(\ol F/E)$ preserves the two summands. Thus the $\Gal(\ol F/E)$-orbit of a weight $\alpha$ in \eqref{eq:root-difference} cannot contain $-\alpha$, which implies that $F_\alpha/F_{\pm\alpha}$ is unramified.
\end{proof}

Kaletha has informed us that in general there is no character of $M(F)_{[x]_G}$ extending the characters $\epsilon_{\sharp,x}^{G/M}$ associated to the various tamely ramified maximal $F$-tori $T$ of $M$ such that $x \in \cB(T)$. The following lemma shows, among other things, that this is not the case if $M$ is an unramified twisted Levi of $G$.

\begin{lemma}\label{lem:sharp-x-char-extension}
Let $M \subset G$ be a tamely ramified twisted Levi, and let $H \subset G$ be an unramified twisted Levi such that $M \cap H$ contains a maximal $F$-torus of $G$. Assume $x \in \cB(M \cap H)$. There exists a unique character $\epsilon_{\sharp,x}^{G/M,H}$ of $(M \cap H)(F)_{[x]_G}$ such that for each tamely ramified maximal $F$-torus $T$ of $M \cap H$ such that $x \in \cB(T)$, the restriction of $\epsilon_{\sharp,x}^{G/M,H}$ to $\ol T_{[x]}(\F_q)$ is equal to $\epsilon_{\sharp,x}^{G/M} \cdot \epsilon_{\sharp,x}^{H/M\cap H}$. If
\[
V = \frg(F)_{x,\frac{r}{2}}/\left(\frg(F)_{x,\frac{r}{2}+} + \frh(F)_{x,\frac{r}{2}} + \mathfrak{m}(F)_{x,\frac{r}{2}}\right),
\]
and there exists a parabolic $F$-subgroup $P \subset G$ with Levi factor $H$, and $V_+$ denotes the image of $\frp(F)_{x,\frac{r}{2}}$ in $V$, then 
\[
\epsilon_{\sharp, x}^{G/M,H} = \sgn_{\F_q^\times}\circ\det(-|V_+).
\]
\end{lemma}

\begin{proof}
Uniqueness follows from \cite[Lemma 3.5]{FKS23}, so we need only prove existence. We may and do assume that $G$ is of adjoint type, in which case $[x]_G = x$. We will apply \cite[Proposition 5.1.13]{FKS23}. Let $Z_{M\cap H}$ be the maximal central $F$-subtorus of $M \cap H$. Following the notation of \textit{loc.\ cit.}, define
\[
\mathfrak{S} = \{\alpha \in \Phi((G/(M \cup H))_{\ol F}, (Z_{M \cap H})_{\ol F})\co r/2 \in \ord_x(\alpha)\}/I_F,
\]
Note that $\mathfrak{S}$ inherits a $\Sigma$-action from the $\Sigma$-action on $\Phi(G_{\ol F}, (Z_{M\cap H})_{\ol F})$. By Lemma~\ref{lemma:unramified-levi-not-ramified-symmetric}, every root $\alpha \in \Phi((G/(M \cup H))_{\ol F}, (Z_{M \cap H})_{\ol F})$ is either asymmetric or symmetric unramified; such an $\alpha$ is asymmetric if and only if its image $\cO_\alpha$ in $\mathfrak{S}$ is asymmetric in the sense that $\cO_{-\alpha} \not\subset \Gamma \cdot \cO_\alpha$. Observe the $(M \cap H)(F)_x$-equivariant decomposition
\[
V_{\ol\F_q} = \bigoplus_{\cO \in \mathfrak{S}} V_\cO,
\]
where $V_\cO$ is the sum of the weight spaces for the action of $Z_{M\cap H}(F)$ on $V_{\ol\F_q}$ corresponding to weights in $\cO$. Given $\cO \in \mathfrak{S}$, let $\F_\cO$ be the residue field of the fixed field of the stabilizer of $\cO$ in $\Gamma$, and define a character $\chi_\cO\co (M \cap H)(F)_x \to \F_\cO^\times$ by
\[
\chi_\cO(\gamma) = \det(\gamma, V_\cO).
\]
Define $\F_{\pm\cO}$ similarly, and let $\F_\cO^1 = \ker(\Nm_{\F_\cO/\F_{\pm\cO}}\co \F_\cO^\times \to \F_{\pm\cO}^\times)$. Let $L_\cO\co \F_\cO^1 \to \F_\cO^\times/\F_{\pm\cO}^\times$ be the isomorphism which is inverse to the natural isomorphism $t \mapsto t\sigma(t)^{-1}$, where $\sigma \in \Gal(\F_\cO/\F_{\pm\cO})$ is the nontrivial element.

By \cite[Proposition 5.1.13]{FKS23}, there is a character $\epsilon_{\mathfrak{S}}\co (M\cap H)(F)_x \to \F^\times/\F^{\times 2}$ which satisfies
\[
\epsilon_{\mathfrak{S}}(\gamma) = \prod_{\cO \in \mathfrak{S}_{\mathrm{asym}}/\Sigma} \Nm_{\F_\cO/\F}(\chi_\cO(\gamma)) \cdot \prod_{\cO \in \mathfrak{S}_{\mathrm{sym}}/\Gamma} \Nm_{\F_\cO/\F}(L_\cO(\chi_\cO(\gamma))) \pmod{\F^{\times 2}}
\]
for all $\gamma \in (M \cap H)(F)_x$. If $T \subset M\cap H$ is a tamely ramified maximal $F$-torus such that $x \in \cB(T)$, then applying \cite[Remark 5.1.12]{FKS23} to the $\Sigma$-equivariant surjection
\[
\{\alpha \in \Phi((G/(M \cup H))_{\ol F}, T_{\ol F})\co r/2 \in \ord_x(\alpha)\} \to \mathfrak{S}
\]
and noting again Lemma~\ref{lemma:unramified-levi-not-ramified-symmetric}, we see that the character $\epsilon_{\mathfrak{S}}$ restricts to $T(F)_x$ as $\epsilon_{\sharp,x}^{G/M}\epsilon_{\sharp,x}^{H/M \cap H}$, as desired.

Now suppose that there exists a parabolic $F$-subgroup $P \subset G$ with Levi factor $H$, and let $V_+$ be as in the lemma statement. Fix a tamely ramified maximal $F$-torus $T$ of $M \cap H$ such that $x \in \cB(T)$. Let $\Phi$ denote the set of roots $\alpha \in \Phi((G/(M \cup H))_{\ol F}, T_{\ol F})$ such that $r/2 \in \ord_x(\alpha)$. We can decompose
\[
V = \bigoplus_{\substack{\omega \in \Phi/\Gamma \\ r/2 \in \ord_x(\omega)}} V_\omega,
\]
where $\ord_x(\omega) \coloneqq \ord_x(\alpha)$ for any $\alpha \in \omega$, and
\[
V_\omega = \left(\bigoplus_{\alpha \in \omega/I_F} (V_{\ol\F_q})_\alpha\right) \cap V.
\]
The subspace $V_+$ of $V$ corresponds to a $\Gamma$-stable subset $\Psi \subset \Phi$ such that $\Phi = \Psi \sqcup -\Psi$. We have therefore
\[
\det(t|V_+) = \prod_{\alpha \in \Psi/I_F} \alpha(t).
\]
Since $\Psi$ is $\Gamma$-stable, every element of $\Phi$ is asymmetric. Fix an orbit $\omega \in \Phi/\Gamma$, and let $\alpha \in \omega$; without loss of generality, we may assume $\omega \subset \Psi$. Then for $t \in \ol T_{[x]}(\F_q)$ we have
\begin{align*}
\sgn_{\F_q^\times}(\det(t|V_+ \cap (V_\omega \oplus V_{-\omega}))) &= \prod_{\alpha \in \omega/I_F} \alpha(t)^{(q-1)/2} \\
    &= \Nm_{\F_\alpha/\F_q}(\alpha(t)^{(q-1)/2}) \\
    &= \sgn_{\F_\alpha^\times}(\alpha(t)),
\end{align*}
as desired.
\end{proof}

\subsection{Sign characters associated to a twisted Levi tower}\label{sssec:eps-chars-descent-Levi}

Suppose we are given a sequence $G^0 \subset G^1 \subset \cdots \subset G^d = G$ of twisted Levi $F$-subgroups each of which splits after a tamely ramified extension of $F$. Fix embeddings $\cB(G^0) \subset \cB(G^1) \subset \cdots \subset \cB(G^d)$ and suppose $x \in \cB(G^0)$.\footnote{This property does not depend on the choice of embedding $\cB(G^0) \subset \cB(G)$.}
For every $i$, let $[x]_{G^{i+1}}$ be the image of $x$ in $\cB(G^{i+1}_{\der})$. Note
\[
G^0(F)_{[x]_G}\subset G^i(F)_{[x]_{G^{i+1}}},
\]
so the restriction in \eqref{eqn:yu-datum-sign-character} is defined.
If we are given good (relative to $G^{i+1}$) elements $X_i^* \in \Lie^*(G^i_{\mathrm{sc,ab}})(F)$ of depth $-r_i$ for $0 \leq i \leq d-1$, then we define a character $\epsilon = \epsilon_G = \epsilon_G(\vec{G}, x, \vec{r})$ on $G^0(F)_{[x]_G}$ by
\begin{equation}\label{eqn:yu-datum-sign-character}
\epsilon = \prod_{i=0}^{d-1} \epsilon_x^{G^{i+1}/G^i}|_{G^0(F)_{[x]_G}}.
\end{equation}
Similarly, if $T \subset G^0$ is a tamely ramified maximal $F$-torus such that $x \in \cB(T)$, then the formulas \eqref{eqn:FKS-1}--\eqref{eqn:FKS-5} define characters of $T(F)_{[x]_G}$. We may therefore define a character $\epsilon_{\sharp,x} = \epsilon_{G,\sharp,x} = \epsilon_{\vec{G},\sharp,x}$ of $T(F)_{[x]_G}$ by
\begin{equation}\label{eqn:eps-sharp-x}
\epsilon_{\sharp,x} = \prod_{i=0}^{d-1} \epsilon_{\sharp,x}^{G^{i+1}/G^i}|_{T(F)_{[x]_G}}.
\end{equation}
We define characters $\epsilon_{\flat,0}$, $\epsilon_f$, $\epsilon_{\flat,1}$, and $\epsilon_{\flat,2}$ of $T(F)_{[x]_G}$ completely analogously. Define also
\begin{equation}\label{eqn:eps_flat}
\epsilon_\flat \coloneqq \epsilon_{G,\flat} \coloneqq \epsilon_{\flat,0}\epsilon_{\flat,1}\epsilon_{\flat,2}.
\end{equation}
Finally, if $H \subset G$ is an unramified twisted Levi $F$-subgroup such that $G^0 \cap H$ contains a maximal $F$-torus of $G$, then define
\begin{equation}\label{eqn:sharp-x-difference}
\epsilon_{\sharp,x}^{\vec{G},H} = \prod_{i=0}^{d-1} \epsilon_{\sharp,x}^{G^{i+1}/G^i, H\cap G^{i+1}}|_{(G^0\cap H)(F)_{[x]_G}}
\end{equation}
with notation as in Lemma~\ref{lem:sharp-x-char-extension}.

\section{$\chi$-data and L-embeddings}\label{ss:chi-data-L-embedding}

As mentioned in the introduction, the explicit (inertial) L-parameter associated to a supercuspidal $\ol\Q_\ell$-representation $\pi = \pi(\Psi)$ arising from a Yu datum $\Psi$ is obtained by first constructing a pair $(T, \theta)$ consisting of a tamely ramified maximal $F$-torus $T \subset G$ and a character $\theta\co T(F) \to \ol\Q_\ell^\times$, and then defining $\rho^{\Kal}(\Psi)$ as the composition of $\ld\theta\co W_F \to \ld T(\ol\Q_\ell)$ and an L-embedding $\ld j_{T,G}\co \ld T \to \ld G$. When $T$ is unramified, the L-embedding $\ld j_{T,G}$ is canonical, but in general it is not, and another innovation of \cite{FKS23} is to select a ``good'' choice of $\ld j_{T,G}$ (depending on both $T$ and $\theta$).

We begin by recalling a number of definitions which originate from \cite{LS87}, though we will largely follow \cite{Kal19}, \cite{Kal21a}, and \cite{Kal21b}. Fix a tamely ramified $F$-torus $S \subset G$. Let $\Phi(G_{\ol F}, S_{\ol F})$ denote the set of nonzero weights for the action of $S_{\ol F}$ on $\mathfrak{g}_{\ol F}$. We follow the notation of \S\ref{ssec:FKS-preliminaries}: let $\Gamma = \Gal(\ol F/F)$ and $\Sigma = \Gamma \times \{\pm 1\}$, and for each nonzero weight $\alpha \in \Phi(G_{\ol F}, S_{\ol F})$, let $F_\alpha$ be the fixed field of the stabilizer of $\alpha$ in $\Gamma$, and let $F_{\pm\alpha} \subset F_\alpha$ be the fixed field of the (setwise) stabilizer of $\{\pm\alpha\}$. Let
\begin{equation}\label{eq:kappa-alpha}\kappa_\alpha\co F_{\pm\alpha}^\times/\Nm_{F_\alpha/F_{\pm\alpha}}(F_\alpha^\times) \to \{\pm 1\}
\end{equation}
denote the unique injection, which is an isomorphism if and only if $\alpha$ is symmetric. Throughout this section, let $k$ be a ring among $\ol\Q_\ell$, $\ol\Z_\ell$, and $\ol\F_\ell$, where $\ell \neq p$ is a prime number. We choose once and for all a square root of $q$ in $k$, and we will assume that the choices for $\ol\Q_\ell$, $\ol\Z_\ell$, and $\ol\F_\ell$ are compatible in the obvious sense.

\subsection{$\chi$-data and $a$-data} In \cite{LS87}, the following definition is introduced.

\begin{defn}\label{defn:chi-data}
A \textit{set of $\chi$-data} $(\chi_\alpha)_{\alpha \in \Phi(G_{\ol F}, S_{\ol F})}$ for $S$ with coefficients in $k$ is a collection of continuous homomorphisms $\chi_\alpha\co F_\alpha^\times \to k^\times$ satisfying the following properties for all $\alpha \in \Phi(G_{\ol F}, S_{\ol F})$.
\begin{enumerate}
    \item $\chi_{-\alpha} = \chi_\alpha^{-1}$.
    \item $\chi_{\gamma(\alpha)} = \chi_\alpha \circ \gamma^{-1}$ for all $\gamma \in \Gamma$.
    \item $\chi_\alpha|_{F_{\pm\alpha}^\times} = \kappa_\alpha$ if $\alpha$ is symmetric.
\end{enumerate}
\end{defn}

In some cases, there are canonical choices of $\chi$-data.


\begin{defn}\label{defn:minimally-ramified-chi-data}If $(\chi_\alpha)_\alpha$ is a set of $\chi$-data such that $\chi_\alpha = 1$ for all asymmetric $\alpha$ and $\chi_\alpha$ is the unramified character $F_\alpha^\times \to k^\times$ sending a uniformizer to $-1$ for all symmetric unramified roots $\alpha$, then we will call $(\chi_\alpha)$ \emph{minimally ramified}. Given a set $\chi = (\chi_\alpha)_\alpha$, we will let $\min \chi = (\min \chi_\alpha)_\alpha$ denote the unique minimally ramified set of $\chi$-data such that $\min \chi_\alpha = \chi_\alpha$ for all symmetric ramified $\alpha$. 

If all roots for $S_{\ol F}$ are either asymmetric or symmetric unramified, then we will refer to the unique minimally ramified set of $\chi$-data $(\chi_\alpha)_\alpha$ as the \textit{canonical set of $\chi$-data}.
\end{defn}

\begin{example}\label{ex:unramified-chi-data}
If $S$ is unramified, or more generally if $S$ becomes maximally unramified after a tame extension of \textit{odd} ramification degree, then every root in $\Phi(G_{\ol F}, S_{\ol F})$ is either asymmetric or symmetric unramified and thus $\Phi(G_{\ol F}, S_{\ol F})$ has a canonical set of $\chi$-data. 
\end{example}

As we will soon recall, if $M \subset G$ is a tamely ramified twisted Levi with maximal central torus $Z$, then a set of $\chi$-data $(\chi_\alpha)$ for $Z$ can be used to induce a $\widehat{G}(k)$-conjugacy class of L-embeddings $\ld j_\chi\co \ld M \to \ld G$.

One of the key technical points in \cite{Kal19} (resp.\ \cite{Kal21b} and \cite{FKS23}) is the definition of a set of $\chi$-data $(\chi_\alpha')$ (resp.\ $(\chi_\alpha'')$) associated to a non-singular torus-character pair $(T, \theta)$. When $T$ is unramified, this set of $\chi$-data is the canonical set described above, but in general it depends on a number of subtle choices.

We now recall the notion of sets of $a$-data, introduced in \cite{LS87}.

\begin{defn}[{\cite[\S 4.6]{Kal19}}]\label{defn:a-data}
If $M \subset G$ is a tamely ramified twisted Levi $F$-subgroup and $T \subset M$ is a tamely ramified maximal $F$-torus, then a \textit{set of $a$-data} $\{a_\alpha\}$ for $\Phi((G/M)_{\ol F}, T_{\ol F})$ is an assignment to each $\alpha \in \Phi((G/M)_{\ol F}, T_{\ol F})$ of an element $a_\alpha \in F_\alpha^\times$ such that $a_{\gamma\alpha} = \gamma(a_\alpha)$ and $a_{-\alpha} = -a_\alpha$ for all $\gamma \in \Gamma$ and all $\alpha$.
\end{defn}

\subsection{The set of $\chi$-data associated to a quadruple}\label{sss:canonical-chi-data}

In this subsection, we assume $p \neq 2$.
Let $\Psi = (\vec{G}, x, \vec{r})$ be a triple as at the end of \S\ref{sssec:eps-chars-descent-Levi}, and for each $0 \leq i \leq d-1$ suppose that we are given a good element $X_i^* \in \Lie^*(G^i_{\mathrm{sc,ab}})(F)_{-r_i}$. Fix a maximally unramified maximal $F$-torus $T \subset G^0$. Recall the character $\psi^0\co \F_q \to k^\times$ fixed in \S\ref{ss:fks}. For each $0 \leq i \leq d-1$ and each $\alpha \in \Phi((G^{i+1}/G^i)_{\ol F}, T_{\ol F})$, define $a_\alpha \in F_\alpha^\times$ by $a_\alpha = \langle X_i^*, H_\alpha\rangle$. As observed in \cite[Notation 4.2.2]{FKS23}, for each $i$ the collection of $a_\alpha$ constitutes a set of $a$-data for $\Phi((G^{i+1}/G^i)_{\ol F}, T_{\ol F})$, and $a_\alpha \in F_{\alpha_i}^\times$.

We now use this set of $a$-data to define a set of $\chi$-data $(\chi_\alpha'')$. By convention, set $G^{-1}=Z^{-1}=T$. Given an integer $-1\leq i\leq d-1$ and $\alpha \in \Phi(G^{i+1}_{\ol F}, T_{\ol F}) - \Phi(G^i_{\ol F}, T_{\ol F})$, let $\alpha_i$ denote the restriction of $\alpha$ to the maximal central torus $Z^i_{\ol F}$ of $G^i_{\ol F}$. Define $\chi''_{\alpha_i} = \chi''_{G,\alpha_i} \co F_{\alpha_i}^\times \to k^\times$ as in \cite[Notation 4.2.2]{FKS23}, as follows.\footnote{Another similar description of a set of $\chi$-data, also denoted by $\chi''$, is described in the arXiv version of \cite[\S 3.5]{Kal21b}; Kaletha has informed us that these are likely different, and that the definition in \cite{FKS23} is the ``correct'' one.}

\begin{enumerate}
\item If $\alpha_i$ is asymmetric, then $\chi_{\alpha_i}'' = 1$.
\item If $\alpha_i$ is symmetric unramified, then $\chi_{\alpha_i}''$ is the unique unramified quadratic character.
\item If $\alpha_i$ is symmetric ramified, then $\chi_{\alpha_i}''$ is the unique tamely ramified character whose restriction to $\cO_{F_{\alpha_i}}^\times$ is the inflation of the quadratic character of the residue field $\F_{\alpha_i}^\times$, whose restriction to $F_{\pm\alpha_i}^\times$ is the $\kappa_{\alpha_i}$ from \eqref{eq:kappa-alpha}, and which has the property that
\begin{equation}\label{eqn:ramified-symmetric-chi-data}
\chi_{\alpha_i}''(\ell_{G,p'}(\alpha^\vee)a_\alpha) = (-1)^{f_{\alpha_i} + 1}\frG_{\F_{\alpha_i}}(\psi^0),
\end{equation}
where $\ell_{G,p'}(\alpha^\vee)$ is the prime-to-$p$ part of the normalized squared length of $\alpha^\vee$ (so $\ell_{G,p'}(\alpha^\vee) \in \{1, 2, 3\}$)\footnote{Since we have assumed $p\neq 2$ throughout this section, the subscript $p'$ is only relevant if $p = 3$ and $G$ has a simple factor of type $\mathrm{G}_2$. The set of $\chi$-data introduced in \cite[\S 4.2]{FKS23} uses $\ell_G(\alpha^\vee)$ in place of $\ell_{G,p'}(\alpha^\vee)$, but that section also assumes that $p$ does not divide the order of any bond in the absolute Dynkin diagram for $G$, which rules out the $\mathrm{G}_2$ case above.}, the positive integer $f_{\alpha_i}$ is the degree of the residue field extension $\F_{\alpha_i}/\F_q$ and for any finite field extension $\F'/\F_q$ we define the normalized Gauss sum\footnote{Recall that we have fixed a square root of $q$.}
\[
\frG_{\F'}(\psi^0) = |\F'|^{-1/2}\sum_{x \in \F'^\times} \sgn_{\F'^\times}(x)\psi^0(\tr_{\F'/\F_q}(x)).
\]
\end{enumerate}
Observe that $\chi''_{\alpha_i}$ only depends on $\alpha_i$ by \cite[Notation 4.2.2]{FKS23}. Finally, define $\chi_\alpha'' = \chi_{G,\alpha}'' = \chi_{\alpha_i}'' \circ \Nm_{F_\alpha/F_{\alpha_i}}$. We will call $\chi''_G = \chi''_{(\vec{G}, x, \vec{r}, \vec{X}, T)} = (\chi''_\alpha)$ the \textit{set of $\chi$-data associated to $(\vec{G}, x, \vec{r}, \vec{X}, T)$}. We note that if $T$ becomes maximally unramified in $G$ after an extension of odd ramification degree, then $\chi''_G$ is the canonical set of $\chi$-data associated to $T$.

\subsection{The L-embedding associated to a set of $\chi$-data}\label{ss:l-embedding-defn}

We now drop the assumption $p \neq 2$ from the previous section.
As we have mentioned, from a set of $\chi$-data $(\chi_\alpha)$ for the maximal central torus $Z$ of a tamely ramified twisted Levi $M \subset G$, we obtain a $\wh G(k)$-conjugacy class $\ld j_\chi\co \ld M \to \ld G$ of L-embeddings. We recall this construction now, following \cite[\S\S 3.3, 3.4, 6.1]{Kal21a}, but we will restrict the generality for simplicity. The content below amounts essentially to unraveling the constructions in \cite{Kal21a}; because our use of sets of $\chi$-data will be limited to a few rather concrete settings, we choose not to recall the notion of gauges or the notations $t_p$, $s_{p/q}$, etc.

\subsubsection{The quasi-split case}\label{sss:quasi-split-l-embedding}
Suppose first that $G$ and $M$ are quasi-split, and let $(T_M, B_M)$ and $(T_G, B_G)$ be $F$-rational Borel pairs in $M$ and $G$, respectively. Choose $g \in G(\ol F)$ such that $(T_M)_{\ol F} = g(T_G)_{\ol F}g^{-1}$ and $(B_M)_{\ol F} \subset g(B_G)_{\ol F}g^{-1}$ and such that the system of simple absolute roots $\Delta((B_M)_{\ol F}, (T_M)_{\ol F})$ for $(B_M, T_M)$ is contained in the system of simple absolute roots $g\Delta((B_G)_{\ol F}, (T_G)_{\ol F})g^{-1}$ for $(g(B_G)_{\ol F}g^{-1}, g(T_G)_{\ol F}g^{-1})$. Given $\gamma \in \Gamma$, define $\gamma_{M,G}$ to be the image of $g^{-1}\gamma(g)$ in the Weyl group $\Omega(G_{\ol F}, (T_G)_{\ol F})$.

Choose a $\Gamma$-stable pinning $(\wh B, \wh T, \{X_{\alpha^\vee}\}_{\alpha \in \Delta(G_{\ol F}, (T_G)_{\ol F})})$ for $\wh G$, so we obtain a standard Levi $k$-subgroup scheme $\wh M \subset \wh G$ containing $\wh T$ with set of simple roots $\Delta(\wh M, \wh T)$ dual to 
\[
g^{-1}\Delta((B_M)_{\ol F}, (T_M)_{\ol F})g \subset \Delta(G_{\ol F}, (T_G)_{\ol F}) = \Delta(\wh G, \wh T)^\vee.
\]
Given $\gamma \in \Gamma$, let $\gamma_G$ denote the pinned automorphism of $\wh G$ induced by $\gamma$, and also use $\gamma_{M, G}$ to denote the element of $\Omega(\wh G, \wh T)$ corresponding to the above element $\gamma_{M, G}$ under the canonical isomorphism $\Omega(\wh G, \wh T) \cong \Omega(G_{\ol F}, (T_G)_{\ol F})$. Observe that $\Delta(\wh M, \wh T) \subset \Delta(\wh G, \wh T)$ is stable under the automorphism $\gamma_{M, G} \rtimes \gamma_G \in \Omega(\wh G, \wh T) \rtimes \Aut(\wh G, \wh B, \wh T, \{X_{\alpha^\vee}\})$, but $\Delta(\wh G, \wh T)$ is not stable under $\gamma_{M, G} \rtimes \gamma_G$ in general.

Let $n\co \Omega(\wh G, \wh T) \to N_{\wh G}(\wh T)(k)$ denote the Tits section of the projection map $N_{\wh G}(\wh T)(k) \to \Omega(\wh G, \wh T)$, given on simple reflections $s_\alpha$ by defining $n(s_\alpha) = x_\alpha(1)x_{-\alpha}(-1)x_\alpha(1)$, where $x_{\pm\alpha}$ are the root subgroup homomorphisms determined by the chosen pinning, and defined in general by $n(w) = n(s_1) \cdots n(s_r)$ if $w = s_1 \cdots s_r$ is any reduced expression of $w$ into simple reflections.

Let $A \subset \Phi \coloneqq \Phi((G/M)_{\ol F}, (T_M)_{\ol F})/\Omega(M_{\ol F}, (T_M)_{\ol F})$ be a chosen system of representatives for the $\Sigma$-orbits. We now define a function $Q = Q_\chi\co W_F \to Z_{\wh M}(k)$, where $Z_{\wh M}$ denotes the maximal central torus of $\wh M$, using the given set of $\chi$-data $(\chi_\alpha)_\alpha$. We will define $Q_\chi$ as a product
\[
Q_\chi = s \cdot \prod_{\alpha \in A} Q_{\chi,\alpha},
\]
so it suffices to define $s$ and each $Q_{\chi,\alpha}$. Observe that there is a natural map 
\[
\Phi \to X_*(Z_{\wh M}),
\]
given by $\cO \mapsto \sum_{\alpha \in \cO} \wh\alpha^\vee$. If $\alpha \in \cO$, then we will abusively use $\wh\alpha^\vee$ to denote the image of $\cO$ under this map. We will also abuse notation by using $\alpha$ to refer to an element of $A$.

For each $\alpha \in A$, let $W_\alpha$ be the stabilizer of $\alpha$ in $W_F$ and let $W_{\pm\alpha}$ be the (setwise) stabilizer of the set $\{\pm\alpha\}$. Fix a system of representatives $w_1, \dots, w_m \in W_F$ for the quotient $W_{\pm\alpha}\backslash W_F$, and let $v_0 \in W_\alpha$ be arbitrary. If $\alpha$ is symmetric, let $v_1 \in W_{\pm\alpha} - W_\alpha$ be arbitrary. Define a function $P \co \Phi \to \{\pm 1\}$ by
\[
P(\beta) = \begin{cases}
1 &\text{if } \beta = w_i^{-1}\alpha \text{ for some } \alpha \in A \text{ and } 1\leq i\leq m, \\
-1 &\text{otherwise.}
\end{cases}
\]
Let $\Phi^+ \subset \Phi$ denote (the image in $\Phi$ of) the positive system of roots corresponding to $g(B_G)_{\ol F}g^{-1}$. Using $P$, for each $w \in W_F$ we define the set
\begin{align*}
M(w) &= \{\alpha \in \Phi\co \alpha \in \Phi^+, w^{-1}\alpha \not\in \Phi^+, P(\alpha) = P(w^{-1}\alpha) = 1\} \\
&\sqcup \{\alpha \in \Phi\co \alpha \in \Phi^+, w^{-1}\alpha \in \Phi^+, P(\alpha) = -1, P(w^{-1}\alpha) = 1\}.
\end{align*}
Using this, for each $w \in W_F$ we define
\[
s(w) = \prod_{\alpha \in M(w)} \wh\alpha^\vee(-1) \in Z_{\wh M}(k).
\]

Now fix $\alpha \in A$. For $w \in W_F$ and $1 \leq i \leq m$, define $u_i(w) \in W_{\pm\alpha}$ by $w_i \cdot w = u_i(w) \cdot w_j$ for some (uniquely determined) $j$. Similarly, for $u \in W_{\pm\alpha}$ and $i \in \{0, 1\}$, define $v_i(u) \in W_\alpha$ by $v_i \cdot u = v_i(u) \cdot v_j$ for some (uniquely determined) $j \in \{0, 1\}$. We define
\begin{equation}\label{eqn:def-of-Q}
Q_{\chi,\alpha}(w) = \prod_{i=1}^m (w_i^{-1}\wh\alpha^\vee)(\chi_\alpha(\rec_{F_\alpha}(v_0(u_i(w))))),
\end{equation}
where $\rec_{F_\alpha}\co W_{F_\alpha}\to F_\alpha^\times$ is the reciprocity map from class field theory (normalized to send a lift of \emph{geometric} Frobenius to a uniformizer of $F_\alpha^\times$). We then define $\ld j_\chi\co \ld M \to \ld G$ by
\begin{equation}\label{eqn:def-of-L-embedding}
\ld j_\chi(h, w) = (hQ(w)n(\gamma_{M,G}), w),
\end{equation}
where $\gamma \in \Gamma$ is the image of $w \in W_F$. It is proven in \cite[Corollary 6.7]{Kal21a} that $\ld j_\chi$ is an L-embedding. This completes the construction in the case that $G$ and $M$ are quasi-split.

\subsubsection{The general case}
In the general case, let $G_0$ be a quasi-split inner form of $G$. By \cite[Lemma 6.4]{Kal21a}, there exists an inner twist $\xi\co G_{0,\ol F} \to G_{\ol F}$ with the property that $\xi^{-1}(M_{\ol F})$ is the base change to $\ol F$ of a quasi-split twisted Levi $M_0 \subset G_0$, and the restricted map $\xi\co M_{0,\ol F} \to M_{\ol F}$ is an inner twist. The inner twist $\xi$ therefore induces compatible isomorphisms $\ld M \cong \ld M_0$ and $\ld G \cong \ld G_0$. If $\chi_0$ is the set of $\chi$-data for the maximal central torus $Z_0$ of $M_0$ induced by $\chi$, then we define $\ld j_\chi = \ld j_{\chi_0}$ via the above isomorphisms.

\subsection{Dual L-homomorphisms}\label{ss:dual-l-homs}

In the remainder of this subsection, fix a tamely ramified maximal $F$-torus $T \subset G$, and retain the notation of the previous subsection. The following remark partially explains the otherwise somewhat mystifying formulas occurring in the definitions of L-embeddings given above.

\begin{remark}[Explicit form of Shapiro's lemma]\label{remark:explicit-shapiro}
Let $\Gamma$ be a locally profinite group, let $\Gamma_0\subset\Gamma$ be an open subgroup of finite index, and let $M$ be a locally profinite group equipped with a continuous $\Gamma_0$-action. The continuous noncommutative Shapiro lemma gives a natural bijection $\rH^1(\Gamma_0,M)\cong\rH^1(\Gamma,\ind_{\Gamma_0}^{\Gamma}M)$. This bijection can be described concretely on the level of cocycles as follows.

First, recall that $\ind_{\Gamma_0}^{\Gamma} M$ can be described as the space of $\Gamma_0$-equivariant functions $f\co \Gamma \to M$, with $\Gamma$-action given by $(\gamma \cdot f)(\gamma_0) = f(\gamma_0\gamma)$. Choose a set of representatives $\eta_1, \dots, \eta_n \in \Gamma$ for $\Gamma_0\backslash\Gamma$, and observe that an element $f \in \ind_{\Gamma_0}^\Gamma M$ is determined by the elements $f(\eta_1), \dots, f(\eta_n)$. For $\gamma \in \Gamma$, we define $\eta_i(\gamma) \in \Gamma_0$ by the equation
\[
\eta_i \cdot \gamma = \eta_i(\gamma) \cdot \eta_j,
\]
where $j$ is uniquely determined. For $1$-cocycles, the map in Shapiro's lemma can be described on cocycles as sending a $1$-cocycle $c_0\co \Gamma_0 \to M$ to the $1$-cocycle $c\co \Gamma \to \ind_{\Gamma_0}^{\Gamma} M$ given by
\[
c(\gamma)(\eta_i) = c_0(\eta_i(\gamma)).
\]
It is easy to check that $c$ is a $1$-cocycle and that it is independent of the choice of representatives $\eta_1, \dots, \eta_n$ up to a coboundary. Moreover, the inverse bijection sends a $1$-cocycle $c$ to the $1$-cocycle $\Gamma_0\to M$, $\gamma\mapsto c(\gamma)(1)$.
\end{remark}

Now we prove a technical lemma which will be necessary later. For each $\alpha \in \Phi(G_{\ol F}, T_{\ol F})$, define an $F$-torus $J_\alpha$ as in \cite[\S 3.1]{Kal21a} as follows: if $\alpha$ is asymmetric, let $J_\alpha=\Res_{F_\alpha/F}\G_m$; if $\alpha$ is symmetric, let
\[
J_\alpha=\Res_{F_{\pm\alpha}/F}\Res^1_{F_\alpha/F_{\pm\alpha}}\G_m
=\ker\bigl(\Res_{F_\alpha/F}\G_m\xrightarrow{\Nm_{F_\alpha/F_{\pm\alpha}}}
\Res_{F_{\pm\alpha}/F}\G_m\bigr).
\]
Observe that the root $\alpha$ is an $F_\alpha$-homomorphism $T_{F_\alpha}\to\G_m$, which corresponds to an $F$-homomorphism $\alpha\co T\to J_\alpha$. In turn, this is dual to an L-homomorphism $\ld\alpha\co\ld J_\alpha\to\ld T$. Let $\wh\alpha^\vee\co\G_m\to\wh T$ denote the cocharacter corresponding to $\alpha$, and let $\wh\alpha\co\wh J_\alpha\to\wh T$ be the restriction of $\ld\alpha$ to dual groups.

Recall that the dual group of $\Res_{F_\alpha/F}\G_m$ can be described (functorially) as the space of functions $W_{F_\alpha}\backslash W_F \to \G_m$. If $\alpha$ is symmetric, then $\wh J_\alpha$ can be described as the space of functions $W_{F_\alpha}\backslash W_F\to\G_m$ modulo pointwise multiplication by functions $W_{F_{\pm\alpha}}\backslash W_F\to\G_m$.

Recall that $k$ is a ring among $\ol\Q_\ell$, $\ol\Z_\ell$, and $\ol\F_\ell$, where $\ell\neq p$.
For each $\alpha$, let $\xi_\alpha\co F_\alpha^\times\to k^\times$ be a character. If $\alpha$ is asymmetric, identify $J_\alpha(F)$ with $F_\alpha^\times$ and set $\xi_\alpha^1=\xi_\alpha$. If $\alpha$ is symmetric, assume that $\xi_\alpha$ factors through the surjection $F_\alpha^\times\surj F_\alpha^1$ given by $t\longmapsto t\,v_1(t)^{-1}$, where $F_\alpha^1 \coloneqq \ker(\Nm_{F_\alpha/F_{\pm\alpha}}\co F_\alpha^\times \to F_{\pm\alpha}^\times)$, and let
\[
\xi_\alpha^1\co J_\alpha(F)=F_\alpha^1\to k^\times
\]
be the factored character. Let $\wh \xi_\alpha^1\co W_F\to\wh J_\alpha(k)$ denote its dual cocycle. Let $\wh \xi_{\alpha,0}\co W_{F_\alpha}\to k^\times$ be the character corresponding to $\xi_\alpha$ as a character of $\G_m(F_\alpha)$, and let $\wh \xi_\alpha\co W_F\to(\Res_{F_\alpha/F}\G_m)^\land(k)$ be the cocycle corresponding to $\xi_\alpha$ as a character of $(\Res_{F_\alpha/F}\G_m)(F)$.

\begin{lemma}\label{lemma:dual-l-homs}
We have
\begin{equation}\label{eqn:l-hom-dual-to-root-sign}
[\wh\alpha \circ \wh \xi_\alpha^1] = \left[w\longmapsto \prod_{i=1}^m (w_i^{-1}\wh\alpha^\vee)(\wh \xi_{\alpha,0}(v_0(u_i(w))))\right]\quad\text{in }\rH^1(W_F, \wh T(k)).
\end{equation}
\end{lemma}

\begin{proof}
Suppose first that $\alpha$ is asymmetric, so $F_\alpha = F_{\pm\alpha}$. The dual L-homomorphism $\wh\alpha\co \wh J_\alpha \to \wh T$ is given by
\begin{equation}\label{eqn:l-hom-dual-to-root}
\wh\alpha(f) = \prod_{i=1}^m (w_i^{-1}\wh\alpha^\vee)(f(v_0w_i)),
\end{equation}
where $f\co W_{F_\alpha}\backslash W_F \to \G_m$ as above. By Remark~\ref{remark:explicit-shapiro}, for $w\in W_F$ we have
\begin{equation}\label{eqn:l-param-dual-to-sign}
\wh \xi_\alpha(w)(v_0w_i) = \wh \xi_{\alpha,0}(v_0(u_i(w))).
\end{equation}
Combining \eqref{eqn:l-hom-dual-to-root} and \eqref{eqn:l-param-dual-to-sign} yields \eqref{eqn:l-hom-dual-to-root-sign}.

Now suppose that $\alpha$ is symmetric. Note that $\wh\alpha\co \wh J_\alpha \to \wh T$ is given by
\begin{equation}\label{eqn:l-hom-dual-to-root-1}
\wh\alpha(f^1) = \prod_{i=1}^m (w_i^{-1}\wh\alpha^\vee)(f^1(v_0w_i)\cdot f^1(v_1w_i)^{-1}),
\end{equation}
where we regard $f^1$ as a function $W_{F_\alpha}\backslash W_F \to \G_m$, up to multiplication by a function $W_{F_{\pm\alpha}}\backslash W_F \to \G_m$. By Remark~\ref{remark:explicit-shapiro}, the dual L-parameter $\wh \xi_\alpha\co W_F \to (\Res_{F_\alpha/F}\G_m)^\land(k)$ is given for $w \in W_F$ by
\begin{equation}\label{eqn:l-param-dual-to-sign-1-2}
\wh \xi_\alpha(w)(v_jw_i) = \wh \xi_{\alpha,0}(v_j(u_i(w))).
\end{equation}
The injective homomorphism $\wh J_\alpha \to (\Res_{F_\alpha/F}\G_m)^\land$ which is dual to $t \mapsto t \cdot v_1(t)^{-1}$ is given by sending the equivalence class of a function $f^1\co W_{F_\alpha}\backslash W_F \to \G_m$ to the function $f\co W_{F_\alpha}\backslash W_F \to \G_m$ given by
\[
f(v_jw_i) = f^1(v_jw_i) \cdot f^1(v_{1-j}w_i)^{-1}
\]
for all $j$ and $i$. For each $w\in W_F$, represent the class $\wh \xi_\alpha^1(w)\in\wh J_\alpha(k)$ by the unique function normalized by $\wh \xi_\alpha^1(w)(v_1w_i)=1$ for every $i$. With this normalization,
\begin{equation}\label{eqn:norm-1-difference}
\wh \xi_\alpha^1(w)(v_0w_i)=\wh \xi_{\alpha,0}(v_0(u_i(w)))
\quad\text{and}\quad
\wh \xi_\alpha^1(w)(v_1w_i)=1
\end{equation}
for all $i$. Combining \eqref{eqn:l-hom-dual-to-root-1}, \eqref{eqn:l-param-dual-to-sign-1-2}, and \eqref{eqn:norm-1-difference}, we obtain \eqref{eqn:l-hom-dual-to-root-sign}.
\end{proof}

\subsection{Associated L-embeddings} We have now associated explicit L-embeddings to twisted Levi $F$-subgroups of $G$ in two important settings.

\begin{defn}\label{defn:canonical-l-embeddings} If $H \subset G$ is an unramified twisted Levi with maximal central torus $Z$ and $\chi$ is the canonical set of $\chi$-data for $\Phi(G_{\ol F}, Z_{\ol F})$ (which exists by Lemma~\ref{lemma:unramified-levi-not-ramified-symmetric}), then we let $\ld j_{H,G}\co \ld H \to \ld G$ denote the L-embedding $\ld j_{H,G} = \ld j_\chi$.

Similarly, if $p \neq 2$ and $(\vec{G},x,\vec{r},\vec{X})$ is a tuple for $G$ as in \S\ref{sss:canonical-chi-data} and $T \subset G^0$ is a maximally unramified maximal $F$-torus and $\chi''$ is the set of $\chi$-data associated to $(\vec{G},x,\vec{r},\vec{X},T)$ in \S\ref{sss:canonical-chi-data}, let $\ld j_{T, G}\co \ld T \to \ld G$ denote the L-embedding $\ld j_{T, G} = \ld j_{\chi''}$. When these two notations overlap, it will be clear from context which is meant.
\end{defn}

We summarize some basic properties of $\ld j_{H,G}$ in the following lemma.

\begin{lemma}\label{lem:unramified-L-embedding-properties}
Let $H \subset G$ be an unramified twisted Levi $F$-subgroup. The L-embedding $\ld j_{H,G}$ enjoys the following properties.
\begin{enumerate}
\item If $H$ is the Levi of an $F$-rational parabolic subgroup of $G$, then $\ld j_{H, G}$ agrees with the canonical dual embedding $(t, w) \mapsto (t, w)$.
\item If $\pi\co G' \to G$ is an $F$-homomorphism which induces an isomorphism on adjoint groups, and we let $H' = \pi^{-1}(H)$, then the diagram
\[
\begin{tikzcd}
    \ld H \arrow[r, "{\ld j_{H,G}}"] \arrow[d]
        &\ld G \arrow[d, "\pi"] \\
    \ld H' \arrow[r, "{\ld j_{H',G'}}"]
        &\ld G'
\end{tikzcd}
\]
commutes.
\item If $G \cong G_1 \times G_2$, so $H \cong H_1 \times H_2$, then $\ld j_{H,G}$ is the product of $\ld j_{H_1, G_1}$ and $\ld j_{H_2, G_2}$.
\end{enumerate}
\end{lemma}

\begin{proof}
(1) We may and do assume that $G$ is quasi-split by \cite[Lemma 6.4]{Kal21a}, since passing to an inner twist which preserves $H$ also preserves the property that $H$ is the Levi of a parabolic $F$-subgroup of $G$. In this case, there is a Borel $F$-subgroup $B \subset G$ such that $B \cap H$ is a Borel $F$-subgroup. In the construction of \S\ref{sss:quasi-split-l-embedding}, we may choose $B_G = B$, $B_H = B \cap H$, and $T_G = T_H = T$ a maximal $F$-torus in $B$. We therefore have $\gamma_{H,G} = 1$ for all $\gamma \in \Gamma$, and it suffices to show $Q = 1$ by \eqref{eqn:def-of-L-embedding}. Note that we may choose representatives $w_i$ as in \S\ref{sss:quasi-split-l-embedding} so that $w_i^{-1}\alpha$ is positive (relative to $B$) whenever $\alpha$ is positive. With this choice, the set $M(w)$ is empty and thus $s = 1$. Moreover, every element $\alpha \in \Phi((G/H)_{\ol F}, T_{\ol F})$ is asymmetric, so the set of $\chi$-data is trivial and thus $Q_{\chi,\alpha} = 1$ for all $\alpha$ by \eqref{eqn:def-of-Q}.

(2) First, it is clear that $H'$ is an unramified twisted Levi in $G'$, so the diagram makes sense. We may again assume that $G$ is quasi-split, hence $G'$ is also quasi-split. In this case, it suffices to observe that the definitions of $s$, $n$, and $Q$ in \S\ref{sss:quasi-split-l-embedding} are all evidently compatible with $\pi$, since the sets of $\chi$-data defining $\ld j_{H,G}$ and $\ld j_{H',G'}$ are the same.

(3) This is clear from the definitions.
\end{proof}

\section{Compatibility of sign characters and L-embeddings}\label{sec:sign-l-embedding-compatibility}

It will be important to understand the behavior of the L-embedding $\ld j_{T,G}$ associated to a quadruple $(\vec G, x, \vec r, \vec X)$ as in Definition~\ref{defn:canonical-l-embeddings} when $G$ is replaced by an unramified twisted Levi subgroup $H$ such that $x \in \cB(H)$. For such an $H$, we obtain a quadruple $(\vec H, x, \vec r, \vec X_H)$, where $H^i = G^i \cap H$ for all $i$ and $X_{H,i}^* = X_i^*|_{\Lie(H^i)}$, and from this we obtain an L-embedding $\ld j_{T,H}\co \ld T \to \ld H$. There is a canonical L-embedding $\ld j_{H,G}\co \ld H \to \ld G$, and we want to compare $\ld j_{T,G}$ and $\ld j_{H,G}\circ\ld j_{T,H}$. It is not true that
\[
\ld j_{T,G}\sim \ld j_{H,G}\circ\ld j_{T,H}
\]
in general; see Example~\ref{example:nontrivial-sign-char}. The technical issues are twofold: first, the definition of $\ld j_{T,G}$ (resp.\ $\ld j_{T,H}$) involves a product over the set of roots $\Phi(G_{\ol F}, T_{\ol F})$ (resp.\ $\Phi(H_{\ol F}, T_{\ol F})$), and these sets are of course different. Second, the definitions both involve the numbers $\ell_{G,p'}(\alpha^\vee)$ (resp.\ $\ell_{H,p'}(\alpha^\vee)$), and these differ even when $\alpha \in \Phi(H_{\ol F}, T_{\ol F})$.

However, we will prove in Proposition~\ref{prop:fks-and-tasho-cancel} that
\[
\ld j_{T,G}\circ\ld\theta
\sim
\ld j_{H,G}\circ\ld j_{T,H}\circ
\ld(\theta\epsilon_G\epsilon_H\epsilon_{\sharp,x,G}\epsilon_{\sharp,x,H}),
\]
where the $\epsilon$ are the sign characters from \cite{FKS23} recalled above, whenever $T$ is elliptic and $\theta\co T(F) \to k^\times$ is a character. We will similarly analyze base change, i.e., the behavior of $\ld j_{T,G}$ under passage from $F$ to a cyclic tamely ramified extension $E/F$. Since the calculations become rather involved, the proofs are mostly relegated to Appendix~\ref{app:rootwise-sign-chi-calculations}.

We retain the notation of Sections~\ref{sec:sign-characters} and~\ref{ss:chi-data-L-embedding}. Fix a tuple $(\vec{G}, x, \vec{r}, \vec{X})$ as in \S\ref{sss:canonical-chi-data}, and let $k$ be a ring among $\ol\Q_\ell$, $\ol\Z_\ell$, and $\ol\F_\ell$.


\subsection{Unramified twisted Levis}\label{ss:unram-levi-signs}

We first deal with the case of descent to unramified twisted Levis. We assume in this section that $p \neq 2$.

\begin{lemma}\label{lemma:fks-unram-levi}
Let $H$ and $M$ be tamely ramified twisted Levi $F$-subgroups of $G$ such that $H$ is an unramified twisted Levi and $H \cap M$ contains a maximal $F$-torus $T$ of $G$\footnote{Note that this automatically implies that $H \cap M$ is a twisted Levi $F$-subgroup of $H$.} such that $x\in\cB(T)$. Then
\[
\epsilon_f^{G/M} = \epsilon_f^{H/(M\cap H)}
\]
as characters of $T(F)_{[x]}$.
\end{lemma}

\begin{proof}
Observe that if $\alpha \in \Phi((G/M)_{\ol F},T_{\ol F})$ is symmetric ramified then $\alpha \in \Phi((H/(M \cap H))_{\ol F}, T_{\ol F})$ by Lemma~\ref{lemma:unramified-levi-not-ramified-symmetric}, and it is clear from the definitions that $f_{(G,T)}(\alpha) = f_{(H,T)}(\alpha)$.
\end{proof}

Throughout this section, let $H \subset G$ be an unramified twisted Levi $F$-subgroup such that $H \cap G^0$ contains a maximal $F$-torus of $G$, and let $H^i = H \cap G^i$ for all $0 \leq i \leq d$. Fix a maximally unramified maximal $F$-torus $T \subset H^0$, and let $\chi''_G$ (resp.\ $\chi''_H$) be the set of $\chi$-data for $\Phi(G_{\ol F}, T_{\ol F})$ (resp.\ $\Phi(H_{\ol F}, T_{\ol F})$) associated to the tuple $(\vec{G},x,\vec{r},\vec{X})$ where $X_i^*$ is as in \S \ref{sss:canonical-chi-data}.\footnote{Note that the restrictions of the good elements $X_i^*$ to $\Lie(H^i)$ are automatically good relative to $H^{i+1}$, since the roots of $H^{i+1}/H^i$ relative to $T$ form a subset of the roots of $G^{i+1}/G^i$.} Below we will use the characters defined in \S \ref{sssec:eps-chars-descent-Levi}.

For the following statements, we recall that for a character $\eta\co T(F)_{\mathrm{b}} \to k^\times$, there is an associated inertial L-parameter $I_F \to \ld T(k)$; we will denote this homomorphism by $\ld\eta|_{I_F}$.\footnote{This amounts to the claim that if $\eta_1, \eta_2\co T(F) \to k^\times$ are two characters extending $\eta$, then $\ld\eta_1|_{I_F}\sim\ld\eta_2|_{I_F}$. From the construction of the Local Langlands Correspondence for tori, this reduces to the case $T = \G_m$, in which case it follows from the fact that the isomorphism $W_F^{\ab} \cong F^\times$ from class field theory sends $I_F$ to $\cO_F^\times = \G_m(F)_{\mathrm{b}}$.} If $T$ is not anisotropic, then the notation $\ld\eta|_{I_F}$ is slightly abusive, because there is no canonical choice of L-parameter $\ld\eta\co W_F \to \ld T(k)$ of which $\ld\eta|_{I_F}$ is the restriction.

\begin{lemma}\label{lemma:comparison-to-minimal-ramified-chi-data}
Let $\theta\co T(F) \to k^\times$ be a character. Recall from Definition~\ref{defn:minimally-ramified-chi-data} that $\min \chi''_G$ denotes the minimally ramified $\chi$-data associated to $\chi''_G$. We have
\[
    \ld j_{\chi''_G} \circ \ld\theta|_{I_F} \sim \ld j_{\min \chi''_G} \circ \ld(\theta \cdot \epsilon_{G,\flat,0})|_{I_F}.
\]
If moreover $T$ is elliptic in $G$\footnote{Observe that the following displayed equation would not make sense if $T$ were not assumed elliptic, because then $T(F) \neq T(F)_{[x]}$, so $\epsilon_{G,\flat,0}$ would not be a character of $T(F)$. Completely similar remarks apply to the subsequent results.}, then we have
\[
\ld j_{\chi''_G} \circ \ld\theta \sim \ld j_{\min \chi''_G} \circ \ld(\theta \cdot \epsilon_{G,\flat,0}).
\]
The completely analogous claims hold for $H$.
\end{lemma}

\begin{proof}
The proof is given in Appendix~\ref{app:comparison-minimal-proof}.
\end{proof}

Let $\chi''_{H,G}$ denote the unique minimally ramified set of $\chi$-data for $\Phi(G_{\ol F}, T_{\ol F})$ such that for all $\alpha \in \Phi(H_{\ol F}, T_{\ol F})$ we have $\chi''_{H,G,\alpha} = \min \chi''_{H,\alpha}$; this determines $\chi''_{H,G}$ by Lemma~\ref{lemma:unramified-levi-not-ramified-symmetric}.

\begin{lemma}\label{lemma:modification-of-symmetric-ramified-chi-data}
If $\theta\co T(F) \to k^\times$ is a character, then
\[
\ld j_{\chi''_{H,G}} \circ \ld\theta|_{I_F} \sim \ld j_{\min\chi''_G} \circ \ld(\theta \cdot \epsilon_{G,\flat,1}\epsilon_{H,\flat,1}\epsilon_{G,\flat,2}\epsilon_{H,\flat,2})|_{I_F}.
\]
If moreover $T$ is elliptic in $G$, then
\[
\ld j_{\chi''_{H,G}} \circ \ld\theta \sim \ld j_{\min\chi''_G} \circ \ld(\theta \cdot \epsilon_{G,\flat,1}\epsilon_{H,\flat,1}\epsilon_{G,\flat,2}\epsilon_{H,\flat,2}).
\]
\end{lemma}

\begin{proof}
The proof is given in Appendix~\ref{app:modification-symmetric-proof}.
\end{proof}

Let $\chi_H$ be the canonical (minimally ramified) set of $\chi$-data for $\Phi((G/H)_{\ol F}, (Z_H)_{\ol F})$, and let $\chi_{H,G}$ be the unique set of $\chi$-data for $\Phi(G_{\ol F}, T_{\ol F})$ whose restriction to $\Phi((G/H)_{\ol F}, T_{\ol F})$ is the inflation of $\chi_H$ in the sense of \cite[Definition 5.16]{Kal21a} and whose restriction to $\Phi(H_{\ol F}, T_{\ol F})$ is equal to the restriction of $\chi''_{H,G}$.

\begin{lemma}\label{lemma:inflation-change}
We have $\ld j_{\chi''_{H,G}}|_{\wh T \rtimes I_F} \sim \ld j_{\chi_{H,G}}|_{\wh T \rtimes I_F}$. If $T$ is elliptic in $G$, then $\ld j_{\chi''_{H,G}} \sim \ld j_{\chi_{H,G}}$.
\end{lemma}

\begin{proof}
The two sets of $\chi$-data agree on $\Phi(H_{\ol F},T_{\ol F})$, so fix a $\Sigma$-orbit represented by $\alpha\in\Phi((G/H)_{\ol F},T_{\ol F})$ and let $\alpha_H$ be its restriction to $Z_H$. Set
\[
\zeta_\alpha\coloneqq\chi_{H,G,\alpha}(\chi''_{H,G,\alpha})^{-1}.
\]
If $\alpha$ is asymmetric, set $\zeta_\alpha^1=\zeta_\alpha$ on $J_\alpha(F)=F_\alpha^\times$; if $\alpha$ is symmetric, then $\zeta_\alpha$ is trivial on $F_{\pm\alpha}^\times$, so it factors through $F_\alpha^\times\to F_\alpha^1=J_\alpha(F)$, and we denote the factored character by $\zeta_\alpha^1$. Put $\eta_\alpha = \zeta_\alpha^1\circ\alpha$, so the product $\prod_{\alpha \in \Phi((G/H)_{\ol F}, T_{\ol F})/\Sigma} \eta_\alpha$ is the character $\zeta_T$ of $T(F)$ from \cite[\S 3.2]{Kal21a}.

By Lemma~\ref{lemma:unramified-levi-not-ramified-symmetric}, neither $\alpha$ nor $\alpha_H$ is symmetric ramified. This implies that the restriction of $\zeta_T$ to $T(F)_{[x]}$ is trivial (as in the proof of \cite[Proposition 5.27]{Kal21a}). Lemma~\ref{lemma:dual-l-homs}, applied to $\zeta_\alpha$, identifies
\[
[Q_{\chi_{H,G},\ol\alpha}Q_{\chi''_{H,G},\ol\alpha}^{-1}]
=[\wh\eta_\alpha]
\quad\text{in }\rH^1(W_F, \wh T(k)).
\]
The right hand class restricts trivially to $I_F$ by local Langlands for tori, so the same is true of the left hand class. By the definitions \eqref{eqn:def-of-L-embedding} and \eqref{eqn:def-of-Q}, this implies
$\ld j_{\chi''_{H,G}}|_{\wh T\rtimes I_F}\sim
\ld j_{\chi_{H,G}}|_{\wh T\rtimes I_F}$. If $T$ is elliptic, then the same argument applies without restricting to $I_F$.
\end{proof}

\begin{prop}\label{prop:fks-and-tasho-cancel}
Let $\theta\co T(F) \to k^\times$ be a character. Then, with $\epsilon_{G,\sharp,x}$ and $\epsilon_{H,\sharp,x}$ defined as in \eqref{eqn:eps-sharp-x}, we have
\[
\ld j_{\chi''_G} \circ \ld\theta|_{I_F} \sim \ld j_{H,G} \circ \ld j_{\chi''_H} \circ \ld(\theta \cdot \epsilon_G\epsilon_H\epsilon_{G,\sharp,x}\epsilon_{H,\sharp,x})|_{I_F}.
\]
If moreover $T$ is elliptic in $G$, then
\[
\ld j_{\chi''_G} \circ \ld\theta \sim \ld j_{H,G} \circ \ld j_{\chi''_H} \circ \ld(\theta \cdot \epsilon_G\epsilon_H\epsilon_{G,\sharp,x}\epsilon_{H,\sharp,x}).
\]
\end{prop}

\begin{proof}
By definition, the restriction of $\chi_{H,G}$ to $\Phi((G/H)_{\ol F}, T_{\ol F})$ is equal to the inflation (in the sense of \cite[Definition 5.16]{Kal21a}) of $\chi_H$, the canonical set of $\chi$-data for $\Phi((G/H)_{\ol F}, (Z_H)_{\ol F})$. The restriction of $\chi_{H,G}$ to $\Phi(H_{\ol F}, T_{\ol F})$ is by definition equal to $\min \chi''_H$. Thus by \cite[Proposition 6.9]{Kal21a}, we have
\[
\ld j_{\chi_{H,G}} \sim \ld j_{H,G} \circ \ld j_{\min \chi''_H}.
\]
By Lemmas~\ref{lemma:comparison-to-minimal-ramified-chi-data}, \ref{lemma:modification-of-symmetric-ramified-chi-data}, and \ref{lemma:inflation-change}, we have
\begin{align*}
\ld j_{\chi''_G} \circ \ld\theta|_{I_F} &\sim \ld j_{\min \chi''_G} \circ \ld(\theta \cdot \epsilon_{G,\flat,0})|_{I_F} \\
    &\sim \ld j_{\chi''_{H,G}} \circ \ld(\theta \cdot \epsilon_{G,\flat,0}\epsilon_{G,\flat,1}\epsilon_{H,\flat,1}\epsilon_{G,\flat,2}\epsilon_{H,\flat,2})|_{I_F} \\
    &\sim \ld j_{\chi_{H,G}} \circ \ld(\theta \cdot \epsilon_{G,\flat,0}\epsilon_{G,\flat,1}\epsilon_{H,\flat,1}\epsilon_{G,\flat,2}\epsilon_{H,\flat,2})|_{I_F} \\
    &\sim \ld j_{H,G} \circ \ld j_{\min \chi''_H} \circ \ld(\theta \cdot \epsilon_{G,\flat,0}\epsilon_{G,\flat,1}\epsilon_{H,\flat,1}\epsilon_{G,\flat,2}\epsilon_{H,\flat,2})|_{I_F} \\
    &\sim \ld j_{H, G} \circ \ld j_{\chi''_H} \circ \ld(\theta \cdot \epsilon_{G,\flat}\epsilon_{H,\flat})|_{I_F},
\end{align*}
where $\epsilon_{G,\flat}$ and $\epsilon_{H,\flat}$ are as in \eqref{eqn:eps_flat}. By \eqref{eqn:def-of-epsilon}, \eqref{eqn:eps_flat}, and Lemma~\ref{lemma:fks-unram-levi}, we obtain
\[
\epsilon_{G,\flat}\epsilon_{H,\flat}
=\epsilon_G\epsilon_H\epsilon_{G,\sharp,x}\epsilon_{H,\sharp,x},
\]
as required. The calculation in the case that $T$ is elliptic is identical.
\end{proof}

\subsection{Base change}

Now we move on to the case of base change functoriality. In this section, we continue to assume $p \neq 2$. For any finite extension $E/F$ and any weight $\alpha$, write $\F_{E,\alpha}$ and $\F_{E,\pm\alpha}$ for the residue fields of $E_\alpha$ and $E_{\pm\alpha}$, respectively.

\begin{lemma}\label{lemma:fks-base-change}
Let $T \subset M$ be a tamely ramified maximal $F$-torus such that $x\in\cB(T)$, and let $E/F$ be a finite unramified extension of odd degree. Then
\[
\epsilon_x^{G/M} \circ \Nm_{E/F}|_{T(E)_{[x]}} = \epsilon_x^{G_E/M_E}|_{T(E)_{[x]}}.
\]
\end{lemma}

\begin{proof}
The proof is given in Appendix~\ref{app:fks-base-change-proof}.
\end{proof}

\begin{lemma}\label{lemma:canonical-l-embedding-base-change}
Let $H \subset G$ be an unramified twisted Levi $F$-subgroup, and let $E/F$ be a finite extension. Recall from \S\ref{ssec:notation} that $E_2/E$ denotes the unramified extension of degree $2$. Then
\[
\ld j_{H_E,G_E}|_{\wh H\rtimes W_{E_2}} \sim \ld j_{H,G}|_{\wh H\rtimes W_{E_2}}.
\]
\end{lemma}

\begin{proof}
By twisting, we may assume that $G$ is quasi-split. In this case, observe that each $\chi_\alpha$ is unramified of order at most $2$ and the maps $s$ and $n$ are essentially independent of $F$, so the result follows from the definitions \eqref{eqn:def-of-Q} and \eqref{eqn:def-of-L-embedding}.
\end{proof}

\begin{remark}\label{remark:canonical-embedding-incompatible}
The restriction to $W_{E_2}$ in Lemma~\ref{lemma:canonical-l-embedding-base-change} may look strange, but it is necessary in general. For instance, let $G=\PGL_2$ and let $T\subset G$ be an unramified elliptic maximal $F$-torus. If $E=F_2$, the representatives furnished by the construction satisfy $\ld j_{T_E,G_E}(1,w)=(1,w)$ for every $w\in W_E$, since $T_E$ is split. On the other hand, if $w\in W_F$ has image $\Fr^2$ in $W_F/I_F$, then
\[
\ld j_{T,G}(1,w)=(-1,w).
\]
We expect a more natural-looking compatibility if one uses the c-group in place of the L-group, but the difference will not ultimately be relevant to this paper.
\end{remark}

Let $E/F$ be a finite separable, tamely ramified field extension. Choose a maximally unramified maximal $F$-torus $T \subset G^0$, and let $\chi''_F$ (resp.\ $\chi''_E$) denote the set of $\chi$-data for $\Phi(G_{\ol F}, T_{\ol F})$ associated to the tuple $(\vec{G},x,\vec{r},\vec{X})$ (resp.\ $(\vec{G}_E,x,\vec{r},\vec{X})$). Let $\ld j_F\co \ld T \to \ld G$ (resp.\ $\ld j_E\co \ld(T_E) \to \ld(G_E)$) denote the L-embedding associated to $\chi''_F$ (resp.\ $\chi''_E$) by the mechanism of \S\ref{ss:chi-data-L-embedding}. Note that since the characters associated to $\chi''_F$ and $\chi''_E$ are valued in $\mu_4$, we may regard $\ld j_F$ and $\ld j_E$ as homomorphisms of $\ol\Z[1/p]$-group schemes. In particular, we may consider the base change $k$-homomorphisms $(\ld j_F)_k$ and $(\ld j_E)_k$ for our chosen coefficient ring $k$.

Put $\psi_E = \psi\circ\Tr_{E/F}$, and let $\psi_E^0\co \F_E\to k^\times$ be the induced character of the residue field; this is the character relative to which we define the L-embedding $\ld j_E$. 

\begin{lemma}\label{lemma:tasho-base-change}
Suppose that $E/F$ is a cyclic extension of prime degree $\ell$.
\begin{enumerate}
    \item If $\ell=2$, then $(\ld j_E)_{\ol\bF_2}\sim(\ld j_F)_{\ol\bF_2}|_{\ld(T_E)}$. If moreover $E/F$ is unramified, then $\ld j_E|_{\wh T \rtimes I_F}\sim\ld j_F|_{\wh T \rtimes I_F}$.
    \item If $\ell$ is odd and $E/F$ is unramified, then $\ld j_E\sim\ld j_F|_{\ld(T_E)}$.
\end{enumerate}
\end{lemma}

\begin{proof}
We may and do assume that $G$ is quasi-split. Note that the characters associated to $\chi''_F$ and $\chi''_E$ are $\mu_4$-valued, so the products $Q$ associated to both by \eqref{eqn:def-of-Q} are trivial after reduction modulo $2$; the factor $s$ is trivial modulo $2$ because it is a product of values $\wh\alpha^\vee(-1)$. Finally, the Tits section $n\co \Omega(\wh G,\wh T)\to N_{\wh G}(\wh T)(\F_2)$ is a homomorphism.  Formula~\eqref{eqn:def-of-L-embedding} therefore proves the first claim of (1).

For the second claim of (1), choose the systems of representatives used in \eqref{eqn:def-of-Q} compatibly with the decomposition of every $\Gal(\ol F/F)$-orbit into $\Gal(\ol F/E)$-orbits. Note that for $w \in I_F$, the Weyl element $\gamma_{T,G}$, its Tits lift $n(\gamma_{T,G})$, and the factor $s(w)$ in \eqref{eqn:def-of-L-embedding} are the same over $F$ and over $E$. Thus norm functoriality in local class field theory reduces us to the claim
\[
    \chi''_{F,\alpha}\circ \Nm_{E_\alpha/F_\alpha}|_{\cO_{E_\alpha}^\times}
    =\chi''_{E,\alpha}|_{\cO_{E_\alpha}^\times}
\]
for every absolute root $\alpha$. In fact, by definition of $\chi_{F,\alpha}''$ and $\chi_{E,\alpha}''$, it is enough to show
\begin{equation}\label{eqn:chi-data-unramified-base-change-restricted-inertia}
    \chi''_{F,\alpha_i}\circ \Nm_{E_{\alpha_i}/F_{\alpha_i}}|_{\cO_{E_{\alpha_i}}^\times}
    =\chi''_{E,\alpha_i}|_{\cO_{E_{\alpha_i}}^\times}
\end{equation}
if $\alpha \in \Phi(G^{i+1}_{\ol F}, T_{\ol F}) - \Phi(G^i_{\ol F}, T_{\ol F})$.

If $\alpha_i$ is either asymmetric or symmetric unramified over $F$, then \eqref{eqn:chi-data-unramified-base-change-restricted-inertia} is clear from the definitions (since both sides are trivial), so we may and do assume that $\alpha_i$ is symmetric ramified over $F$. Since $E/F$ is unramified, it follows that $\alpha_i$ is symmetric ramified over $E$, and thus both sides of \eqref{eqn:chi-data-unramified-base-change-restricted-inertia} are the unique nontrivial sign character of $\cO_{E_{\alpha_i}}^\times$, and in particular they are equal.

For (2), arguing as above reduces us to proving
\begin{equation}\label{eqn:chi-data-unramified-base-change-restricted}
    \chi''_{F,\alpha_i}\circ \Nm_{E_{\alpha_i}/F_{\alpha_i}}
    =\chi''_{E,\alpha_i}.
\end{equation}
for every absolute root $\alpha$. Unramified base change of odd degree preserves whether $\alpha_i$ is asymmetric, symmetric unramified, or symmetric ramified.  If it is asymmetric, both characters in \eqref{eqn:chi-data-unramified-base-change-restricted} are trivial.  If it is symmetric unramified, both are the unramified quadratic character: the degree $[E_{\alpha_i}:F_{\alpha_i}]$ is odd, so pullback by the norm preserves the value $-1$ on a uniformizer.

It remains to treat the symmetric ramified case. The pullback $\chi''_{F,\alpha_i}\circ \Nm_{E_{\alpha_i}/F_{\alpha_i}}$ has the required restriction to $\cO_{E_{\alpha_i}}^\times$, because the quadratic character of $\cO_{E_{\alpha_i}}^\times$ is compatible with norms; it has the required restriction to $E_{\pm\alpha_i}^\times$ by norm functoriality for the quadratic characters $\kappa_{\alpha_i}$.  Thus it remains only to check the final condition in \eqref{eqn:ramified-symmetric-chi-data}.

Put $f_{F,\alpha_i} = [\F_{\alpha_i}:\F]$ and $f_{E,\alpha_i} = [\F_{E,\alpha_i}: \F_E]$, and let $m = [E_{\alpha_i}:F_{\alpha_i}]$. The Hasse--Davenport relation gives the identity of Gauss sums
\[
    -\frG_{\F_{E,\alpha_i}}(\psi_E^0)
    =\left(-\frG_{\F_{\alpha_i}}(\psi^0)\right)^m.
\]
Since $m$ is odd, we have $f_{F,\alpha_i}\equiv f_{E,\alpha_i}\pmod 2$ and thus \eqref{eqn:ramified-symmetric-chi-data} gives
\begin{align*}
    (\chi''_{F,\alpha_i}\circ \Nm_{E_{\alpha_i}/F_{\alpha_i}})
    (\ell_{G,p'}(\alpha^\vee)a_\alpha)
    &=\chi''_{F,\alpha_i}(\ell_{G,p'}(\alpha^\vee)a_\alpha)^m \\
    &= \left((-1)^{f_{F,\alpha_i} +1}\frG_{\F_{\alpha_i}}(\psi^0)\right)^m \\
    &=(-1)^{f_{E,\alpha_i} +1}\frG_{\F_{E,\alpha_i}}(\psi_E^0).
\end{align*}
This is precisely the remaining defining condition, so \eqref{eqn:chi-data-unramified-base-change-restricted} follows.
\end{proof}

We now deal with the analogue of Lemma~\ref{lemma:tasho-base-change} for ramified extensions of odd degree. Suppose that $G^0 = T$, so the image of the point $x$ in $\cB((G^0_{\der})_E)$ is a vertex. Let $\epsilon_F$ (resp.\ $\epsilon_E$) denote the sign character associated to the tuple $(\vec G, x, \vec r, \vec X)$ (resp.\ $(\vec G_E, x, \vec r, \vec X_E)$) as in \S\ref{sssec:eps-chars-descent-Levi}.

\begin{prop}\label{prop:ramified-kal-bc}
Let $E/F$ be a cyclic extension of prime degree $\ell$. If $\ell \neq 2$ and $\theta\co T(F)_{[x]} \to k^\times$ is a character and $\theta_E \coloneqq \theta \circ \Nm_{E/F}|_{T(E)_{[x]}}$, then
\[
    \ld j_E \circ \ld(\theta_E\cdot\epsilon_E\epsilon_{E,\sharp,x}\cdot (\epsilon_F\epsilon_{F,\sharp,x}) \circ \Nm_{E/F})|_{I_E} \sim \ld j_F \circ \ld\theta|_{I_E}.
\]
If moreover $T_E$ is elliptic, then
\[
    \ld j_E \circ \ld(\theta_E\cdot\epsilon_E\cdot \epsilon_F \circ \Nm_{E/F}) \sim \ld j_F \circ \ld\theta|_{W_E}.
\]
The same assertions hold if $\ell = 2$ and $k = \ol \F_2$.
\end{prop}

\begin{proof}
If $\ell = 2$ and $k = \ol\F_2$, then the result follows from Lemma~\ref{lemma:tasho-base-change}(1) and the fact that $\epsilon_E$ and $\epsilon_F$ are both trivial; thus one may assume $\ell \neq 2$. If $E/F$ is unramified, then the result follows from Lemma~\ref{lemma:fks-base-change} and Lemma~\ref{lemma:tasho-base-change}(2); thus one may assume that $E/F$ is totally ramified. In this case, the proof is given in Appendix~\ref{app:ramified-bc-proof}.
\end{proof}

\subsection{Composition of canonical L-embeddings}

Now drop the assumption $p \neq 2$. In this subsection, let $M \subset H \subset G$ be unramified twisted Levi $F$-subgroups. Let $Z_M$ and $Z_H$ denote the maximal central $F$-tori of $M$ and $H$, respectively. Recall that we have defined a canonical L-embedding $\ld j_{M, H}\co \ld M \to \ld H$, and similarly for $\ld j_{H,G}$ and $\ld j_{M,G}$ in Definition~\ref{defn:canonical-l-embeddings}. For clarity, let $\chi_{M,H}$, $\chi_{H,G}$, and $\chi_{M,G}$ denote the associated sets of $\chi$-data. We also define a set of $\chi$-data
\[
\chi_{M,H,G} = (\chi_{M,H,G,\alpha})_{\alpha \in \Phi(G_{\ol F}, (Z_M)_{\ol F})}
\]
as follows: for $\alpha \in \Phi(G_{\ol F}, (Z_M)_{\ol F})$, we set
\[
\chi_{M,H,G,\alpha} = \begin{cases}
\chi_{M,H,\alpha} &\text{if } \alpha \in \Phi(H_{\ol F}, (Z_M)_{\ol F}), \\
\chi_{H,G,\alpha_H} \circ \Nm_{F_\alpha/F_{\alpha_H}} &\text{if } \alpha \in \Phi((G/H)_{\ol F}, (Z_M)_{\ol F}),
\end{cases}
\]
where $\alpha_H$ denotes the restriction of $\alpha$ to $Z_H$.

The following two results will not be used in this paper; they will be used in \cite{CF26b} to accommodate inductive arguments.

\begin{lemma}\label{lemma:canonical-change-of-chi-data}
We have $\ld j_{\chi_{M,G}}|_{\wh M \rtimes I_F} \sim \ld j_{\chi_{M,H,G}}|_{\wh M \rtimes I_F}$. If $Z(M)/Z(G)$ is anisotropic, then $\ld j_{\chi_{M,G}} \sim \ld j_{\chi_{M,H,G}}$.
\end{lemma}

\begin{proof}
The proof of Lemma~\ref{lemma:inflation-change} adapts nearly verbatim to this setting, so we omit it.
\end{proof}

The following proposition is a technical extension (in view of Lemma~\ref{lemma:canonical-change-of-chi-data}) of \cite[Proposition 6.9]{Kal21a} to the above choices of $\chi$-data.\footnote{Note that, despite appearances, \cite[Proposition 6.9]{Kal21a} if $M$ is a torus then only applies because of Lemma~\ref{lemma:canonical-change-of-chi-data}; the sets of $\chi$-data implicitly appearing in \cite[Proposition 6.9]{Kal21a} are $\chi_{M,H}$, $\chi_{H,G}$, and $\chi_{M,H,G}$, not $\chi_{M,G}$.}

\begin{prop}\label{prop:composition-of-canonical-l-embeddings}
We have $\ld j_{M,G}|_{\wh M \rtimes I_F} \sim \ld j_{H,G} \circ \ld j_{M,H}|_{\wh M \rtimes I_F}$. If $Z(M)/Z(G)$ is anisotropic, then $\ld j_{M,G} \sim \ld j_{H,G} \circ \ld j_{M,H}$.
\end{prop}

\begin{proof}
The proof is nearly identical to the proof of \cite[Proposition 6.9]{Kal21a}, but we recall the outline for convenience of the reader. We will only deal with the case that $Z(M)/Z(G)$ is anisotropic, the general case being completely similar. By Lemma~\ref{lemma:canonical-change-of-chi-data}, we have $\ld j_{M,G} \sim \ld j_{\chi_{M,H,G}}$. Applying \cite[Lemma 6.4]{Kal21a} twice, we may and do assume that $M$, $H$, and $G$ are quasi-split. Let $Q_{H,G}$, $Q_{M,H}$, and $Q_{M,H,G}$ be the functions as defined in \eqref{eqn:def-of-Q} using $\chi_{H,G}$, $\chi_{M,H}$, and $\chi_{M,H,G}$, respectively. Unraveling the construction in \S\ref{ss:l-embedding-defn} (following the proof of \cite[Proposition 6.9]{Kal21a}), one reduces to showing that the $1$-cocycle $w \mapsto Q_{M,H}(w)Q_{H,G}(w)Q_{M,H,G}(w)^{-1}$ is a $1$-coboundary.

Let $T \subset M$ be a maximal $F$-torus. If $\Sigma = \Gal(\ol F/F) \times \{\pm 1\}$ as in \S\ref{ss:l-embedding-defn}, then $Q_{M,H,G}$ is defined as a product over $\Phi((G/M)_{\ol F}, T_{\ol F})/(\Sigma \ltimes \Omega(M_{\ol F}, T_{\ol F}))$, and $Q_{M,H}$ is the subproduct indexed only over the orbits in $\Phi((H/M)_{\ol F}, T_{\ol F})/(\Sigma\ltimes \Omega(M_{\ol F}, T_{\ol F}))$. Thus $Q_{M,H}Q_{M,H,G}^{-1}$ is equal, in the notation of \cite[\S6]{Kal21a}, to the $1$-cochain from \cite[Definition 3.17]{Kal21a} for the torus $M_{\ab}$ and the set $\Phi((G/H)_{\ol F}, (M_{\ab})_{\ol F})$, and the set of $\chi$-data inflated (in the sense of \cite[Definition 5.16]{Kal21a}) from $\chi_{H,G}$ as above. Also, $Q_{H,G}$ is equal to the $1$-cochain from \cite[Definition 3.17]{Kal21a} for the torus $H_{\ab}$, the set $\Phi((G/H)_{\ol F}, (H_{\ab})_{\ol F})$, and the set of $\chi$-data $\chi_{H,G}$. Thus \cite[Proposition 5.23(2)]{Kal21a} applied to the surjection $M_{\ab} \to H_{\ab}$ shows that $Q_{H,G}$ is cohomologous to $Q_{M,H}Q_{M,H,G}^{-1}$, as desired.
\end{proof}

\section{Yu's construction}\label{ss:yu-construction}

In this section we recall Yu's construction of cuspidal representations from Yu data \cite{Yu01}, following the exposition (but not the notation) of \cite{Fin21b}. We use the ``twisted Yu construction'' from \cite{FKS23}, which incorporates a sign character derived from those recalled in \S\ref{ss:fks} and is better suited for functoriality. With an eye toward later applications, we work with $\ol\Q_\ell$-, $\ol\F_\ell$-, and $\ol\Z_\ell$-coefficients, and we record the behavior of the construction under tamely ramified extensions of $F$ and restriction to twisted Levi subgroups of $G$. Our definitions will be slightly broader than the standard ones because of these concerns.

\subsection{Yu data}\label{sss:yu-data}
Assume $p \neq 2$, let $\ell \neq p$ be another prime, let $k$ be a ring among $\ol\Q_\ell$, $\ol\Z_\ell$, and $\ol\F_\ell$, and use the fixed additive character $\psi\co F \to k^\times$ introduced at the beginning of \S \ref{sec:sign-characters}. As in \S\ref{ssec:notation}, $\cB(G)$ denotes the enlarged Bruhat--Tits building of \(G\) over \(F\). In this paper, a $k$-valued \textit{Yu datum} for $G$ is a $5$-tuple $\Psi = ((G^i)_{0 \leq i \leq d}, x, (r_i)_{0 \leq i \leq d}, \tau, (\phi_i)_{0 \leq i \leq d})$, where
\begin{enumerate}[label=(\alph*)]
\item $G^0 \subset G^1 \subset \cdots \subset G^d = G$ is a sequence of twisted Levi $F$-subgroups\footnote{In most references, it is further assumed that $G^i \neq G^{i+1}$ for all $0 \leq i \leq d-1$. However, relaxing this requirement does not significantly affect the theory, and for the purposes of functoriality with respect to unramified twisted Levi subgroups of $G$, our convention is more convenient. This generalization (with slightly different conditions on the $r_i$ and $\phi_i$) was introduced in \cite[\S 15]{Yu01}, where it was used to analyze the irreducible supercuspidal representations of products.} of $G$ which split over a tamely ramified extension of $F$,
\item $x$ is a point of $\cB(G^0) \subset \cB(G)$,
\item $(r_i)_{0 \leq i \leq d}$ is a sequence of real numbers such that $0 < r_0 < r_1 < \cdots < r_{d-1} \leq r_d$ if $d > 0$, and $0 \leq r_0$ if $d = 0$,
\item $\tau$ is a $k$-representation of $G^0(F)_{[x]}$ (where $G^0(F)_{[x]} \subset G^0(F)$ is the stabilizer of the image $[x]$ of $x$ in the reduced Bruhat--Tits building of $G/F$) which is trivial on $G^0(F)_{x, 0+}$ and whose underlying $k$-module is finite free,
\item $(\phi_i)_{0 \leq i \leq d}$ is a sequence of characters $\phi_i\colon G^i(F) \to k^\times$. For $0 \leq i \leq d-1$, we assume that $\phi_i$ is of depth $r_i$. 

By convention we set $r_{-1}\coloneqq0$ if $d=0$. If $r_{d-1}<r_d$, then $\phi_d$ is of depth $r_d$; if $r_{d-1}=r_d$, then $\phi_d$ is trivial.
\end{enumerate}
The constituents of this $5$-tuple are moreover required to satisfy the following properties:
\begin{enumerate}
\item $Z(G^0)/Z(G)$ is anisotropic,
\item $[x]$ is a vertex in $\cB(G^0_{\der})$,
\item $\tau \otimes_k {\Frac(k)}$ is a cuspidal representation of $G^0(F)_{[x]}/G^0(F)_{x, 0+}$ in the sense of \cite[\S 2.10]{Cot26b}. (See also \cite[Proposition~2.1.4]{Cot26b}.)
\item $\phi_i$ is $G^{i+1}$-generic of depth $r_i$ for $0 \leq i \leq d-1$: this means that $\phi_i$ is of depth $r_i$ and under the Moy--Prasad isomorphism $G^i(F)_{x,r_i}/G^i(F)_{x,r_i+} \cong \frg^i(F)_{x,r_i}/\frg^i(F)_{x,r_i+}$, there is an element $X_i^* \in \Lie^*(G^i)^{G^i}(F)_{-r_i}$ which is $G^{i+1}$-generic in the sense of \cite[\S 8]{Yu01} (see \cite[Remark 4.1.3]{FKS23} for a slight correction) such that $\phi_i|_{G^i(F)_{x,r_i}}$ is induced by $X_i^*$ and the fixed additive character $\psi$.\footnote{Note that if $G^i = G^{i+1}$, then every character of $G^i(F)$ which is induced by an element of $\Lie^*(G^i)^{G^i}$ is $G^{i+1}$-generic. (This can fail; it is essentially Hypothesis $\mathrm{C}(\vec G)$ from \cite[\S 2.6]{HM08}.}
\end{enumerate}
If $\tau \otimes_k \Frac(k)$ is an irreducible representation of $G^0(F)_{[x]}/G^0(F)_{x,0+}$, then we will call $\Psi$ \textit{irreducible}. If $G^i \neq G^{i+1}$ for all $0\leq i<d$, then we will call $\Psi$ \textit{irredundant}.

\begin{remark}
Typically, one only considers Yu data for which $k$ is an algebraically closed field of characteristic $\neq p$ and $\Psi$ is irreducible. Our ultimate interest is in representations arising from irreducible Yu data over $\ol\Q_\ell$, but for the study of L-parameters we will need to understand the modular reductions of $\ol\Z_\ell$-lattices in these; such modular reductions may a priori arise from Yu's construction applied to Yu data which are not irreducible.
\end{remark}

From a Yu datum $\Psi$, \cite{Yu01} constructs a compact-mod-center open subgroup $K \subset G(F)$, a $k$-representation $\rho$ of $K$, and then \cite{FKS23} defines a sign character $\epsilon\co K \to k^\times$ and sets
\[
\pi(\Psi) \coloneqq \cInd_K^{G(F)}(\epsilon \otimes \rho).
\]
We will now recall the definitions of $K$, $\rho$, and $\epsilon$.

\subsection{The inducing subgroup and extended characters}\label{sss:yu-characters}
From a $k$-valued Yu datum $\Psi$, one constructs a smooth representation $\pi=\cInd_K^{G(F)}(\epsilon\otimes\rho)$ as follows. First, we define a compact-mod-center open subgroup $K \subset G(F)$ by
\begin{equation}\label{eq:K}
K = K(\vec{G}, x, \vec{r}) = G^0(F)_{[x]}G^1(F)_{x, \frac{r_0}{2}} \cdots G^d(F)_{x, \frac{r_{d-1}}{2}}.
\end{equation}
See Figure \ref{fig:K} for an illustration of $K$.

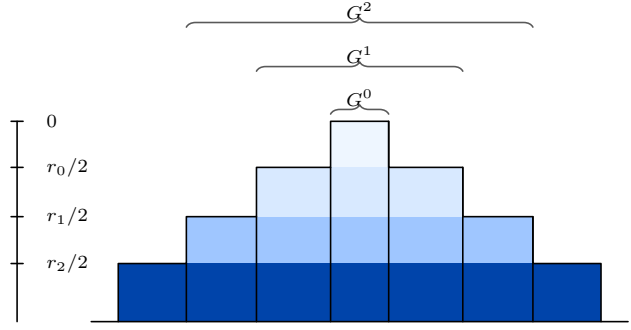
\begin{figure}[!h]
\centering
\begin{tikzpicture}[
line cap=round,
line join=round,
x=0.53cm,
y=0.53cm,
outline/.style={draw=black, line width=0.6pt},
brace/.style={decorate, decoration={brace, amplitude=3.5pt}, draw=black!65, line width=0.6pt},
every node/.style={font=\scriptsize}
]

\definecolor{depthzero}{HTML}{EEF6FF}
\definecolor{depthone}{HTML}{D8E9FF}
\definecolor{depthtwo}{HTML}{9FC6FF}
\definecolor{depththree}{HTML}{0044AA}

\coordinate (a0) at (-6.025,0);
\coordinate (a1) at (-6.025,1.45);
\coordinate (b1) at (-4.325,1.45);
\coordinate (b2) at (-4.325,2.62);
\coordinate (c2) at (-2.575,2.62);
\coordinate (c3) at (-2.575,3.85);
\coordinate (d3) at (-0.725,3.85);
\coordinate (d4) at (-0.725,5.00);
\coordinate (e4) at (0.725,5.00);
\coordinate (e3) at (0.725,3.85);
\coordinate (f3) at (2.575,3.85);
\coordinate (f2) at (2.575,2.62);
\coordinate (g2) at (4.325,2.62);
\coordinate (g1) at (4.325,1.45);
\coordinate (h1) at (6.025,1.45);
\coordinate (h0) at (6.025,0);

\begin{scope}
\clip (a0) -- (a1) -- (b1) -- (b2) -- (c2) -- (c3) -- (d3) -- (d4)
    -- (e4) -- (e3) -- (f3) -- (f2) -- (g2) -- (g1) -- (h1) -- (h0) -- cycle;
\fill[depthzero]  (-6.025,3.85) rectangle (6.025,5.00);
\fill[depthone]   (-6.025,2.62) rectangle (6.025,3.85);
\fill[depthtwo]   (-6.025,1.45) rectangle (6.025,2.62);
\fill[depththree] (-6.025,0.00) rectangle (6.025,1.45);
\end{scope}

\draw[outline] (-6.7,0) -- (6.7,0);
\draw[outline] (a0) -- (a1) -- (b1) -- (b2) -- (c2) -- (c3) -- (d3) -- (d4)
-- (e4) -- (e3) -- (f3) -- (f2) -- (g2) -- (g1) -- (h1) -- (h0);
\foreach \x/\ytop in {-4.325/1.45,-2.575/2.62,-0.725/3.85,0.725/5.00,2.575/3.85,4.325/2.62}
\draw[outline] (\x,0) -- (\x,\ytop);

\draw[outline] (-8.55,0) -- (-8.55,5.00);
\foreach \y in {5.00,3.85,2.62,1.45}
\draw[outline] (-8.70,\y) -- (-8.40,\y);
\node[anchor=west] at (-8.05,5.00) {$0$};
\node[anchor=west] at (-8.05,3.85) {$r_0/2$};
\node[anchor=west] at (-8.05,2.62) {$r_1/2$};
\node[anchor=west] at (-8.05,1.45) {$r_2/2$};

\draw[brace] (-0.725,5.18) -- (0.725,5.18);
\node at (0,5.58) {$G^0$};

\draw[brace] (-2.575,6.25) -- (2.575,6.25);
\node at (0,6.65) {$G^1$};

\draw[brace] (-4.325,7.35) -- (4.325,7.35);
\node at (0,7.75) {$G^2$};

\end{tikzpicture}
\caption{Illustration of the subgroup $K$ from \eqref{eq:K}. The vertical axis measures depth, and the horizontal axis records the roots present in $G^i$.}\label{fig:K}
\end{figure}

Next, we recall the definition of the representation $\rho$ as above. We note that $K$ acts on $\tau$ by requiring that $G^1(F)_{x, \frac{r_0}{2}} \cdots G^d(F)_{x, \frac{r_{d-1}}{2}}$ acts trivially. We will define $\rho = \tau \otimes \kappa \otimes \phi_d$, where $\kappa$ is a representation of $K$ which we will define below. Roughly speaking, we have
\[
\kappa = \bigotimes_{i=0}^{d-1} V_{\wh\phi_i},
\]
where each $V_{\wh\phi_i}$ is constructed using a Weil--Heisenberg representation whose central character is derived from the character $\phi_i$.\footnote{This notation is not standard; for example, \cite{HM08} use $\phi_i'$ to denote $V_{\wh\phi_i}$, while \cite{Fin21b} uses $(\omega_i, V_{\omega_i})$.}

Fix a maximal $F$-torus $T' \subset G^0$ which splits over a tamely ramified Galois extension $E'/F$ and satisfies $x \in \cA(T'_{E'})$. As in \cite[\S 4]{Yu01}, for $0 \leq i \leq d-1$ let $\widehat{\phi}_i$ be the unique $k^\times$-valued character of $G^0(F)_{[x]}G^i(F)_{x,0}G(F)_{x,\frac{r_i}{2}+}$ 
satisfying the following two properties:
\begin{itemize}
\item $\widehat{\phi}_i|_{G^0(F)_{[x]}G^i(F)_{x,0}} = \phi_i|_{G^0(F)_{[x]}G^i(F)_{x,0}}$.
\item $\widehat{\phi}_i|_{G(F)_{x, \frac{r_i}{2}+}}$ factors through
\begin{align*}
G(F)_{x, \frac{r_i}{2}+}/G(F)_{x, r_i+} &\cong \mathfrak{g}(F)_{x, \frac{r_i}{2}+}/\mathfrak{g}(F)_{x, r_i+} \\
    &= (\mathfrak{g}^i(F) \oplus \mathfrak{n}^i(F))_{x, \frac{r_i}{2}+}/(\mathfrak{g}^i(F) \oplus \mathfrak{n}^i(F))_{x, r_i+} \\
    &\to \mathfrak{g}^i(F)_{x, \frac{r_i}{2}+}/\mathfrak{g}^i(F)_{x, r_i+} \\
    &\cong G^i(F)_{x, \frac{r_i}{2}+}/G^i(F)_{x, r_i+},
\end{align*}
where $\mathfrak{n}^i$ is defined as $\mathfrak{g} \cap \bigoplus_{\alpha \in \Phi(G, T'_{E'}) - \Phi(G^i, T'_{E'})} (\mathfrak{g}_{E'})_{\alpha}$, and the map $\mathfrak{g}^i(F) \oplus \mathfrak{n}^i(F) \to \mathfrak{g}^i(F)$ is the first projection.
\end{itemize}
Observe that the induced map $\widehat{\phi}_i\co G^i(F)_{x, \frac{r_i}{2}+} \to k^\times$ is equal to $\phi_i|_{G^i(F)_{x,\frac{r_i}{2}+}}$, so the two bullet points are compatible, hence this construction is well-defined. By convention, we set $\wh\phi_d = \phi_d$. 












\subsection{The Weil--Heisenberg factors}\label{sss:yu-heisenberg-weil}
Let $\wt\RR = \RR \cup \{r+\co r \in \RR\}$, ordered such that for $r < s \in \wt\RR$ we have $r < r+ < s$. For $\wt r \geq \wt r' \geq \frac{\wt r}{2} > 0$ in $\wt\RR$ and $1 \leq i\leq d$, let
\begin{align*}
G^i(F)_{x,\wt r, \wt r'} &= G(F) \\
&\quad\cap \langle T'(E')_{x,\wt r}, U_\alpha(E')_{x,\wt r}, U_\beta(E')_{x,\wt r'}\co \\
&\qquad \alpha \in \Phi(G^{i-1}_{E'}, T'_{E'}), \\
&\qquad \beta \in \Phi((G^i/G^{i-1})_{E'}, T'_{E'}) \rangle.
\end{align*}
In particular, define $J_i = J_i(\vec{G},x,\vec{r}) = G^i(F)_{x,r_{i-1}, \frac{r_{i-1}}{2}}$ and $J_i^+ = G^i(F)_{x,r_{i-1},\frac{r_{i-1}}{2}+}$. Observe that
\[
K = G^0(F)_{[x]}J_1 \cdots J_d.
\]
Figure~\ref{fig:Ji} illustrates $J_i$; we learned this picture from Stephen DeBacker.

\begin{figure}[!h]
\centering
\begin{tikzpicture}[
line cap=round,
line join=round,
x=0.60cm,
y=0.60cm,
outline/.style={draw=black, line width=0.6pt},
brace/.style={decorate, decoration={brace, amplitude=3.5pt}, draw=black!65, line width=0.6pt},
every node/.style={font=\scriptsize}
]

\definecolor{Jihalfdepth}{HTML}{D8E9FF}
\definecolor{Jifulldepth}{HTML}{0044AA}

\begin{scope}
\clip (-4.20,0) -- (-4.20,3.00) -- (-1.35,3.00) -- (-1.35,1.35)
-- (1.35,1.35) -- (1.35,3.00) -- (4.20,3.00) -- (4.20,0) -- cycle;
\fill[Jihalfdepth] (-4.20,1.35) rectangle (4.20,3.00);
\fill[Jifulldepth] (-4.20,0) rectangle (4.20,1.35);
\end{scope}

\draw[outline] (-4.70,0) -- (4.70,0);
\draw[outline] (-4.20,0) -- (-4.20,3.00) -- (-1.35,3.00) -- (-1.35,1.35)
-- (1.35,1.35) -- (1.35,3.00) -- (4.20,3.00) -- (4.20,0);
\foreach \x in {-1.35,1.35}
\draw[outline] (\x,0) -- (\x,1.35);

\draw[outline] (-6.55,0) -- (-6.55,4.25);
\foreach \y in {4.25,3.00,1.35}
\draw[outline] (-6.70,\y) -- (-6.40,\y);
\node[anchor=west] at (-6.15,4.25) {$0$};
\node[anchor=west] at (-6.15,3.00) {$r_{i-1}/2$};
\node[anchor=west] at (-6.15,1.35) {$r_{i-1}$};

\draw[brace] (-1.35,4.18) -- (1.35,4.18);
\node at (0,4.58) {$G^{i-1}$};
\draw[brace] (-4.20,5.18) -- (4.20,5.18);
\node at (0,5.58) {$G^i$};

\end{tikzpicture}
\caption{Illustration of $J_i=G^i(F)_{x,r_{i-1},r_{i-1}/2}$. The vertical axis measures depth, and the horizontal axis records the roots present in $G^i$. The roots present in $G^{i-1}$ enter at depth $r_{i-1}$, while the remaining roots enter at depth $r_{i-1}/2$.}
\label{fig:Ji}
\end{figure}
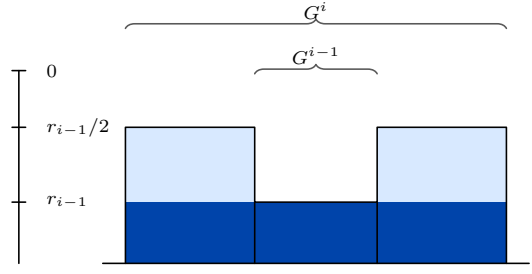

\begin{table}[h]
\centering
\begin{tabular}{|c|c|l|}
\hline 
symbol & range & role \\
\hline
\noalign{\vskip 2pt}
$\wh\phi_i$ & $0 \leq i \leq d-1$ & extension of $\phi_i$ used to define the $(i+1)$st Heisenberg datum \\
$J_i, J_i^+$ & $1 \leq i \leq d$ & compact subgroups attached to the jump from $G^{i-1}$ to $G^i$ \\
$V_i$ & $1 \leq i \leq d$ & symplectic quotient $J_i/J_i^+$, with central character from $\wh\phi_{i-1}$ \\
$B_i$ & $1 \leq i \leq d$ & $\F_q$-valued symplectic form on $V_i$ induced by $X_{i-1}^*$ \\
$V_{\wh\phi_i}$ & $0 \leq i \leq d-1$ & Weil--Heisenberg representation attached to $J_{i+1}$ \\
\noalign{\vskip 2pt}
\hline
\noalign{\vskip 2pt}
\end{tabular}
\caption{Table of objects and indices appearing in Yu's construction.}
\end{table}

Recall that $\ol\psi\co \F_p \to k^\times$ is the nontrivial additive character fixed in \S\ref{ss:fks}. It is proven in \cite[Proposition 11.4]{Yu01} that if we let $V_i = V_i(\vec{G},x,\vec{r}) = J_i/J_i^+$ for $1 \leq i\leq d$ and equip $V_i$ with the pairing $\langle \cdot, \cdot\rangle_i$ defined by $\langle a, b\rangle_i = \ol\psi^{-1}(\widehat{\phi}_{i-1}(aba^{-1}b^{-1}))$, then $(V_i, \langle\cdot, \cdot\rangle_i)$ is a symplectic $\F_p$-vector space.

\begin{lemma}\label{lemma:yu-symplectic-form}
For all $1 \leq i \leq d$, the Moy--Prasad isomorphism 
\[
G^i(F)_{x,\frac{r_{i-1}}{2}}/G^i(F)_{x,\frac{r_{i-1}}{2}+} \cong \frg^i(F)_{x,\frac{r_{i-1}}{2}}/\frg^i(F)_{x,\frac{r_{i-1}}{2}+}
\]
restricts to an isomorphism
\begin{equation}\label{eqn:mp-iso}
V_i \cong \frg^i(F)_{x,\frac{r_{i-1}}{2}}/(\frg^i(F)_{x,\frac{r_{i-1}}{2}+} + \frg^{i-1}(F)_{x,\frac{r_{i-1}}{2}}).
\end{equation}
Under this isomorphism, for all $v, w \in \frg^i(F)_{x,\frac{r_{i-1}}{2}}$ we have
\begin{equation}\label{eqn:yu-symplectic}
\langle \ol v, \ol w\rangle_i = \Tr_{\F_q/\F_p}\overline{X_{i-1}^*([v, w])},
\end{equation}
where the bar denotes reduction modulo $\frp_F$ of the element $X_{i-1}^*([v,w])\in\cO_F$.
\end{lemma}

\begin{proof}
The identity \eqref{eqn:mp-iso} is standard and is observed in the proof of \cite[Lemma 11.1]{Yu01}. The identity
\[
\ol\psi(\langle\ol v, \ol w\rangle_i) = \psi(X_{i-1}^*([v,w]))
\]
is also observed in the proof of \cite[Lemma 11.1]{Yu01}, and \eqref{eqn:yu-symplectic} follows from the fact that $\ol\psi$ is injective and $\psi^0=\ol\psi\circ\Tr_{\F_q/\F_p}$.
\end{proof}

Let $B_i = B_i(\vec{G},x,\vec{r})\co V_i \times V_i \to \F_q$ denote the symplectic form which is transported under the Moy--Prasad isomorphism to the symplectic form
\[
(\ol v, \ol w) \mapsto \overline{X_{i-1}^*([v,w])}.
\]
Our interest in this symplectic form (as opposed to $\langle\cdot,\cdot\rangle_i$, which is $\F_p$-valued) comes from the following lemma. Recall that $F_r$ denotes the unramified extension of $F$ of degree $r$.

\begin{lemma}\label{lemma:algebraic-action}
The maps $f_{i,r}\co \ol G^0_{[x]}(\F_{q^r}) \to \Sp(V_i(\vec{G}_{F_r}), B_i(\vec{G}_{F_r}))$ (for $r \in \Z_{\geq 1}$) induced by the actions of $G^0(F_r)_{[x]}$ on $J_i(\vec{G}_{F_r})$ and $J_i^+(\vec{G}_{F_r})$ are algebraic, i.e., there is an $\F_q$-homomorphism $\ol G^0_{[x]} \to \mathbf{Sp}(V_i, B_i)$ which restricts to $f_{i,r}$ on $\F_{q^r}$-points for all $r$.
\end{lemma}

\begin{proof}
Put
\[
L_i=\frg(F)_{x,r_{i-1}/2},\qquad
L_i^+=\frg(F)_{x,r_{i-1}/2+},
\]
considered as vector groups over $\cO_F$, and let $\cG^0_{[x]}$ be the smooth separated $\cO_F$-model of $G^0$ with $\cG^0_{[x]}(\cO_{F^{\unr}})=G^0(F^{\unr})_{[x]}$. The relative identity component of $\cG^0_{[x]}$ is affine and open, and its cosets form an affine open cover whose $\cO_{F^{\unr}}$-points cover $\cG^0_{[x]}(\cO_{F^{\unr}})$. Applying \cite[Corollary~2.10.10]{KP} to each coset, we see that the maps $G^0(F^{\unr})_{[x]} \to \Aut(L_i \otimes \cO_{F^{\unr}})$ and $G^0(F^{\unr})_{[x]} \to \Aut(L_i^+ \otimes \cO_{F^{\unr}})$ induced by the adjoint action extend uniquely to homomorphisms of $\cO_{F^{\unr}}$-group schemes $(\cG^0_{[x]})_{\cO_{F^{\unr}}} \to \GL(L_i \otimes \cO_{F^{\unr}})$, and similarly for $L_i^+$. These maps are Galois-invariant, and hence they are defined over $\cO_F$ by descent.
Since $\varpi L_i\subset L_i^+$, these maps induce an algebraic action of $(\cG_{[x]}^0)_{\F_q}$ on $L_i/L_i^+$. Thus the action factors through $\ol G^0_{[x]}$. The inverse image of the unipotent radical of $(\cG^0_{[x]})_{\F_q}^{\circ}$ is $G^0(F^{\unr})_{x,0+}$, which acts trivially on $L_i/L_i^+$ by the Moy--Prasad commutator relation. Since $G^0\subset G^{i-1}\subset G^i$, equation~\eqref{eqn:mp-iso} induces an algebraic action of $\ol G^0_{[x]}$ on $V_i$. Finally, the $G^{i-1}$-invariance of $X_{i-1}^*$ shows that the action preserves $B_i$. This gives the required algebraic homomorphism, and unramified base change gives the asserted maps on all $\F_{q^r}$-points.
\end{proof}

By \cite[Proposition 11.4]{Yu01} again, for $1\leq i\leq d$ there is a canonical special isomorphism
\begin{equation}\label{eq:special-isom}
j_i = j_i(\vec{G},x,\vec{r},\vec{\phi}) \colon J_i/\bigl(J_i^+ \cap \ker \widehat{\phi}_{i-1}\bigr) \to V_i^\sharp,
\end{equation}
where $V_i^\sharp$ is the Heisenberg group corresponding to $(V_i, \frac{1}{2}\langle\cdot, \cdot\rangle)$. In other words, $V_i^\sharp$ is the group with underlying set $V_i \times \F_p$ and group law $(v, a)\cdot(w, b) = (v + w, a + b + \frac{1}{2}\langle v, w\rangle_i)$.

For $0\leq i\leq d-1$, let $V_{\widehat{\phi}_i}$ denote the Heisenberg representation of $J_{i+1}/(J_{i+1}^+ \cap \ker \widehat{\phi}_i)$ with central character the character of $J_{i+1}^+/(J_{i+1}^+ \cap \ker \widehat{\phi}_i)$ induced by $\widehat{\phi}_i|_{J_{i+1}^+}$. By \eqref{eq:special-isom} applied with index $i+1$, the group $J_{i+1}/(J_{i+1}^+ \cap \ker \widehat{\phi}_i)$ is identified with the Heisenberg group $V_{i+1}^\sharp$ attached to the symplectic space $V_{i+1}=J_{i+1}/J_{i+1}^+$. When $k$ is a field, $V_{\widehat{\phi}_i}$ is the unique irreducible $k$-representation of $V_{i+1}^\sharp$ with central character $\wh\phi_i$. For $k=\ol\Z_\ell$, take a $V_{i+1}^\sharp$-stable lattice in the corresponding Heisenberg representation over $\ol\Q_\ell$. Existence and uniqueness up to isomorphism follow from \cite[Part III, no.\ 15.5, Proposition 43]{Serre77} and the fact that $V_{i+1}^\sharp$ is a $p$-group and $\ell \neq p$.

The representation $\kappa$ is defined to have underlying space $\bigotimes_{i=0}^{d-1} V_{\widehat{\phi}_i}$; we will define the action of $K$ separately on each $V_{\widehat{\phi}_i}$. We let $J_{i+1}$ act on $V_{\widehat{\phi}_i}$ through its quotient $V_{i+1}^\sharp$. If $j \neq i$, let $J_{i+1}$ act on $V_{\widehat{\phi}_j}$ through $\widehat{\phi}_j|_{J_{i+1}}$. Moreover, $G^0(F)_{[x]}$ acts on $V_{\widehat{\phi}_i}$ through $\phi_i \cdot W_i$, where $W_i$ is the representation
\[
G^0(F)_{[x]}/G^0(F)_{x,0+} \to \Sp(V_{i+1}) \to \GL(V_{\widehat{\phi}_i}),
\]
where $\Sp(V_{i+1}) \to \GL(V_{\widehat{\phi}_i})$ is the Weil representation, as defined in \cite[Theorem 2.4]{Ger77}. (See also \cite[Lemma 2.5]{Fin22} for an explanation of the existence of these representations over $\ol\Z_\ell$, and by extension $\ol\F_\ell$.)

\subsection{Compatibility with tame base change and restriction to twisted Levis}

\begin{ansatz}\label{ansatz:yu-functoriality-setup}
In the functoriality statements below, let $H \subset G$ be a twisted Levi $F$-subgroup containing a tamely ramified elliptic maximal $F$-torus in common with $G^0$, assume $x \in \cB(H)$, and let $E/F$ be a finite tamely ramified Galois extension. We set $H^i = G^i \cap H$ for all $i$ and write $\vec H = (H^i)_{0 \leq i \leq d}$. If $\phi_i$ is a character of $G^i(F)$, we write $\phi_{H,i} = \phi_i|_{H^i(F)}$.
\end{ansatz}

\begin{lemma}\label{lemma:yu-subgroup-fixed}
Assume we are in the setting of Ansatz~\ref{ansatz:yu-functoriality-setup}.
\begin{enumerate}
    \item $K({\vec{G}},x,\vec{r}) \cap H(F) = K(\vec{H},x,\vec{r})$.
    \item $K(\vec{G}_E, x_E, \vec{r})^{\Gal(E/F)} = K(\vec{G}, x, \vec{r})$.
\end{enumerate}
\end{lemma}

\begin{proof}
For (2), we induct on $d$, the case $d = 0$ being clear. If $d > 0$, then there is a short exact sequence
\begin{equation}\label{eqn:yu-group-bc-1}
1 \to G^d(E)_{x,\frac{r_{d-1}}{2}} \to K(\vec{G}_E,x_E,\vec{r}) \to K((G^i_E)_{0 \leq i < d}, x_E, (r_i)_{0\leq i < d})/G^{d-1}(E)_{x,\frac{r_{d-1}}{2}} \to 1.
\end{equation}
There is also a tautological short exact sequence
\begin{equation}\label{eqn:yu-group-bc-2}
1 \to G^{d-1}(E)_{x,\frac{r_{d-1}}{2}} \to K((G^i_E)_{0 \leq i < d}, x_E, (r_i)_{0\leq i < d}) \to K((G^i_E)_{0 \leq i < d}, x_E, (r_i)_{0\leq i < d})/G^{d-1}(E)_{x,\frac{r_{d-1}}{2}} \to 1.
\end{equation}
Taking $\Gal(E/F)$-fixed points in \eqref{eqn:yu-group-bc-1} and \eqref{eqn:yu-group-bc-2}, applying \cite[Proposition 2.2]{Yu01}, and using the fact that the Moy--Prasad isomorphism restricts well with respect to finite field extensions of $F$ (which follows easily from the proof of \cite[Theorem 13.5.1]{KP}), we obtain
\begin{equation}\label{eqn:yu-bc-ses}
1 \to G^d(F)_{x,\frac{r_{d-1}}{2}} \to K(\vec{G}_E,x_E,\vec{r})^{\Gal(E/F)} \to K((G^i_E)_{0 \leq i < d}, x_E, (r_i)_{0\leq i < d})^{\Gal(E/F)}/G^{d-1}(F)_{x,\frac{r_{d-1}}{2}} \to 1.
\end{equation}
By induction, we have
\begin{equation}\label{eqn:yu-bc-equality}
K((G^i_E)_{0 \leq i < d}, x_E, (r_i)_{0\leq i < d})^{\Gal(E/F)} = K((G^i)_{0\leq i < d}, x, (r_i)_{0\leq i<d}).
\end{equation}
Since clearly $K(\vec{G}, x, \vec{r}) \subset K(\vec{G}_E, x_E, \vec{r})^{\Gal(E/F)}$, it follows from \eqref{eqn:yu-bc-ses} and \eqref{eqn:yu-bc-equality} that this inclusion is an equality, as desired.

For (1), let $Z$ denote the maximal central torus of $H$. By (2), we may pass to a finite tamely ramified extension of $F$ to assume that there is an element $z \in Z(F)$ of finite order prime to $p$ such that $Z_G(z) = H$. Then the claim is that
\begin{equation}\label{eqn:yu-levi-equality}
K(\vec{G}, x, \vec{r})^z = K(\vec{H}, x, \vec{r}).
\end{equation}
Since $z$ is of order prime to $p$, the set $\rH^1(z, G^i(F)_{x, \frac{r_{i-1}}{2}})$ is a singleton. Moreover, we have $(G^i(F)_{x,\frac{r_{i-1}}{2}})^z = H^i(F)_{x,\frac{r_{i-1}}{2}}$ for all $i$ and $(G^0(F)_{[x]})^z = H^0(F)_{[x]}$, so the same argument as in the previous paragraph allows us to conclude inductively that \eqref{eqn:yu-levi-equality} holds.
\end{proof}

To study the behavior of the $\wh\phi_i$ under base change to a tamely ramified Galois extension $E/F$, we need to extend each $\phi_i$ from $G^i(F)$ to a $\Gal(E/F)$-invariant character of $G^i(E)$; this is not possible in complete generality, and it will lead to some small restrictions on $p$.

\begin{lemma}\label{lemma:yu-hat-functoriality}
Assume we are in the setting of Ansatz~\ref{ansatz:yu-functoriality-setup}.
\begin{enumerate}
    \item $\wh\phi_{H,i} = \wh\phi_i|_{H^0(F)_{[x]}H^i(F)_{x,0}H(F)_{x,\frac{r_i}{2}+}}$ for all $0 \leq i \leq d-1$.
    \item If $\phi_i$ extends to a character $\phi_{E,i}$ of $G^i(E)$ which is generic of depth $r_i$, then $\wh\phi_i = \wh\phi_{E,i}|_{G^0(F)_{[x]}G^i(F)_{x,0}G(F)_{x,\frac{r_i}{2}+}}$ for all $0 \leq i \leq d-1$.
\end{enumerate}
\end{lemma}

\begin{proof}
We begin with (1). It is clear that $\wh\phi_i|_{H^0(F)_{[x]}H^i(F)_{x,0}H(F)_{x,\frac{r_i}{2}+}}$ satisfies the first defining property in \S\ref{sss:yu-characters}. Compatibility of the Moy--Prasad isomorphism with restriction to twisted Levi subgroups shows that the map
\[
H(F)_{x,\frac{r_i}{2}+}/H(F)_{x,r_i+} \to H^i(F)_{x,\frac{r_i}{2}+}/H^i(F)_{x,r_i+}
\]
defined there is the restriction of the corresponding map for $G$. Since $\phi_{H,i}$ is the restriction of $\phi_i$, it is therefore clear that the second defining property in \S\ref{sss:yu-characters} holds as well. The argument for (2) is completely similar, using compatibility of the Moy--Prasad isomorphism with tame extensions, as in \cite[Corollary 2.4]{Yu01}.
\end{proof}

\begin{lemma}\label{lemma:yu-Ji-fixed}
With notation and assumptions as in Ansatz~\ref{ansatz:yu-functoriality-setup}, we have
\begin{enumerate}
    \item $J_i(\vec{G}, x, \vec{r}) \cap H(F) = J_i(\vec{H}, x, \vec{r})$ and similarly for $J_i^+$,
    \item $J_i(\vec{G}_E,x_E,\vec{r})^{\Gal(E/F)} = J_i(\vec{G},x,\vec{r})$ and similarly for $J_i^+$.
\end{enumerate}
\end{lemma}

\begin{proof}
The proof is completely analogous to that of Lemma~\ref{lemma:yu-subgroup-fixed}.
\end{proof}

\begin{lemma}\label{lemma:yu-special-iso-functoriality}
Assume we are in the setting of Ansatz~\ref{ansatz:yu-functoriality-setup}.
\begin{enumerate}
\item Write $\vec\phi_H=(\phi_{H,j})_j$, and let $\iota_{H,i}\colon V_i(\vec H)^\sharp\lhook\joinrel\longrightarrow V_i(\vec G)^\sharp$ be the canonical embedding. Then
\[
\iota_{H,i}\circ j_i(\vec H,x,\vec r,\vec\phi_H)
=j_i(\vec G,x,\vec r,\vec\phi)\big|_{J_i(\vec H,x,\vec r)/(J_i^+(\vec H,x,\vec r)\cap\ker\widehat\phi_{H,i-1})}.
\]
\item If $\vec\phi_E=(\phi_{E,j})_j$ extends $\vec\phi$, with each $\phi_{E,j}$ being $G_E^{j+1}$-generic of depth $r_j$ for $0\leq j<d$, let $\iota_{E,i}\colon V_i(\vec G)^\sharp\lhook\joinrel\longrightarrow V_i(\vec G_E)^\sharp$ be the canonical embedding. Then
\[
\iota_{E,i}\circ j_i(\vec G,x,\vec r,\vec\phi)
=j_i(\vec G_E,x_E,\vec r,\vec\phi_E)\big|_{J_i(\vec G,x,\vec r)/(J_i^+(\vec G,x,\vec r)\cap\ker\widehat\phi_{i-1})}.
\]
\end{enumerate}
\end{lemma}

\begin{proof}
We briefly recall the construction of $j_i$ as in the proof of \cite[Proposition 11.4]{Yu01}. First suppose that $G^i$ is split, let $S \subset G^i$ be a split maximal $F$-torus, and choose a system of positive roots $\Phi^+ \subset \Phi(G^i, S)$. Let $J_i(+)$ (resp.\ $J_i(-)$) denote the subgroup of $J_i$ generated by $U_\alpha(F)_{x,\frac{r_{i-1}}{2}}$ for $\alpha \in \Phi^+$ (resp.\ $\alpha \in -\Phi^+$) not lying in $\Phi(G^{i-1},S)$. If $N_i = J_i^+ \cap \ker\wh\phi_{i-1}$, then the maps $\pi_+\co J_i(+)N_i/N_i \to J_i/J_i^+$ and $\pi_-\co J_i(-)N_i/N_i \to J_i/J_i^+$ are injective, and they form a complete polarization of $J_i/J_i^+$. By \cite[Lemma 10.1]{Yu01}, there is a canonical isomorphism $j_i\co J_i/N_i \to V_i^\sharp$ given by
\[
j_i(w_+w_-c) = (\pi_+(w_+)+\pi_-(w_-), c + \frac{1}{2}\langle\pi_+(w_+), \pi_-(w_-)\rangle)
\]
for $w_+ \in J_i(+)N_i/N_i$, $w_- \in J_i(-)N_i/N_i$, and $c \in J_i^+/N_i$, identified with an element of $\F_p$ via $\ol\psi^{-1}\circ\wh\phi_{i-1}$. It is shown in the proof of \cite[Proposition 11.4]{Yu01} that $j_i$ is independent of the choices of $S$ and $\Phi^+$. If $G^i$ is not necessarily split, let $T \subset G^i$ be a tamely ramified maximal $F$-torus, let $F'/F$ be a finite tamely ramified extension such that $T_{F'}$ is $F'$-split, and let $j_i$ denote the restriction to $J_i/N_i$ of the isomorphism $J_i(F')/N_i(F') \to (J_i(F')/J_i^+(F'))^\sharp$ constructed above. As noted in the proof of \cite[Proposition 11.4]{Yu01} (see also the discussion preceding \cite[Definition 3.15]{HM08}), this gives rise to a special isomorphism over $F$ by \cite[Lemma 10.3]{Yu01} which is independent of the choice of $F'$. From this construction, both points are evident.
\end{proof}

\subsection{The FKS twist and compact induction}\label{sss:yu-data-fks}
Finally, for a Yu datum $\Psi$, we define the character $\epsilon$ of $K$ to be trivial on $G^i(F)_{x,\frac{r_{i-1}}{2}}$ for $1 \leq i \leq d$, and to restrict to the character $\epsilon$ on $G^0(F)_{[x]}$ defined in \eqref{eqn:yu-datum-sign-character} using the twisted Levi sequence $(G^i)$ and the generic elements $X_i^*$ in the definition.

Finally, we define
\begin{equation}\label{eq:yu-compact-induction}
\pi = \pi(\Psi) \coloneqq \cInd_K^{G(F)}(\epsilon \otimes \tau \otimes \bigotimes_{i=0}^{d-1} V_{\widehat{\phi}_i} \otimes \phi_d).
\end{equation}
For $k\in\{\ol\Q_\ell,\ol\F_\ell\}$, \cite[Theorem~3.1]{Fin22} and exactness of compact induction show that $\pi(\Psi)$ has finite length; it is irreducible and cuspidal when $k \in \{\ol\Q_\ell, \ol\F_\ell\}$ and $\Psi$ is irreducible. If moreover $p$ does not divide the order of the absolute Weyl group of $G$, then \cite[Theorem~4.1]{Fin22} shows that every irreducible cuspidal smooth $k$-representation of $G(F)$ arises from this construction.

\subsection{Equivalence and normalization of Yu data}\label{sss:hm}

It is natural to ask when $\pi(\Psi) \cong \pi(\Psi')$ for two different Yu data $\Psi, \Psi'$. This leads to the notion of \textit{$G$-equivalence} on irredundant Yu data, an equivalence relation introduced by Hakim--Murnaghan in \cite[Definition 6.1]{HM08}, whose definition we will shortly recall.

\subsubsection{Removing redundant stages} Recall from \S\ref{sss:yu-data} that a Yu datum is irredundant if $G^i\neq G^{i+1}$ for all $0\leq i<d$. First, we show that one can reduce to irredundant Yu data.

\begin{lemma}\label{lemma:irredundant-yu-data}
Let 
\[
\Psi = ((G^i)_{0 \leq i \leq d}, x, (r_i)_{0\leq i\leq d}, \tau, (\phi_i)_{0\leq i\leq d})
\]
be a Yu datum for $G$. Suppose that $G^m = G^{m+1}$ for some $0\leq m < d$, and let 
\[
\Psi' = ((G'^i)_{0\leq i\leq d-1}, x, (r'_i)_{0\leq i\leq d-1}, \tau, (\phi'_i)_{0\leq i\leq d-1})
\]
be the Yu datum defined by
\begin{enumerate}
    \item $G'^i = G^i$ for $0\leq i< m$ and $G'^i = G^{i+1}$ for $m\leq i\leq d-1$,
    \item $r'_i = r_i$ for $0\leq i<m$ and $r'_i = r_{i+1}$ for $m\leq i\leq d-1$,
    \item $\phi'_i=\phi_i$ for $0\leq i<m$ and $\phi'_m = \phi_m\phi_{m+1}$ and $\phi'_i = \phi_{i+1}$ for $m < i\leq d-1$.
\end{enumerate}
    Then $\pi(\Psi) \cong \pi(\Psi')$. In particular, there is an irredundant Yu datum $\Psi_0$ such that $\pi(\Psi) \cong \pi(\Psi_0)$.
\end{lemma}

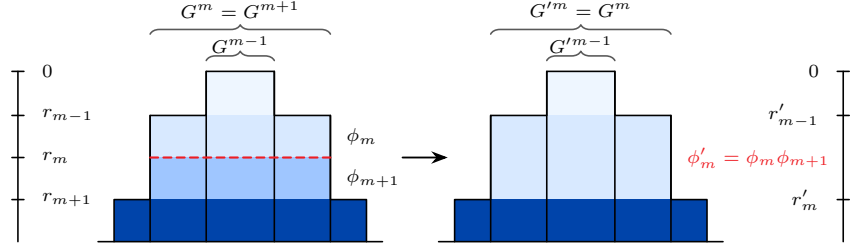
\begin{figure}[H]
\centering
\begin{tikzpicture}[
line cap=round,
line join=round,
x=0.53cm,
y=0.53cm,
outline/.style={draw=black, line width=0.6pt},
brace/.style={decorate, decoration={brace, amplitude=3.5pt}, draw=black!65, line width=0.6pt},
every node/.style={font=\scriptsize}
]

\definecolor{collapsezero}{HTML}{EEF6FF}
\definecolor{collapseone}{HTML}{D8E9FF}
\definecolor{collapsetwo}{HTML}{9FC6FF}
\definecolor{collapsethree}{HTML}{0044AA}

\draw[outline] (-5.55,0) -- (-5.55,4.25);
\foreach \y in {4.25,3.15,2.10,1.05}
\draw[outline] (-5.70,\y) -- (-5.40,\y);
\node[anchor=west] at (-5.18,4.25) {$0$};
\node[anchor=west] at (-5.18,3.15) {$r_{m-1}$};
\node[anchor=west] at (-5.18,2.10) {$r_m$};
\node[anchor=west] at (-5.18,1.05) {$r_{m+1}$};

\begin{scope}[shift={(0,0)}]
\begin{scope}
\clip (-3.15,0) -- (-3.15,1.05) -- (-2.25,1.05) -- (-2.25,3.15) -- (-0.85,3.15)
    -- (-0.85,4.25) -- (0.85,4.25) -- (0.85,3.15) -- (2.25,3.15)
    -- (2.25,1.05) -- (3.15,1.05) -- (3.15,0) -- cycle;
\fill[collapsezero] (-3.15,3.15) rectangle (3.15,4.25);
\fill[collapseone] (-3.15,2.10) rectangle (3.15,3.15);
\fill[collapsetwo] (-3.15,1.05) rectangle (3.15,2.10);
\fill[collapsethree] (-3.15,0) rectangle (3.15,1.05);
\end{scope}
\draw[outline] (-3.55,0) -- (3.55,0);
\draw[outline] (-3.15,0) -- (-3.15,1.05) -- (-2.25,1.05) -- (-2.25,3.15) -- (-0.85,3.15)
-- (-0.85,4.25) -- (0.85,4.25) -- (0.85,3.15) -- (2.25,3.15)
-- (2.25,1.05) -- (3.15,1.05) -- (3.15,0);
\foreach \x/\ytop in {-2.25/3.15,-0.85/4.25,0.85/4.25,2.25/3.15}
\draw[outline] (\x,0) -- (\x,\ytop);
\draw[Red, densely dashed, line width=0.8pt] (-2.25,2.10) -- (2.25,2.10);
\node[anchor=west] at (2.42,2.62) {$\phi_m$};
\node[anchor=west] at (2.42,1.58) {$\phi_{m+1}$};
\draw[brace] (-0.85,4.52) -- (0.85,4.52);
\node at (0,4.91) {$G^{m-1}$};
\draw[brace] (-2.25,5.18) -- (2.25,5.18);
\node at (0,5.80) {$G^m=G^{m+1}$};
\end{scope}

\draw[-{Stealth[length=2.2mm]}, line width=0.7pt] (4.05,2.10) -- (5.18,2.10);

\begin{scope}[shift={(8.5,0)}]
\begin{scope}
\clip (-3.15,0) -- (-3.15,1.05) -- (-2.25,1.05) -- (-2.25,3.15) -- (-0.85,3.15)
    -- (-0.85,4.25) -- (0.85,4.25) -- (0.85,3.15) -- (2.25,3.15)
    -- (2.25,1.05) -- (3.15,1.05) -- (3.15,0) -- cycle;
\fill[collapsezero] (-3.15,3.15) rectangle (3.15,4.25);
\fill[collapseone] (-3.15,1.05) rectangle (3.15,3.15);
\fill[collapsethree] (-3.15,0) rectangle (3.15,1.05);
\end{scope}
\draw[outline] (-3.55,0) -- (3.55,0);
\draw[outline] (-3.15,0) -- (-3.15,1.05) -- (-2.25,1.05) -- (-2.25,3.15) -- (-0.85,3.15)
-- (-0.85,4.25) -- (0.85,4.25) -- (0.85,3.15) -- (2.25,3.15)
-- (2.25,1.05) -- (3.15,1.05) -- (3.15,0);
\foreach \x/\ytop in {-2.25/3.15,-0.85/4.25,0.85/4.25,2.25/3.15}
\draw[outline] (\x,0) -- (\x,\ytop);
\node[anchor=west, text=Red] at (2.42,2.10) {$\phi'_m=\phi_m\phi_{m+1}$};
\draw[outline] (6.55,0) -- (6.55,4.25);
\foreach \y in {4.25,3.15,1.05}
\draw[outline] (6.40,\y) -- (6.70,\y);
\node[anchor=east] at (6.18,4.25) {$0$};
\node[anchor=east] at (6.18,3.15) {$r'_{m-1}$};
\node[anchor=east] at (6.18,1.05) {$r'_m$};
\draw[brace] (-0.85,4.52) -- (0.85,4.52);
\node at (0,4.91) {$G'^{m-1}$};
\draw[brace] (-2.25,5.18) -- (2.25,5.18);
\node at (0,5.80) {$G'^m=G^m$};
\end{scope}

\end{tikzpicture}
\caption{Collapsing a redundant step in a Yu datum. The vertical axis records character depth, and the horizontal axis records the roots present in the corresponding twisted Levi.}
\label{fig:normalized-yu-collapse}
\end{figure}

\begin{proof}
The final claim is immediate from the first by induction, so it suffices to prove the first claim. The fact that $\Psi'$ is a Yu datum is trivial.
One checks that $K(\Psi) = K(\Psi')$; moreover, $\epsilon_\Psi = \epsilon_{\Psi'}$, since $\epsilon_x^{G^{m+1}/G^m}$ is trivial. It is just as straightforward to check that the inducing representations $\rho(\Psi)$ and $\rho(\Psi')$ are isomorphic, so we leave this to the reader.
\end{proof}

\subsubsection{$G$-equivalence}\label{sss:G-equivalence} We aim to explicitly describe the equivalence relation on Yu data induced by isomorphism of the corresponding representations.

\begin{defn}\label{def:G-equivalence}
Let $\Psi = ((G^i), x, (r_i), \tau, (\phi_i))$ and $\Psi' = ((G'^i), x', (r'_i), \tau', (\phi_i'))$ be irredundant Yu data over $k$. We say that $\Psi$ and $\Psi'$ are related by an \textit{elementary transformation} if $(G^i) = (G'^i)$, $x = x'$, $(r_i) = (r_i')$, $\tau \cong \tau'$, and $(\phi_i) = (\phi_i')$. We say that $\Psi$ and $\Psi'$ are \textit{$G$-conjugate} if there exists some $g \in G(F)$ such that every term of $\Psi'$ is obtained from the corresponding term of $\Psi$ by $g$-conjugation, where we interpret $g \cdot (r_i) = (r_i)$. We say that $\Psi'$ is a \textit{refactorization} of $\Psi$ if $G^i = G'^i$ for all $i$ and $[x] = [x']$, and the following conditions on the characters $\chi_i\co G^i(F) \to k^\times$ defined by $\chi_i(g) = \prod_{j=i}^d \phi_j(g)\phi_j'(g)^{-1}$ hold:
\begin{enumerate}
\item If $\phi_d = 1$, then $\phi_d' = 1$.
\item The character $\chi_i$ is of depth at most $r_{i-1}$ for all $0 \leq i \leq d$, where $r_{-1} = 0$.
\item $\tau' = \tau \otimes \chi_0$.
\end{enumerate}
Finally, we say that $\Psi$ and $\Psi'$ are \textit{$G$-equivalent} if they are equivalent under the equivalence relation generated by elementary transformations, $G$-conjugacy, and refactorization. 

\end{defn}

By \cite[Corollary 3.5.5]{Kal19} (which builds on \cite[Theorem 6.6]{HM08}), one knows that the supercuspidal representations of $G(F)$ over $k = \ol\Q_\ell$ constructed from irreducible Yu data $\Psi$ and $\Psi'$ are isomorphic if and only if $\Psi$ and $\Psi'$ are $G$-equivalent. We expect that the same claim holds for $k = \ol\F_\ell$, but we have not checked it. However, we will use one direction of this result, proven in \cite[Proposition 4.24]{HM08}. 

\begin{prop}\label{prop:hakim-murnaghan-4.24}
Let $\Psi$ and $\Psi'$ be irredundant Yu data for $G$ over a ring $k$ among $\ol\Q_\ell$, $\ol\Z_\ell$, and $\ol\F_\ell$. If $\Psi$ and $\Psi'$ are $G$-equivalent, then $\pi(\Psi) \cong \pi(\Psi')$.
\end{prop}

\begin{proof}
The proof of \cite[Proposition 4.24]{HM08} is largely self-contained, and it extends nearly verbatim to our setting once Hypothesis $\mathrm{C}(\vec G)$ from \cite[\S 4.3]{HM08} is arranged to hold by passing to a z-extension as in \cite[Lemmas 3.5.2--3.5.4]{Kal19}.
\end{proof}

\subsubsection{Normalized Yu data}
We introduce the following definition of ``normalized'' characters, whose nomenclature was chosen for consistency with the notion of normalized Yu data in \cite[Definition 3.7.1]{Kal19}.

\begin{defn}\label{def:normalized-character}
Let $H$ be a connected reductive $F$-group and let $k$ be a
commutative ring.  A smooth character
$\phi\co H(F)\to k^\times$ is \textit{normalized} if it is trivial on
the image of $H_{\mathrm{sc}}(F)\to H(F)$, where $H_{\mathrm{sc}}$ is
the simply connected cover of $H_{\der}$.
\end{defn}

With this terminology, \cite[Definition 3.7.1]{Kal19} says that a Yu datum $\Psi$ is \textit{normalized} if each $\phi_i$ is normalized as a character of $G^i(F)$. By \cite[Lemma 3.7.2]{Kal19}\footnote{The errata on Kaletha's website notes that Section 3.7 (among others) of \cite{Kal19} does not, but should, require that $p$ is a good prime for $G$ and that $p$ does not divide the order of $\pi_0(Z(G))$. However, this error only arises because the proof of \cite[Lemma 3.6.8]{Kal19} uses a misstated form of \cite[Lemma 8.1]{Yu01} in its proof. The proof of \cite[Lemma 3.7.2]{Kal19} does not use \cite[Lemma 3.6.8]{Kal19}, and it is still correct. Moreover, the proof applies without change over $\ol\Q_\ell$ and $\ol\F_\ell$.}, if $p$ does not divide $|\pi_1(G_{\mathrm{der}})|$, then every irredundant Yu datum is $G$-equivalent to a normalized Yu datum.

The following lemma will both facilitate the study of modular reductions and allow us to prove finiteness of L-packets later.

\begin{lemma}\label{lemma:yu-datum-finite-order}
If $\Psi$ is an irredundant Yu datum over $k$, then there is a Yu datum $\Psi'$ which is a refactorization of $\Psi$ such that each $\phi_i'$ is of finite order. Moreover, for each $i$ there is some integer $N_i$ depending only on $G^i$, $x$, and $r_i$ such that $\phi_i'$ is of order at most $N_i$. If $\Psi$ is normalized, then $\Psi'$ can be chosen to be normalized. 
\end{lemma}

\begin{proof}
Recall the subgroup $G^i(F)^1 \subset G^i(F)$ defined by
\[
G^i(F)^1 = \{g \in G^i(F)\co |\chi(g)| = 1 \text{ for all } \chi\co G^i \to \G_m\}.
\]
Note that $G^i(F)/G^i(F)^1$ is a free abelian group, so the map $G^i(F)/G^i(F)_{\der} \to G^i(F)/G^i(F)^1$ of abelian groups admits a section, and thus
\begin{equation}\label{eqn:abelianization-decomposition}
G^i(F)/G^i(F)_{\der} \cong G^i(F)^1/G^i(F)_{\der} \times G^i(F)/G^i(F)^1.
\end{equation}
For each $i$, let $\psi_i$ be the pullback, under projection to the second factor in \eqref{eqn:abelianization-decomposition}, of the restriction of $\phi_i$ to that factor. Thus $\psi_i$ is trivial on $G^i(F)^1$. Put $\phi_i'=\phi_i\psi_i^{-1}$. Then $\phi_i'$ factors through the compact group $G^i(F)^1/G^i(F)_{\der}$ and, since it is trivial on $G^i(F)_{x,r_i+}$, its order divides
\[
N_i\coloneqq\left|G^i(F)^1/\bigl(G^i(F)_{\der}G^i(F)_{x,r_i+}\bigr)\right|.
\]

Put $\chi_i=\prod_{j=i}^d\psi_j|_{G^i(F)}$ and $\tau'=\tau\otimes\chi_0$. If $j\geq i$, then $G^i(F)^1\subset G^j(F)^1$, so $\chi_i$ is trivial on $G^i(F)^1$ and hence has depth at most $0\leq r_{i-1}$. Moreover,
\[
\prod_{j=i}^d\phi_j\phi_j'^{-1}=\chi_i,
\]
while $\phi_d=1$ implies $\phi_d'=1$. Thus $\Psi' = ((G^i), x, (r_i), \tau', (\phi_i'))$ is a refactorization of $\Psi$. Finally, the image of $G^i_{\mathrm{sc}}(F)$ lies in $G^i(F)^1$, so each $\phi_i'$ has the same restriction as $\phi_i$ to this image; therefore $\Psi'$ is normalized whenever $\Psi$ is. 
\end{proof}

\section{Explicit construction of L-parameters}\label{sec:kaletha-param}

In this section, we recall Kaletha's definition of a Local Langlands Correspondence for non-singular tame supercuspidal representations. Moreover, we extend Kaletha's candidate to positive characteristic coefficients and partially to singular tame cuspidal representations. The key results established in this section are certain compatibilities of Kaletha's construction with Langlands functoriality, which will be needed later for comparison to the Fargues--Scholze correspondence.

Throughout this section, assume that $p\neq 2$. Recall from \S\ref{ssec:notation} that $F_n/F$ denotes the unramified extension of degree $n$. The notation $\ld j_{T,G}$ for the L-embedding of a torus $T$ associated to a Yu datum is defined in Definition~\ref{defn:canonical-l-embeddings}.

\subsection{Almost equivalence of torus-character pairs}\label{ss:inertial-equiv}

Recall from \S\ref{ssec:notation} the notation $\cG_{[x]}$ and $\ol G_{[x]}$ associated to a point $x \in \cB(G)$. Let $F^{\unr}$ denote the maximal unramified extension of $F$. If $T \subset G$ is a maximally unramified maximal torus and a point $x \in \cB(G)$ is given which lies in $\cB(T)$, then we let $\ol T$ denote the image in $\ol G_{[x]}$ of the special fiber of the schematic closure of $T$ in $\cG_{[x]}$, so $\ol T(\F_q) = T(F)_{[x]}/T(F)_{x,0+}$. Note in particular that depth $0$ characters of $T(F)_{[x]}$ are equivalent to characters of $\ol T(\F_q)$.

\begin{defn}\label{defn:almost-equivalent}
Let $T, T' \subset G$ be two tamely ramified maximal $F$-tori such that $x \in \cB(T) \cap \cB(T')$, let $k$ be a field among $\ol\Q_\ell$ and $\ol\F_\ell$, and let $\theta\co T(F)_{[x]} \to k^\times$ and $\theta'\co T'(F)_{[x]} \to k^\times$ be characters. Since $Z(G)(F)$ acts trivially on $\cB(G_{\der})$, it is contained in both $T(F)_{[x]}$ and $T'(F)_{[x]}$. We will say that $(T,\theta)$ and $(T',\theta')$ are \textit{almost equivalent} if the following conditions hold:
\begin{enumerate}
    \item $\theta|_{Z(G)(F)} = \theta'|_{Z(G)(F)}$.
    \item There exists a positive integer $n$ such that the pairs
    \[
        \left(T_{F_n},\theta\circ \Nm_{F_n/F}|_{T(F_n)_{[x]}}\right)
        \quad\text{and}\quad
        \left(T'_{F_n},\theta'\circ \Nm_{F_n/F}|_{T'(F_n)_{[x]}}\right)
    \]
    are $G(F_n)$-conjugate.
\end{enumerate}
\end{defn}

The notion of almost equivalence should be considered as a weak analogue for $p$-adic groups of the equivalence relation on tori arising from rational Lusztig series, as explained in \cite[\S 2.8]{Cot26b}. The terminology is justified by Lemma~\ref{lemma:conjugacy-of-torus-character-pairs}.

\begin{prop}\label{prop:extend-normalized-characters}
Let $k \in \{\ol\F_\ell, \ol\Z_\ell, \ol\Q_\ell\}$, let $E/F$ be a finite extension of degree $n$, and let $\phi\co G(F)\to k^\times$ be a normalized character (cf.\ Definition \ref{def:normalized-character}). There is a canonical normalized character $\phi_E\co G(E)\to k^\times$ such that
\[
\phi_E|_{T(E)}=\phi|_{T(F)}\circ \Nm_{E/F}
\]
for every maximal $F$-torus $T\subset G$. Moreover, $\phi_E|_{G(F)}=\phi^n$, and $\phi_E$ is $\Gal(E/F)$-stable when $E/F$ is Galois.
\end{prop}

\begin{proof}
Choose a $z$-extension
\[
1\longrightarrow Z\longrightarrow\widetilde G\xrightarrow{q}G\longrightarrow1
\]
with $Z$ induced and $\widetilde G_{\der}$ simply connected, so that $q|_{\widetilde G_{\der}}$ is the universal covering map onto $G_{\der}$. Put $D=\widetilde G/\widetilde G_{\der}$, write $\pi_D\co\widetilde G\to D$ for the quotient, and let $i\co Z\to D$ be the induced map. For $L\in\{F,E\}$, Hilbert 90 and Kneser's theorem \cite[Theorem 10.6.4]{KP} give
\[
A_L\coloneqq G(L)/q(\widetilde G_{\der}(L))
\xrightarrow{\ \sim\ }D(L)/i(Z(L)).
\]
Since $\phi$ is normalized, the pullback $q^*\phi$ descends to a character $\phi_D\co D(F)\to k^\times$ which is trivial on $i(Z(F))$. The composite
\[
\widetilde G(E)\xrightarrow{\pi_D}D(E)
\xrightarrow{\Nm_{E/F}}D(F)
\xrightarrow{\phi_D}k^\times
\]
is trivial on $Z(E)$ by functoriality of the norm, and hence descends to a character $\phi_E$ of $G(E)$. It is normalized, and it is Galois-stable when $E/F$ is Galois. If $g\in G(F)$ and $\widetilde g\in\widetilde G(F)$ lifts it, then
\[
\phi_E(g)=\phi_D\bigl(\Nm_{E/F}(\pi_D(\widetilde g))\bigr)
=\phi_D(\pi_D(\widetilde g)^n)=\phi(g)^n.
\]
Finally, if $T\subset G$ is a maximal $F$-torus and $\widetilde T=q^{-1}(T)$, then $\widetilde T(E)\to T(E)$ is surjective. Lifting $t\in T(E)$ to $\widetilde t\in\widetilde T(E)$ and using functoriality of norms gives
\[
\phi_E(t)
=\phi_D\bigl(\Nm_{E/F}(\pi_D(\widetilde t))\bigr)
=\phi\bigl(\Nm_{E/F}(t)\bigr),
\]
which proves the proposition.
\end{proof}

\begin{lemma}\label{lemma:conjugacy-of-torus-character-pairs}
Let $T, T' \subset G$ be maximally unramified maximal $F$-tori such that $x \in \cB(T) \cap \cB(T')$, let $k$ be a field among $\ol\Q_\ell$ and $\ol\F_\ell$, and let $\theta_0\co T(F)_{[x]} \to k^\times$ and $\theta'_0\co T'(F)_{[x]} \to k^\times$ be characters of depth $0$. Let $\phi\co G(F) \to k^\times$ be a normalized character, and let $\theta = \theta_0 \cdot \phi|_{T(F)_{[x]}}$ and $\theta' = \theta'_0 \cdot \phi|_{T'(F)_{[x]}}$. If $\cE(\ol G_{[x]}, [\ol T, \theta_0]) = \cE(\ol G_{[x]}, [\ol T', \theta_0'])$, with notation as in \cite[Definition~2.8.2]{Cot26b}, then the pairs $(T, \theta)$ and $(T', \theta')$ are almost equivalent.
\end{lemma}

\begin{proof}
We first reduce to the case $\phi=1$. By geometric conjugacy of $(\ol T, \theta_0)$ and $(\ol T', \theta_0')$, there is a finite unramified extension $E/F$ and an element of $G(E)$ conjugating $(T_E,\theta_0 \circ \Nm_{E/F})$ to $(T',\theta'_0 \circ \Nm_{E/F})$. To show that the same element also conjugates $(T,\theta \circ \Nm_{E/F})$ to $(T',\theta' \circ \Nm_{E/F})$, it suffices to construct a character $\phi_E$ of $G(E)$ whose restrictions to $T(E)$ and $T'(E)$ are the corresponding norm pullbacks of $\phi$. Such a character is provided by Proposition~\ref{prop:extend-normalized-characters}, so we have made the desired reduction.

Let $n$ be a positive integer such that the pairs $(\ol T_{\F_{q^n}}, \theta \circ \Nm_{\F_{q^n}/\F_q})$ and $(\ol T'_{\F_{q^n}}, \theta' \circ \Nm_{\F_{q^n}/\F_q})$ are $\ol G_{[x]}(\F_{q^n})$-conjugate, and let $E = F_n$. Let $\ol g \in \ol G_{[x]}(\F_{q^n})$ be an element such that
\[
\ol g(\ol T_{\F_{q^n}}, \theta \circ \Nm_{\F_{q^n}/\F_q})\ol g^{-1} = (\ol T'_{\F_{q^n}}, \theta' \circ \Nm_{\F_{q^n}/\F_q}).
\]
Let $S$ (resp.\ $S'$) denote the maximal unramified $F$-subtorus of $T$ (resp.\ $T'$), and let $\cS$ (resp.\ $\cS'$) denote the $\cO_F$-torus with generic fiber $S$ (resp.\ $S'$). Note that there are canonical closed embeddings $\cS \to \cG_{[x]}$ and $\cS' \to \cG_{[x]}$ by \cite[Axiom 4.1.20]{KP}.

Let
\[
    \pi_x\co(\cG_{[x]})_{\F_{q^n}}\longrightarrow(\ol G_{[x]})_{\F_{q^n}}
\]
be the quotient by the unipotent radical of $(\cG_{[x]})_{\F_{q^n}}^\circ$. By Lang's theorem, $\pi_x$ is surjective on $\F_{q^n}$-points. Thus $\ol g$ lifts to $\cG_{[x]}(\F_{q^n})$, and then, by smoothness of $\cG_{[x]}$ and completeness of $\cO_E$, to some $g\in\cG_{[x]}(\cO_E)$. The images of $\cS_{\F_{q^n}}$ and $\cS'_{\F_{q^n}}$ under $\pi_x$ are $\ol T^\circ_{\F_{q^n}}$ and $\ol T'^\circ_{\F_{q^n}}$, respectively. Hence the tori $g\cS g^{-1}$ and $\cS'$ have the same image under $\pi_x$, so they are conjugate by $\ker\pi_x$. After passing to a translate of $g$ by an element of $\cG_{[x]}(\cO_E)$ whose special fiber lies in $(\ker \pi_x)(\F_{q^n})$, we may assume that $g\cS g^{-1}$ and $\cS'$ have the same special fiber. By deformation theory for tori \cite[Exp.\ IX, Corollaire 7.3]{SGA3II}, there is some
\[
    h\in\ker\bigl(\cG_{[x]}^\circ(\cO_E)\to\ol G_{[x]}^\circ(\F_{q^n})\bigr)
    =G(E)_{x,0+}
\]
such that $hgSg^{-1}h^{-1}=S'$. Since $T$ and $T'$ are maximally unramified, we have $T = Z_G(S)$ and $T' = Z_G(S')$, so also $hg T g^{-1}h^{-1} = T'$. Since $h$ has trivial image in $\ol G_{[x]}^\circ(\F_{q^n})$ and $\theta,\theta'$ have depth $0$, we deduce
\[
hg(T_E, \theta \circ \Nm_{E/F}|_{T(E)_{[x]}})g^{-1}h^{-1} = (T'_E, \theta' \circ \Nm_{E/F}|_{T'(E)_{[x]}}),
\]
as desired.
\end{proof}

\subsection{Associated torus-character pairs}\label{ss:kal-param}

In \cite{Kal21b}, Kaletha introduced his Local Langlands Correspondence for non-singular tame supercuspidal $\ol\Q_\ell$-representations of $G(F)$. In this section, we recall this correspondence and partially extend it to general tame cuspidal representations.

Let $\Psi = ((G^i), x, (r_i), \tau, (\phi_i))$ be an irreducible Yu datum for $G$ over a field $k$ which is either $\ol\Q_\ell$ or $\ol\F_\ell$. Let $\ol Z$ be the center of $\ol G^0_{[x]}$, and recall that $\ol G^0_{[x]}$ is a paraductive $\F_q$-group scheme in the sense of \cite[Definition~2.1.1]{Cot26b}.

\subsubsection{Associated depth $0$ torus-character pair}
By \cite[Proposition~2.6.2]{Cot26b}, there is a maximal $\F_q$-torus $\ol T^\circ \subset (\ol G^0_{[x]})^\circ$ and a character $\theta_0\co \ol T(\F_q) \to k^\times$, where $\ol T = \ol T^\circ \cdot \ol Z$, such that $\tau$ lies in the semi-rational Lusztig series $\cE(\ol G^0_{[x]}, [\ol T, \theta_0])$ (with notation as in \cite[Definition~2.8.2]{Cot26b}). 
The pair $(\ol T, \theta_0)$ is not unique in general; however, the Lusztig series $\cE(\ol G^0_{[x]}, [\ol T, \theta_0])$ is unique by \cite[Lemma~2.8.4]{Cot26b}. Moreover, if there is such a pair with $\theta_0$ non-singular, then the pair $(\ol T, \theta_0)$ is unique up to $\ol G^0_{[x]}(\F_q)$-conjugacy by \cite[Lemma~2.9.2]{Cot26b}.

\subsubsection{Associated $p$-adic torus-character pair}\label{sss:p-adic-torus-character-pair}
Choose an $\cO_F$-torus $\cT_0$ of $\cG^0_{[x]}$ whose special fiber $\ol\cT_0$ has image $\ol T^\circ$ in $\ol G^0_{[x]}$ (which exists by \cite[Proposition 11.14(1)]{Bor91} and deformation theory for tori \cite[Expos\'e IX, Th\'eor\`eme 3.6, Th\'eor\`eme 7.1]{SGA3II}), and let $T_0$ be the generic fiber of $\cT_0$, an $F$-torus of $G^0$. If $T = Z_{G^0}(T_0)$, then $T$ is a maximally unramified maximal $F$-torus of $G^0$ and thus also a maximal $F$-torus of $G$, and $x \in \cB(T)$. Define a character $\theta\co T(F)_{[x]} \to k^\times$ by
\[
\theta = \theta_0 \cdot \prod_{i=0}^d \phi_i|_{T(F)_{[x]}}.
\]
If $\tau$ is non-singular, then it follows that $\ol T^\circ$ is elliptic and thus $T$ is also elliptic since $Z(G^0)/Z(G)$ is anisotropic.



\begin{defn}\label{defn:associated-torus-character-pairs}
Define $\cT(\Psi)$ as the set of $G(F)$-conjugacy classes $[(T, \theta)]$ of pairs $(T, \theta\co T(F)_{[x]}\to k^\times)$ constructed above. It is clear that $\cT(\Psi)$ only depends on the $G$-equivalence class of $\Psi$, and it is unaffected by the procedure of Lemma~\ref{lemma:irredundant-yu-data}.

Let $\cT_0(\Psi) \subset \cT(\Psi)$ denote the set of $[(T,\theta)]$ constructed as above such that $\tau \in \cE_0(\ol G, [\ol T, \theta])$, with notation as in \cite[Definition~2.8.2]{Cot26b}.
\end{defn}

Note that if $k = \ol\Q_\ell$ and $\tau$ is cuspidal, then for every $[(T, \theta)] \in \cT_0(\Psi)$, the torus $T$ is elliptic in $G$. The set $\cT_0(\Psi)$ will be important for studying the Fargues--Scholze correspondence in \cite{CF26b} because it retains direct contact with Deligne--Lusztig induction, but the following lemma will eventually show that it is irrelevant for the study of Kaletha's correspondence.

\begin{lemma}\label{lemma:associated-tori-almost-equiv}
Assume that $\Psi$ is normalized. If $[(T, \theta)]$ and $[(T', \theta')]$ lie in $\cT(\Psi)$, then the pairs $(T,\theta)$ and $(T',\theta')$ are almost equivalent in the sense of Definition~\ref{defn:almost-equivalent}.
The central character of $\pi(\Psi)$ is equal to $\theta|_{Z(G)(F)}$.
\end{lemma}

\begin{proof}
Put $\phi=\prod_{i=0}^d\phi_i|_{G^0(F)}$. The inclusions $G^0_{\der}\to G^i_{\der}$ lift to the simply connected covers, so the hypothesis that $\Psi$ is normalized implies that $\phi$ is normalized. The fact that $(T, \theta)$ and $(T', \theta')$ are almost equivalent is immediate from Lemma~\ref{lemma:conjugacy-of-torus-character-pairs}. The final claim is clear from the fact that the inducing representation $\epsilon\otimes\tau\otimes\kappa\otimes\phi_d$ in Yu's construction has central character $\theta|_{Z(G)(F)}$ (since $\epsilon|_{Z(G)(F)} = 1$).
\end{proof}

\begin{lemma}\label{lemma:associated-tori-inertial-parameter}
Assume that $\Psi$ is normalized. If $[(T,\theta)]$ and $[(T',\theta')]$ lie in $\cT(\Psi)$, then
\[
\ld j_{T,G}\circ\ld(\theta|_{T(F)_{\mathrm b}})|_{I_F}
\sim
\ld j_{T',G}\circ\ld(\theta'|_{T'(F)_{\mathrm b}})|_{I_F}.
\]
\end{lemma}

\begin{proof}
By Lemma~\ref{lemma:associated-tori-almost-equiv}, there are an unramified extension $E/F$ and $g\in G^0(E)$ such that $g(T_E, \theta \circ \Nm_{E/F})g^{-1} = (T'_E, \theta' \circ \Nm_{E/F})$. Let $c_g \co T_E\xrightarrow{\sim}T'_E$ be the conjugation map, and let $\ld c_g \co\ld T'_E\xrightarrow{\sim}\ld T_E$ be the associated L-isomorphism. It is visible from \eqref{eqn:def-of-L-embedding} that
\begin{equation}\label{eq:associated-tori-conjugate-E-embeddings}
    \ld j_{T'_E,G_E}
    \sim  \ld j_{T_E,G_E}\circ\ld c_g.
\end{equation}
Let
\[
    \theta_E\coloneqq\theta\circ \Nm_{E/F}|_{T(E)_{[x]}},
    \qquad
    \theta'_E\coloneqq\theta'\circ \Nm_{E/F}|_{T'(E)_{[x]}}.
\]
The choice of $g$ implies that $c_g^*\theta'_E=\theta_E$, so naturality of the inertial Local Langlands Correspondence for tori and \eqref{eq:associated-tori-conjugate-E-embeddings} yield
\[
    \ld j_{T'_E,G_E}\circ\ld(\theta'_E|_{T'(E)_{\mathrm b}})|_{I_E}
    \sim 
    \ld j_{T_E,G_E}\circ\ld(\theta_E|_{T(E)_{\mathrm b}})|_{I_E}.
\]
Thus we conclude by norm functoriality for the inertial Local Langlands Correspondence for tori and Lemma~\ref{lemma:tasho-base-change}(1)--(2) (applied to a filtration of $E/F$ by extensions of prime degree).
\end{proof}

\begin{lemma}\label{lemma:character-restricted-to-torus-depth}
If $T\subset G^0$ is a tamely ramified maximal $F$-torus such that $x\in\cB(T)$, then $\phi_i|_{T(F)}$ is of depth $r_i$ for all $0\leq i<d$, and $\phi_d|_{T(F)}$ is of depth $r_d$ if $\phi_d\neq1$.
\end{lemma}

\begin{proof}
Fix $i < d$. By genericity, the character induced by $\phi_i$ on $G^i(F)_{x,r_i}/G^i(F)_{x,r_i+}$ is represented, under the Moy--Prasad isomorphism and the fixed additive character $\psi$, by the class of some $X_i^*\in\Lie^*(G^i)^{G^i}(F)_{x,-r_i}$. Choose $Y\in\frg^i(F)_{x,r_i}$ whose pairing with $X_i^*$ lies in $\cO_F^\times$; this exists since $X_i^*$ has nontrivial image in $\Lie^*(G^i)(F)_{x,-r_i:-r_i+}$. Let $E/F$ be a tame extension splitting $T$, so there is a direct sum decomposition
\[
    \frg^i(E)_{x,r_i}
    =\frt(E)_{r_i}\oplus\bigoplus_\alpha\frg^i_\alpha(E)_{x,r_i}.
\]
Since $X_i^*$ is $G^i$-invariant, it annihilates every root space, each of which lies in $[\frg^i,\frg^i]$. Thus the projection of $Y$ to $\frt(E)_{r_i}$ has the same nonzero pairing with $X_i^*$. This projection lies in $\frt(F)_{r_i}$, so $\phi_i|_{T(F)_{r_i}}$ is nontrivial. Thus $\phi_i|_{T(F)}$ has depth $r_i$.

Now assume $\phi_d \neq 1$, so $\phi_d$ is of depth $d$ by assumption. Choose a z-extension $\pi\co \wt G \to G$, and note that by \cite[Lemma 3.5.3]{Kal19}, the map $\wt G(F)_{x,r_d} \to G(F)_{x,r_d}$ is surjective; thus the induced character $\wt\phi_d$ of $\wt G(F)$ is of depth $r_d$. Now \cite[Lemma 3.5.2]{Kal19} shows that $\wt\phi_d$ is represented by an element $X_i^* \in \Lie^*(\wt G)^{\wt G}(F)_{x,-r_d}$, so the same argument as above shows that if $\wt T = \pi^{-1}(T)$ then $\wt\phi_d|_{\wt T(F)}$ is of depth $r_d$. Since $\wt T(F)_{r_d} \to T(F)_{r_d}$ is surjective by \cite[Lemma 3.5.3]{Kal19} again, we conclude.
\end{proof}

\subsubsection{Non-singular characters}
We adapt the following definition from \cite[Definition 3.1.1]{Kal21b}.\footnote{In fact, for reasons explained below, our definition differs slightly from \cite[Definition 3.1.1]{Kal21b}.}

\begin{defn}\label{defn:depth-0-nonsingular}
Let $k$ be a field of characteristic not equal to $p$. Let $T \subset G$ be a maximally unramified maximal $F$-torus, fix a point $x \in \cB(T) \subset \cB(G)$, let $E_0/F$ be a finite tamely ramified extension splitting $T$, and let $F'/F$ be the maximal unramified subextension, so $F'$ splits the maximal unramified $F$-subtorus $S \subset T$. Let $\F_{F'}$ be the residue field of $F'$, and let $\ol S^\circ$ denote the special fiber of the relative identity component of the lft N\'eron model of $S$, viewed as a maximal torus of $\ol G_{[x]}^\circ$. Fix a subgroup $T^\dagger \subset T(F)$ containing $T(F)_{\mathrm{b}}$ and a character $\theta\co T^\dagger\to k^\times$. Assume
\begin{equation}\label{eqn:nonvanishing-on-absolute-coroots}
\theta \circ \Nm_{E_0/F} \circ \alpha^\vee|_{1 + \frp_{E_0}} = 1 \text{ for every coroot } \alpha^\vee \in \Phi^\vee(G_{E_0}, T_{E_0}).
\end{equation}
We say that
\begin{enumerate}
\item $\theta$ is \textit{$\F$-non-singular} if $(\theta \circ \Nm_{F'/F} \circ \alpha_{\mathrm{res}}^\vee)|_{\cO_{F'}^\times}$ is nontrivial\newline for each $\alpha_{\mathrm{res}} \in \Phi((\ol G_{[x]}^\circ)_{\F_{F'}}, \ol S^\circ_{\F_{F'}}) \subset \Phi(G_{F'}, S_{F'})$.
\item $\theta$ is \textit{$F$-non-singular} if $(\theta \circ \Nm_{F'/F} \circ \alpha_{\mathrm{res}}^\vee)|_{\cO_{F'}^\times}$ is nontrivial for each $\alpha_{\mathrm{res}} \in \Phi(G_{F'}, S_{F'})$.
\end{enumerate}
\end{defn}

We now generalize Definition~\ref{defn:depth-0-nonsingular} beyond the case of maximally unramified tori as in \cite[Definition 3.4.1]{Kal21b}, for which we must introduce some notation. Let $k$ be a field of characteristic different from $p$, fix a tamely ramified maximal $F$-torus $T \subset G$, a subgroup $T^\dagger$ with $T(F)_{\mathrm b}\subset T^\dagger\subset T(F)$, and a character $\theta\co T^\dagger\to k^\times$. Define the depth of $\theta$ to be the depth of $\theta|_{T(F)_{\mathrm b}}$. Let $E/F$ be the splitting field of $T$. For each positive real number $r$, we define
\[
\Phi_r = \Phi_r(T, \theta) = \{ \alpha \in \Phi(G_E, T_E) \co \theta(\Nm_{E/F}(\alpha^\vee (E_r^\times))) = 1 \}
\]
where $E_r^\times \coloneqq 1 + \mf{p}_E^{\lceil e r \rceil}$ for $e$ the ramification degree of $E/F$. Put $E_{r+}^\times\coloneqq\bigcup_{s>r}E_s^\times$. Let $\Phi_{r+} = \bigcap_{s>r} \Phi_s$ and say that $r \in \RR_{>0}$ is a \emph{break} in the filtration $\Phi_\bullet(T,\theta)$ if $\Phi_r \neq \Phi_{r+}$. There are finitely many (possibly $0$) breaks, which we denote $r_{d-1} > r_{d-2} > \ldots  > r_0 > 0$. We define by convention $r_{-1} \coloneqq 0$ and $r_d \coloneqq \mathrm{depth}(\theta)$. If $d>0$, then this gives a sequence $r_d \geq r_{d-1} > \cdots > r_0 > r_{-1}$, while if $d=0$ (no breaks in $\Phi_r$) then $r_0 \geq r_{-1} = 0$.

For $d \geq i \geq 0$, we define $G^i = G^i(T,\theta) \subset G$ to be the connected reductive subgroup with maximal torus $T$ and root system $\Phi_{r_{i-1}+}$. We define $G^{-1} \coloneqq T$, so we have 
\[
T = G^{-1}  \subset G^0 \subset \cdots \subset G^d = G.
\]
Note that if $p$ is good for $G$ and does not divide the order of $\pi_1(G_{\der})$, then $G^i$ is a twisted Levi subgroup of $G$ by \cite[Lemma 3.6.1]{Kal19}. In fact, Lemma~\ref{lemma:ns-yu-data-and-characters} will show that this is true for arbitrary $p$ when $(T, \theta) \in \cT(\Psi)$ for a normalized Yu datum $\Psi$.

\begin{defn}\label{defn:non-singular}
We say that $\theta\co T^\dagger\to k^\times$ is \emph{$\F$-non-singular} (resp.\ \textit{$F$-non-singular}) if $T$ is maximally unramified in $G^0$ and, when $T$ is considered inside $G^0$, the character $\theta$ (which satisfies the hypotheses of Definition~\ref{defn:depth-0-nonsingular}) is $\F$-non-singular (resp.\ $F$-non-singular) with respect to $G^0$ in the sense of Definition~\ref{defn:depth-0-nonsingular}. These notions and the groups $G^i(T,\theta)$ depend only on $\theta|_{T(F)_{\mathrm b}}$.
\end{defn}

As in \cite[Fact 3.1.4(1)]{Kal21b}, an $F$-non-singular character is automatically $\F$-non-singular.

\begin{remark}\label{rmk:difference-in-defn}
Definitions~\ref{defn:depth-0-nonsingular} and \ref{defn:non-singular} differ slightly from \cite[Definition 3.1.1, Definition 3.4.1]{Kal21b}: indeed, \cite[Definition 3.4.1]{Kal21b} requires that $\theta$ is non-singular with respect to $G^0$ in the sense of \cite[Definition 3.1.1]{Kal21b}, which strictly speaking implies that $\theta$ is of depth $0$; however, this is clearly an unintended restriction.
\end{remark}


\begin{defn}\label{defn:kF-non-singular-yu-datum}
We will say that an irreducible Yu datum $\Psi = (\vec{G},x,\vec{r},\tau,\vec{\phi})$ over $k \in \{\ol\F_\ell, \ol\Q_\ell\}$ is \textit{$\F$-non-singular} if $\tau$ is non-singular in the sense of \cite[Definition~2.9.3]{Cot26b}.
\end{defn}

\begin{lemma}\label{lemma:ns-yu-data-and-characters}
Assume $p \neq 2$, let $\Psi = (\vec{G}, x, \vec{r}, \tau, \vec{\phi})$ be an irredundant normalized irreducible Yu datum for $G$ over a field $k$ among $\ol\Q_\ell$ and $\ol\F_\ell$, and let $[(T, \theta)] \in \cT(\Psi)$.
\begin{enumerate}
    \item $G^i = G^i(T, \theta)$ for all $0\leq i\leq d$, where $G^i(T,\theta)$ is defined from the filtration $\Phi_\bullet(T,\theta)$ preceding Definition~\ref{defn:non-singular}.
    \item The canonical ($\wh G(k)$-conjugacy class of) embedding(s) $\wh{G^0} \to \wh G$ realizes $\wh{G^0}$ as the ($\wh G(k)$-conjugacy class of) connected centralizer(s) $Z_{\wh G}((\ld j_{T, G} \circ \ld\theta)(P_F))^\circ$.
    \item $\Psi$ is $\F$-non-singular if and only if $\theta$ is $\F$-non-singular. In this case, $T$ is elliptic and $\cT(\Psi)$ consists of the single conjugacy class $[(T, \theta)]$.
\end{enumerate}
\end{lemma}

\begin{proof}
We begin with (1). Let $E/F$ be a finite tamely ramified Galois extension splitting $T$. It is enough to prove
\[
    \Phi_{r_{i-1}+}(T,\theta)=\Phi(G^i_E,T_E)
    \qquad(0\leq i\leq d).
\]
If $j < i$, then $\phi_j$ has depth at most $r_{i-1}$, so we have
\begin{equation}\label{eqn:factorization-of-theta}
\theta(\Nm_{E/F}(\alpha^\vee(E_{r_{i-1}+}^\times))) = \prod_{j=i}^d\phi_j(\Nm_{E/F}(\alpha^\vee(E_{r_{i-1}+}^\times))).
\end{equation}
If $\alpha \in \Phi(G^i_E, T_E)$, then the coroot $\alpha^\vee\co \G_m \to G^i_E$ factors through a map $\alpha_{\mathrm{sc}}^\vee\co \G_m \to (G^i_{\mathrm{sc}})_E$, where $\pi\co G^i_{\mathrm{sc}} \to G^i_{\der}$ is the simply connected cover. Since $\Psi$ is normalized, functoriality of the norm on the inverse-image torus in $G^j_{\mathrm{sc}}$ gives
\begin{equation}\label{eqn:normalized-character-coroot-norm}
\phi_j\circ \Nm_{E/F}\circ\alpha^\vee=1
\qquad(i\leq j\leq d).
\end{equation}
By \eqref{eqn:factorization-of-theta}, this means that $\alpha \in \Phi_{r_{i-1}+}(T, \theta)$.

Conversely, suppose that $\alpha\in\Phi_{r_{i-1}+}(T,\theta)$. If $\alpha\in\Phi(G^j_E,T_E)$ for $j \leq i$, then $\alpha\in\Phi(G^i_E,T_E)$ immediately. Otherwise there is a unique $i\leq j<d$ such that
\[
\alpha\in\Phi(G^{j+1}_E,T_E) - \Phi(G^j_E,T_E).
\]
Equation~\eqref{eqn:normalized-character-coroot-norm} shows that $\phi_m\circ \Nm_{E/F}\circ\alpha^\vee$ is trivial for $m\geq j+1$, while $\phi_m\circ \Nm_{E/F}\circ\alpha^\vee$ has depth smaller than $r_j$ for $m < j$. The assumption $\alpha\in\Phi_{r_{i-1}+}(T,\theta)$ and the inequality $r_j>r_{i-1}$ therefore give
\begin{equation}\label{eqn:condition-for-contradiction}
\phi_j(\Nm_{E/F}(\alpha^\vee(E_{r_j}^\times))) = 1.
\end{equation}
Let $v_E$ denote the extension to $E$ of the normalized valuation of $F$. Let $X_j^*\in\Lie^*(G^j)^{G^j}(F)_{x,-r_j}$ represent $\phi_j|_{G^j(F)_{x,r_j}}$ through the Moy--Prasad isomorphism and $\psi$. Genericity means that $v_E(\langle X_j^*,H_\alpha\rangle)=-r_j$. Let $z \in E$ be such that $v_E(z) = r_j$ and
\[
\psi\bigl(\Tr_{E/F}(z\langle X_j^*,H_\alpha\rangle)\bigr)\neq1;
\]
such a $z$ exists because $E/F$ is tame. The Moy--Prasad isomorphism $
T(F)_{r_j}/T(F)_{r_j+}
    \xrightarrow{\ \sim\ }
\frt(F)_{r_j}/\frt(F)_{r_j+}$ and its analogue for $E$ sends $\Nm_{E/F}\bigl(\alpha^\vee(1+z)\bigr)$ to $\Tr_{E/F}(zH_\alpha)$. Since $X_j^*$ is $F$-rational and represents $\phi_j$ on $\mf{g}^j(F)_{x,r_j}/\mf{g}^j(F)_{x,r_j+}$, it follows that
\[
\begin{aligned}
\phi_j\bigl(\Nm_{E/F}(\alpha^\vee(1+z))\bigr)
    &=\psi\left(\left\langle X_j^*,
    \Tr_{E/F}(zH_\alpha)\right\rangle\right)\\
    &=\psi\bigl(\Tr_{E/F}(z\langle X_j^*,H_\alpha\rangle)\bigr)\neq1,
\end{aligned}
\]
contradicting Equation~\eqref{eqn:condition-for-contradiction}. This proves $\Phi_{r_{i-1}+}(T,\theta)=\Phi(G^i_E,T_E)$.

For (2), note that since $T$ is tamely ramified, the restriction $\ld j_{T,G}|_{\wh T \rtimes P_F}$ is simply the canonical embedding $\wh T \times P_F \to \wh G \times P_F$. Thus $Z_{\wh G}((\ld j_{T,G} \circ \ld\theta)(P_F))^\circ$ is the reductive $k$-subgroup of $\wh G$ with maximal $k$-torus $\wh T$ and root system consisting of those $\wh\alpha$ (for $\alpha \in \Phi(G_{\ol F}, T_{\ol F})$) such that $\wh\alpha \circ \wh\theta|_{P_F} = 1$. By duality, this means $\theta(\Nm_{E/F}(\alpha^\vee(E_{0+}^\times))) = 1$, i.e., $\alpha \in \Phi(G^0_{\ol F}, T_{\ol F})$. So the claim follows from (1).

We now prove (3). Let $S\subset T$ be the maximal unramified subtorus, and choose a finite unramified extension $F'/F$ splitting $S$. Part~(1) shows that the restriction of $\theta$ to every coroot occurring in $G^0$ has depth $0$. More precisely, for each restricted coroot $\alpha_0^\vee$ of $S_{F'}$ in $G^0_{F'}$, the same argument which led to Equation~\eqref{eqn:normalized-character-coroot-norm} gives
\[
\phi_i\circ \Nm_{F'/F}\circ\alpha_0^\vee=1
\qquad(0\leq i\leq d).
\]
Thus $\theta$ is $\F$-non-singular if and only if $\theta_0$ is non-singular. Since $\tau$ lies in $\cE(\ol G^0_{[x]}, [\ol T,\theta_0])$, \cite[Definition~2.9.3 and Lemma~2.9.2]{Cot26b} show that $\tau$ is non-singular if and only if $\theta_0$ is non-singular. Under these equivalent conditions, \cite[Lemma~2.9.4]{Cot26b} shows that $\ol T^\circ$ is elliptic. Hence $T$ is elliptic in $G^0$; the anisotropy of $Z(G^0)/Z(G)$ then makes $T$ elliptic in $G$. Moreover, \cite[Lemma~2.9.2]{Cot26b} shows that $(\ol T, \theta_0)$ is unique up to $\ol G_{[x]}^0(\F_q)$-conjugacy. Lemma~\ref{lemma:conjugacy-of-torus-character-pairs} therefore shows that $\cT(\Psi)$ consists of the single orbit class $[(T,\theta)]$.
\end{proof}

\begin{defn}\label{defn:non-singular-yu-datum}
Let $\Psi$ be an irreducible Yu datum over $k\in\{\ol\F_\ell,\ol\Q_\ell\}$. We say that $\Psi$ is \textit{$F$-non-singular} (or just \textit{non-singular}) if some class $[(T,\theta)]\in\cT(\Psi)$ has the property that $\theta$ is $F$-non-singular in the sense of Definition~\ref{defn:non-singular}. We will then say that $\pi = \pi(\Psi)$ is non-singular.
\end{defn}

\begin{remark}\label{remark:toral-yu-ns}
If $\Psi$ is a normalized irreducible Yu datum with $G^0 = T$, then $\Psi$ is $F$-non-singular.
\end{remark}

Note that Lemma~\ref{lemma:ns-yu-data-and-characters}(3) and \cite[Fact 3.1.4(1)]{Kal21b} show that any $F$-non-singular Yu datum is $\F$-non-singular.

\begin{defn}\label{def:expected-L-param}
Let $\Psi=(\vec G,x,\vec r,\tau,\vec\phi)$ be an irreducible Yu datum for $G$ over $k\in\{\ol\F_\ell,\ol\Q_\ell\}$, let $\Psi^{\mathrm{irr}}$ be the irredundant datum obtained by iterating Lemma~\ref{lemma:irredundant-yu-data}, and assume that $\Psi^{\mathrm{irr}}$ is $G$-equivalent to a normalized Yu datum. Choose a class $[(T,\theta)]\in\cT(\Psi)$. When $\Psi$ is $F$-non-singular, Lemma~\ref{lemma:ns-yu-data-and-characters}(3) shows that $\cT(\Psi)$ consists of this single $G(F)$-conjugacy class, with $T$ elliptic and $\theta\colon T(F)\to k^\times$ non-singular. Define $\rho^{\Kal}(\Psi)$ to be the $\wh G(k)$-conjugacy class represented by
\begin{equation}\label{eq:nonsingular-Kal-param}
W_F \xrightarrow{\ld \theta} \ld T(k) \xrightarrow{ \ld j_{T,G}} \ld G(k)
\end{equation}
where $\ld\theta\co W_F\to\ld T(k)$ is dual to $\theta$ under the Local Langlands Correspondence for tori, and $\ld j_{T,G}$ is the embedding from Definition~\ref{defn:canonical-l-embeddings}.

Under the same hypothesis on $\Psi^{\mathrm{irr}}$, without assuming that $\Psi$ is $F$-non-singular, define $\rho_I^{\Kal}(\Psi)$ to be the $\wh G(k)$-conjugacy class represented by
\begin{equation}\label{eq:inertial-Kal-param}
I_F \xrightarrow{\ld(\theta|_{T(F)_{\mathrm b}})} \ld T(k) \xrightarrow{ \ld j_{T,G}} \ld G(k)
\end{equation}
where $\ld(\theta|_{T(F)_{\mathrm b}})\co I_F\to\ld T(k)$ is the inertial L-parameter corresponding to $\theta|_{T(F)_{\mathrm{b}}}$ under the inertial Local Langlands Correspondence for tori, as in the discussion preceding Lemma~\ref{lemma:comparison-to-minimal-ramified-chi-data}. The fact that this is well-defined is proven in Lemma~\ref{lemma:kaletha-parameter-choice-independence}.
\end{defn}

The superscript ``Kal'' is used here in recognition of Kaletha's work, although strictly speaking Kaletha only defines $\rho^{\Kal}(\Psi)$ in \cite{Kal21b} in the case that $k$ is of characteristic $0$ and $\Psi$ is non-singular. Note that if $\Psi$ is $\F$-non-singular but not $F$-non-singular, then the recipe defining $\rho^{\Kal}(\Psi)$ still makes sense, but it does not yield the ``correct'' semisimple L-parameter.

\begin{lemma}\label{lemma:kaletha-parameter-choice-independence}
The conjugacy classes $\rho_I^{\Kal}(\Psi)$ and, when $\Psi$ is $F$-non-singular, $\rho^{\Kal}(\Psi)$, only depend on the $G$-equivalence class of $\Psi$.
\end{lemma}

\begin{proof}
Independence of the chosen pair $[(T, \theta)] \in \cT(\Psi)$ for the inertial parameter is Lemma~\ref{lemma:associated-tori-inertial-parameter}; when $\Psi$ is $F$-non-singular, there is a unique such pair. The fact that $G$-equivalent Yu data yield the same (inertial) L-parameter is simply a matter of unraveling the definitions recalled in \S\ref{sss:G-equivalence}, which we leave to the reader.
\end{proof}

\begin{remark}
If $k = \ol\Q_\ell$, then it follows from Lemma~\ref{lemma:kaletha-parameter-choice-independence} and \cite[Corollary 3.5.5]{Kal19} (which builds on \cite[Theorem 6.6]{HM08}) that $\rho^{\Kal}(\Psi)$ and $\rho_I^{\Kal}(\Psi)$ only depend on $\pi(\Psi)$. We expect that the analogous assertion over $\ol\F_\ell$ is true, and that it can be established by similar methods, but it will not be used or proven here. Our eventual results compare $\rho^{\Kal}(\Psi)$ and $\rho^{\Kal}_I(\Psi)$ to the parameters arising from the Fargues--Scholze correspondence, establishing a certain amount of independence a posteriori.
\end{remark}

\begin{remark}
We regard $\rho^{\Kal}_I(\Psi)$ as our guess for the restriction to $I_F$ of the L-parameter associated to $\pi(\Psi)$ under the ``true'' semisimple Local Langlands Correspondence. As a sanity check, we will show in Lemma~\ref{lemma:non-singular-irreducible}(1) that $\rho^{\Kal}_I(\Psi)$ is semisimple; this is nontrivial if $k$ is of positive characteristic.

It seems to be a difficult problem to give a reasonably concrete definition of a full (semisimple) L-parameter $\rho^{\Kal}(\Psi)$ when $\Psi$ is a Yu datum over $\ol\Q_\ell$ which is not $F$-non-singular, even when $\pi(\Psi)$ is of depth $0$.
\end{remark}

The following lemma is a straightforward generalization of \cite[Lemma~4.1.10]{Kal21b} beyond the case of characteristic $0$ coefficients.

\begin{lemma}\label{lemma:toral-inertial-centralizer-criterion}
Let $k\in\{\ol\Q_\ell,\ol\F_\ell\}$, let $\Psi$ be an irredundant normalized irreducible Yu datum for $G$ with coefficients in $k$, and choose $[(T,\theta)]\in\cT(\Psi)$. Let $\rho\co I_F\to\ld G(k)$ be the representative of $\rho_I^{\Kal}(\Psi)$ given by \eqref{eq:inertial-Kal-param} using $(T,\theta)$, and identify $\wh T$ with its image under $\ld j_{T,G}$.
The following are equivalent:
\begin{enumerate}
\item $Z_{\wh G}(\rho(I_F))^\circ_{\red}$ is a torus,
\item $\theta$ is non-singular.
\end{enumerate}
Under these hypotheses, $Z_{\wh G}(\rho(I_F))^\circ_{\red}=(\wh T^{I_F})^\circ_{\red}$.
\end{lemma}

\begin{proof}
The proof of \cite[Lemma~4.1.10]{Kal21b} can be read nearly verbatim in our setting, so we omit it.
\end{proof}

\subsection{Modular reduction}

Our ultimate aim is to prove results about the L-parameters of supercuspidal $\ol\Q_\ell$-representations of $G(F)$ obtained from Yu's construction. However, in order to analyze these we need to relate such supercuspidal representations to those constructed over $\ol\F_\ell$.

\begin{lemma}\label{lemma:yu-datum-modular-reduction}
Let $\Psi=((G^i),x,(r_i),\tau,(\phi_i))$ be a Yu datum over $\ol\Z_\ell$. Put
\[
    \ol\tau=\tau\otimes_{\ol\Z_\ell}\ol\F_\ell,
    \qquad
    \ol\Psi=((G^i),x,(r_i),\ol\tau,(\ol\phi_i)),
\]
where $\ol\phi_i$ is the reduction of $\phi_i$. The tuple $\ol\Psi$ is a (possibly reducible) Yu datum, and there is a natural isomorphism
\begin{equation}\label{eq:yu-modular-base-change}
    \pi(\ol\Psi)\cong
    \pi(\Psi)\otimes_{\ol\Z_\ell}\ol\F_\ell.
\end{equation}
Moreover:
\begin{enumerate}
    \item Every irreducible subquotient $\ol\tau_0$ of $\ol\tau$ is cuspidal, and the tuple
    \[
    \ol\Psi_0=((G^i),x,(r_i),\ol\tau_0,(\ol\phi_i))
    \]
    is an irreducible Yu datum. The representation $\pi(\ol\Psi_0)$ is an irreducible subquotient of $\pi(\Psi)\otimes_{\ol\Z_\ell}\ol\F_\ell$.
    \item Every irreducible subquotient of $\pi(\Psi)\otimes_{\ol\Z_\ell}\ol\F_\ell$ is isomorphic to $\pi(\ol\Psi_0)$ for some irreducible subquotient $\ol\tau_0$ of $\ol\tau$.
\end{enumerate}
If $\Psi$ is irredundant (resp. normalized), then every $\ol\Psi_0$ is irredundant (resp. normalized).
\end{lemma}

\begin{proof}
First observe that \cite[Lemma~2.10.1]{Cot26b} and exactness of the formation of Jacquet modules (since $\ell \neq p$) imply that $\ol\tau$ and each of its irreducible subquotients are cuspidal. Next, note that modular reduction preserves the depth and genericity of each $\phi_i$, since $\ell \neq p$. Thus $\ol\Psi$ is a Yu datum and every $\ol\Psi_0$ is an irreducible Yu datum. Now \eqref{eq:yu-modular-base-change} and (1) follow from unraveling the definitions. 

Since Yu's construction is exact as a functor of the depth $0$ representation, applying it to a composition series of $\ol\tau$ proves part~(2); the resulting representations are irreducible by \cite[Theorem~3.1]{Fin22}. If $\Psi$ is normalized, its reductions are normalized because each $\ol\phi_i$ remains normalized. Preservation of irredundancy is also clear. 
\end{proof}

\begin{lemma}\label{lemma:torus-character-modular-reduction}
Retain the notation of Lemma~\ref{lemma:yu-datum-modular-reduction}, and assume that $\Psi \otimes_{\ol\Z_\ell} \ol\Q_\ell$ is irreducible. For an irreducible subquotient $\ol\tau_0$ of $\ol\tau$ let $\ol\Psi_0$ be the associated normalized irreducible Yu datum. For every class $[(T,\theta)]\in\cT(\Psi_{\ol\Q_\ell})$, the character $\theta$ takes values in $\ol\Z_\ell^\times$, and its reduction $\ol\theta$ satisfies $[(T,\ol\theta)]\in\cT(\ol\Psi_0)$. In particular,
\begin{enumerate}
    \item $\rho_I^{\Kal}(\ol\Psi_0)$ is the (semisimplified) modular reduction of $\rho_I^{\Kal}(\Psi_{\ol\Q_\ell})$.
    \item If $\ol\Psi_0$ is $F$-non-singular, then $\Psi_{\ol\Q_\ell}$ is $F$-non-singular and $\rho^{\Kal}(\ol\Psi_0)$ is the (semisimplified) modular reduction of $\rho^{\Kal}(\Psi_{\ol\Q_\ell})$.
\end{enumerate}
\end{lemma}

\begin{proof}
Fix $[(T,\theta)]\in\cT(\Psi_{\ol\Q_\ell})$, arising from a depth $0$ character $\theta_0$. The group $\ol T(\F_q)$ is finite modulo $Z(\ol G^0_{[x]})(\F_q)$. On this central subgroup, \cite[Definition~2.8.2]{Cot26b} identifies $\theta_0$ with the central character of $\tau\otimes_{\ol\Z_\ell}\ol\Q_\ell$, which is $\ol\Z_\ell^\times$-valued because $\tau$ is a finite free $\ol\Z_\ell$-lattice. Hence $\theta_0$ is $\ol\Z_\ell^\times$-valued on $\ol T(\F_q)$, while each $\phi_i$ is $\ol\Z_\ell^\times$-valued by hypothesis. Thus
\[
\theta=\theta_0\prod_{i=0}^d\phi_i|_{T(F)_{[x]}}
\]
is $\ol\Z_\ell^\times$-valued. Since $\ol\tau_0$ is an irreducible subquotient of the reduction of $\tau$, \cite[Definition~2.8.2]{Cot26b} shows $\ol\tau_0 \in \cE(\ol G^0_{[x]}, [\ol T,\ol\theta_0])$. Hence $[(T,\ol\theta)]\in\cT(\ol\Psi_0)$. From this, (1) and (2) are both clear from the definitions.
\end{proof}

\subsection{Compatibility with central isogenies}

Finally, we record a basic compatibility, which is improved considerably for non-singular supercuspidal representations with characteristic $0$ coefficients in \cite{BM25} under some technical conditions. Fix an $F$-homomorphism $f\co G \to G'$ inducing an isomorphism $G_{\ad} \cong G'_{\ad}$, let $k \in \{\ol\Q_\ell, \ol\F_\ell\}$, and let $\Psi' = ((G'^i), x', (r'_i), \tau', (\phi_i'))$ be a normalized irreducible irredundant Yu datum for $G'$ with coefficients in $k$. Choose a Yu datum $\Psi = ((G^i), x, (r_i), \tau, (\phi_i))$ as follows:
\begin{enumerate}
\item $G^i = f^{-1}(G'^i)$ for all $i$,
\item $x \in \cB(G)$ is any point with the same image as $x'$ in $\cB(G_{\ad}) = \cB(G'_{\ad})$,
\item $r_i = r'_i$ for all $i$,
\item $\tau$ is an irreducible constituent of $\tau'$ as a $\ol G^0_{[x]}(\F_q)$-representation,
\item $\phi_i = \phi_i' \circ f$ for all $i$.
\end{enumerate}
Suppose that each $\phi_i$ has the same depth as $\phi_i'$. In this case, the fact that $\Psi$ is a normalized irredundant irreducible Yu datum is clear.

\begin{lemma}\label{lemma:kaletha-compatibility-with-central-isogenies}
If $\ld f\co \ld G' \to \ld G$ is the dual L-homomorphism, then $\rho_I^{\Kal}(\Psi) \sim \ld f \circ \rho_I^{\Kal}(\Psi')$. If $\Psi$ is $F$-non-singular, then $\Psi'$ is also $F$-non-singular and $\rho^{\Kal}(\Psi) \sim \ld f \circ \rho^{\Kal}(\Psi')$.
\end{lemma}

\begin{proof}
If $[(T', \theta')] \in \cT(\Psi')$ and we define $T = f^{-1}(T')$, $\theta = \theta' \circ f|_{T(F)_{[x]}}$, then $[(T, \theta)] \in \cT(\Psi)$: indeed, the definition of $\epsilon_{G'}$ involves passing to $G'_{\ad}$, so it is clear that $\epsilon_G = \epsilon_{G'} \circ f$, and it is enough to invoke \cite[Lemma~2.10.3]{Cot26b}. Moreover, we have $\ld j_{T,G} \circ \ld(f|_T) \sim \ld f\circ \ld j_{T',G'}$ by Lemma~\ref{lem:unramified-L-embedding-properties}(2). This implies the claim.
\end{proof}

\section{Base change functoriality for Yu data}\label{sec:bc-functoriality-yu-data}

Throughout this section, assume that $p\neq 2$. We will study the behavior of $\rho^{\Kal}$ and $\rho_I^{\Kal}$ under base change in two cases: base change of ``large'' prime degree for arbitrary cuspidal representations, base change of ``small'' prime degree for ``toral'' representations (i.e., those arising from Yu data such that $G^0$ is a torus). Both cases have important simplifying features. 
\begin{itemize}
\item In the first case, individual apartments in the Bruhat--Tits building are unchanged after base change of large prime degree (see \cite[Proposition~4.1.2]{Cot26b}).
\item In the second case, the reduced building of a torus is a point, so no issues from Bruhat--Tits theory arise. Furthermore, the irreducible representations of an abelian group are $1$-dimensional, so no complications from Deligne--Lusztig theory arise.
\end{itemize} 

\subsection{Preliminaries on characters}\label{ss:technical-character-results}

In this section, we collect a number of technical results which will allow us to extend characters of unramified twisted Levi subgroups after base change. Throughout this subsection, let $k\in\{\ol\F_\ell,\ol\Q_\ell,\ol\Z_\ell\}$, and let $G^i \subset G^{i+1} \subset G$ be tamely ramified twisted Levi subgroups.

Let $E/F$ be a tamely ramified cyclic extension of prime degree $\ell\neq p$, let $\phi\co G^i(F) \to k^\times$ be a normalized character (in the sense of Definition~\ref{def:normalized-character}), and let $\phi_E\co G^i(E) \to k^\times$ be the character defined in Proposition~\ref{prop:extend-normalized-characters}.

\begin{lemma}\label{lemma:depth-stays-same}
The characters $\phi$ and $\phi_E$ are of the same depth.
\end{lemma}

\begin{proof}
Suppose that $\phi$ is of depth $r$. By definition, this means that
\begin{enumerate}
    \item there exists a point $x \in \cB(G^i)$ such that $\phi|_{G^i(F)_{x,r+}}$ is trivial,
    \item if $r> 0$, then for any point $y \in \cB(G^i)$, the restriction $\phi|_{G^i(F)_{y,r}}$ is nontrivial.
\end{enumerate}
First, we show that $\phi_E|_{G^i(E)_{x,r+}}$ is trivial. Let $x \in \cB(G^i)$ be a point such that condition (1) holds for $\phi$, and let $T \subset G^i$ be a tamely ramified maximal $F$-torus with $x \in \cB(T)$. By Lemma~\ref{lemma:character-restricted-to-torus-depth}, the restriction $\phi|_{T(F)}$ is of depth $r$. The norm map $\Nm_{E/F}\colon T(E)_{x,r+} \to T(F)_{x,r+}$ is surjective by \cite[Lemma 3.1.3]{Kal19}, so indeed (1) holds.

Now suppose $r > 0$, let $y \in \cB(G^i_E)$ be any point, and suppose for the sake of contradiction that $\phi_E|_{G^i(E)_{y,r}}$ is trivial. Let $c$ be the center of mass of the $\Gal(E/F)$-orbit of $y$, so $c\in\cB(G^i)$ by \cite[Theorem 12.9.2]{KP}. Since $\phi_E$ is a character, it is trivial on $G^i(E)_{g\cdot y,r}$ for every $g\in G^i(E)$. Choose an apartment $A$ containing $c$ and a point $y_0\in A\cap G^i(E)\cdot y$, let $S$ be the maximal $E$-split torus defining $A$, and put $Z=Z_{G^i_E}(S)$. For each relative root $\alpha$, choose a vector $v_\alpha$ in the image of $S(E)$ in $\Aut(A)$ with $\langle \alpha, v_\alpha\rangle>0$. Then $y_0+n v_\alpha$ lies in $A\cap G^i(E)\cdot y$, and \cite[Lemma 6.1.6]{KP} gives
\[
U_\alpha(E)_{c,r}\subset U_\alpha(E)_{y_0+n v_\alpha,r}
\]
for all sufficiently large $n$: indeed, once $\alpha(y_0)+n\alpha(v_\alpha)\geq\alpha(c)$, we have
\[
r-\alpha(c)\geq r-\alpha(y_0)-n\alpha(v_\alpha),
\]
so the condition defining $U_\alpha(E)_{c,r}$ is stronger than the one defining $U_\alpha(E)_{y_0+n v_\alpha,r}$.

By \cite[Definition 13.2.1]{KP}, the group $G^i(E)_{c,r}$ is generated by $Z(E)_r$ and these $U_\alpha(E)_{c,r}$. Hence $\phi_E$ is trivial on $G^i(E)_{c,r}$. By \cite[Proposition 12.9.4]{KP}, this implies that $\phi^\ell$ is trivial on $G^i(F)_{c,r}$, contradicting the depth of $\phi$. Thus $\phi_E|_{G^i(E)_{y,r}}$ is nontrivial, as desired.
\end{proof}

\begin{lemma}\label{lemma:generic-stays-same}
If $\phi$ is of depth $r > 0$, then $\phi$ is $G^{i+1}$-generic of depth $r$ if and only if $\phi_E$ is $G^{i+1}_E$-generic of depth $r$.
\end{lemma}

\begin{proof}
Choose a z-extension $\wt G \to G$. By \cite[Lemma 3.5.3]{Kal19}, the map $\wt G(F)_{x,r} \to G(F)_{x,r}$ is surjective, so by construction of $\phi_E$ we may pass from $G$ to $\wt G$ to assume that $G_{\der}$ is simply connected. Therefore hypothesis $\mathrm{C}(\vec G)$ of \cite[\S 2.6]{HM08} holds by \cite[Lemma 3.5.2]{Kal19}. Recall that we have fixed a choice of additive character $\psi\co F \to k^\times$. Let $x \in \cB(G^i)$ be a point such that $\phi|_{G^i(F)_{x,r}}$ factors through a character of $G^i(F)_{x,r}/G^i(F)_{x,r+}$. Under the Moy--Prasad isomorphism 
\[
G^i(F)_{x,r}/G^i(F)_{x,r+} \cong \frg^i(F)_{x,r}/\frg^i(F)_{x,r+},
\]
Hypothesis $\mathrm{C}(\vec G)$ shows that there is an element $X^* \in \Lie^*(G^i)^{G^i}(F)_{-r}$ such that $\phi|_{G^i(F)_{x,r}}$ is induced by $X^*$ and $\psi$. We may consider $X^*$ as an element of $\Lie^*(G^i)^{G^i}(E)_{-r}$, in which case we will denote it by $X^*_E$ for clarity. Then the restriction of $\phi_E$ to $G^i(E)_{x,r}$ is induced by $X^*_E$ and the additive character $\psi \circ \Tr_{E/F}$, since both of these characters are the $\ell$th powers of the unique $\Gal(E/F)$-stable extension of $\phi|_{G^i(F)_{x,r}}$ to $G^i(E)_{x,r}$.



For a twisted Levi $M \subset G$, let $M_{\mathrm{sc}}$ denote the preimage of $M$ in the universal cover $G_{\mathrm{sc}}$ of $G_{\der}$, and fix a maximal $F$-torus $T$ of $G^i$. For a root $\alpha \in \Phi(G^{i+1}_{\ol F}, T_{\ol F})$, let $\alpha_{\mathrm{sc}}^\vee$ denote the $\ol F$-coroot $\G_m \to (G^{i+1}_{\mathrm{sc}})_{\ol F}$, and let $H_\alpha = \mathrm{d}\alpha_{\mathrm{sc}}^\vee(1) \in \Lie(G^{i+1})(\ol F)$. Recall from \cite[\S 8]{Yu01} (noting the mild correction of \cite[Remark 4.1.3]{FKS23}) that genericity of $\phi$ is equivalent to the conjunction of the following two conditions:
\begin{enumerate}
    \item We have $v(\langle X^*, H_\alpha\rangle) = -r$ for all $\alpha \in \Phi((G^{i+1}/G^i)_{\ol F}, T_{\ol F})$.
    \item Recall that there is a natural isomorphism
    \[
    \Lie^*(G^i_{\mathrm{sc},\ab}) \otimes_F \ol F \cong X^*((G^i_{\mathrm{sc},\ab})_{\ol F}) \otimes_{\Z} \ol F \subset X^*((T_{\mathrm{sc}})_{\ol F}) \otimes_{\Z} \ol F.
    \]
    Let $\varpi_r \in \ol F$ denote an element of valuation $r$, and let $\ol X^*$ denote the residue class of $\varpi_r X^*$ in $X^*((G^i_{\mathrm{sc},\ab})_{\ol F}) \otimes_{\Z} \ol\F_q$. Then the stabilizer of $\ol X^*$ in the absolute Weyl group of $G^{i+1}$ is equal to the absolute Weyl group of $G^i$.
\end{enumerate}
The conditions for genericity of $\phi_E$ are completely similar. It is clear that these two conditions hold for $\phi$ if and only if they hold for $\phi_E$.
\end{proof}

\subsection{Large prime degree base change}\label{ss:large-prime-deg-bc-kal}

In this section, let $k\in\{\ol\F_\ell,\ol\Z_\ell,\ol\Q_\ell\}$ be the coefficient ring. Suppose that $\ell$ is a banal prime for $G$ such that $\ell > \rk G + 1$, and note that $\ell$ is also banal for $G_{F_\ell}$ by \cite[Proposition~4.1.2(2)]{Cot26b}. Throughout the subsection, let
\[
\Psi=((G^i)_{0\leq i\leq d},x,(r_i),\tau,(\phi_i))
\]
be a normalized irreducible Yu datum for $G$ over $k$.

\begin{lemma}\label{lemma:depth-zero-twisting-character}
There is a character $\chi\co G^0(F)\to k^\times$, trivial on $G^0(F)^1$, such that $\tau \otimes \chi^{-1}$ factors through a finite quotient of $G^0(F)_{[x]}$ of order prime to $\ell$.
\end{lemma}

\begin{proof}

Choose a splitting $Z(G)(F) \cong Z(G)(F)_{\mathrm{b}} \times M$, where $M$ is a finitely generated free abelian group. The induced map $M \to G^0(F)/G^0(F)^1$ is injective, and since $k^\times$ is divisible there exists an extension of the central character of $\tau|_M$ to $G^0(F)/G^0(F)^1$. Since $\ell$ is banal for $G(F)$, it follows that the central character of $\tau \otimes\chi^{-1}$ is finite of order prime to $\ell$, and in particular $\tau\otimes\chi^{-1}$ has finite image.
\end{proof}

\begin{defn}\label{def:depth-zero-base-change}
For $k \in \{\ol\F_\ell,\ol\Z_\ell,\ol\Q_\ell\}$, the \emph{base change of $\tau$} is the $k$-representation $\tau_\ell$ of $G^0(F_\ell)_{[x]}$ defined as follows: first, choose a character $\chi$ as in Lemma~\ref{lemma:depth-zero-twisting-character}. Let $\ol{\wt\tau}_\chi$ denote the Frobenius twist of the $\ell$-modular reduction of the Glauberman correspondent (in the sense of \cite[\S 3.5.3]{Cot26b}) of $\tau \otimes \chi^{-1}$, and let $\wt\tau_\chi$ denote the unique irreducible lift of $\ol{\wt\tau}_\chi$ which factors through a finite quotient of $G^0(F_\ell)_{[x]}$ of order prime to $\ell$. Note that $\wt\tau_\chi$ is irreducible by \cite[Part III, no.\ 15.5, Proposition 43]{Serre77}. Then set
\begin{equation}\label{eq:def-depth-zero-base-change}
\tau_\ell
\coloneqq
\wt\tau_{\chi}
\otimes
\chi_{F_\ell}\big|_{G^0(F_\ell)_{[x]}}.
\end{equation}
\cite[Remark~3.5.10]{Cot26b} shows that $\tau_\ell$ does not depend on the choice of $\chi$; in particular, if $\tau$ itself factors through a finite quotient of order prime to $\ell$, one may take $\chi=1$.
If $G^0=T$ is a torus, then \eqref{eq:def-depth-zero-base-change} shows that $\tau_\ell$ is the character $\tau_{F_\ell}$ of Proposition~\ref{prop:extend-normalized-characters}.
\end{defn}

We define a Yu datum $\Psi_\ell = ((G_\ell^i), x_\ell, (r_{\ell,i}), \tau_\ell, (\phi_{\ell,i}))$ as follows.
\begin{enumerate}
\item $G_\ell^i = (G^i)_{F_\ell}$ for $0 \leq i \leq d$.
\item $x_\ell = x$, considered as a point of $\cB(G^0_{F_\ell})$; its image in $\cB((G^0_{F_\ell})_{\der})$ is a vertex by \cite[Proposition~4.1.2(3)]{Cot26b}.
\item $r_{\ell,i} = r_i$ for all $i$.
\item Let $\tau_\ell$ be the representation of Definition~\ref{def:depth-zero-base-change}; this is irreducible and cuspidal by \cite[Lemma~3.5.9 and Proposition~4.2.2]{Cot26b}.
\item Let $\phi_{\ell,i}$ be the character of $G^i(F_\ell)$ defined as $(\phi_i)_{F_\ell}$ in Proposition~\ref{prop:extend-normalized-characters}. Note that $\phi_{\ell,i}$ is $(G^{i+1})_{F_\ell}$-generic of depth $r_i$ for $0\leq i<d$ by Lemma~\ref{lemma:depth-stays-same} and Lemma~\ref{lemma:generic-stays-same}.
\end{enumerate}
Observe that the isomorphism class of $\pi(\Psi_\ell)$ is $\Gal(F_\ell/F)$-stable. If $k = \ol\Z_\ell$ and $\ol\Psi$ denotes the modular reduction of $\Psi$, then the modular reduction of $\Psi_\ell$ is equal to $(\ol\Psi)_\ell$.

\begin{prop}\label{prop:kal-base-change}
Let $\ell$ be a banal prime for $G$ such that $\ell > \rk G + 1$. Assume that $k\in\{\ol\F_\ell,\ol\Q_\ell\}$, and let $\Psi_\ell$ be as above. Then
\[
\rho_I^{\Kal}(\Psi_\ell)\sim
\rho_I^{\Kal}(\Psi)|_{I_{F_\ell}}.
\]
If $\Psi$ is moreover $F$-non-singular, then so is $\Psi_\ell$, and
\[
\rho^{\Kal}(\Psi_\ell)\sim
\rho^{\Kal}(\Psi)|_{W_{F_\ell}}.
\]
\end{prop}

\begin{proof}
If $k=\ol\F_\ell$, then Lemma~\ref{lemma:depth-zero-twisting-character}, \cite[Lemma~3.5.16]{Cot26b}, and Teichm\"uller lifting give a normalized integral lift of $\Psi$ with cuspidal generic fiber. After applying Lemma~\ref{lemma:irredundant-yu-data} to assume $\Psi$ is irredundant, Lemmas~\ref{lemma:yu-datum-modular-reduction} and~\ref{lemma:torus-character-modular-reduction} reduce both assertions to the case $k = \ol\Q_\ell$. Let $[(T, \theta)] \in\cT(\Psi)$ as in \S\ref{ss:kal-param}. For the first statement, write $\theta\coloneqq\theta_0\prod_{i=0}^d\phi_i|_{T(F)_{[x]}}$ as in the construction of $\cT(\Psi)$, so $\tau \in \cE(\ol G^0_{[x]}, [\ol T, \theta_0])$. By \cite[Proposition~4.3.4(2)]{Cot26b} and the definitions, the representation $\tau_\ell$ lies in the semi-rational Lusztig series $\cE((\ol G^0_{[x]})_{\F_{q^\ell}}, [\ol T_{\F_{q^\ell}}, \theta_{\ell,0}])$, where $\theta_{\ell,0} = \theta_0 \circ \Nm_{\F_{q^\ell}/\F_q}$.

Let
\[
\theta_\ell
\coloneqq
\theta_{\ell,0}
\prod_{i=0}^d\phi_{\ell,i}|_{T(F_\ell)_{[x]}}.
\]
The definition of $\Psi_\ell$ gives
\[
[(T_{F_\ell},\theta_\ell)]
=
[(T_{F_\ell},\theta\circ \Nm_{F_\ell/F})]
\in\cT(\Psi_\ell).
\]
Norm functoriality for the inertial Local Langlands Correspondence for tori identifies the parameter of $\theta\circ \Nm_{F_\ell/F}$ with $\ld\theta|_{I_{F_\ell}}$.
Thus
\begin{align*}
\rho_I^{\Kal}(\Psi)|_{I_{F_\ell}}
    &\sim \ld j_{T,G}\circ\ld\theta|_{I_{F_\ell}}\\
    &\sim \ld j_{T_{F_\ell},G_{F_\ell}}
        \circ\ld(\theta\circ \Nm_{F_\ell/F})|_{I_{F_\ell}}\\
    &\sim \rho_I^{\Kal}(\Psi_\ell).
\end{align*}
The second relation follows from Lemma~\ref{lemma:associated-tori-inertial-parameter}, and the third follows from Lemma~\ref{lemma:tasho-base-change}(2) and norm functoriality. If $\Psi$ is $F$-non-singular, then so is $\Psi_\ell$, and the same calculation holds on all of $W_{F_\ell}$.
\end{proof}

\subsection{Small prime degree base change}\label{ss:small-prime-deg-bc-kal}
We do not currently have a robust general theory of low degree base change for Yu data, but for our main results we need only the following special case. Let $k\in\{\ol\F_\ell,\ol\Q_\ell\}$, and let
\[
\Psi=((G^i)_{0\leq i\leq d},x,(r_i),\tau,(\phi_i))
\]
be a normalized irreducible Yu datum for $G$ over $k$ such that $G^0=T$ is a torus. In particular, $\tau$ is a character. Let $E/F$ be a Galois extension of prime degree $\ell\ne p$, and assume that $T_E$ is an elliptic maximal $E$-torus of $G_E$. This implies that $T(E)=T(E)_{[x]}$ and $T(F)=T(F)_{[x]}$.

Define a Yu datum
\[
\Psi_E=((G_E^i),x_E,(r_{E,i}),\tau_E,(\phi_{E,i}))
\]
as follows.
\begin{enumerate}
\item $G_E^i = (G^i)_E$ for $0 \leq i \leq d$.
\item $x_E = x$, considered as a point of $\cB(G^0_E)$; its image in $\cB((G^0_E)_{\der})$ is the unique point, hence a vertex.
\item $r_{E,i} = r_i$ for $0 \leq i \leq d$.
\item $\tau_E = ((\tau\epsilon_F\epsilon_{F,\sharp,x}) \circ \Nm_{E/F})\epsilon_E\epsilon_{E,\sharp,x}$, where $\epsilon_F$, $\epsilon_{F,\sharp,x}$ and $\epsilon_E$, $\epsilon_{E,\sharp,x}$ are the sign characters associated to $(\vec G, x, \vec r, \vec\phi)$ and $(\vec G_E, x, \vec r, \vec\phi_E)$, respectively, as in \S\ref{sssec:eps-chars-descent-Levi}; this is trivially irreducible and cuspidal.
\item $\phi_{E,i} = (\phi_i)_E$ as in Proposition~\ref{prop:extend-normalized-characters}; this is $G_E^{i+1}$-generic of depth $r_i$ for each $0 \leq i < d$ by Lemmas~\ref{lemma:depth-stays-same} and~\ref{lemma:generic-stays-same}.
\end{enumerate}
Observe that $\Psi_E$ is normalized.

\begin{prop}\label{prop:kal-small-degree-bc}
If $\ell$ is odd, then
\[
\rho^{\Kal}(\Psi)|_{W_E}
\sim
\rho^{\Kal}(\Psi_E).
\]
The same holds if $\ell$ is even and $k = \ol\F_2$.
\end{prop}

\begin{proof}
Ellipticity of $T_E$ implies that $T(E)=T(E)_{[x]}$ and $T(F)=T(F)_{[x]}$. Since $G^0=T$, the torus--character pairs associated to $\Psi$ and $\Psi_E$ are $(T,\theta)$ and $(T_E,\theta_E)$, where
\[
\theta=\tau\prod_{i=0}^d\phi_i|_{T(F)},
\qquad
\theta_E=\tau_E\prod_{i=0}^d\phi_{E,i}|_{T(E)}.
\]
By construction, we have $
\theta_E=((\theta\epsilon_F\epsilon_{F,\sharp,x})\circ \Nm_{E/F}) \cdot \epsilon_E\epsilon_{E,\sharp,x}$. By Proposition~\ref{prop:ramified-kal-bc}, we therefore have
\begin{align*}
    \rho^{\Kal}(\Psi)|_{W_E} \sim \ld j_{T,G}\circ\ld\theta|_{W_E} \sim \ld j_{T_E,G_E}\circ\ld\theta_E \sim \rho^{\Kal}(\Psi_E),
\end{align*}
as desired.
\end{proof}

\section{Unramified twisted Levi functoriality for Yu data}\label{ss:kaletha-functoriality}

Throughout this section, continue to assume $p \neq 2$ and let $\Psi = ((G^i)_{0\leq i\leq d}, x, (r_i)_{0\leq i\leq d}, \tau, (\phi_i)_{0\leq i\leq d})$ be a normalized irreducible Yu datum for $G$ with coefficients in a field $k \in \{\ol\Q_\ell,\ol\F_\ell\}$. Fix an unramified twisted Levi subgroup $H\subset G$ for which there exists $[(T, \theta)] \in \cT(\Psi)$ with $T \subset H$, let $H^i=H\cap G^i$, and fix an embedding $\cB(H^0) \subset \cB(G^0)$.

This section constructs a Yu datum $\Psi_L$ on a Levi subgroup $L$ of an unramified twisted Levi subgroup $H\subset G$ for which
\[
\rho_I^{\Kal}(\Psi)\sim \ld j_{L,G}\circ\rho_I^{\Kal}(\Psi_L).
\]
When $\Psi$ is $F$-non-singular, the construction gives $L=H$ and proves the full comparison
\[
\rho^{\Kal}(\Psi)\sim \ld j_{L,G}\circ\rho^{\Kal}(\Psi_L).
\]
In general, one cannot take $L = H$; this reflects the fact that functoriality for twisted Levis does not always preserve cuspidality. For instance, if $\theta_{10}$ is the unipotent supercuspidal $\ol\Q_\ell$-representation of $\Sp_4(F)$ and $H \subset G$ is a twisted Levi such that $H \cong \SL_2 \times S$ for an anisotropic $F$-torus $S$, then each smooth representation of $H(F)$ functorially associated to $\theta_{10}$ is unipotent, and $H(F)$ has no supercuspidal unipotent representations.

\subsection{Constructing the Levi}\label{sss:descended-levi-construction}

Let $(\ol T', \theta_0')$ be a pair consisting of a generalized maximal torus $\ol T'$ of $\ol H^0_{[x]_G}$ and a character $\theta_0'\co \ol T'(\F_q) \to k^\times$ such that $\tau$ lies in the semi-rational Lusztig series $\cE(\ol G^0_{[x]}, [\ol T', \theta_0'])$, with notation as in \cite[Definition~2.8.2]{Cot26b}. By \cite[Lemma~2.8.4]{Cot26b}, this set is uniquely determined by $\tau$.

We first extract an irreducible $k$-representation of $\ol H^0_{[x]_G}(\F_q)$.\footnote{We are careful to write $\ol H^0_{[x]_G}$ instead of $\ol H^0_{[x]}$ for clarity, because in this generality it may happen that $[x]_H \neq [x]_G$.} Let $\tau_{H,0}$ be an irreducible $k$-representation of $\ol H^0_{[x]_G}(\F_q)$ occurring in $\cE(\ol H^0_{[x]_G}, [\ol T', \theta_0'])$. Note that $\tau_{H,0}$ might not be cuspidal (see \cite[Remark~3.4.4]{Cot26b}), so it cannot be the depth $0$ representation in a Yu datum for $H$. Instead we will use parabolic induction to descend further to a Levi subgroup of $H$.

By \cite[Lemma~2.10.2]{Cot26b}, there is a split $\F_q$-subtorus $\ol S^\circ \subset \ol H^0_{[x]_G}$ which is the maximal central split $\F_q$-torus of $\ol L \coloneqq Z_{\ol H^0_{[x]_G}}(\ol S^\circ)$, a parabolic $\F_q$-subgroup $\ol P^\circ \subset (\ol H^0_{[x]_G})^\circ$ with Levi factor $\ol L^\circ$, and an irreducible cuspidal $k$-representation $\tau_{L,0}$ of $\ol L(\F_q)$ such that $\tau_{H,0}$ is an irreducible subrepresentation of $\ind_{\ol P(\F_q)}^{\ol H^0_{[x]_G}(\F_q)}(\tau_{L,0})$, where $\ol P = \ol L \cdot \ol P^\circ$.

Let $\cH^0_{[x]_G}$ be the smooth separated $\cO_F$-model of $H^0$ with $\cH^0_{[x]_G}(\cO_F) = H^0(F)_{[x]_G}$, so $\cH^0_{[x]_G}$ is an open $\cO_F$-subgroup scheme of $\cH^0_{[x]_H}$. Let $\cS \subset \cH^0_{[x]_G}$ be an $\cO_F$-subtorus lifting $\ol S^\circ$, as in the construction \S\ref{sss:p-adic-torus-character-pair}. Let $S$ be the generic fiber of $\cS$ and set $L=Z_H(S)$. For all $0\leq i\leq d$, put
\[
L^i \coloneqq L \cap G^i = L \cap H^i.
\]

Since $\cS\subset\cH^0_{[x]_G}$, \cite[\S9.7.6(2)]{KP} shows that there is a maximally unramified maximally split maximal $F$-torus $T_x\subset H^0$ containing $S$ whose apartment in $\cB(H^0)$ contains $x$; we will also regard $x$ as a point of $\cB(T_x)$ via a choice of embedding $\cB(T_x) \subset \cB(L^0)$. By definition of $L$, we must then have $T_x\subset L^0$. Fixing an embedding of $\cB(L^0)$ in $\cB(H^0)$, we may regard $x$ as a point of $\cA(T_x) \subset \cB(L^0)$. Note that the Moy--Prasad filtration depends only on the image $[x]_L$ of $x$ in $\cB(L^0_{\der})$, hence is independent of this choice of lift. Recall the notation $\cL^0_{[x]_L}$ and $\ol L^0_{[x]_L}$ from \S\ref{ssec:notation}. 

\begin{lemma}\label{lemma:descended-yu-datum-anisotropic}
The torus $S$ is the maximal $F$-split central subtorus of $L^0$, and thus $Z_{L^0}/Z_L$ is $F$-anisotropic.
\end{lemma}

\begin{proof}
Let $A$ be the maximal $F$-split central subtorus of $L^0$, so $S \subset A$. By \cite[Axiom 4.1.20]{KP}, if $\cA$ is the split $\cO_F$-torus with generic fiber $A$, then the inclusion $A \subset L^0$ extends to an $\cO_F$-homomorphism $\cA \to \cL^0_{[x]_L}$ which is a closed embedding. But then the special fiber $\ol A^\circ$ is an $\F_q$-split central subtorus of $(\ol L^0_{[x]_L})^\circ$, so $\ol A^\circ = \ol S^\circ$ by choice of $\ol S^\circ$. For dimension reasons, it follows that $S=A$.
\end{proof}

Note that Lemma~\ref{lemma:descended-yu-datum-anisotropic} shows that $\cB(L^0_{\der}) \subset \cB(L_{\der})$, so the notation $[x]_L$ above is consistent with \S\ref{ssec:notation}.

\begin{lemma}\label{lemma:descended-yu-datum-vertex}
The point $[x]_L$ is a vertex in $\cB(L^0_{\der})$.
\end{lemma}

\begin{proof}
Let $\cF$ be the facet of $\cB(L^0_{\der})$ containing $[x]_L$, and let $y$ be a vertex in the closure of $\cF$; we aim to show $y=[x]_L$. By \cite[Axiom 4.1.22]{KP}, there is an associated $\cO_F$-homomorphism $(\cL^0_{[x]_L})^\circ \to (\cL^0_y)^\circ$ between the associated parahoric models, and the associated map on special fibers gives rise to an isomorphism of $(\ol L^0_{[x]_L})^\circ$ with a Levi factor of a proper parabolic $\F_q$-subgroup of $(\ol L^0_y)^\circ$. In particular, if $[x]_L\neq y$ then the maximal $\F_q$-split central torus of $(\ol L^0_{[x]_L})^\circ$ is of strictly larger dimension than that of $(\ol L^0_y)^\circ$.

By construction, the $\F_q$-split central torus $\ol S^\circ$ of $(\ol L^0_{[x]_L})^\circ$ is maximal, and $\ol S^\circ$ is the special fiber of an $\cO_F$-torus $\cS \subset \cL^0_{[x]_L}$ with generic fiber $S$. Since $S$ is an $F$-split central torus in $L^0$, it follows that the maximal $\F_q$-split central torus of $(\ol L^0_y)^\circ$ is of dimension at least $\dim S$, concluding the proof.
\end{proof}

\subsection{Constructing the depth $0$ representation}\label{sssec:descended-depth-zero-rep}

We are almost ready to define a Yu datum $\Psi_L$, but there are two remaining technical subtleties:
\begin{itemize}
\item $\tau_{L,0}$ is a representation of $\ol L^0_{[x]_G}(\F_q)$, which is generally smaller than $\ol L^0_{[x]_L}(\F_q)$.
\item Even when $\ol L^0_{[x]_G}(\F_q) = \ol L^0_{[x]_L}(\F_q)$, the representation $\tau_{L,0}$ must be twisted by a sign character to give the ``correct'' $\Psi_L$.
\end{itemize}

Let $\cL^0_{[x]_G}\coloneqq Z_{\cH^0_{[x]_G}}(\cS)$, a smooth separated $\cO_F$-model of $L^0$ with
\[
\cL^0_{[x]_G}(\cO_F)=L^0(F)_{[x]_G}\coloneqq L^0(F)\cap G^0(F)_{[x]_{G^0}}.
\]
Let $\ol L^0_{[x]_G}$ denote the quotient of $(\cL^0_{[x]_G})_{\F_q}$ by the unipotent radical of $(\cL^0_{[x]_G})_{\F_q}^\circ$, so $\ol L^0_{[x]_G}$ is a paraductive $\F_q$-group scheme.

By \cite[Proposition~2.6.2]{Cot26b}, there is a generalized maximal torus-character pair $(\ol T_L,\eta_0)$ in $\ol L^0_{[x]_G}$ such that
\[
\tau_{L,0}\in\cE(\ol L^0_{[x]_G},[\ol T_L,\eta_0]).
\]
By \cite[Proposition~2.8.5]{Cot26b}, for either $k\in\{\ol\Q_\ell,\ol\F_\ell\}$ we have
\[
\tau_{H,0}\in\cE(\ol H^0_{[x]_G},[\ol T_L,\eta_0]).
\]

\subsubsection{The sign twisting}
Fix an $F$-parabolic subgroup $P\subset H$ with Levi factor $L$. Since $Z(G^0)/Z(G)$ is $F$-anisotropic, the groups $G^0$ and $G$ share a maximal $F$-split central subtorus. By compatibility of the Moy--Prasad filtration with respect to Levi $F$-subgroups, we have
\[
L^0(F)_{x,0+}=L^0(F)\cap H^0(F)_{x,0+}.
\]
Let $\epsilon_{\sharp,x}^{\vec G, L}$ be the (depth $0$) sign character of $L^0(F)_{[x]_G}$ defined in \eqref{eqn:sharp-x-difference}.
Define a sign character $\kappa_{G,L,x}$ on $L^0(F)_{[x]_G}$ by
\begin{equation}\label{eq:ambient-full-sharp-character}
\kappa_{G,L,x}
=\epsilon_G\epsilon_L\epsilon_{\sharp,x}^{\vec G, L}.
\end{equation}



\subsubsection{The extension} 

By \cite[Lemma~2.1.3]{Cot26b}, we have
\begin{equation}\label{eq:ambient-smaller-quotient}
\ol L^0_{[x]_G}=Z_{\ol H^0_{[x]_G}}(\ol S^\circ).
\end{equation}
By construction, the character $\kappa_{G,L,x}$ of \eqref{eq:ambient-full-sharp-character} factors through $\ol L^0_{[x]_G}(\F_q)$; we use the same notation for the resulting quotient character. Let $\ol T_{L,1} = Z_{\ol L^0_{[x]_L}}(\ol T_L^\circ)$. By \cite[Lemma~2.10.3]{Cot26b}, if $\tau_L$ is a representation of $\ol L^0_{[x]_L}(\F_q)$ whose restriction to $\ol L^0_{[x]_G}(\F_q)$ admits $\tau_{L,0} \otimes \kappa_{G,L,x}$ as an irreducible subquotient, then there exists a character $\eta_1\co\ol T_{L,1}(\F_q)\to k^\times$ and an irreducible cuspidal $k$-representation $\tau_L$ of $\ol L^0_{[x]_L}(\F_q)$ such that
\begin{align}
\eta_1|_{\ol T_L(\F_q)}
&=\eta_0\,\kappa_{G,L,x}|_{\ol T_L(\F_q)},
\label{eq:signed-descendant-pair}\\
\tau_{L,0}\otimes\kappa_{G,L,x}
&\subset\tau_L|_{\ol L^0_{[x]_G}(\F_q)},
\qquad
\tau_L\in\cE(\ol L^0_{[x]_L},[\ol T_{L,1},\eta_1]).
\label{eq:signed-descendant-restriction}
\end{align}

\subsection{The descended Yu datum and functoriality}\label{ss:unram-levi-yu-datum}
With notation as above, we also use $\tau_L$ to denote the inflation of $\tau_L$ to $L^0(F)_{[x]_L}$.
For every $0\leq i\leq d$, define $\phi_{L,i}\coloneqq \phi_i|_{L^i(F)}$. We may now define the $5$-tuple
\begin{equation}\label{eq:L-Yu-datum}
\Psi_{L,\tau_L} = ((L^i)_{0 \leq i \leq d}, x, (r_i)_{0 \leq i \leq d}, \tau_L, (\phi_{L,i})_{0\leq i \leq d}).
\end{equation}
We proceed to verify that it is a Yu datum.

\begin{lemma}\label{lemma:generic-character-restriction}For every $0\leq i<d$, the character $\phi_{L,i}$ is $L^{i+1}$-generic of depth $r_i$. Moreover, if $r_{d-1}<r_d$, then $\phi_{L,d}$ has depth $r_d$.
\end{lemma}

\begin{proof}
Since each $L^i$ contains a tamely ramified maximal $F$-torus of $G^0$, the depth claims follow from Lemma~\ref{lemma:character-restricted-to-torus-depth} and the fact that the Moy--Prasad filtration restricts in the obvious way from $G^i$ to $L^i$. For the genericity assertion, fix $0 \leq i<d$. If $L^i = L^{i+1}$, then genericity is vacuous, so we may and do assume that $L^i \neq L^{i+1}$.

Observe that $X_i^*|_{\Lie(L^i)}$ is $L^{i+1}$-generic of depth $-r_i$ by definition of genericity. Functoriality of the Moy--Prasad isomorphisms shows that $X_i^*|_{\Lie(L^i)}$ induces $\phi_{L,i}$ on the depth $r_i$ quotient at $x$, so $\phi_{L,i}$ is $L^{i+1}$-generic.
\end{proof}

\begin{prop}\label{prop:descent-to-levi-yu-datum}
The $5$-tuple $\Psi_{L,\tau_L}$ from \eqref{eq:L-Yu-datum} is a normalized irreducible Yu datum for $L$ over $k$. 
\end{prop}

\begin{proof}
It is clear that the constituents of $\Psi_L$ are of the kind described in (a)-(e) of \S\ref{sss:yu-data}, the depth requirement in (e) for $\phi_d|_{L^d(F)}$ holding by Lemma~\ref{lemma:generic-character-restriction}. Condition (1) follows from Lemma~\ref{lemma:descended-yu-datum-anisotropic}. Condition (2) is the statement of Lemma~\ref{lemma:descended-yu-datum-vertex}. For (3), note that $\tau_L$ is an irreducible cuspidal $k$-representation of $\ol L^0_{[x]_L}(\F_q)$ by definition. Finally, condition (4) follows from Lemma~\ref{lemma:generic-character-restriction}. It is obvious that $\Psi_{L,\tau_L}$ is normalized and irreducible.
\end{proof}


Since $L$ is an unramified twisted Levi $F$-subgroup of $G$, Definition~\ref{defn:canonical-l-embeddings} shows that there is a canonical L-embedding $\ld j_{L,G}\co\ld L\to\ld G$.

\begin{thm}\label{thm:kaletha-functoriality}
Suppose that $p \neq 2$. We have
\begin{equation}\label{eqn:kaletha-inertial-functoriality}
\rho_I^{\Kal}(\Psi) \sim \ld j_{L,G} \circ \rho_I^{\Kal}(\Psi_{L,\tau_L}).
\end{equation}
If $\Psi$ is $F$-non-singular, then $L = H$, the datum $\Psi_{L,\tau_L}$ is $F$-non-singular, and
\begin{equation}\label{eqn:kaletha-non-singular-functoriality}
\rho^{\Kal}(\Psi) \sim \ld j_{L, G} \circ \rho^{\Kal}(\Psi_{L,\tau_L}).
\end{equation}
\end{thm}

\begin{proof}
We consider the general case first. We will use the notation of the preceding sections freely. Choose, by the construction of \S\ref{ss:kal-param}, an $\cO_F$-torus $\cS'\subset\cH^0_{[x]_G}$ whose special fiber maps onto $\ol T'^{\circ}$ (with notation as in \S\ref{sss:descended-levi-construction}), and let $S'$ be its generic fiber. Let
\[
T'\coloneqq Z_{G^0}(S')=Z_{H^0}(S')\subset H^0,
\]
so $T'$ is a maximally unramified maximal $F$-torus of both $G^0$ and $H^0$. Let $\theta'_0$ also denote its inflation along the reduction map from $T'(F)_{[x]_G}$ to $\ol T'(\F_q)$. The pair $[(T',\theta')]\in\cT(\Psi)$ associated to $(\ol T',\theta'_0)$ satisfies
\[
\theta'=\theta'_0\prod_{i=0}^d\phi_i|_{T'(F)_{[x]}}.
\]
Note that $(\cL^0_{[x]_G})^\circ=(\cL^0_{[x]_L})^\circ$. Let $T_L \subset L^0$ be the torus associated to $\ol T_L$ in the same way that $T'$ is associated to $\ol T'$. Choose compatible embeddings $\cB(T_L)\subset\cB(L^0)\subset\cB(H^0)$, so we may regard $x$ as a point of $\cB(T_L)$ and $\cB(L^0)$.

Define
\[
\wt\eta_0 =\eta_0\prod_{i=0}^d\phi_i|_{T_L(F)_{[x]_G}}, \quad
\wt\eta_1
=\eta_1\prod_{i=0}^d\phi_{L,i}|_{T_L(F)_{[x]_L}}
\]
on $T_L(F)_{[x]_G}$ and $T_L(F)_{[x]_L}$, respectively.
By \eqref{eq:signed-descendant-pair}, we have 
\[
\wt\eta_1|_{T_L(F)_{[x]_G}}
=\bigl(\wt\eta_0\,\kappa_{G,L,x}\bigr)|_{T_L(F)_{[x]_G}}
=\bigl(\wt\eta_0\,\epsilon_G\epsilon_L
\epsilon_{G,\sharp,x}\epsilon_{L,\sharp,x}\bigr)|_{T_L(F)_{[x]_G}}.
\]
By construction, $T$, $T'$, and $T_L$ are maximally unramified in $G^0$. By \cite[Lemma~2.8.4]{Cot26b}, we have $\cE(\ol G^0_{[x]}, [\ol T, \theta_0]) = \cE(\ol G^0_{[x]}, [\ol T', \theta_0'])$ since both sets contain $\tau$. Meanwhile, the definition of $\tau_{H,0}$ shows that $\cE(\ol H^0_{[x]_G}, [\ol T',\theta_0']) = \cE(\ol H^0_{[x]_G}, [\ol T_L,\eta_0])$. Thus Lemma~\ref{lemma:conjugacy-of-torus-character-pairs} shows that $(T,\theta)$, $(T',\theta')$, and $(T_L,\wt\eta_0)$ are almost equivalent in $G^0$, and hence they are almost equivalent in $G$.

By Lemma~\ref{lemma:associated-tori-inertial-parameter}, it follows that
\[
\ld j_{T,G}\circ\ld(\theta|_{T(F)_{\mathrm b}})|_{I_F}
\sim
\ld j_{T_L,G}\circ\ld(\wt\eta_0|_{T_L(F)_{\mathrm b}})|_{I_F}.
\]
Hence \eqref{eqn:kaletha-inertial-functoriality} reduces to the claim
\begin{equation}\label{eqn:kaletha-functoriality-reduced}
\ld j_{L,G}\circ\ld j_{T_L,L}\circ\ld(\wt\eta_1|_{T_L(F)_{\mathrm b}})|_{I_F}
\sim\ld j_{T_L,G}\circ\ld(\wt\eta_0|_{T_L(F)_{\mathrm b}})|_{I_F},
\end{equation}
where $\ld j_{T_L,L}$ and $\ld j_{T_L,G}$ are the L-embeddings from Definition~\ref{defn:canonical-l-embeddings}. Since $k^\times$ is divisible and $T_L(F)_{\mathrm b}$ is open, $\wt\eta_0|_{T_L(F)_{\mathrm b}}$ extends to a smooth character of $T_L(F)$. Then \eqref{eqn:kaletha-functoriality-reduced} follows from Proposition~\ref{prop:fks-and-tasho-cancel} applied to any such extension.

Now suppose that $\Psi$ is $F$-non-singular. The proof of \eqref{eqn:kaletha-non-singular-functoriality} is similar to the above except easier, since no parabolic induction is needed, so we will be concise. 

Lemma~\ref{lemma:kaletha-parameter-choice-independence}, Lemma~\ref{lemma:ns-yu-data-and-characters}(3), and \cite[Fact~3.1.4(1)]{Kal21b} show that the chosen $\theta$ is $F$-non-singular and that $\tau$ is non-singular. \cite[Lemma~2.9.4]{Cot26b}, applied to $\cE(\ol H^0_{[x]_G},[\ol T',\theta_0'])$, implies that $\tau_{H,0}$ is non-singular and cuspidal; thus $L = H$ and $\tau_H = \tau_L$ is non-singular and cuspidal by definition. We now follow the same argument as above for the general case, except that we use \cite[Lemma~2.9.2]{Cot26b} in place of \cite[Lemma~2.8.4]{Cot26b}. Since $\Psi$ and $\Psi_{H,\tau_H}$ are $\F$-non-singular as above, Lemma~\ref{lemma:ns-yu-data-and-characters} and Lemma~\ref{lemma:conjugacy-of-torus-character-pairs} imply that $(T',\theta')$ is $G^0(F)$-conjugate to $(T,\theta)$.
From the definitions, it follows that $\Psi_{H,\tau_H}$ is $F$-non-singular. Note that Lemma~\ref{lemma:ns-yu-data-and-characters}(3) implies $T$ is elliptic in $G$, so $T'$ is also elliptic in $G$. Then Proposition~\ref{prop:fks-and-tasho-cancel} and \eqref{eq:signed-descendant-pair} yield
\[
\rho^{\Kal}(\Psi)
\sim \ld j_{T_L,G}\circ\ld(\wt\eta_0)
\sim \ld j_{L,G}\circ\ld j_{T_L,L}\circ\ld(\wt\eta_1)
=\ld j_{L,G}\circ\rho^{\Kal}(\Psi_L).
\]
This proves \eqref{eqn:kaletha-non-singular-functoriality}.
\end{proof}

\subsection{An example}\label{ss:sign-example}

The following example justifies the pains to which we have gone in defining sign twistings.

\begin{example}\label{example:nontrivial-sign-char}
It can happen that $\kappa_{G,L,x} \neq 1$. For example, let $S_1, S_2 \subset \SL_2$ be elliptic $F$-tori such that $S_1$ is unramified and $S_2$ is ramified. Let $G = \Sp_4$, and let $T = S_1 \times S_2 \subset \SL_2 \times \SL_2 \subset G$. Let $H = S_1 \times \SL_2$, so $H$ is an unramified twisted Levi $F$-subgroup of $G$. Let
\[
\Psi = ((G^0 \subsetneq G^1), x, (r_0=r_1=r>0), \tau, (\phi_0, 1))
\]
be a normalized Yu datum for $G$ over $\ol\Q_\ell$, where $G^0 = \SL_2 \times S_2$. Choose compatible embeddings $\cB(T)\subset\cB(G^0)\subset\cB(G)$, denote the images of the unique point of $\cB(T)$ by $x$, and suppose that $\tau$ is non-singular. Every irreducible constituent of ${}^*R^{\ol G^0_{[x]}}_{\ol H^0_{[x]_G}}(\tau)$ is non-singular and cuspidal by \cite[Lemma~2.9.4]{Cot26b}, so the construction of this section gives $L=H$ and a Yu datum $\Psi_H$ with $H^0=T$ and $H^1=H$. If $\epsilon_G$ and $\epsilon_H$ are constructed from $\Psi$ and $\Psi_H$ as in \S\ref{sss:yu-data-fks}, then we claim $\epsilon_G\epsilon_H\epsilon_{G,\sharp,x}\epsilon_{H,\sharp,x}|_{T(F)_x} \neq 1$. By Lemma~\ref{lemma:fks-unram-levi}, we have
\[
\epsilon_G\epsilon_H\epsilon_{G,\sharp,x}\epsilon_{H,\sharp,x} = \epsilon_{G,\flat}\epsilon_{H,\flat},
\]
where $\epsilon_{G,\flat} = \epsilon_{\flat,0}^{G^1/G^0}\epsilon_{\flat,1}^{G^1/G^0}\epsilon_{\flat,2}^{G^1/G^0}$, and similarly for $\epsilon_{H,\flat}$.

Let $e_i \in X^*((S_i)_{\ol F})$ be a $\Z$-basis element for $i = 1, 2$ such that the absolute roots of $T_{\ol F}$ are (up to sign) given by
\[
\alpha=e_1-e_2,\qquad \beta=2e_2,
\qquad \alpha+\beta=e_1+e_2,
\qquad 2\alpha+\beta=2e_1.
\]
Thus $\Phi(H_{\ol F},T_{\ol F})=\{\pm\beta\}$ and $\Phi(G^0_{\ol F},T_{\ol F})=\{\pm(2\alpha+\beta)\}$. Let $E/F$ be the totally ramified quadratic splitting field of $S_2$, and let $F_2/F$ be the unramified quadratic splitting field of $S_1$, so $E_2$ is the splitting field of $T$. Let $E'/F$ be the other ramified quadratic subfield of $E_2/F$. Thus both $E_2/E$ and $E_2/E'$ are unramified quadratic extensions. Let $\sigma_1$ be the nontrivial element of $\Gal(E_2/E)$ and $\sigma_2$ the nontrivial element of $\Gal(E_2/F_2)$. The elements $\sigma_1$ and $\sigma_2$ alter the signs of $e_1$ and $e_2$, respectively, and the $\Gal(\ol F/F)$-orbits on $\Phi(G_{\ol F},T_{\ol F})$ are as follows:
\begin{equation}\label{eq:sp4-root-orbit-table}
\begin{array}{c|c|c|c}
\Gal(\ol F/F)\text{-orbit} & F_\gamma/F_{\pm\gamma} & \text{type} & \text{location} \\ \hline
\{\pm\beta\} & E/F & \text{symmetric ramified} & G^1/G^0\text{ and }H^1/H^0 \\
\{\pm(2\alpha+\beta)\} & F_2/F & \text{symmetric unramified} & G^0 \\
\{\pm\alpha,\pm(\alpha+\beta)\} & E_2/E' & \text{symmetric unramified} & G^1/G^0
\end{array}
\end{equation}

We now compute the three sign factors. Fix $t\in\ol T^\circ(\F_q)$. For $\epsilon_{\flat,0}^{G^1/G^0}$, the restriction $\alpha_0$ of $\alpha$ to $1\times S_2$ is symmetric ramified, and $F_{\alpha_0}=E$, so $e(\alpha/\alpha_0)=1$. Thus \eqref{eqn:FKS-2} gives
\[
\epsilon_{\flat,0}^{G^1/G^0}(t)=\sgn_{\F_\alpha^1}(\alpha(t)),
\qquad
\epsilon_{\flat,0}^{H^1/H^0}(t)=1.
\]
The two $\epsilon_{\flat,1}$ factors come from the orbit of $\beta$ and, by \eqref{eqn:FKS-4}, are equal to
\[
\epsilon_{\flat,1}^{G^1/G^0}(t)=\epsilon_{\flat,1}^{H^1/H^0}(t)=\begin{cases}
    \sgn_{\F_q^\times}(2) &\text{if } \beta(t)=-1,\\
    1 &\text{otherwise.}
\end{cases}
\]
Similarly, \eqref{eqn:FKS-5} gives
\[
\epsilon_{\flat,2}^{G^1/G^0}(t)=\epsilon_{\flat,2}^{H^1/H^0}(t).
\]
Thus we conclude
\[
(\epsilon_{G,\flat}\epsilon_{H,\flat})(t)=\sgn_{\F_\alpha^1}(\alpha(t)).
\]
Note that this character is nontrivial: the restriction $\alpha|_{S_1\times\{1\}}$ identifies $S_1$ with $\Res^1_{F_2/F}\G_m$, and hence identifies $S_1(F)$ with $\ker\bigl(\Nm_{F_2/F}:F_2^\times\to F^\times\bigr)$. Reduction identifies $S_1(F)/S_1(F)_{0+}$ with $\F_\alpha^1$, on which $\sgn_{\F_\alpha^1}$ is nontrivial.
\end{example}

\begin{remark}\label{remark:nontrivial-sign-char}
One consequence of Example~\ref{example:nontrivial-sign-char} is that, if $\theta\co T(F) \to \ol\Q_\ell^\times$ is a regular character (in the sense of \cite[Definition 3.7.5]{Kal19}) and $\pi$ is the $\ol\Q_\ell$-representation constructed from $(T,\theta)$ as in \cite[Corollary 3.7.10]{Kal19}, then the representation $\pi^H_{(T,\theta)}$ does not always satisfy the obvious guess $\rho^{\Kal}\left(\pi^G_{(T,\theta)}\right) \sim \ld j_{H,G} \circ \rho^{\Kal}\left(\pi^H_{(T,\theta)}\right)$. By Theorem~\ref{thm:kaletha-functoriality}, we have
\[
\rho^{\Kal}\left(\pi^G_{(T,\theta)}\right) \sim \ld j_{H,G} \circ \rho^{\Kal}\left(\pi^H_{(T,\theta\epsilon)}\right),
\]
where $\epsilon\coloneqq \epsilon_{G,\flat}\epsilon_{H,\flat}\co T(F)\to\{\pm1\}$. This character can be nontrivial, as Example~\ref{example:nontrivial-sign-char} shows. If $\theta$ is of odd order, then we claim
\[
\ld j_{H,G}\circ\ld j_{T,H}\circ\ld\theta\not\sim\ld j_{H,G}\circ\ld j_{T,H}\circ\ld(\theta\epsilon)
\]
To see this, note that if these two L-parameters were $\wh G(\ol\Q_\ell)$-conjugate, then they must be $\wh T(\ol\Q_\ell)$-conjugate: indeed, any element of $\wh G(\ol\Q_\ell)$ not lying in $\wh T(\ol\Q_\ell)$ changes the restrictions of these L-parameters to $W_E$ for some quadratic extension $E/F$ since $\theta \circ \Nm_{E/F}$ is regular.\footnote{Note that $x$ was chosen to have hyperspecial image in $\cB(G^0_{\der})$, so the rational Weyl group appearing in \cite[Definition 3.7.5]{Kal19} does not change after base extension.} Since $\theta$ and $\theta\epsilon$ are different characters, the Local Langlands Correspondence for tori implies that these two L-parameters cannot be $\wh T(\ol\Q_\ell)$-conjugate.
\end{remark}

\section{A partial characterization of the semisimple Local Langlands Correspondence}\label{sec:partial-characterization}

This section contains the main theorem of this paper (Theorem~\ref{thm:llc-partial-characterization}), which gives a partial characterization of $\rho^{\Kal}$ in terms of the various functoriality results that we have established in the previous sections.

\subsection{Generalities on L-parameters}\label{ss:generalities-1}

We begin with a number of technical results on reductive groups. The following lemma generalizes \cite[Lemme 1.10]{DM94}.

\begin{lemma}\label{lemma:finite-fixed-points}
If $\alpha$ is an automorphism of a connected reductive group $H$ over a field $k$ such that $H^\alpha/Z(H)^\alpha$ is finite, then $H$ is a torus. 
\end{lemma}

\begin{proof}
By passing from $H$ to $H_{\mathrm{der}}$, we may assume that $H$ is semisimple. Every automorphism of $H$ lifts to an automorphism of the universal cover of $H$, and using this fact one reduces easily to the case that $H$ is simply connected. Let $\alpha_0$ denote a pinning-preserving automorphism of $H$ lying in the same component of the automorphism scheme $\mathrm{Aut}_{H/k}$ as $\alpha$. By \cite[Proposition 1]{Spr06} (which applies since $H$ is simply connected), the group $H^{\alpha_0}$ is semisimple, say of rank $r$. In fact, $r > 0$: for instance, if $\alpha_0$ preserves the pinning $(B, T, \{X_\alpha\})$ and $\beta^\vee$ is a simple coroot with respect to this pinning, then $\sum_{\gamma^\vee \in \alpha_0^{\Z}(\beta^\vee)} \gamma^\vee\colon \G_{\mathrm{m}} \to H$ is a nontrivial cocharacter factoring through $H^{\alpha_0}$.

Since the identity component of $\mathrm{Aut}_{H/k}$ is $H/Z(H)$, acting via inner automorphisms, it follows that $\alpha = c_h \circ \alpha_0$ for some $h \in H(k)$, where $c_h$ denotes the inner automorphism of $H$ induced by $h$. In the notation of \cite[\S 5.2]{XZ19}, we have $H^\alpha = I_h$, and it follows from \cite[Remark 5.2.2(1)]{XZ19} that $\dim H^\alpha \geq r$. This proves the lemma.
\end{proof}

\begin{lemma}\label{lemma:wild-centralizer-components}
Let $k$ be a field of characteristic $\neq p$, let $H$ be a reductive $k$-group such that $\pi_0(H)$ is a finite $p$-group, and let $A$ be a finite $p$-group acting on $H$. Then $\pi_0(H^A)$ is a finite $p$-group.
\end{lemma}

\begin{proof}
By induction, we may assume that $A$ is cyclic, and we may further assume that $H$ is connected. In this case, the result is \cite[Corollary 2.16(b)]{St75}.
\end{proof}

\begin{lemma}\label{lemma:align-tame-components}
Let $k \in \{\ol\Q_\ell, \ol\F_\ell\}$ with $\ell\ne p$, and let $H$ be a (possibly disconnected) reductive $k$-group. Let $I$ be a finite group with a normal $p$-subgroup $P$ such that $I/P$ is cyclic of order prime to $p$. Let $\rho_1,\rho_2\co I\to H(k)$ be homomorphisms such that
\begin{enumerate}
    \item $\rho_1|_P = \rho_2|_P$,
    \item the induced maps $I \to \pi_0(H)(k)$ are equal,
    \item $Z_H(\rho_1(P))^\circ$ is a torus.
\end{enumerate}
Then there is some $c\in Z_{H^\circ}(\rho_1(P))(k)$ such that $c\rho_1c^{-1}$ and $\rho_2$ induce the same action of $I$ on $Z_H(\rho_1(P))^\circ$.
\end{lemma}

\begin{proof}
By (1), the map
\[
x \mapsto \rho_1(x)\rho_2(x)^{-1}
\]
is a $1$-cocycle $\varphi\co I/P \to Z_H(\rho_1(P))(k)$; by (2), the map $\varphi$ is valued in $Z_{H^\circ}(\rho_1(P))(k)$. The map $\varphi$ induces a $1$-cocycle $\ol\varphi\co I/P \to \pi_0(Z_{H^\circ}(\rho_1(P)))(k)$, and the lemma is equivalent to the statement that $\ol\varphi$ is a coboundary (since $Z_H(\rho_1(P))^\circ$ is commutative by (3)). But $I/P$ is of order prime to $p$ and $\pi_0(Z_{H^\circ}(\rho_1(P)))(k)$ is a $p$-group by Lemma~\ref{lemma:wild-centralizer-components}, so this is automatic.
\end{proof}

\subsubsection{Semisimple and irreducible L-parameters}

Recall from \cite[Lemma 2.1.4 and the following paragraph]{CGP15} that if $k$ is a field, $H$ is a reductive $k$-group, and $\lambda\colon \G_{\mathrm{m}} \to H$ is a cocharacter, then there are associated subgroups $Z_H(\lambda)$ and $P_H(\lambda)$ of $H$: by definition, $Z_H(\lambda)$ is the centralizer of the cocharacter $\lambda$, while $P_H(\lambda)$ represents the functor parameterizing local sections $h$ of $H$ such that the limit $\lim_{t \to 0} \lambda(t)h\lambda(t)^{-1}$ exists. Following \cite[\S 6]{BMR05}, we call $P_H(\lambda)$ an \textit{R-parabolic subgroup} of $H$, and we call $Z_H(\lambda)$ an \textit{R-Levi subgroup} of $P_H(\lambda)$. By \cite[Proposition 2.2.9]{CGP15}, the identity component $P_H(\lambda)^\circ$ is a parabolic subgroup of $H^\circ$. Moreover, by \cite[Proposition 2.1.8]{CGP15} we have $P_H(\lambda) = Z_H(\lambda) \ltimes U_H(\lambda)$, where $U_H(\lambda)$ is the unipotent radical of $P_H(\lambda)^\circ$; in particular, there is a natural projection map $P_H(\lambda) \to Z_H(\lambda)$, given informally by ``sending $h$ to $\lim_{t \to 0} \lambda(t)h\lambda(t)^{-1}$''.

\begin{defn}
Let $\Gamma$ be a(n ``abstract'') group, and let $H$ be a reductive group over a field $k$. If $\rho \co \Gamma \to H(k)$ is a homomorphism, then $\rho$ is \textit{semisimple} if for every R-parabolic $\ol{k}$-subgroup $P \subset H_{\ol k}$ such that $\rho(\Gamma) \subset P(\ol k)$, there is an R-Levi $\ol k$-subgroup $L \subset P$ such that $\rho(\Gamma) \subset L(\ol k)$.

If $\rho\colon \Gamma \to H(k)$ is any homomorphism, then a \textit{semisimplification} of $\rho$ is a semisimple homomorphism $\rho^{\mathrm{ss}}\colon \Gamma \to H(k)$ such that there exists an R-parabolic $P \subset H$ with R-Levi $L \subset P$ such that $\rho$ factors through $P(k)$ and $\rho^{\mathrm{ss}}$ factors as $\Gamma \xrightarrow{\rho} P(k) \to L(k) \subset H(k)$. We will use $\rho^{\mathrm{ss}}$ to denote a choice of semisimplification of $\rho$.
\end{defn}

The proof of the following lemma is surprisingly complicated, and we do not know whether one should expect a simpler proof.

\begin{lemma}\label{lemma:general-semisimplicity-criterion}
Let $k$ be a field, and let $H$ be a reductive $k$-group. Let $\Gamma$ be a group which admits a normal subgroup $\Gamma_0$ such that $\Gamma/\Gamma_0$ is cyclic. If $\rho\co \Gamma \to H(k)$ is a homomorphism, then the following are equivalent:
\begin{enumerate}
    \item $\rho|_{\Gamma_0}$ is semisimple and there is a $\rho(\Gamma)$-stable Borel-torus pair $(B, T)$ in $Z_H(\rho(\Gamma_0))^\circ_{\red}$,
    \item $\rho$ is semisimple.
\end{enumerate}
\end{lemma}

\begin{proof}
We may and do assume $k = \ol k$. Passing from $\Gamma$ to the Zariski closure of $\rho(\Gamma)$, we may assume that $\rho$ is the inclusion of a smooth closed subgroup $\Gamma$ such that $\Gamma/\Gamma_0$ is topologically generated by an element $\gamma \in \Gamma(k)$. If $\rho$ is semisimple, then $\rho|_{\Gamma_0}$ is semisimple by (the disconnected version of) \cite[Theorem 3.10]{BMR05}. Thus by \cite[Corollary 3.7]{BMR08}, the homomorphism $\rho$ is semisimple if and only if $\rho|_{\Gamma_0}$ is semisimple and the induced homomorphism $\alpha\co \Gamma/\Gamma_0 \to N_H(\Gamma_0)_{\red}/\Gamma_0$ is semisimple. The identity component of $N_H(\Gamma_0)_{\red}/\Gamma_0$ admits a central isogeny from $Z_H(\Gamma_0)^\circ_{\red}$, so $\alpha$ is semisimple if and only if $\Gamma$ preserves a Borel-torus pair of $Z_H(\Gamma_0)^\circ_{\red}$ (if and only if conjugation by $\gamma$ induces a quasi-semisimple automorphism of $Z_H(\Gamma_0)^\circ_{\red}$) by \cite[Proposition 3]{Spr06}.
\end{proof}

\begin{lemma}\label{lemma:functorial-semisimplicity}
Let $k \in \{\ol\Q_\ell, \ol\F_\ell\}$ be a field.
\begin{enumerate}
\item If $E/F$ is finite Galois, $G=\Res_{E/F}(H_E)$, and $\rho\co W_F\to\ld H(k)$ is semisimple, then the composition $\Delta \circ \rho$ is semisimple, where $\Delta\co \ld H\to\ld G$ is the canonical homomorphism arising from base change.
\item If $H\subset G$ is an unramified twisted Levi and $\rho\co W_F\to\ld H(k)$ is semisimple, then $(\ld j_{H,G}\circ\rho)|_{I_F}$ is semisimple. Consequently,
\[
\bigl(z\cdot\ld j_{H,G}\circ\rho\bigr)^{\ss}|_{I_F}
\sim (\ld j_{H,G}\circ\rho)|_{I_F}
\]
for every unramified cocycle $z$ valued in $Z(\wh H)(k)^{I_F}$.
\item In (2), if $\ell>\rk G+1$, then $z\cdot\ld j_{H,G}\circ\rho$ itself is semisimple.
\end{enumerate}
\end{lemma}

\begin{proof}
In (1), the Shapiro isomorphism identifies $\Delta \circ \rho$ with $\rho|_{W_E}$. This isomorphism is easily checked to preserve semisimplicity, so (1) follows from the fact that $\rho|_{W_E}$ is semisimple by \cite[Theorem 3.10, \S 6.3]{BMR05}.

For (2) and (3), write $H=Z_G(S)$ for an unramified $F$-torus $S$, and choose an unramified extension $E/F$ splitting $S$. Lemmas~\ref{lemma:canonical-l-embedding-base-change} and~\ref{lem:unramified-L-embedding-properties}(1) identify $\ld j_{H,G}|_{\wh H\rtimes W_{E_2}}$ with the standard embedding, which preserves semisimplicity by \cite[Corollary 2.10]{BMR08}. Since $I_F\subset W_{E_2}$, this proves the first assertion of (2); the second follows from \cite[Theorem 3.10, \S 6.3]{BMR05}.

For (3), we may assume $G$ is not a torus, so $\ell > 2$. Let $E/F$ be splitting field of $S$, so every prime $r$ dividing $[E:F]$ satisfies $r-1\leq\dim S\leq\rk G$. Thus $[E_2:F]$ is prime to $\ell$. Twisting $\rho$ by the central cocycle $z$ preserves semisimplicity, so $z \cdot \ld j_{H,G} \circ \rho|_{W_{E_2}}$ is semisimple. The image of this restriction is normal of index prime to $\ell$ in the image of $z \cdot \ld j_{H,G} \circ \rho$, so (3) follows from \cite[Corollary 3.7(ii)]{BMR08}.
\end{proof}

\begin{lemma}\label{lemma:general-irreducibility-criterion}
Let $k$ be a field, and let $H$ be a reductive $k$-group. Let $\Gamma$ be a group which admits a normal subgroup $\Gamma_0$ such that $\Gamma/\Gamma_0$ is cyclic. If $\rho\colon \Gamma \to H(k)$ is a homomorphism, then the following are equivalent:
\begin{enumerate}
    \item $\rho|_{\Gamma_0}$ is semisimple and $Z_{H^\circ}(\rho(\Gamma))/Z(H^\circ)^{\rho(\Gamma)}$ is finite,
    \item $\rho$ is irreducible.
\end{enumerate}
\end{lemma}

\begin{proof}
For (1) $\Rightarrow$ (2), note that $Z_H(\rho(\Gamma_0))^\circ_{\mathrm{red}}$ is connected reductive by \cite[Proposition 3.12]{BMR05}, so applying Lemma~\ref{lemma:finite-fixed-points} to $\alpha = \rho(\gamma)$, we find that $S = Z_H(\rho(\Gamma_0))^\circ_{\mathrm{red}}$ is a torus. Thus $\rho$ is semisimple by Lemma~\ref{lemma:general-semisimplicity-criterion}.

For (2) $\Rightarrow$ (1), note that $Z_{H^\circ}(\rho(\Gamma))^\circ_{\mathrm{red}}$ is connected reductive by \cite[Proposition 3.12]{BMR05}. Thus if $Z_{H^\circ}(\rho(\Gamma))/Z(H^\circ)^\Gamma$ is not finite then there exists a non-central torus $S_0 \subset H^\circ$ centralized by $\Gamma$. This implies that $\rho(\Gamma) \subset Z_H(S_0)(k)$ and $Z_H(S_0) \subsetneq H$, contradicting irreducibility of $\rho$.
\end{proof}

\subsubsection{Non-singular characters}

We now use Lemma~\ref{lemma:general-irreducibility-criterion} to study non-singular characters of tori. Recall that an $F$-torus $T$ is \textit{totally ramified} if $T$ does not admit a nontrivial unramified $F$-subtorus.

\begin{lemma}\label{lemma:elliptic-finite-invariants}
If $T \subset G$ is a maximal $F$-torus, then $T$ is elliptic if and only if $\widehat{T}^{W_F}/Z(\wh G)^{W_F}$ is finite. Moreover, $T \cap G_{\der}$ is totally ramified if and only if $\wh T^{I_F}/Z(\wh G)^{I_F}$ is finite.
\end{lemma}

\begin{proof}
We may and do assume that $G$ is semisimple, so $\widehat{T}^{W_F}/Z(\wh G)^{W_F}$ is finite if and only if $\widehat{T}^{W_F}$ is finite. The cocharacter group $X_*(\widehat{T}^{W_F})$ is equal to $X_*(\widehat{T})^{W_F} = X^*(T)^{W_F}$ and this is clearly trivial if and only if $T$ is elliptic. The same argument works for the second claim.
\end{proof}

\begin{lemma}\label{lemma:irreducibility-criterion}
If $k$ is either $\ol{\F}_\ell$ or $\ol{\Q}_\ell$ and $\rho\colon W_F \to \ld G(k)$ is an L-parameter, then $\rho$ is irreducible if and only if $\rho|_{I_F}$ is semisimple and $Z_{\widehat{G}}(\rho(W_F))/Z(\widehat{G})^{W_F}$ is finite, and in this case $Z_{\widehat{G}}(\rho(I_F))^\circ_{\mathrm{red}}$ is a torus. Similarly, $\rho|_{I_F}$ is irreducible if and only if $Z_{\widehat{G}}(\rho(I_F))/Z(\widehat{G})^{I_F}$ is finite, and in this case $Z_{\widehat{G}}(\rho(P_F))^\circ$ is a torus. 
\end{lemma}

\begin{proof}
The first claim is immediate from Lemma~\ref{lemma:general-irreducibility-criterion} applied to $\Gamma = \rho(W_F)$ and $\Gamma_0 = \rho(I_F)$. For the second claim, note that $\rho|_{P_F}$ is automatically semisimple since $\ell \neq p$ \cite[Lemma 2.6]{BMR05}, and $Z_{\widehat{G}}(\rho(P_F))$ is automatically reductive (and in particular smooth) by \cite[Theorem 2.1]{PY02}. Thus the result again follows immediately from Lemma~\ref{lemma:general-irreducibility-criterion}.
\end{proof}

The following lemma is similar to \cite[Lemma 4.1.3]{Kal21b}. The main difference is that our lemma handles L-parameters $\rho$ for which $\rho|_{P_F}$ may not factor through a maximal torus of $\wh G$; this generality will show up when studying the Fargues--Scholze correspondence, for which this property is not known a priori when $p$ is small.

\begin{lemma}\label{lemma:toral-l-parameter-variant}
Let $k$ be a field among $\ol\Q_\ell$ and $\ol\F_\ell$, and let $\rho\co W_F \to \ld G(k)$ be a semisimple L-parameter. There exists a $\rho(I_F)$-stable Borel-torus pair $(\wh B, \wh T)$ in $Z_{\wh G}(\rho(P_F))$, and for any such pair the torus $(\wh T^{I_F})^\circ_{\red}$ is a maximal $k$-torus of $Z_{\wh G}(\rho(I_F))^\circ_{\red}$. Moreover, $\wh T$ can be chosen to be $\rho(W_F)$-stable.
\end{lemma}

\begin{proof}
The first claim follows from Lemma~\ref{lemma:general-semisimplicity-criterion} applied to suitable finite quotients $\Gamma$ and $\Gamma_0$ of $I_F$ and $P_F$, respectively. The second claim follows from \cite[Th\'eor\`eme 1.8(iii)]{DM94}. For the final claim, note that another application of Lemma~\ref{lemma:general-semisimplicity-criterion} shows that there is a $\rho(W_F)$-stable Borel-torus pair $(\wh B_0, \wh T_0)$ in $Z_{\wh G}(\rho(I_F))^\circ_{\red}$. By \cite[Th\'eor\`eme 1.8(iv)]{DM94}, the pair $(\wh B, \wh T)$ can be chosen such that $\wh T = Z_{Z_{\wh G}(\rho(P_F))^\circ}(\wh T_0)$ and $\wh B_0 \subset \wh B$, proving the claim.
\end{proof}

\begin{lemma}\label{lemma:non-singular-irreducible}
Let $k$ be a field among $\ol\F_\ell$ and $\ol\Q_\ell$, let $\Psi$ be a normalized irreducible Yu datum for $G$ with coefficients in $k$, and let $[(T,\theta)] \in \cT(\Psi)$.
\begin{enumerate}
    \item $\rho_I^{\Kal}(\Psi)$ is semisimple.
    \item The $k$-group scheme $Z_{\wh G}(\rho_I^{\Kal}(\Psi)(I_F))^\circ_{\red}$ is a torus if and only if $\Psi$ is non-singular, and in this case
    \begin{equation}\label{eqn:inertial-fixed-points-torus}
    Z_{\wh G}(\rho^{\Kal}(\Psi)(I_F))^\circ_{\red} = (\wh T^{I_F})^\circ_{\red}.
    \end{equation}
    In particular, in this case $\rho^{\Kal}(\Psi)$ is irreducible.
    \item $\rho_I^{\Kal}(\Psi)$ is irreducible if and only if $\Psi$ is non-singular and $T \cap G_{\der}$ is totally ramified.
\end{enumerate}
\end{lemma}

\begin{proof}
By Lemma~\ref{lemma:irredundant-yu-data}, we may assume that $\Psi$ is irredundant. For (1), note that if $(G^i)$ is the twisted Levi sequence of $\Psi$, then by construction $T$ is a maximally unramified maximal $F$-torus of $G^0$. By Lemma~\ref{lemma:ns-yu-data-and-characters}(2), the canonical embedding $\wh{G^0} \to \wh G$ identifies $\wh{G^0}$ with $Z_{\wh G}(\rho_I^{\Kal}(\Psi)(P_F))^\circ$. Since $T$ is maximally unramified in $G^0$, there is a $\rho_I^{\Kal}(\Psi)(I_F)$-stable Borel subgroup $\wh B \subset Z_{\wh G}(\rho_I^{\Kal}(\Psi)(P_F))^\circ$ containing $\wh T$. Since $\rho_I^{\Kal}(\Psi)|_{P_F}$ is automatically semisimple by \cite[Lemma 2.6]{BMR05}, the claim follows from Lemma~\ref{lemma:general-semisimplicity-criterion}.

The first claim of (2) follows from Lemma~\ref{lemma:toral-inertial-centralizer-criterion} and Definition~\ref{defn:non-singular-yu-datum}. To prove \eqref{eqn:inertial-fixed-points-torus}, it is enough to show that $(\wh T^{\rho_I^{\Kal}(\Psi)(I_F)})^\circ_{\red}$ is a maximal torus of $Z_{\wh G}(\rho_I^{\Kal}(\Psi)(I_F))^\circ_{\red}$, and this follows from the previous paragraph and Lemma~\ref{lemma:toral-l-parameter-variant}.

Now assume that $\Psi$ is non-singular, so $\theta$ is non-singular and thus $T$ is elliptic by Lemma~\ref{lemma:ns-yu-data-and-characters}. Let $\rho = \rho^{\Kal}(\Psi)$. First, (1) shows that $\rho|_{I_F}$ is semisimple. Moreover, $Z_{\wh G}(\rho(I_F))^\circ_{\red}$ is a torus by Lemma~\ref{lemma:toral-inertial-centralizer-criterion}, so $Z_{\wh G}(\rho(I_F))^\circ_{\red} \subset \wh T^{I_F}$ by Lemma~\ref{lemma:toral-l-parameter-variant} and thus $Z_{\wh G}(\rho(W_F))^\circ_{\red} \subset \wh T^{W_F}$. Since $T$ is elliptic, it follows that $Z_{\wh G}(\rho(W_F))/Z(\wh G)^{W_F}$ is finite.
Thus Lemma~\ref{lemma:irreducibility-criterion} shows that $\rho$ is irreducible, proving the final claim of (2).

For (3), first suppose that $T \cap G_{\der}$ is totally ramified. Then $G^0 = T$ since the elliptic torus $T_{F^{\unr}}$ is a maximally split torus of $G^0_{F^{\unr}}$, which is quasi-split by Steinberg's theorem. Thus $\Psi$ is non-singular by Remark~\ref{remark:toral-yu-ns}. By (2), we have $Z_{\wh G}(\rho_I^{\Kal}(\Psi)(I_F))^\circ_{\red}=(\wh T^{I_F})^\circ_{\red}$. By Lemma~\ref{lemma:elliptic-finite-invariants}, the statement that $T\cap G_{\der}$ is totally ramified is equivalent to finiteness of $\wh T^{I_F}/Z(\wh G)^{I_F}$. It follows that $Z_{\wh G}(\rho_I^{\Kal}(\Psi)(I_F))/Z(\wh G)^{I_F}$ is finite, and Lemma~\ref{lemma:irreducibility-criterion} implies that $\rho_I^{\Kal}(\Psi)$ is irreducible. Conversely, if $\rho_I^{\Kal}(\Psi)$ is irreducible, then Lemma~\ref{lemma:irreducibility-criterion} shows that $Z_{\wh G}(\rho_I^{\Kal}(\Psi)(I_F))/Z(\wh G)^{I_F}$ is finite; by (2), this implies that $\Psi$ is non-singular. Since $\wh T^{I_F}\subset Z_{\wh G}(\rho_I^{\Kal}(\Psi)(I_F))$, the quotient $\wh T^{I_F}/Z(\wh G)^{I_F}$ is finite, and thus $T \cap G_{\der}$ is totally ramified by Lemma~\ref{lemma:elliptic-finite-invariants}, as desired.
\end{proof}

\subsection{The main theorem}

For a field $k$ among $\ol\Q_\ell$ and $\ol\F_\ell$, let $\Pi_k(G)_{\mathrm{tc}}$ denote the set of isomorphism classes of $k$-representations of $G(F)$ which arise from Yu's construction applied to a normalized irreducible Yu datum for $G$ over $k$, and let $\Phi_k^{\ss}(G)$ denote the set of semisimple L-parameters $W_F \to \ld G(k)$ up to $\wh G(k)$-conjugacy.  

\begin{thm}\label{thm:llc-partial-characterization}
Let $\cS$ be a set of prime numbers containing $p$, let $F$ be a non-archimedean local field of residue characteristic $p \neq 2$, let $G$ be a connected reductive $F$-group, and let $k$ be a field among $\ol\Q_\ell$ and $\ol\F_\ell$. Suppose that we are given, for each finite tamely ramified extension $E/F$ and each connected reductive $E$-group $H$ of semisimple rank at most that of $G$, a map of sets
\[
\rho\co \Pi_k(H)_{\mathrm{tc}} \to \Phi_k^{\ss}(H).
\]
Abbreviate $\rho_I \coloneqq \rho|_{I_F}$. Suppose that $\rho$ enjoys the following properties. 
\begin{enumerate}
    \item When $G$ is a torus, $\rho$ agrees with the Local Langlands Correspondence for tori.
    \item $\rho$ is compatible with central characters.
    \item $\rho$ is compatible with homomorphisms which are isomorphisms up to center, i.e., Lemma~\ref{lemma:kaletha-compatibility-with-central-isogenies} holds with $\rho_I$ (resp.\ $\rho$) in place of $\rho_I^{\Kal}$ (resp.\ $\rho^{\Kal}$)
    \item for every $\ell\notin\cS$, if $\Psi = ((G^i), x, (r_i), \tau, (\phi_i))$ is a normalized irreducible Yu datum for $G$ over $k$ such that $G^0$ is a torus, and every tamely ramified extension $E/F$ of degree $\ell$ such that $G^0_E$ is elliptic, we have
    \[
    \rho(\pi(\Psi))|_{P_F}
    \sim
    \rho(\pi(\Psi_E))|_{P_F},
    \]
    with notation as in Proposition~\ref{prop:kal-small-degree-bc}.
    \item If $\Psi = ((G^i), x, (r_i), \tau, (\phi_i))$ is a normalized irreducible Yu datum for $G$ over $k$ such that $G^0 \cap G_{\der}$ is not a totally ramified torus, then there exists a proper unramified twisted Levi $L \subsetneq G$ and a Yu datum $\Psi_L$ for $L$ over $k$ such that the conclusion of Theorem~\ref{thm:kaletha-functoriality} holds for both $\rho_I$ and $\rho_I^{\Kal}$. Moreover,
    \begin{enumerate}[label=(\alph*)]
        \item if $S_\Psi = Z_{\wh G}(\rho_I(\pi(\Psi))(I_F))^\circ_{\red}$ is a torus, then one can choose representatives such that $\rho_I(\pi(\Psi)) = \ld j_{L,G} \circ \rho_I(\pi(\Psi_L))$ and the actions of $W_F$ on $S_\Psi$ through $\rho(\pi(\Psi))$ and $\ld j_{L,G} \circ \rho(\pi(\Psi_L))$ are the same.
        \item if $\rho(\pi(\Psi))$ is irreducible, so $L = H$, then there exists a cocycle $z\co W_F/I_F \to Z(\wh H \cap \wh G_{\der})(k)^{I_F}$ such that
        \[
        \rho(\pi(\Psi)) \sim z \cdot \ld j_{H,G} \circ \rho(\pi(\Psi_H)).
        \]
    \end{enumerate}
\end{enumerate}
Let $\Psi$ be a normalized irreducible Yu datum for $G$ with $k$-coefficients, and let $[(T, \theta)] \in \cT(\Psi)$ (with notation as in \S\ref{ss:kal-param}). Let $T_0$ be the maximal unramified subtorus of $T$, and let $n$ be the ramification degree of the splitting field of the torus $T \cap Z_G(T_0)_{\der}$.
\begin{enumerate}[label=(\Alph*)]
    \item If no prime factor of $n \cdot [F(\mu_n):F]$ lies in $\cS$, then
    \begin{equation}\label{eqn:abstract-agreement-on-inertia}
    \rho_I(\pi(\Psi)) \sim \rho_I^{\Kal}(\Psi).
    \end{equation}
    Therefore $\Psi$ is non-singular if and only if $S_\Psi = Z_{\wh G}(\rho_I(\pi(\Psi))(I_F))^\circ_{\red}$ is a torus.
    \item Suppose moreover that $p$ is good for $Z_G(T_0)$ and $\Psi$ is non-singular, and let 
    \[
    X = X_*((T/Z(Z_G(T_0))^\circ_{\red})_{\ol F})_{I_F} \quad \text{and} \quad X_{\ad} = X_*((T/Z(Z_G(T_0)))_{\ol F})_{I_F}. 
    \]
    If $N$ is the least integer such that the map
    \[
    \rH^1(W_F/I_F, \Hom(X_{\ad}, k^\times)) \to \rH^1(W_{F_N}/I_F, \Hom(X, k^\times))
    \]
    is trivial, then
    \begin{equation}\label{eqn:abstract-agreement-on-finite-index}
    \rho(\pi(\Psi))|_{W_{F_N}} \sim \rho^{\Kal}(\Psi)|_{W_{F_N}}.
    \end{equation}
    Moreover, there exist $\wh G(k)$-conjugates such that $\rho(\pi(\Psi))|_{I_F} = \rho^{\Kal}(\Psi)|_{I_F}$ and the actions of $\rho(\pi(\Psi))$ and $\rho^{\Kal}(\Psi)$ on $S_\Psi$ are the same. In particular, a normalized irreducible Yu datum $\Psi_0$ is non-singular if and only if $\rho(\pi(\Psi_0))$ is irreducible.
\end{enumerate} 
\end{thm}

\begin{remark}
The main case of interest for Theorem~\ref{thm:llc-partial-characterization} is the Fargues--Scholze correspondence $\rho = \rho^{\FS}$. This is known to satisfy hypotheses (1), (2), and (3) of Theorem~\ref{thm:llc-partial-characterization} by \cite[Theorem I.9.6]{FS}. In \cite{CF26b}, we verify (4) with $\cS=\{p\}$ and (5) using \cite[Theorem~9.1.1 and Example~9.1.2]{F24}, as discussed in Remark~\ref{rem:modular-functoriality-black-box-II}. The reason for the added generality of $\cS$ is that we previously did not know that \cite[Theorem~9.1.1]{F24} held for all $\ell$.
\end{remark}

\subsubsection{Some bounds on the index}

Before proving Theorem~\ref{thm:llc-partial-characterization}, we explain how one can get some handle on the complicated-looking integer $N$ appearing in (B). First, if $T$ is maximally unramified, then $X = 0$ in Theorem~\ref{thm:llc-partial-characterization}(B), and thus $N = 1$.

\begin{example}\label{ex:llc-slr}
Suppose $G_{F^{\unr}} \cong \SL_r$ and $T \subset \SL_r$ is a totally ramified maximal $F$-torus, so $I_F/P_F$ acts on $X_*(T_{\ol F})$ through an elliptic Weyl element. The only elliptic Weyl element of the absolute Weyl group of $G$ is the Coxeter element, so $X \cong \Z/r\Z \cong X_{\ad}$, and the map $X \to X_{\ad}$ is trivial. Thus $N = 1$ in Theorem~\ref{thm:llc-partial-characterization}(B). In fact, we will see in \cite[Theorem~10.4.3]{CF26b} (by a longer argument) that one can take $N = 1$ for arbitrary $\Psi$ in this case.
\end{example}

Now we establish some elementary methods for computing $N$.

\begin{lemma}\label{lemma:explicit-finite-index-bound}
Let $Y$ be a finite abelian group equipped with an automorphism $\alpha$. Let $m$ be an integer such that for each prime $\ell_0$ dividing $|Y|$, the action of $\alpha^m$ on $Y/\ell_0 Y$ is of order prime to $\ell_0$, and let $n_0$ be the exponent of the image $Z$ of $Y \to Y$, $y \mapsto \sum_{i=0}^{m-1} \alpha^i(y)$. Then the restriction map
\[
\rH^1(\alpha, Y) \to \rH^1(\alpha^{mn_0}, Y)
\]
is trivial.
\end{lemma}

\begin{proof}
Note that
\[
\rH^1(\alpha, Y) = Y_\alpha = \bigoplus_{\ell_0 \mid |Y|} Y[\ell_0^\infty]_\alpha,
\]
and similarly for $\rH^1(\alpha^{mn_0}, Y)$, so we may pass to $Y[\ell_0^\infty]$ to assume that $|Y| = \ell_0^e$ for some $e$. The image of the map $\rH^1(\alpha, Y) \to \rH^1(\alpha^m, Y)$ lies in the image of $\rH^1(\alpha^m, Z)$, so it is enough to show that the map $\rH^1(\alpha^m, Z) \to \rH^1(\alpha^{mn_0})$ is trivial. Note that the elements $\alpha^m$ and $\alpha^{\ell_0 m}$ generate the same cyclic subgroup of $\Aut(Y/\ell_0 Y)$, so the restriction map $\rH^1(\alpha^m, Y/\ell_0 Y) \to \rH^1(\alpha^{\ell_0 m}, Y/\ell_0 Y)$, which is induced by the map $Y \to Y$, $y \mapsto \sum_{i=0}^{\ell_0-1} \alpha^{mi}(y)$, is identified with multiplication by $\ell_0$ on $(Y/\ell_0 Y)_{\alpha^m}$ and hence vanishes. Thus the restriction map $\rH^1(\alpha^m, Z) \to \rH^1(\alpha^{\ell_0 m}, Y)$ factors through $\ell_0 \rH^1(\alpha^{\ell_0 m}, Y)$, and the result follows by induction.
\end{proof}

The following special case is useful in examples.

\begin{lemma}\label{lemma:2-group-explicit-restriction}
Let $Y = (\Z/\ell)^a \oplus (\Z/\ell^2)^b$ for some $a, b \in \Z_{\geq 0}$, and let $\alpha$ be an automorphism of $Y$. If $r = \ell^{\lceil\log_\ell(\max(1, a+b, \ell b))\rceil+1}$, then the restriction map
\[
\rH^1(\alpha, Y) \to \rH^1(\alpha^r, Y)
\]
is trivial.
\end{lemma}

\begin{proof}
By passing from $Y$ to $Y_{\alpha^r}$, we may and do assume that $\alpha^r = 1$. Let $Y_0$ be the image of $Y[\ell]$ in $Y/\ell Y$, and let $Y_1$ be the image of $\{1\} \oplus (\Z/\ell^2)^b$ in $Y/\ell Y$, so $Y/\ell Y = Y_0 \oplus Y_1$ and $\alpha$ stabilizes $Y_0$. If $r_1 = \ell^{\lceil\log_\ell(\max(1,b))\rceil}$, then since $\dim_{\F_\ell} Y_1 = b$, the automorphism $\alpha_1$ of $Y/(\ell Y + Y[\ell])$ induced by $\alpha$ satisfies $\alpha_1^{r_1} = 1$. On the other hand, if $r_2 = \ell^{\lceil\log_\ell(\max(1, a+b))\rceil}$, then the automorphism $\ol\alpha$ of $Y/\ell Y$ induced by $\alpha$ satisfies $\ol\alpha^{r_2} = 1$. It follows that $\alpha^{\max(\ell r_1, r_2)} = 1$. By choice of $r_1$, the map $Y \to Y$, $y \mapsto \sum_{i=0}^{\ell r_1 - 1} \alpha^i(y)$, factors through $Y[\ell]$, and the claim therefore follows from Lemma~\ref{lemma:explicit-finite-index-bound}.
\end{proof}

\begin{example}\label{ex:llc-classical}
Let $G$ be a connected reductive $F$-group such that $G_{F^{\unr}}$ is isomorphic to one of $\GL_r$, $\SO_{2r+1}$, $\Sp_{2r}$, and $\SO_{2r}$. Suppose $p \neq 2$ and $\rho$ is a map as in Theorem~\ref{thm:llc-partial-characterization}.

Assume that $T \cap G_{\der}$ is totally ramified. In this case, $T$ corresponds to an elliptic Weyl element of the absolute Weyl group of $G$ (namely, the element through which $I_F/P_F$ acts on $X_*(T_{\ol F})$; here we use that $G$ splits after an unramified extension). We split into cases:
\begin{enumerate}
    \item If $G_{F^{\unr}}$ is isomorphic to $\GL_r$, then the only elliptic Weyl element is the Coxeter element, and in the notation of Theorem~\ref{thm:llc-partial-characterization}(B) we have $X \cong X_{\ad} \cong \Z/r$. By Lemma~\ref{lemma:explicit-finite-index-bound} applied to $Y = \Hom(X, k^\times)$, we may take $N = r$: indeed, $Y/\ell_0Y$ is either trivial or isomorphic to $\Z/\ell_0$ for all primes $\ell_0$, and $\Z/\ell_0$ has no automorphisms of order $\ell_0$.
    \item If $G_{F^{\unr}}$ is either $\SO_{2r+1}$ or $\Sp_{2r}$, then \cite[\S3.4]{GP00} shows that every elliptic Weyl element is the Coxeter element of a Weyl group of type $\mathrm{B}_{a_1} \times \cdots \times \mathrm{B}_{a_s}$, where $(a_1, \dots, a_s)$ is a partition of $r$. For such an element, the group $X$ of Theorem~\ref{thm:llc-partial-characterization}(B) is isomorphic to $(\Z/2)^s$, so $m$ divides $2^{\lceil \log_2(s)\rceil}$ in the notation of Lemma~\ref{lemma:explicit-finite-index-bound} and $n_0 | 2$. In particular, since $s \leq r$ we may take $N \leq 2^{\lceil\log_2(2r)\rceil}$ to be a power of $2$.
    \item If $G_{F^{\unr}}$ is isomorphic to $\SO_{2r}$, then by \cite[\S3.4]{GP00} every elliptic Weyl element is a twisted Coxeter element of a Weyl group of type $\mathrm{D}_{a_1} \times \cdots \times \mathrm{D}_{a_s}$, where $(a_1, \dots, a_s)$ is a partition of $r$ with $s$ even. For such an element, the group $X$ of Theorem~\ref{thm:llc-partial-characterization}(B) is $(\Z/2)^a \oplus (\Z/4)^b$, where $a+2b \leq s$, and Lemma~\ref{lemma:2-group-explicit-restriction} shows $N \leq 2^{\lceil\log_2(2r)\rceil}$ is a power of $2$.
\end{enumerate}
Note that in all cases above, every unramified twisted Levi $F$-subgroup of $G$ becomes isomorphic to a product of groups, each of which is either isomorphic to $\GL_s$ ($s \leq r$) or of the same type as $G$ (of smaller rank). From this, if one assumes that $\rho$ respects products and Weil restriction then one can obtain similar upper bounds to the above for general $T$. However, the bounds are of a more combinatorial nature and appear tricky to state cleanly.
\end{example}

Using Lemma~\ref{lemma:explicit-finite-index-bound} and the classification of elliptic (twisted) Weyl elements as in \cite[\S3.4]{GP00}, \cite[\S3]{GM97}, \cite[\S\S3 and 4]{GKP00}, and \cite[\S7]{He07}, it should be possible to give reasonably sharp explicit bounds on the integer $N$ in Theorem~\ref{thm:llc-partial-characterization}(B) in all cases. Indeed, it is possible to bound both integers $m$ and $n_0$ in Lemma~\ref{lemma:explicit-finite-index-bound} using the characteristic polynomial of $\alpha$, as the following well-known elementary lemma (inspired by \cite[Lemma 4.1]{Ree11}) shows.

\begin{lemma}\label{lemma:minimal-characteristic-poly-size}
Let $Y_0$ be a finite free $\Z$-module, and let $\alpha_0$ be an automorphism of $Y_0$ such that $(Y_0)_{\alpha_0}$ is finite. Let $\mu$ (resp.\ $\chi$) be the minimal (resp.\ characteristic) polynomial of $\alpha_0$. Then
\begin{enumerate}
    \item $(Y_0)_{\alpha_0}$ is killed by $\mu(1)$,
    \item $(Y_0)_{\alpha_0}$ is of order $|\chi(1)|$.
\end{enumerate}
\end{lemma}

\begin{proof}
For (1), observe that $\mu(1)$ acts by $\mu(\alpha_0) = 0$ on $(Y_0)_{\alpha_0}$. For (2), note that $Y_0$ admits a finite index $\alpha_0$-stable submodule $Y_1$ which is a direct sum of $\Z$-submodules on which $\alpha_0$ acts by a matrix in rational canonical form. Applying Tate cohomology for $\alpha_0$ to the exact sequence $0 \to Y_1 \to Y_0 \to Y_0/Y_1 \to 0$ and noting that $(Y_0)^{\alpha_0} = (Y_1)^{\alpha_0} = 0$ and $|\rT^i(\alpha_0, Y_0/Y_1)|$ is independent of $i$ (by standard results on Herbrand quotients), we thereby reduce from $Y_0$ to $Y_1$. But the claim is easily checked by direct calculation for $Y_1$.
\end{proof}

\begin{example}\label{ex:llc-exceptional}
Let $G$ be an unramified absolutely simple $F$-group of exceptional type, suppose that $p$ is good for $G$, and suppose that $\rho$ is a map as in Theorem~\ref{thm:llc-partial-characterization} with $\cS = \{p\}$ which is compatible with Weil restrictions and products. Let $\pi$ be a non-singular supercuspidal $\ol\Q_\ell$-representation of $G(F)$ with associated torus-character pair $(T, \theta)$. The torus $T/Z(Z_G(T_0))$ is a totally ramified torus of (a semisimple quotient of) a Levi subgroup of $G$, so to bound the integer $N$ in Theorem~\ref{thm:llc-partial-characterization}(B) it suffices to bound $N$ when $T$ is totally ramified and $G$ has the same type as a Levi subgroup of an exceptional group.

For each element of the Weyl group $\Omega$, there is a Weyl subgroup $\Omega_0$ of $\Omega$ (i.e., the Weyl group of a root subsystem of $\Phi(G_{\ol F}, T_{\ol F})$, as in \cite[\S3, Example~(iii)]{Car72}) such that $w \in \Omega_0$ and $w$ is not contained in a proper Weyl subgroup of $\Omega_0$. If $w$ is elliptic, then this root subsystem must span $X^*(T_{\ol F}) \otimes \Q$. The set of possible $\Omega_0$ which occur can be described using Borel--de Siebenthal theory, as explained in \cite[\S4]{Car72}. Moreover, given $\Omega_0$, the set of $w$ satisfying this property is classified in \cite[\S5, Theorem A]{Car72}, and the possible characteristic polynomials of such $w$ are described in \cite[\S6, Proposition 21]{Car72}.

Using these considerations and Lemmas~\ref{lemma:explicit-finite-index-bound} and \ref{lemma:minimal-characteristic-poly-size}, we can bound the integer $N$ in Theorem~\ref{thm:llc-partial-characterization}(B) effectively. The calculations are straightforward but tedious, and we will not use them, so we only summarize here:
\begin{enumerate}
    \item If $G$ is of type $\mathrm{G}_2$, then $N \leq 4$.
    \item If $G$ is of type $\mathrm{F}_4$, then $N \leq 9$.
    \item If $G$ is of type $\mathrm{E}_6$, then $N \leq 18$.
    \item If $G$ is of type $\mathrm{E}_7$, then $N \leq 24$.
    \item If $G$ is of type $\mathrm{E}_8$, then $N \leq 36$.
\end{enumerate}
By inspecting \cite[Table 1]{Ree11}, one sees that if $G$ is of type $\mathrm{E}_8$ and $T$ is totally ramified, then there are $30$ possible elliptic Weyl elements $w$ corresponding to $T$, and for most such elements $N$ is considerably smaller. For instance, $N = 1$ for nine possible such $w$ (those with $X_w = 0$, with notation as in \cite[Table 1]{Ree11}), and $N$ divides $4$ for seven more (those with $X_w = 2^2$). There are only five such $w$ for which we cannot ensure $N \leq 9$ (those labeled $A_1^8$, $A_3D_5(a_1)$, $A_4^2$, $A_2^4$, and $A_5A_2A_1$, for which $N$ divides $16$, $16$, $25$, $27$, and $36$, respectively).
\end{example}

\subsection{Proof of Theorem \ref{thm:llc-partial-characterization}}

Fix a connected reductive $F$-group $G$ and a normalized irreducible Yu datum $\Psi = ((G^i),x,(r_i),\tau,(\phi_i))$ for $G$ with $k$-coefficients. We start by proving (A).

\subsubsection{Reductions} Let $T_{\der}\coloneqq(T\cap G_{\der})^\circ$, and let $T_{\der,0}$ be its maximal unramified subtorus. Let $\pi = \pi(\Psi)$. We will prove \eqref{eqn:abstract-agreement-on-inertia} and the first statement of (B) by induction on the dimension of $G_{\mathrm{der}}$ and on the size of the image of $I_F$ in $\Aut(X^*((T_{\der})_{\ol F}))$. The base case in which $G$ is a torus is clear from (1). If $T_{\der,0}\neq1$, then $H \coloneqq Z_G(T_{\der,0})$ is a proper unramified twisted Levi subgroup of $G$, so the induction hypothesis applies to $H$, and it proves (A) for $G$ by (5) and Theorem~\ref{thm:kaletha-functoriality}. We may therefore assume $T_{\der,0}=1$, so $T_0$ is central in $G$ and $T_{\der}$ is totally ramified. Since $T$ is maximally unramified in $G^0$, it follows that $Z_{G^0}(T_0) = T$ and thus $G^0 = T$ since $T_0$ is central in $G^0$. In this case, observe that $\Psi$ is automatically $F$-non-singular by Remark~\ref{remark:toral-yu-ns}. 


Since $T_{\der}$ is totally ramified, it follows from Lemma~\ref{lemma:non-singular-irreducible}(3) that $\rho^{\Kal}(\Psi)|_{I_F}$ is irreducible. Lemma~\ref{lemma:irreducibility-criterion} then shows that $Z_{\wh G}(\rho^{\Kal}(\Psi)(P_F))^\circ$ is a torus, and since it contains the maximal torus $\wh T$ through which $\rho^{\Kal}(\Psi)|_{P_F}$ factors, it equals $\wh T$. Irreducibility of $\rho^{\Kal}(\Psi)|_{I_F}$ implies that $\wh T^{I_F/P_F}/Z(\wh G)^{I_F/P_F}$ is finite.

We first prove
\begin{equation}\label{eqn:abstract-agreement-on-wild-inertia}
\rho(\pi)|_{P_F} \sim \rho^{\Kal}(\Psi)|_{P_F}.
\end{equation}

\subsubsection{Killing Frobenius} Note that Lemma~\ref{lemma:non-singular-irreducible}(3) implies that $\rho^{\Kal}(\Psi)|_{I_F}$ is irreducible. Let $M$ be a positive integer such that the image of $W_{F_M}$ in $\Aut(X^*((T_{\der})_{\ol F}))$ is equal to the image of $I_F$. Write $M = M_0M_1$, where $M_0$ divides a power of $n$ and $M_1$ is prime to $n$. Write $M_0 = \ell_1 \cdots \ell_b$ for primes $\ell_i$, with $b=0$ if $M_0=1$. If $b > 0$, then since $T_{\der}$ is totally ramified, $T_{F_{\ell_1}}$ is still elliptic. If $\Psi_{F_{\ell_1}}$ is the Yu datum from (4), then
\begin{equation}\label{eq:prime-switching-wild-comparison}
\rho(\pi(\Psi_{F_{\ell_1}}))|_{P_F}
\sim \rho(\pi(\Psi))|_{P_F},
\qquad
\rho^{\Kal}(\Psi_{F_{\ell_1}})|_{P_F}
\sim \rho^{\Kal}(\Psi)|_{P_F}.
\end{equation}
and thus we may pass from $F$ to $F_{\ell_1}$ in order to prove \eqref{eqn:abstract-agreement-on-wild-inertia}. Repeating this process for the $b$ prime factors of $M_0$, we may pass from $F$ to $F_{M_0}$ and thus reduce to the case that $M$ is prime to $n$. Repeating this procedure for the prime factors of $[F(\mu_n):F]$, all of which lie outside $\cS$ by assumption, we may further assume that $\mu_n\subset F$.

\subsubsection{Killing ramification} The torus $T_{\der}$ is now totally ramified, and $n$ is the ramification degree of its splitting field. We will now make further base changes to kill the ramification.

By the choice of $M$, the image of $W_{F_M}$ in $\Aut(X^*((T_{\der})_{\ol F}))$ is the same as the (cyclic) image of $I_F$, and it is of order $n$. Hence the splitting field $E/F_M$ of $(T_{\der})_{F_M}$ is a totally ramified cyclic extension of degree $n$. We may therefore write $E = F_M(\sqrt[n]{\varpi})$, where $\varpi \in F_M^\times$ is a uniformizer. Let $\varpi_0 \in F^\times$ be a uniformizer. We claim that there is an element $c \in F^\times$ such that the extension $E' = F(\sqrt[n]{c\varpi_0})$ (which is Galois since $F$ contains all $n$th roots of unity in $\ol F$) is a tamely ramified degree $n$ extension of $F$, is linearly disjoint from $E$, and has the property that $EE'/E'$ is unramified.
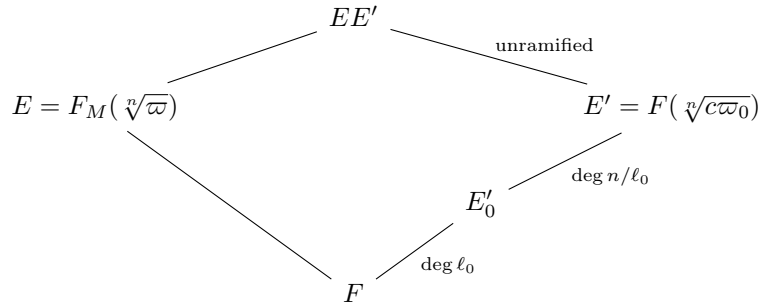
\[
\begin{tikzcd}
& & EE' \ar[dll, dash] \ar[drr, "\text{unramified}", dash] \\
E  = F_M(\sqrt[n]{\varpi}) \ar[ddrr, dash] & & & & E' = F(\sqrt[n]{c\varpi_0}) \ar[dl, dash, "\deg n/\ell_0"] \\
& & & E_0' \ar[dl, dash, "\deg \ell_0"] \\
& & F
\end{tikzcd}
\]
To see this, let $n = \ell_1'^{e_1} \cdots \ell_t'^{e_t}$ be the prime factorization of $n$, and for each $i$ choose a unit $c_i'\in\cO_F^\times$ which is an $n/\ell_i'^{e_i}$th power but not an $\ell_i'$th power; such a unit exists because $\mu_n\subset F$. Since $\ell_i' \nmid M$, the extensions $F_M(\sqrt[\ell_i']{\varpi_0})/F_M$ and $F_M(\sqrt[\ell_i']{c_i'\varpi_0})/F_M$ are linearly disjoint, so at least one is linearly disjoint from $E/F_M$. Choose $c_i \in \{1, c_i'\}$ giving that extension and set $c = \prod_{i=1}^t c_i$. One checks that $E$ and $E'$ are linearly disjoint and $EE'$ is unramified by Abhyankar's lemma, proving the claim.

\subsubsection{Congruence on wild inertia} Let $E'_0/F$ be a cyclic subextension of $E'/F$ of prime degree $\ell_0$.
Since $E'_0$ is linearly disjoint from $E$, base change to $E'_0$ preserves the image of $W_F$ in $\Aut(X^*((T_{\der})_{\ol F}))$ and thus $T$ is still elliptic. Thus \eqref{eq:prime-switching-wild-comparison} holds with $E_0'$ in place of $F_{\ell_1}$. By construction, the image of $I_{E_0'}$ in $\Aut(X^*((T_{\der})_{\ol F}))$ is of order $n/\ell_0$, so the induction hypothesis applied to $\Psi_{E'_0}$ therefore gives
\begin{equation}\label{eq:wild-congruence}
\rho(\pi(\Psi_{E'_0}))|_{P_F} \sim \rho^{\Kal}(\Psi_{E'_0})|_{P_F}.
\end{equation}
Combining the two comparisons in \eqref{eq:prime-switching-wild-comparison} (with $E_0'$ in place of $F_{\ell_1}$) with the induction hypothesis gives
\[
\begin{aligned}
\rho(\pi(\Psi))|_{P_F}
\stackrel{\eqref{eq:prime-switching-wild-comparison}}\sim \rho(\pi(\Psi_{E'_0}))|_{P_F}
    \stackrel{\eqref{eq:wild-congruence}}\sim \rho^{\Kal}(\Psi_{E'_0})|_{P_F}
\stackrel{\eqref{eq:prime-switching-wild-comparison}}\sim \rho^{\Kal}(\Psi)|_{P_F}.
\end{aligned}
\]
Thus \eqref{eqn:abstract-agreement-on-wild-inertia} follows by induction on the size of the image of $I_F$ in $\Aut(X^*((T_{\der})_{\ol F}))$.

\subsubsection{Promoting from wild inertia to all of inertia} With \eqref{eqn:abstract-agreement-on-wild-inertia} established, we may pass to $\wh G(k)$-conjugates to assume that 
\[
\rho(\pi)|_{P_F} = \rho^{\Kal}(\Psi)|_{P_F}.
\] 
Identify the dual torus $\wh T$ with its image in $\wh G$ via $\ld j_{T,G}|_{\wh T \times \{1\}}$. We have already seen $Z_{\wh G}(\rho(\pi)(P_F))^\circ = Z_{\wh G}(\rho^{\Kal}(\Psi)(P_F))^\circ = \wh T$, so Lemma~\ref{lemma:align-tame-components} implies that the two actions on $\wh T$ agree after conjugating $\rho(\pi)$ by some element of $Z_{\wh G}(\rho(\pi)(P_F))$. After making this conjugation and renaming $\rho(\pi)$, let $s\in I_F$ lift a topological generator of $I_F/P_F$, define $h_1=\rho(\pi)(s)$ and $h_2=\rho^{\Kal}(\Psi)(s)$.


Set $\wh{T}_{\der}\coloneqq \wh T\cap\wh G_{\der}$, so $\wh{T}_{\der}^{h_1}$ is finite. Let $f\co\ld G\to\ld G/\wh G_{\der}$ be the quotient map. Compatibility with central characters in hypothesis~(2) shows that there is an element $\ol t\in(\wh G/\wh G_{\der})(k)$ such that $f(h_2)=\ol t f(h_1)\ol t^{-1}$; let $t \in \wh T(k)$ lift $\ol t$, so $d\coloneqq h_2^{-1}t h_1t^{-1} \in \wh G_{\der}(k)$. Since $h_1$ and $h_2$ induce the same action on $\wh T$ and have the same image in $W_F$, we have $d\in Z_{\wh G_{\der}}(\wh T)(k) =\wh T_{\der}(k)$. This means that the images of $h_1$ and $h_2$ in $N_{\wh G\rtimes W_F}(\wh T)/\wh{T}_{\der}$ are conjugate under $\wh T/\wh{T}_{\der}$, so \cite[Lemma~5.3.4]{Cot26b} implies that $h_1$ and $h_2$ are $\wh T(k)$-conjugate. This means precisely that $\rho_I(\pi) \sim \rho_I^{\Kal}(\Psi)$, proving \eqref{eqn:abstract-agreement-on-inertia}. The final claim of (A) follows from Lemma~\ref{lemma:non-singular-irreducible}(2).

\subsubsection{Proof of (B)}
Next, we prove (B), so assume $p$ is good for $Z_G(T_0)$ and $\Psi$ is non-singular. In particular, $S_\Psi$ is a torus in this case.

We begin with some reductions. Let $H \subset \wt H$ be an embedding with $H_{\der} = \wt H_{\der}$ such that $Z(\wt H)$ is an induced torus, and let $\wt T = T \cdot Z(\wt H)$. If $\wt X$ and $\wt X_{\ad}$ are the $\Z[W_F/I_F]$-modules attached to $(\wt H, \wt T)$ as in (B), then $\wt X = \wt X_{\ad} = X_{\ad}$. By (3) and Lemma~\ref{lemma:kaletha-compatibility-with-central-isogenies}, we may and do therefore assume that $Z(G)$ is an induced torus, which implies that $\wh G$ has simply connected derived group.

Note that the torus $S_\Psi$ lies in the image of $\ld j_{T, G}|_{\wh T \rtimes 1}$ by Lemma~\ref{lemma:non-singular-irreducible}(2), and thus it does not change when applying Theorem~\ref{thm:kaletha-functoriality} in the above reduction to the case that $T \cap G_{\der}$ is totally ramified. By condition (5)(a), the action of $\rho(\pi(\Psi))$ on $S_\Psi$ is also unchanged by this reduction step. Thus to check that there are $\wh G(k)$-conjugates such that $\rho(\pi(\Psi))|_{I_F} = \rho^{\Kal}(\Psi)|_{I_F}$ and the two actions of $W_F$ on $S_\Psi$ are the same, we may argue as above to reduce to the case that $G^0 = T$ and $T_{\der,0} = 1$.

Once we have established the action claim of (B), we can also reduce \eqref{eqn:abstract-agreement-on-finite-index} to the case that $T \cap G_{\der}$ is totally ramified: indeed, if the result is known in this case and we let $H = Z_G(T_0)$, then condition (5)(b) shows that we may choose representatives for $\rho(\pi(\Psi))$ and $\rho^{\Kal}(\Psi_H)$ such that there is a cocycle $z\co W_F/I_F \to Z(\wh H \cap \wh G_{\der})(k)^{I_F}$ such that
\[
\rho(\pi(\Psi))|_{W_{F_N}} = \left(z \cdot \ld j_{H,G} \circ \rho^{\Kal}(\Psi_H)\right)|_{W_{F_N}}.
\]
Thus for the reduction step it suffices to show that $z|_{W_{F_N}/I_F}$ is a coboundary.

For this, note the short exact sequence
\[
1 \to Z(\wh H_{\der}) \to Z(\wh H \cap \wh G_{\der}) \to (\wh H \cap \wh G_{\der})/\wh H_{\der} \to 1,
\]
where $(\wh H \cap \wh G_{\der})/\wh H_{\der}$ is dual to $Z(H)/Z(G)$. Since $Z(H)/Z(G)$ is the center of an unramified twisted Levi subgroup of $G_{\ad}$, it follows that $(Z(H)/Z(G))_{F^{\unr}}$ is an induced torus and thus $((\wh H \cap \wh G_{\der})/\wh H_{\der})^{I_F}$ is a torus. Since $((\wh H \cap \wh G_{\der})/\wh H_{\der})^{W_F}$ is finite, \cite[Lemma~5.3.4]{Cot26b} implies that $z$ is cohomologous to a $Z(\wh H_{\der})(k)^{I_F}$-valued cocycle. If $\wh T_{\der} = \wh T \cap \wh H_{\der}$, then $\wh T_{\der}(k)^{I_F} \cong \Hom(X, k^\times)$, and thus indeed $z|_{W_{F_N}}$ is a coboundary by definition. We have thus reduced (B) to the case that $T \cap G_{\der}$ is totally ramified. Since $T_0$ is now central in $G$, our assumption implies that $p$ is good for $G$.

By \eqref{eqn:abstract-agreement-on-inertia}, we may pass to $\wh G(k)$-conjugates to assume $\rho(\pi)|_{I_F} = \rho^{\Kal}(\Psi)|_{I_F}$. Let
$C_P\coloneqq Z_{\wh G}(\rho(\pi)(P_F))$, so $C_P^\circ = \wh T$ as above. By \cite[Proposition 2.1]{GH91} and \cite[Theorem 2.15]{St75} (using that $\wh G$ has simply connected derived group, $G$ is tamely ramified\footnote{This follows from the existence of a Yu datum for $G$.}, and $p$ is good for $G$), if $\gamma \in P_F$ is any element then $Z_{\wh G}(\rho(\pi)(\gamma))$ is a Levi $k$-subgroup of $\wh G$. Applying this inductively to a composition series of $\rho(\pi)(P_F)$ with cyclic subquotients, we learn that $C_P$ is connected, and thus $C_P = \wh T$. 

Let $\wh S = Z_{\wh G}(\rho(\pi)(I_F))^\circ_{\red}$, so $\wh S$ is a (central) torus. Let $\rho_1 = \rho(\pi)$, and let $\rho_2 = \rho^{\Kal}(\Psi) = \ld j_{T,G} \circ \ld\theta$. Let $w \in W_F$ be any element. For $\gamma \in I_F$, we have
\begin{equation}\label{eq:common-inertial-conjugation}
^{\rho_2(w)}\rho_2(\gamma) = \rho_2({}^w\gamma) = \rho_1({}^w\gamma) = {}^{\rho_1(w)}\rho_1(\gamma) = {}^{\rho_1(w)}\rho_2(\gamma),
\end{equation}
so $\rho_1(w)$ and $\rho_2(w)$ have the same effect on $\rho_1(I_F) = \rho_2(I_F)$. Since $Z_{\wh G}(\rho_1(P_F)) =\wh T$, the parameters $\rho_1$ and $\rho_2$ differ by a cocycle valued in $\wh T(k)$ and induce the same action on $\wh T$, hence also on $\wh S$, as desired. The final two claims of (B) now follow, the latter from Lemma~\ref{lemma:non-singular-irreducible}(2).

Finally, we turn to the proof of \eqref{eqn:abstract-agreement-on-finite-index}. For this, we first pin down the images of $\rho_1$ and $\rho_2$ in $\ld G / \wh{G}_{\der}$. As above, let $s \in I_F$ be a lift of a pro-generator of tame inertia. Let $f\co \ld G \to \ld G/\wh G_{\der}$ be the natural quotient map, and note that $f \circ \rho_1$ and $f \circ \rho_2$ are $(\wh G/\wh G_{\der})(k)$-conjugate by (1) and (2). Since $Z(G)$ is an induced torus, the coinvariant group $X_*(Z(G_{\ol F}))_{I_F}$ is torsion-free and thus the $k$-group scheme $(\wh G/\wh G_{\der})^{I_F}$ is connected. For dimension reasons, it follows that the map $Z(\wh G)^{I_F} \to (\wh G/\wh G_{\der})^{I_F}$ is surjective. Thus, after conjugating by an element of $Z(\wh G)^{I_F}(k)$, we may assume that the central characters of $\rho_1$ and $\rho_2$ agree.

For $w \in W_F$, we define $\varphi(w)\coloneqq\rho_1(w)\rho_2(w)^{-1} \in \wh G_{\der}(k)$. Now, \eqref{eq:common-inertial-conjugation} implies that $\rho_1(w)$ and $\rho_2(w)$ induce the same automorphism of the common subgroup $\rho_1(I_F)=\rho_2(I_F)$, so in fact we have $\varphi(w)\in Z_{\wh G_{\der}}(\rho_1(I_F))(k)$. The map $\varphi$ is constant on the right cosets of $I_F$ and hence factors through a map, again denoted by $\varphi$,
\[
\varphi\co W_F/I_F\longrightarrow Z_{\wh G_{\der}}(\rho_1(I_F))(k) = \wh T_{\der}^{\rho_1(I_F)}(k),
\qquad
\varphi(\ol w)=\rho_1(w)\rho_2(w)^{-1},
\]
and it is straightforward to check that this defines a cocycle. Note that the character group of $\wh T_{\der}^{\rho_1(I_F)}$ is isomorphic to $X = X_*((T/Z(Z_G(T_0))^\circ_{\red})_{\ol F})_{I_F}$, so $\wh T_{\der}^{\rho_1(I_F)}(k) \cong \Hom(X, k^\times)$. By hypothesis on $N$, the restriction $\varphi|_{W_{F_N}/I_F}$ is a $1$-coboundary, proving \eqref{eqn:abstract-agreement-on-finite-index}.
\qed 


\appendix
\section{Some calculations with sign characters}\label{app:rootwise-sign-chi-calculations}

In this appendix, we establish certain technical lemmas related to the FKS sign character and $\chi$-data.

\subsection{Proof of Lemma~\ref{lemma:comparison-to-minimal-ramified-chi-data}}\label{app:comparison-minimal-proof}

\begin{proof}
By passing to an inner twist as in the construction of $\ld j_\chi$, we may assume that $G$ is quasi-split. We only prove the assertion for $G$ when $T$ is elliptic; if $T$ is not elliptic, the same calculation works after restricting the characters to $T(F)_{\mathrm b}$ and the corresponding cocycles to $I_F$. Write $\ld\theta(w)=(\wh\theta(w),w)$ for a $1$-cocycle $\wh\theta\co W_F\to\wh T(k)$.
By definition \eqref{eqn:def-of-L-embedding}, our claim reduces to the following equality of cohomology classes:
\begin{equation}\label{eqn:signs-cancel-1}
[Q_{\chi''_G}Q_{\min\chi''_G}^{-1}]=[\wh\epsilon_{G,\flat,0}]\quad\text{in }\rH^1(W_F, \wh T(k)).
\end{equation}

Observe that in \eqref{eqn:def-of-Q}, the elements $v_0(u_i(w))$ and $w_i^{-1}\wh\alpha^\vee$ do not depend on the choice of set of $\chi$-data. Moreover, the terms $\chi''_{G,\alpha}$ and $\min\chi''_{G,\alpha}$ only differ if $\alpha$ is either asymmetric or symmetric unramified. Thus we have
\[
Q_{\chi''_G}Q_{\min\chi''_G}^{-1} = \prod_{\ol\alpha \in (\Phi(G_{\ol F},T_{\ol F})_{\mathrm{asym}} \cup \Phi(G_{\ol F}, T_{\ol F})_{\mathrm{sym,unram}})/\Sigma} Q_{\chi''_G,\ol\alpha}Q_{\min\chi''_G,\ol\alpha}^{-1}.
\]

Now set $G^{-1} \coloneqq T$ and fix a root $\alpha \in \Phi((G^{j+1}/G^j)_{\ol F}, T_{\ol F})$, where $0 \leq j \leq d-1$, which is either asymmetric or symmetric unramified, and let $\alpha_j = \alpha|_{(Z_{G^j})_{\ol F}} \in \Phi((G^{j+1}/G^j)_{\ol F}, (Z_{G^j})_{\ol F})$ denote the corresponding restricted character. There are five cases, as in the proof of \cite[Proposition 5.27]{Kal21a}:
\begin{enumerate}
    \item $\alpha$ and $\alpha_j$ are both asymmetric,
    \item $\alpha$ is asymmetric and $\alpha_j$ is symmetric unramified,
    \item $\alpha$ is asymmetric and $\alpha_j$ is symmetric ramified,
    \item $\alpha$ and $\alpha_j$ are both symmetric unramified,
    \item $\alpha$ is symmetric unramified and $\alpha_j$ is symmetric ramified.
\end{enumerate}
Define a character $\epsilon_{\flat,0,\alpha}\co T(F) \to k^\times$ as follows: in cases (1), (2), and (4), let $\epsilon_{\flat,0,\alpha}$ be the trivial character. In case (3), define
\[
\epsilon_{\flat,0,\alpha}(t) = \sgn_{\F_\alpha^\times}(\alpha(t))^{e(\alpha/\alpha_j)}.
\]
In case (5), define
\[
\epsilon_{\flat,0,\alpha}(t) = \sgn_{\F_\alpha^1}(\alpha(t))^{e(\alpha/\alpha_j)}.
\]
By definition, we have
\[
\epsilon_{\flat,0} = \prod_{\alpha \in (\Phi(G_{\ol F},T_{\ol F})_{\mathrm{asym}}\cup \Phi(G_{\ol F},T_{\ol F})_{\mathrm{sym,unram}})/\Sigma} \epsilon_{\flat,0,\alpha},
\]
so it suffices to show the following equality of cohomology classes for each $\alpha$:
\begin{equation}\label{eqn:min-ram-root-by-root}
[Q_{\chi''_G,\ol\alpha}Q_{\min\chi''_G,\ol\alpha}^{-1}]=[\wh\epsilon_{\flat,0,\alpha}]\quad\text{in }\rH^1(W_F, \wh T(k)).
\end{equation}
In cases (1) and (4), observe that $\chi''_{G,\alpha} = \min \chi''_{G,\alpha}$, so \eqref{eqn:min-ram-root-by-root} is immediate from the definitions. Next, suppose that $\alpha$ is asymmetric; in this case, we have $\min\chi''_{G,\ol\alpha} = 1$, and by definition of $Q$, for $w \in W_F$ we have
\begin{equation}\label{eqn:q-difference-asymmetric}
Q_{\chi''_{G,\ol\alpha}}(w)Q_{\min\chi''_{G,\ol\alpha}}(w)^{-1} = \prod_{i=1}^m (w_i^{-1}\wh\alpha^\vee)(\chi''_{G,\alpha}(\rec_{F_\alpha}(v_0(u_i(w))))).
\end{equation}
Suppose first that we are in case (2). Then $\chi''_{G,\alpha_j}$ is the unique unramified quadratic character of $F_{\alpha_j}^\times$, so $\chi''_{G,\alpha}$ is also unramified of order $\leq 2$. Lemma~\ref{lemma:dual-l-homs} identifies the cohomology class represented by the left side of \eqref{eqn:q-difference-asymmetric} with the cocycle associated to $\chi''_{G,\alpha}\circ\alpha$, where $\alpha\co T\to J_\alpha=\Res_{F_\alpha/F}\G_m$ is the $F$-homomorphism arising from the root $\alpha$. Since $T$ is elliptic, $\alpha(T(F))\subset\cO_{F_\alpha}^\times$, on which $\chi''_{G,\alpha}$ is trivial. The resulting cohomology class is therefore trivial in case (2).

Next suppose that we are in case (3). Once again, $\alpha(T(F))\subset\cO_{F_\alpha}^\times$. On $\cO_{F_\alpha}^\times$, the character $\chi''_{G,\alpha}$ is equal to
\[
\sgn_{\F_{\alpha_j}^\times} \circ \Nm_{\F_\alpha/\F_{\alpha_j}}^{e(\alpha/\alpha_j)} = \sgn_{\F_\alpha^\times}^{e(\alpha/\alpha_j)}.
\]
Lemma~\ref{lemma:dual-l-homs} therefore identifies the class represented by \eqref{eqn:q-difference-asymmetric} with the cohomology class dual to $\epsilon_{\flat,0,\alpha}$, which proves \eqref{eqn:min-ram-root-by-root} in this case.

Finally, suppose we are in case (5). Since both $\chi''_{G,\alpha}$ and $\min\chi''_{G,\alpha}$ restrict to $\kappa_\alpha$ on $F_{\pm\alpha}^\times$, the quotient $\chi''_{G,\alpha} \cdot (\min\chi''_{G,\alpha})^{-1}$ factors through the map $F_\alpha^\times \to F_\alpha^1$, $t \mapsto t \cdot v_1(t)^{-1}$. By the same argument as in the asymmetric case, the factored map is $\sgn_{\F_\alpha^1}^{e(\alpha/\alpha_j)}\co F_\alpha^1 \to k^\times$. Therefore Lemma~\ref{lemma:dual-l-homs}, applied to this quotient character, gives the following identity of classes:
\[
[\wh\epsilon_{\flat,0,\alpha}]=\left[\,w\longmapsto\prod_{i=1}^m (w_i^{-1}\wh\alpha^\vee)((\chi''_{G,\alpha}\cdot(\min\chi''_{G,\alpha})^{-1})(\rec_{F_\alpha}(v_0(u_i(w)))))\,\right]
\]
This is exactly \eqref{eqn:min-ram-root-by-root} in case (5).
\end{proof}

\subsection{Proof of Lemma~\ref{lemma:modification-of-symmetric-ramified-chi-data}}\label{app:modification-symmetric-proof}

\begin{proof}
By twisting, we may and do assume that $G$ is quasi-split. We will assume that $T$ is elliptic; the general case is the same after restricting characters to $T(F)_{\mathrm b}$ and cocycles to $I_F$. As in the proof of Lemma~\ref{lemma:comparison-to-minimal-ramified-chi-data}, the proof reduces to the following equality of cohomology classes:
\begin{equation}\label{eqn:signs-cancel-2}
[Q_{\chi''_{H,G}}Q_{\min\chi''_G}^{-1}]=[\wh\epsilon_{G,\flat,1}\wh\epsilon_{H,\flat,1}\wh\epsilon_{G,\flat,2}\wh\epsilon_{H,\flat,2}]\quad\text{in }\rH^1(W_F, \wh T(k)).
\end{equation}
Since $\chi''_{H,G}$ and $\min\chi''_G$ are both minimally ramified, these sets of $\chi$-data only differ at symmetric ramified roots, at which they agree with $\chi''_H$ and $\chi''_G$, respectively. Thus we have
\[
Q_{\chi''_{H,G}}Q_{\min\chi''_G}^{-1} = \prod_{\ol\alpha \in \Phi(G_{\ol F}, T_{\ol F})_{\mathrm{sym,ram}}/\Gamma} Q_{\chi''_H,\ol\alpha}Q_{\chi''_G,\ol\alpha}^{-1}.
\]
If we further decompose $\epsilon_{G,\flat,i} = \prod_{\alpha \in \Phi(G_{\ol F},T_{\ol F})_{\mathrm{sym,ram}}/\Gamma} \epsilon_{G,\flat,i,\alpha}$ for $i = 1, 2$ as in the definitions \eqref{eqn:FKS-4}, \eqref{eqn:FKS-5}, then it suffices to prove the following equality for all $\alpha \in \Phi(G_{\ol F}, T_{\ol F})_{\mathrm{sym,ram}}$:
\begin{equation}\label{eqn:q-difference-root-by-root}
    [Q_{\chi''_H,\ol\alpha}Q_{\chi''_G,\ol\alpha}^{-1}]=[\wh\epsilon_{G,\flat,1,\alpha}\wh\epsilon_{H,\flat,1,\alpha}\wh\epsilon_{G,\flat,2,\alpha}\wh\epsilon_{H,\flat,2,\alpha}]\quad\text{in }\rH^1(W_F, \wh T(k)).
\end{equation}

Set $G^{-1} \coloneqq T$ and fix a root $\alpha \in \Phi((G^{j+1}/G^j)_{\ol F}, T_{\ol F})_{\mathrm{sym,ram}}$. By Lemma~\ref{lemma:unramified-levi-not-ramified-symmetric}, we have $\alpha \in \Phi(H_{\ol F}, T_{\ol F})$. For $w \in W_F$, we have
\[
Q_{\chi''_H,\ol\alpha}(w)Q_{\chi''_G,\ol\alpha}(w)^{-1} = \prod_{i=1}^m (w_i^{-1}\wh\alpha^\vee)((\chi''_{H,\alpha} \cdot (\chi''_{G,\alpha})^{-1})(\rec_{F_\alpha}(v_0(u_i(w))))).
\]
Let $\alpha_{G,j}$ (resp.\ $\alpha_{H,j}$) be the restriction of $\alpha$ to $Z_{G^j}$ (resp.\ $Z_{H^j}$). Note that $\chi''_{G,\alpha}$ is a tamely ramified character of $F_\alpha^\times$ whose restriction to $\cO_{F_\alpha}^\times$ is $\sgn_{\F_{\alpha_{G,j}}^\times}\circ \Nm_{F_\alpha/F_{\alpha_{G,j}}}$, and whose restriction to $F_{\pm\alpha}^\times$ is $\kappa_{\alpha_{G,j}} \circ \Nm_{F_{\pm\alpha}/F_{\pm\alpha_{G,j}}}$. By the final paragraph of the proof of \cite[Lemma 4.2.7]{FKS23} (the statement of which assumes, but the proof of which does not really use, that $p$ does not divide the order of any bond in the Dynkin diagram), replacing $\ell_G(\alpha^\vee)$ from \textit{loc.\ cit.}\ by $\ell_{G,p'}(\alpha^\vee)$, we have 
\begin{align}
    \chi''_{G,\alpha}(\ell_{G,p'}(\alpha^\vee)a_\alpha) 
        &= (-1)^{f_\alpha+1} \cdot \frG_{\F_\alpha}(\psi^0) \cdot \kappa_\alpha(-1)^{(e(\alpha/\alpha_{G,j})-1)/2}.\label{eqn:chi-double-prime-G-root-value}
\end{align}
These properties uniquely characterize $\chi''_{G,\alpha}$; there is a similar characterization of $\chi''_{H,\alpha}$. Note that the definition of $a_\alpha$ does not change when one passes from $G$ to $H$, justifying the notation.

We have
\[
\sgn_{\F_{\alpha_{G,j}}^\times}\circ \Nm_{F_\alpha/F_{\alpha_{G,j}}}
=\sgn_{\F_\alpha^\times}^{e(\alpha/\alpha_{G,j})}
=\sgn_{\F_\alpha^\times}
\quad\text{on }\cO_{F_\alpha}^\times,
\]
because $e(\alpha/\alpha_{G,j})$ is odd by \cite[Lemma 5.6.5]{FKS23}. The same calculation applies to $H$, so the quotient
\[
\zeta_\alpha\coloneqq\chi''_{H,\alpha}(\chi''_{G,\alpha})^{-1}
\]
is unramified. Moreover, both $\chi''_{H,\alpha}$ and $\chi''_{G,\alpha}$ restrict to $\kappa_\alpha$ on $F_{\pm\alpha}^\times$, so $\zeta_\alpha|_{F_{\pm\alpha}^\times}=1$. Applying \eqref{eqn:chi-double-prime-G-root-value} to $G$ and $H$, we see that $\zeta_\alpha$ is the unique unramified character with these properties and satisfying
\begin{equation}\label{eqn:chi-double-prime-HG-quotient-value}
\zeta_\alpha(a_\alpha) = \sgn_{\F_\alpha^\times}(\ell_{G,p'}(\alpha^\vee)\ell_{H,p'}(\alpha^\vee))\cdot\kappa_\alpha(-1)^{(e(\alpha/\alpha_{G,j})-1)/2 + (e(\alpha/\alpha_{H,j})-1)/2}.
\end{equation}
Because $\zeta_\alpha$ is trivial on $F_{\pm\alpha}^\times$, it factors uniquely through $t\mapsto tv_1(t)^{-1}$, giving a character $\zeta^1_\alpha\co F_\alpha^1\to\{\pm1\}$. Since $a_\alpha$ has odd valuation (with respect to the normalized valuation on $F_\alpha$), we have $a_\alpha v_1(a_\alpha)^{-1}\in-1+\frp_\alpha$, and \eqref{eqn:chi-double-prime-HG-quotient-value} therefore yields, directly from the definitions,
\[
\zeta^1_\alpha\circ\alpha
=\epsilon_{G,\flat,1,\alpha}\epsilon_{H,\flat,1,\alpha}
    \epsilon_{G,\flat,2,\alpha}\epsilon_{H,\flat,2,\alpha}
\]
on $T(F)$, or on $T(F)_{\mathrm b}$ in the inertial case. Lemma~\ref{lemma:dual-l-homs} identifies the dual cohomology class of the left side with $[Q_{\chi''_H,\ol\alpha}Q_{\chi''_G,\ol\alpha}^{-1}]$, proving \eqref{eqn:q-difference-root-by-root} and hence the lemma.
\end{proof}

\subsection{Rootwise formulation of the FKS sign character}\label{app:fks-rootwise-formulation}

Recall the notation $\ol T$ from \S\ref{ss:fks}.
There is another useful way of organizing $\epsilon_x^{G/M}|_{\ol T(\F)}$ as a product of characters: namely, one can write
\[
\epsilon_x^{G/M}|_{\ol T(\F)}
= \prod_{\alpha \in
\Phi((G/M)_{\ol F},T_{\ol F})/\Sigma}
\epsilon_{x,\alpha}^{G/M},
\]
where $\epsilon_{x,\alpha}^{G/M}$ is a $\{\pm1\}$-valued character of $\ol T(\F)$, defined as follows. If $\alpha$ is asymmetric, then we define
\[
\epsilon_{x,\alpha}^{G/M}(\gamma) = \begin{cases}
\sgn_{\F_\alpha^\times}(\alpha(\gamma))
&\begin{aligned}
    &\text{if } \frac{r}{2} \in \ord_x(\alpha) \text{ xor}\\[-2pt]
    &\quad\bigl(\alpha_M \in
    \Phi((G/M)_{\ol F}, (Z_M)_{\ol F})_{\textrm{sym,ram}}\\[-2pt]
    &\qquad\text{and } 2\nmid e(\alpha/\alpha_M)\bigr),
\end{aligned}\\[6pt]
1 &\text{otherwise.}
\end{cases}
\]
If $\alpha$ is symmetric unramified, then we define
\[
\epsilon_{x,\alpha}^{G/M}(\gamma) = \begin{cases}
\sgn_{\F_\alpha^1}(\alpha(\gamma))
&\begin{aligned}
    &\text{if } \frac{r}{2} \in \ord_x(\alpha) \text{ xor}\\[-2pt]
    &\quad\bigl(\alpha_M \in
    \Phi((G/M)_{\ol F}, (Z_M)_{\ol F})_{\textrm{sym,ram}}\\[-2pt]
    &\qquad\text{and } 2\nmid e(\alpha/\alpha_M)\bigr),
\end{aligned}\\[6pt]
1 &\text{otherwise.}
\end{cases}
\]
Finally, if $\alpha$ is symmetric ramified, then we define
\[
\epsilon_{x,\alpha}^{G/M}(\gamma) = \begin{cases}
\begin{aligned}
    &f_{(G,T)}(\alpha)(-1)^{[\F_\alpha\colon \F]+1}\\[-2pt]
    &\quad{}\cdot \sgn_{\F_\alpha^\times}\!\left(
    e_\alpha \ell_{p'}(\alpha^\vee)
    (-1)^{\frac{e(\alpha/\alpha_M)-1}{2}}
    \right)
\end{aligned}
&\text{if } \alpha(\gamma) \in -1 + \frp_\alpha, \\[6pt]
1 &\text{otherwise.}
\end{cases}
\]

We will use this notation freely in the next two subsections.

\subsection{Proof of Lemma~\ref{lemma:fks-base-change}}\label{app:fks-base-change-proof}

\begin{proof}
By definition of the sign characters, we may assume that $G$ is of adjoint type and thus $T(E)_{[x]} = T(E)_{\mathrm b}$.
Filtering the extension $E/F$ by finite extensions of prime degree, we may assume that $E/F$ is cyclic of odd prime degree $\ell$ (possibly equal to $p$). For each $\Sigma$-orbit in $\Phi((G/M)_{\ol F},T_{\ol F})$, choose a $\Gamma$-orbit $\omega$ contained in it.  Then either
\begin{enumerate}
    \item $\omega$ is a single $\Gal(\ol F/E)$-orbit, or
    \item $\omega = \bigcup_{i=0}^{\ell-1} \omega_i$ for pairwise distinct $\Gal(\ol F/E)$-orbits $\omega_0, \dots, \omega_{\ell-1} \subset \omega$.
\end{enumerate}
In case~(1), $E\cap F_\alpha=F$ and $[E_\alpha:F_\alpha]=\ell$, while in case~(2), $E\subset F_\alpha$ and $E_\alpha=F_\alpha$.
Fix some such $\omega$ and fix $\alpha \in \omega$. In case (2), fix $\beta_i \in \omega_i$ for $0 \leq i < \ell$. We will show
\begin{equation}\label{eqn:FKS-bc-root}
\epsilon_{x,\alpha}^{G/M} \circ \Nm_{E/F}|_{T(E)_{\mathrm{b}}} = \begin{cases}
    \epsilon_{x,\alpha}^{G_E/M_E} &\text{in case (1)},\\
\prod_{i=0}^{\ell-1} \epsilon_{x,\beta_i}^{G_E/M_E} &\text{in case (2).}
\end{cases}
\end{equation}
We separate further into cases according to whether $\alpha$ is asymmetric, symmetric unramified, or symmetric ramified. Since $E/F$ is unramified of odd degree, it follows that $\alpha$ is asymmetric (resp.\ symmetric unramified, resp.\ symmetric ramified) for the action of $\Gal(\ol F/F)$ if and only if it is the same for the action of $\Gal(\ol F/E)$. In case (2), each such condition for $\alpha$ is equivalent to the corresponding condition for each $\beta_i$. Completely similar remarks apply to $\alpha_M$ in place of $\alpha$.

Suppose first that $\alpha$ is asymmetric. Note that the good element $X$ remains good over $E$, so the real number $r$ remains the same. Since $E/F$ is unramified, we have $\ord_x(\alpha)_F = \ord_x(\alpha)_E$. Moreover,
\[
e(\alpha/\alpha_M)_F \equiv e(\alpha/\alpha_M)_E \pmod{2}.
\]
Thus by definition, in case (1) above it suffices to show that for $t \in T(E)_{\mathrm{b}}$, we have
\[
\sgn_{\F_{E,\alpha}^\times}(\ol{\alpha(t)}) = \sgn_{\F_\alpha^\times}(\ol{\alpha(\Nm_{E/F}(t))}),
\]
where the overline denotes reduction modulo the maximal ideal. Note that $\ol{\alpha(\Nm_{E/F}(t))} = \alpha(\Nm_{\F_E/\F}(\ol t))$ since $E/F$ is unramified. Since $[\F_{E,\alpha}\co \F_\alpha]$ is odd, we have $\sgn_{\F_\alpha^\times} = \sgn_{\F_{E,\alpha}^\times}|_{\F_\alpha^\times}$. If $\gamma$ is a generator for $\Gal(\F_{E,\alpha}/\F_\alpha)$, then for all $\ol t \in \ol T(\F_E)$ we have
\begin{align*}
\sgn_{\F_\alpha^\times}(\alpha(\Nm_{\F_E/\F}(\bar t))) = \sgn_{\F_{E,\alpha}^\times}(\alpha(\Nm_{\F_E/\F}(\bar t))) &= \prod_{i=0}^{\ell-1} \sgn_{\F_{E,\alpha}^\times}(\gamma^i(\alpha(\bar t)))\\
    &= \sgn_{\F_{E,\alpha}^\times}(\alpha(\bar t))^\ell\\
    &= \sgn_{\F_{E,\alpha}^\times}(\alpha(\bar t)),
\end{align*}
where the second equality follows from the fact that $\Gal(\F_{E,\alpha}/\F_\alpha)$ stabilizes $\alpha$ and the third follows from the fact that $\sgn_{\F_{E,\alpha}^\times}$ is $\Gal(\F_{E,\alpha}/\F_\alpha)$-invariant. This proves \eqref{eqn:FKS-bc-root} in case (1) when $\alpha$ is asymmetric.

Using the same reasoning as above, in case (2) (still assuming that $\alpha$ is asymmetric) we reduce to showing that for $\bar t \in \ol T(\F_E)$ we have
\[
\prod_{i=0}^{\ell-1} \sgn_{\F_{E,\beta_i}^\times}(\beta_i(\bar t)) = \sgn_{\F_\alpha^\times}(\alpha(\Nm_{\F_E/\F}(\bar t))).
\]
In this case, we have $E \subset F_\alpha$, so $\F_{E,\beta_i} = \F_\alpha$ for $0 \leq i < \ell$. Fix $\bar t \in \ol T(\F_E)$. Let $\sigma_i \in \Gal(\ol F/F)$ be such that $\sigma_i^{-1}\alpha = \beta_i$ for all $i$, and note
\[
\prod_{i=0}^{\ell-1} \sgn_{\F_\alpha^\times}(\beta_i(\bar t)) = \prod_{i=0}^{\ell-1} \sgn_{\F_\alpha^\times}(\sigma_i^{-1}(\alpha(\sigma_i\bar t))) = \sgn_{\F_\alpha^\times}(\alpha(\Nm_{\F_E/\F}(\bar t))),
\]
and \eqref{eqn:FKS-bc-root} holds in general when $\alpha$ is asymmetric.

The case that $\alpha$ is symmetric unramified is entirely similar to the previous case; in fact, the previous two paragraphs may be read verbatim in this case if one replaces each instance of $\F_\alpha^\times$ (resp.\ $\F_{E,\alpha}^\times$, resp.\ $\F_{E,\beta_i}^\times$) by $\F_\alpha^1$ (resp.\ $\F_{E,\alpha}^1$, resp.\ $\F_{E,\beta_i}^1$).

Finally, suppose that $\alpha$ is symmetric ramified. Unramified base change preserves $e_\alpha$, $e(\alpha/\alpha_M)$, and the element $c_\alpha$ from \eqref{eq:c-alpha} defining $f_{(G,T)}(\alpha)$. Moreover,
\[
    \kappa_{E,\alpha}
    =\kappa_{F,\alpha}\circ \Nm_{E_{\pm\alpha}/F_{\pm\alpha}},
\]
and $[E_{\pm\alpha}:F_{\pm\alpha}]$ is odd, so the $f_{(G,T)}$-factors in \eqref{eqn:FKS-3} agree. The quadratic characters in \eqref{eqn:FKS-4} agree for the same reason, while the factors in \eqref{eqn:FKS-5} agree because $e(\alpha/\alpha_M)$ is unchanged upon passage from $F$ to $E$.
Since $\ell$ is odd, the exponents occurring in \eqref{eqn:FKS-4} have the same parity in both cases. Since the elements $\alpha$ and $\beta_i$ take values in $\{\pm 1\}$ on $\ol T(\F_E)$, by definition of $\epsilon_{x,\alpha}^{G/M}$ it suffices to show that for each $\ol t \in \ol T(\F_E)$ we have
\[
\alpha(\Nm_{\F_E/\F}(\ol t)) = \begin{cases}
    \alpha(\ol t) &\text{in case (1),} \\
    \prod_{i=0}^{\ell-1} \beta_i(\ol t) &\text{in case (2).}
\end{cases}
\]
In both cases, one shows this via exactly the same calculations as before, using that $E/F$ is of odd degree.
\end{proof}

\subsection{Proof of Proposition~\ref{prop:ramified-kal-bc}}\label{app:ramified-bc-proof}

\begin{proof}
We will assume that $E/F$ is totally ramified and $\ell$ is odd. We may and do assume that $G$ is of adjoint type. By \cite[Lemma 6.4]{Kal21a}, we may pass to an inner twist of $G$ to assume that $G$ and each of the $G^i$ are quasi-split: indeed, $G^0(F)_{[x]}/G^0(F)_{x,0+}$ is unaffected by such twisting since $G^0 = T$ is a torus. We will only prove the claim when $T$ is elliptic, in which case $T(F) = T(F)_{[x]}$; the general case is obtained by replacing all cocycles in the argument below by their restrictions to $I_F$. Write $\ld\theta(w)=(\wh\theta(w),w)$ for a $1$-cocycle $\wh\theta\co W_F\to\wh T(k)$. By definition \eqref{eqn:def-of-L-embedding}, our claim reduces to the following equality of cohomology classes:
\begin{equation}\label{eqn:signs-cancel-0}
    [Q_{\chi''_F}Q_{\chi''_E}^{-1}] = [\wh{\epsilon_F\epsilon_{F,\sharp,x}}][\wh{\epsilon_E\epsilon_{E,\sharp,x}}] \quad\text{in }\rH^1(W_E, \wh T(k)).
\end{equation}
Indeed, since $[E:F]$ is odd, the character $s$ of \S\ref{sss:quasi-split-l-embedding} does not change upon passage from $F$ to $E$. Let $\Sigma_F = \Gal(\ol F/F) \times \{\pm 1\}$ and similarly for $\Sigma_E$. By definition of $Q_{\chi''_F}$ and $Q_{\chi''_E}$, it suffices to show that for each $\alpha \in \Phi((G^{j+1}/G^j)_{\ol F}, T_{\ol F})$ one has the following equality of cohomology classes:
\begin{equation}\label{eqn:signs-cancel-local-0}
    \left[Q_{\chi''_{F,\alpha}} \cdot \prod_{\beta \in \Sigma_F \cdot \alpha/\Sigma_E} Q_{\chi''_{E,\beta}}^{-1}\right] = \left[\wh{\epsilon_{x,\alpha}^{G^{j+1}/G^j}\epsilon_{\sharp,x,\alpha}^{G^{j+1}/G^j}}\right]\cdot \prod_{\beta \in \Sigma_F \cdot \alpha/\Sigma_E}\left[\wh{\epsilon_{x,\beta}^{G^{j+1}_E/G^j_E}\epsilon_{\sharp,x,\beta}^{G^{j+1}_E/G^j_E}}\right]
\end{equation}
in $\rH^1(W_E, \wh T(k))$.

Fix a root $\alpha \in \Phi((G^{j+1}/G^j)_{\ol F}, T_{\ol F})$, and let $\alpha_j = \alpha|_{Z^j_{\ol F}}$. Since $\ell$ is odd, the root $\alpha$ is asymmetric (resp.\ symmetric unramified, resp.\ symmetric ramified) for the action of $\Gal(\ol F/F)$ if and only if it is the same for the action of $\Gal(\ol F/E)$. The same equivalence holds for $\alpha_j$. By precisely the same argument as in the proof of Lemma~\ref{lemma:fks-base-change}, we have
\begin{equation}\label{eqn:norm-relation-asym-sym-unram}
(\epsilon_{x,\alpha}^{G^{j+1}/G^j}\epsilon_{\sharp,x,\alpha}^{G^{j+1}/G^j}) \circ \Nm_{E/F} = \prod_{\beta \in \Sigma_F\cdot\alpha/\Sigma_E} \epsilon_{x,\beta}^{G^{j+1}_E/G^j_E}\epsilon_{\sharp,x,\beta}^{G^{j+1}_E/G^j_E}
\end{equation}
whenever $\alpha$ is asymmetric or symmetric unramified. Moreover, if $\alpha$ is symmetric ramified, $\ol t \in \ol T(\F_E)$, and $m$ is the order of $\Sigma_F \cdot \alpha/\Sigma_E$, then the formulas \eqref{eqn:FKS-3}--\eqref{eqn:FKS-5} give
\begin{equation}\label{eqn:norm-relation-sym-ram}
\begin{aligned}
&((\epsilon_{x,\alpha}^{G^{j+1}/G^j} \circ \Nm_{E/F}) \cdot \prod_{\beta \in \Sigma_F \cdot \alpha/\Sigma_E} \epsilon_{x,\beta}^{G^{j+1}_E/G^j_E})(\ol t) \\
&\qquad= \prod_{\substack{\beta \in \Sigma_F \cdot \alpha/\Sigma_E \\ \beta(\ol t)=-1}}
\sgn_{\F_\beta^\times}\left(e_{\beta,F}e_{\beta,E}(-1)^{\frac{e(\beta/\beta_j)_E-e(\beta/\beta_j)_F}{2}}\right).
\end{aligned}
\end{equation}
Here $e_{\alpha,F}$ (resp.\ $e_{\alpha,E}$) is the ramification degree of $F_\alpha/F$ (resp.\ $E_\alpha/E$), and similarly for $e(\alpha/\alpha_j)_F$ and $e(\alpha/\alpha_j)_E$. If $\alpha_j$ is asymmetric or symmetric unramified, then $\chi''_{F,\alpha_j} \circ \Nm_{E_{\alpha_j}/F_{\alpha_j}} = \chi''_{E,\alpha_j}$ by definition since $[E_{\alpha_j}:F_{\alpha_j}]$ is odd, and thus 
\begin{equation}\label{eqn:asym-sym-unram-chi-data-bc}
    \chi''_{F,\alpha} \circ \Nm_{E_\alpha/F_\alpha} = \chi''_{E,\alpha}.
\end{equation}
Note that the element $a_\alpha$ does not change upon passage from $F$ to $E$, so if $\alpha_j$ is symmetric ramified then by \eqref{eqn:chi-double-prime-G-root-value} we have
\begin{equation}\label{eqn:symmetric-ramified-chi-data-bc}
\begin{aligned}
\chi''_{F,\alpha}(\ell_{p'}(\alpha^\vee)a_\alpha) &= (-1)^{f_{F,\alpha} + 1}\frG_{\F_\alpha}(\psi^0)\sgn_{\F_\alpha^\times}(-1)^{\frac{e(\alpha/\alpha_j)_F-1}{2}}, \\
\chi''_{E,\alpha}(\ell_{p'}(\alpha^\vee)a_\alpha) &= (-1)^{f_{E,\alpha} + 1}\frG_{\F_{E,\alpha}}(\psi_E^0)\sgn_{\F_{E,\alpha}^\times}(-1)^{\frac{e(\alpha/\alpha_j)_E-1}{2}}.
\end{aligned}
\end{equation}

There are now six cases, namely cases (1)--(5) in the proof of Lemma~\ref{lemma:comparison-to-minimal-ramified-chi-data}, and an additional case that $\alpha$ (and hence also $\alpha_j$) is symmetric ramified.
In cases (1), (2), and (4), the displayed equations \eqref{eqn:norm-relation-asym-sym-unram} and \eqref{eqn:asym-sym-unram-chi-data-bc} immediately imply \eqref{eqn:signs-cancel-local-0} by definition \eqref{eqn:def-of-Q}. Cases (3) and (5) proceed in almost precisely the same way as in the proof of Lemma~\ref{lemma:comparison-to-minimal-ramified-chi-data}, so we omit them.

It remains to understand the case that $\alpha$ (and hence also $\alpha_j$) is symmetric ramified. Choosing the representatives in \eqref{eqn:def-of-Q} compatibly with the $\Sigma_E$-orbits, Lemma~\ref{lemma:dual-l-homs} and norm functoriality identify the left side of \eqref{eqn:signs-cancel-local-0} with the class dual to $\prod_{\beta\in\Sigma_F\cdot\alpha/\Sigma_E}(\xi_\beta^1\circ\beta)$, where $\xi_\beta^1\co E_\beta^1\to k^\times$ is induced by the unramified character
\begin{equation}\label{eqn:char-dual-to-q-bc}
(\chi''_{F,\beta}\circ\Nm_{E_\beta/F_\beta})\cdot(\chi''_{E,\beta})^{-1}\co E_\beta^\times\to k^\times
\end{equation}
through $c\mapsto c v_1(c)^{-1}$. The character $\xi_\beta^1$ is trivial on $E_\beta^1\cap(1+\frp_{E,\beta})$. Since $\ell_{p'}(\beta^\vee)a_\beta$ maps to $-1$, equation~\eqref{eqn:norm-relation-sym-ram} reduces the claim to the following identity for each orbit, written for $\beta=\alpha$:
\begin{equation}\label{eqn:eps-vs-chi-bc}
\sgn_{\F_\alpha^\times}\bigl(e_{\alpha,F}e_{\alpha,E}(-1)^{\frac{e(\alpha/\alpha_j)_E-e(\alpha/\alpha_j)_F}{2}}\bigr)
=\bigl((\chi''_{F,\alpha}\circ\Nm_{E_\alpha/F_\alpha})\cdot(\chi''_{E,\alpha})^{-1}\bigr)(\ell_{p'}(\alpha^\vee)a_\alpha).
\end{equation}

First, we compute the right side of \eqref{eqn:eps-vs-chi-bc}. Since $\ell$ is odd we have $(-1)^{f_{F,\alpha}+1} = (-1)^{f_{E,\alpha}+1}$. We compute
\[
\frG_{\F_{E,\alpha}}(\psi_E^0) = \begin{cases}
    \sgn_{\F_\alpha^\times}(\ell)\frG_{\F_\alpha}(\psi^0) &\text{if } \F_{E,\alpha} = \F_\alpha, \\
    \sgn_{\F_\alpha^\times}(\ell) \frG_{\F_\alpha}(\psi^0)^\ell &\text{if } \F_{E,\alpha} \neq \F_\alpha,
\end{cases}
\]
where the first case comes from the fact that $\psi^0_E(x) = \psi^0(x)^\ell$ by definition, and the second case comes from the Hasse--Davenport relation. Since $[E_\alpha:F_\alpha]=\ell/m$, in the case $\F_{E,\alpha} = \F_\alpha$ we have
\begin{equation}
\begin{aligned}
&\bigl((\chi''_{F,\alpha} \circ \Nm_{E_\alpha/F_\alpha}) \cdot (\chi''_{E,\alpha})^{-1}\bigr)(\ell_{p'}(\alpha^\vee)a_\alpha) = \\
    &\qquad= \frG_{\F_\alpha}(\psi^0)^{\ell/m} \sgn_{\F_\alpha^\times}(-1)^{\frac{e(\alpha/\alpha_j)_F-1}{2}}\frG_{\F_{E,\alpha}}(\psi^0_E)^{-1} \sgn_{\F_{E,\alpha}^\times}(-1)^{\frac{e(\alpha/\alpha_j)_E-1}{2}} \\
    &\qquad= \sgn_{\F_\alpha^\times}\bigl(\ell \cdot (-1)^{\frac{e(\alpha/\alpha_j)_E - e(\alpha/\alpha_j)_F + \ell/m - 1}{2}}\bigr)\label{eqn:chi-calc-ram-bc}
\end{aligned}
\end{equation}
since $\frG_{\F_\alpha}(\psi^0)^2 = \sgn_{\F_\alpha^\times}(-1)$. If instead  $\F_{E,\alpha} \neq \F_\alpha$, then $m = 1$ and thus similarly
\begin{equation}
\begin{aligned}
&\bigl((\chi''_{F,\alpha} \circ \Nm_{E_\alpha/F_\alpha}) \cdot (\chi''_{E,\alpha})^{-1}\bigr)(\ell_{p'}(\alpha^\vee)a_\alpha) = \\
    &\qquad= \frG_{\F_\alpha}(\psi^0)^\ell \sgn_{\F_\alpha^\times}(-1)^{\frac{e(\alpha/\alpha_j)_F-1}{2}}\frG_{\F_{E,\alpha}}(\psi^0_E)^{-1} \sgn_{\F_{E,\alpha}^\times}(-1)^{\frac{e(\alpha/\alpha_j)_E-1}{2}} \\
    &\qquad= \sgn_{\F_\alpha^\times}(\ell \cdot (-1)^{\frac{e(\alpha/\alpha_j)_E - e(\alpha/\alpha_j)_F }{2}})\label{eqn:chi-calc-unram-bc}
\end{aligned}
\end{equation}
Since $E/F$ is a \emph{Galois} extension of degree $\ell$, it follows that $|\F_\alpha| \equiv 1 \pmod{\ell}$ and thus 
\begin{equation}\label{eqn:quadratic-reciprocity-identity}
    \sgn_{\F_\alpha^\times}(\ell \cdot (-1)^{\frac{\ell-1}{2}}) = 1
\end{equation}
by quadratic reciprocity.\footnote{This can be proven by splitting into cases according to the classes of $|\F_\alpha|$ and $\ell$ modulo $4$.} Thus by \eqref{eqn:chi-calc-ram-bc} and \eqref{eqn:chi-calc-unram-bc}, we have
\begin{equation}\label{eqn:chi-calculation}
\begin{aligned}
    &\bigl((\chi''_{F,\alpha} \circ \Nm_{E_\alpha/F_\alpha}) \cdot (\chi''_{E,\alpha})^{-1}\bigr)(\ell_{p'}(\alpha^\vee)a_\alpha) = \\
    &\qquad= \begin{cases}
        \sgn_{\F_\alpha^\times}(-1)^{\frac{e(\alpha/\alpha_j)_E-e(\alpha/\alpha_j)_F}{2}} &\text{if } E_\alpha/F_\alpha \text{ is ramified}, \\
        \sgn_{\F_\alpha^\times}(\ell \cdot (-1)^{\frac{e(\alpha/\alpha_j)_E-e(\alpha/\alpha_j)_F}{2}}) &\text{otherwise.}
    \end{cases}
\end{aligned}
\end{equation}

Next, we compute the left side of \eqref{eqn:eps-vs-chi-bc}. There are three cases:
\begin{enumerate}[label=(\Alph*)]
    \item $E \subset F_\alpha$,
    \item $E_\alpha/F_\alpha$ is unramified of degree $\ell$,
    \item $E_\alpha/F_\alpha$ is ramified of degree $\ell$.
\end{enumerate}

In case (A), we have $e_{\alpha,F}=\ell e_{\alpha,E}$ and $m=\ell$. Thus the left side of \eqref{eqn:eps-vs-chi-bc} is
\[
\sgn_{\F_\alpha^\times}(\ell\cdot(-1)^{\frac{e(\alpha/\alpha_j)_E-e(\alpha/\alpha_j)_F}{2}}),
\]
which agrees with \eqref{eqn:chi-calculation} since $E_\alpha=F_\alpha$. 

In case (B), we have $e_{\alpha,F} = \ell e_{\alpha,E}$ and $m = 1$. Thus the left side of \eqref{eqn:eps-vs-chi-bc} is 
\[
\sgn_{\F_\alpha^\times}(\ell \cdot (-1)^{\frac{e(\alpha/\alpha_j)_E-e(\alpha/\alpha_j)_F}{2}}),
\]
which again agrees with \eqref{eqn:chi-calculation} since $\F_{E,\alpha} \neq \F_\alpha$.

Finally, in case (C) we have $e_{\alpha,F} = e_{\alpha,E}$ and $m = 1$. Thus the left side of \eqref{eqn:eps-vs-chi-bc} is
\[
\sgn_{\F_\alpha^\times}(-1)^{\frac{e(\alpha/\alpha_j)_E - e(\alpha/\alpha_j)_F}{2}},
\]
which again agrees with \eqref{eqn:chi-calculation} since $E_\alpha/F_\alpha$ is ramified of degree $\ell$.
\end{proof}

\bibliographystyle{amsalpha-with-labels}
\bibliography{Bibliography}

\end{document}